\documentclass[a4paper, 11pt, extrafontsizes]{memoir}

\usepackage[ngerman, main=UKenglish]{babel}
\usepackage[utf8]{inputenc}
\usepackage[T1]{fontenc}
\usepackage[charter, cal=cmcal]{mathdesign}
\usepackage{amsmath}
\usepackage{amsthm}
\usepackage{enumitem}
\usepackage{flafter}
\usepackage{tikz}
\usetikzlibrary{babel, positioning, graphs, trees, backgrounds, fit, calc} % always load the babel library when using babel package
\usepackage[unicode=true]{hyperref} % should be loaded last
\usepackage{memhfixc} % make hyperref work with memoir; load after hyperref

\newcounter{LabelHelp} % for label-ref cross references; is never displayed
\settypeblocksize{*}{14,3cm}{1.618}
\setlrmargins{*}{*}{1.2}
\setulmargins{3.5cm}{*}{*}
\setheadfoot{\baselineskip}{3\baselineskip}
\setheaderspaces{*}{*}{0.618}
\checkandfixthelayout[nearest] % implement layout (height may change slightly)

\nonzeroparskip % sets  ''reasonable non-zero value'' for spacing between paragraphs
\makeevenhead{ruled}{\leftmark}{}{} % default uses \scshape \leftmark
\copypagestyle{NoSecruled}{ruled} % for chapters without sections: ruled style with chapter title on even and odd pages; not for ToC, which uses its own marks
\makeoddhead{NoSecruled}{}{}{\leftmark}

\addto{\captionsUKenglish}{}
\settocdepth{subsection} % adds subsections and higher to ToC
\addto{\captionsUKenglish}{}
\let\OLDthebibliography\thebibliography % help macro for bibliography
\renewcommand\thebibliography[1]{ % redefine bibliography with reduced space between items
     \OLDthebibliography{#1}
     \setlength{\parskip}{0pt}
}
\makeindex
\hypersetup{hidelinks}
\hypersetup{bookmarksnumbered}

\setsecnumdepth{subsection} % numbering of subsections and higher
\chapterstyle{ger}
\copypagestyle{chapter}{plain}% first page of chapter has chapter style; to change, redefine as copy of plain instead of alias
\makeoddfoot{chapter}{}{}{\thepage}
\makeheadrule{chapter}{\textwidth}{\normalrulethickness}
\namerefoff % disables name reference for chapters etc. -> more stable captions
\setfloatlocations{figure}{!ht} % floats never placed before their definition because of flafter-package
\captionnamefont{\bfseries}
\captiontitlefont{\itshape}
\hangcaption
\newsubfloat{figure} % define subfigures with subcaption, etc. 
\hangsubcaption

\epigraphposition{center}
\epigraphtextposition{flushleftright}
\newlist{tightitemize}{itemize}{2} % {name}{type}{max-depth}
\newlist{tightdescription}{description}{2}
\newlist{tightenumerate}{enumerate}{2}
\setlist[tightitemize]{label=\textbullet, nosep, leftmargin=*, align=left, itemindent=0pt, labelindent=1.5em, labelsep*=0.3em}
\setlist[tightdescription]{nosep, leftmargin=!, align=left, itemindent=0pt, labelindent=1.5em, labelsep*=0.3em, labelwidth=4em}
\setlist[tightenumerate, 1]{label=(\alph*), nosep, leftmargin=*, align=left, itemindent=0pt, labelindent=1.5em, labelsep*=0.3em} % optional argument = level 1 of tightenumerate

\theoremstyle{plain}
\newtheorem{thm}{Theorem}[chapter] % number as chapter.thm and reset thm counter for each chapter
\newtheorem{prop}[thm]{Proposition} % use thm counter

\newtheorem{cor}[thm]{Corollary}

\theoremstyle{definition}
\newtheorem{defi}[thm]{Definition}

\newtheorem{HelpRem}[thm]{Remark} % help environment to establishes heading of Remark that ends with diamond
\newtheorem{HelpExmpl}[thm]{Example}
\newenvironment{exmpl}{\renewcommand{\qedsymbol}{\( \blacklozenge \)}\pushQED{\qed}\begin{HelpExmpl}}{\popQED\end{HelpExmpl}}
\newenvironment{rem}{\renewcommand{\qedsymbol}{\( \blacklozenge \)}\pushQED{\qed}\begin{HelpRem}}{\popQED\end{HelpRem}}

\renewcommand{\qedsymbol}{$\blacksquare$} % use black qed-square

\renewcommand{\in}{\smallin} % use mathdesign font
\renewcommand{\notin}{\notsmallin} % use mathdesign font
\newcommand{\DisjCup}{\mathbin{\dot{\cup}}} % disjoint union, inline math
\DeclareMathOperator*{\DisjCupDisp}{\dot{\bigcup}} % disjoint union, display style ( * sets limits above and below)
\DeclareMathOperator{\GenLin}{GL} % general linear group
\DeclareMathOperator{\SpLin}{SL} % special linear group
\newcommand{\GL}{\GenLin(2 , \RR)}
\newcommand{\SL}{\SpLin(2 , \RR)}
\newcommand{\Id}{I\!d} % identity
\DeclareMathOperator{\Tr}{tr} % trace
\newcommand{\NN}{\mathbb{N}}
\newcommand{\ZZ}{\mathbb{Z}}
\newcommand{\QQ}{\mathbb{Q}}
\newcommand{\RR}{\mathbb{R}}
\newcommand{\CC}{\mathbb{C}}
\newcommand{\Set}[1]{\{ \, #1 \, \}} % inserts spaces between brackets and content
\newcommand{\Card}[1]{\# #1} % cardinality of a set #1
\newcommand{\Cmplmt}[1]{#1^{\mathsf{c}}} % complement of a set #1
\newcommand{\Close}[1]{\overline{#1}} % closure of a set #1
\newcommand{\CharFkt}[1]{\mathbb{1}_{#1}} % characteristic function of the set #1
\newcommand{\DefAs}{:=} % symbol ''defined as''
\newcommand{\AsDef}{=:} % symbol ''defined as''

\newcommand{\Alphab}{\mathscr{A}} % alphabet
\newcommand{\AlphabEv}{\widetilde{\Alphab}} % eventual alphabet
\newcommand{\AEv}{\widetilde{a}} % element of the eventual alphabet
\newcommand{\EventNr}{\widetilde{K}} % when eventual alphabet is reached
\newcommand{\Length}[1]{\lvert #1 \rvert} % length of the word #1
\newcommand{\Restr}[3]{#1|_{[#2 , \, #3]}} % restriction of word #1 to interval [#2 , #3 ] 
\newcommand{\Hole}{?} % hole in Toeplitz definition
\newcommand{\InfWord}{\omega} % (generic)  two-sided infinite word
\newcommand{\RWord}{p} % one-sided infinite word. Used for right half of leading sequences and for p-Blocks ( because of p^(\infty) )
\newcommand{\PBlock}[1]{\RWord^{(#1)}} % p-block
\newcommand{\SpW}[1]{\InfWord^{(#1)}} % distinguished words p^(\infty) | a p^(\infty), with '' -1 | 0 ''
\newcommand{\Origin}[1]{#1 \, |} % denotes origin in a word (  -1 | 0 ) 
\newcommand{\Cylin}[2]{C_{#2}(#1)} % cylinder set of word #1 starting at position #2 
\newcommand{\Langu}[1]{\mathscr{L}(#1)} % language of the subshift #1
\newcommand{\Rev}[1]{\overleftarrow{#1}} % reverse (word read backwards)
\makeatletter
\newcommand{\Word}[1]{#1 \@ifnextchar\bgroup{\LoadNextArg}{}} % display first argument. Checks if argument is followed by another \begingroup, i.e. by ''{''. If yes, call \LoadNextArg. If not (last argument, i.e. end of word), do nothing.
\newcommand{\LoadNextArg}[1]{\, #1 \@ifnextchar\bgroup{\LoadNextArg}{}} % insert space, then display next argument. If followed by another ''{'', repeat with next argument. Otherwise (last letter), do nothing.
\makeatother

\newcommand{\Subshift}{\Omega} % (simple Toeplitz) subshift
\newcommand{\Shift}{T} % shift map
\newcommand{\ErgodMeas}{\mu} % unique ergodic measure of the subshift
\newcommand{\MaxCylin}{\eta} % maximal measure of a cylinder set (for Boshernitzan condition)
\newcommand{\Rot}[1]{R_{#1}} % rotation map (rotation of unit circle by #1)
\newcommand{\LWord}[1]{\InfWord^{(#1)}} % leading sequence of (LSC) subshift
\newcommand{\LOrigin}[1]{v^{(#1)}} % finite word at origin of leading sequence, between \RWord-reveresed and \RWord
\newcommand{\LNr}{r} % number of leading sequences in the subshift

\newcommand{\Comp}{\mathcal{C}} % subword complexity
\newcommand{\Growth}{\mathcal{G}} % growth of complexity; G(n) = C(n)- C(n-1)
\newcommand{\Pali}{\mathcal{P}} % palindrome complexity 
\newcommand{\Repe}{\mathcal{R}} % repetitivity function
\newcommand{\AllL}{F} % smallest number such that all letters appear in a_{k+1}, ... a_{F( k )}
\newcommand{\AllLInc}[1]{m_{#1}} % positions where \AllL actually increases
\newcommand{\Debruijn}[1]{\mathcal{G}_{#1}} % #1-th de Bruijn-Graph
\newcommand{\Vertices}[1]{\mathcal{V}_{#1}} % vertices of  #1-th de Bruijn-Graph
\newcommand{\Edges}[1]{\mathcal{E}_{#1}} % edges of  #1-th de Bruijn-Graph
\newcommand{\Copies}[2]{\#_{#1}( #2 )} % number of copies of #1 in #2
\newcommand{\DisCopies}[2]{\widetilde{\#}_{#1}( #2 )} % number of disjoint copies of #1 in #2

\newcommand{\Jac}{H} % Jacobi operator
\newcommand{\Spec}{\Sigma} % spectrum of SO resp. JO
\newcommand{\AC}{\mathrm{ac}} % absolute continuous part of the spectrum
\newcommand{\SC}{\mathrm{sc}} % singular continuous part of the spectrum
\newcommand{\PP}{\mathrm{pp}} % pure point part of the spectrum
\newcommand{\Dig}{g} % function on diagonal in Jacobi operator 
\newcommand{\NDig}{f} % function above / below diagonal in Jacobi operator
\newcommand{\TrMat}{M_{E}} % transfer matrix with energy E
\newcommand{\TrMod}{\widetilde{M}_{E}} % modified transfer matrix (determinant = 1) with energy E 
\newcommand{\Cocyc}{A} % (generic) symbol for a cocycle
\newcommand{\Lyapu}[1]{\Lambda_{#1}} % Lyapunov exponent of subadditive function #1
\newcommand{\AsAv}[1]{\overline{#1}} % asymptotic average of a subadditive function ( \lim_{L -> \infty} \max F(u) / L  : u \in Language )

\newcommand{\GenSeq}{\alpha} % sequence that defines the generators of the group
\newcommand{\SchreiSp}[1]{\mathcal{S}_{#1}} % space of Schreier graphs without isolated points, associated to member G_{#1} of the family of Grigorchuk's groups 
\newcommand{\Lap}[2]{L_{#1}(#2)} % Laplacian on the Schreier graph #1 with edge weight function #2

\newcommand{\LCond}[1]{(LSC#1)} % leading word condition   !!! use as \LCond{} because of spacing !!!
\newcommand{\LConda}[1]{\hyperlink{LSCa}{\LCond{\(\ \alpha\)}}} % write ``(\LCond{} alpha)'' and link to hypertarget LSCa
\newcommand{\LCondb}[1]{\hyperlink{LSCb}{\LCond{\(\ \beta\)}}}
\newcommand{\LCondc}[1]{\hyperlink{LSCc}{\LCond{\(\ \gamma\)}}}
\newcommand{\LCondaa}[1]{\hyperlink{LSCaa}{\LCond{\(\ \alpha '\)}}}
\newcommand{\LCondbb}[1]{\hyperlink{LSCbb}{\LCond{\(\ \beta '\)}}}
\newcommand{\LCondcc}[1]{\hyperlink{LSCcc}{\LCond{\(\ \gamma '\)}}}

\title{Simple Toeplitz subshifts: combinatorial properties and uniformity of cocycles}
\author{Daniel Sell}
\date{30.08.2019}
\hypersetup{pdftitle=\thetitle, pdfauthor=\theauthor}

\begin{document}
\allowdisplaybreaks % allow page break in mathmode; prevent with \\*
\frontmatter % page numbers in lowercase roman numerals. no numbering of sectional divisions. float caption, equations, etc., are numbered continuously.

\thispagestyle{empty}
\begingroup
\centering
\vspace*{-0.2cm} % without *, space at beginning of page is ignored 

\HUGE%
\rule{\textwidth}{0.4ex}\\[0.5cm]%
\textsf{%
Simple Toeplitz subshifts:\\[1ex]%
combinatorial properties and\\[-0.5ex]%
uniformity of cocycles}\\%
\rule{\textwidth}{0.4ex}%

\vspace{3cm}

\huge
\textbf{%
Dissertation\\%
zur Erlangung des akademischen Grades%
}\\%
\textsl{doctor rerum naturalium (Dr. rer. nat.)}\\[2cm]%
\Large%
\textbf{%
vorgelegt dem Rat der Fakultät für Mathematik und Informatik\\%
der Friedrich-Schiller-Universität Jena\\[1.7cm]%
\huge%
von M. Sc. Daniel Sell\\[0.5cm]%
\Large%
geboren am 10.03.1991 in Quedlinburg%
}

\endgroup

\clearpage

\thispagestyle{empty}

\vspace*{\fill}

\begingroup
\Large%
\textbf{Gutachter:\\[\baselineskip]
%\rule{10cm}{\normalrulethickness}
1. \; Professor Dr. Daniel Lenz \textmd{(Friedrich-Schiller-Universität Jena)}\\[\baselineskip]
2. \; Professor Dr. Michael Baake \textmd{(Universität Bielefeld)}\\[\baselineskip]
3. \; Professor Dr. David Damanik \textmd{(Rice University, Houston)}\\[2\baselineskip]
Tag der öffentlichen Verteidigung: \,26.02.2020 
}

\vspace{3ex}
\endgroup

%%%%%%%%%%%%%%%%%%%%%%%%
%%%%%%%%%%%%%%%%%%%%%%%%
%%%%%%%%%%%%%%%%%%%%%%%%
\chapter{Abstract}

In this thesis we investigate the properties of so-called simple Toeplitz subshifts. They are generated by a subclass of Toeplitz words that is distinguished by the extraordinary simple structure of their approximating partial words. We describe their construction principle in detail and discuss various examples. Our study focuses on the combinatorial properties of the subshifts themself and on the spectral properties of Jacobi operators that are defined by the subshifts.

Regarding the combinatorial properties we treat the complexity, the de Bruijn graphs, the palindrome complexity,  the repetitivity and \( \alpha \)-repetitivity as well as the Boshernitzan condition. We derive explicit formulas for the complexity, the palindrome complexity and, for sufficiently large word length, for the repetitivity. In addition we give a complete description of the de Bruijn graphs. We characterise \( \alpha \)-repetitivity and, based on a result by Liu and Qu from 2011, the Boshernitzan condition for all simple Toeplitz subshifts.

We begin the treatment of our second topic by reviewing the notions of a Jacobi operator, its spectrum, the associated transfer matrices and Gordon-type arguments for excluding eigenvalues. Our first main result in this context is that Jacobi operators \( \Jac_{\InfWord} \) have empty pure point spectrum for almost all \( \InfWord \) in a simple Toeplitz subshift. This generalises a result of Grigorchuk, Lenz and Nagnibeda which dealt only with one particular simple Toeplitz subshift. As our second main result we show that the spectrum of a Jacobi operator on a simple Toeplitz subshift is always a Cantor set of Lebesgue measure zero. In fact we prove a stronger statement: we introduce a new condition for subshifts that we call the leading sequence condition, and then we show that every locally constant \(\SL\)-cocycle is uniform whenever our condition is satisfied. As a special case we infer uniformity of the (modified) transfer matrix cocycle. We show that simple Toeplitz subshifts satisfy the leading sequence condition. Combined with their aperiodicity, this yields the aforementioned Cantor spectrum. The notion of leading sequences and their usage in the study of spectral properties stem from a collaboration with Rostislav Grigorchuk, Daniel Lenz and Tatiana Nagnibeda.

On simple Toeplitz subshifts, Cantor spectrum was previously only known for Schrödinger operators (obtained by Liu/Qu 2011 through an analysis of the trace map). Our generalisation to Jacobi operators is particularly interesting since the latter are connected to Laplacians on Schreier graphs of self-similar groups. This is briefly reviewed in the appendix.

As an aside we prove that Sturmian subshifts satisfy the leading sequence condition, too. Our approach has therefore the additional advantage that it establishes uniformity of cocycles for two disjoint classes of subshifts (namely simple Toeplitz and Sturmian) in a unified way.

%%%%%%%%%%%%%%%%%%%%%%%%
%%%%%%%%%%%%%%%%%%%%%%%%
%%%%%%%%%%%%%%%%%%%%%%%%
\chapter{Zusammenfassung}
\begin{otherlanguage}{ngerman}
Die vorliegende Dissertation behandelt die Eigenschaften sogenannter \foreignlanguage{UKenglish}{simple Toeplitz subshifts}. Diese werden von Toeplitz-Wörtern erzeugt, die sich durch eine besonders einfach Periodenstruktur auszeichnen. Die Arbeit behandelt vor allem die kombinatorischen Eigenschaften dieser \foreignlanguage{UKenglish}{subshifts} sowie die spektralen Eigenschaften von Jacobi-Operatoren, die von den \foreignlanguage{UKenglish}{subshifts} definiert werden.

Bei den kombinatorischen Eigenschaften konzentrieren wir uns auf die Komplexität, die de-Bruijn-Graphen, die Palindromkomplexität, die Repetitivität und \( \alpha \)-Repetitivität sowie die Boshernitzan-Bedingung. Wir leiten explizite Formeln für die Komplexität, die Palindromkomplexität und, bei genügend großer Wortlänge, für die Repetitivität her. Weiterhin geben wir eine detailliere Beschreibung der de-Bruijn-Graphen an. Wir charakterisieren sowohl \( \alpha \)-Repetitivität als auch, basierend auf einem Resultat von Liu und Qu, die Boshernitzan-Bedingung für \foreignlanguage{UKenglish}{simple Toeplitz subshifts}.

Die Behandlung des zweiten Themenbereiches beginnen wir mit einer Wiederholung der Begriffe des Jacobi-Operators, seines Spektrums und der zugehörigen Transfermatrizen, sowie einer auf Gordon zurückgehenden Methode zum Ausschließen von Eigenwerten. Wir beweisen anschließend, dass Jacobi-Operatoren \( \Jac_{\InfWord} \) für fast alle Elemente \( \InfWord \) aus einem \foreignlanguage{UKenglish}{simple Toeplitz subshift} leeres Punktspektrum haben. Dies verallgemeinert ein Resultat von Grigorchuk, Lenz und Nagnibeda, welches nur einen einzelnen, speziellen \foreignlanguage{UKenglish}{simple Toeplitz subshift} behandelte. Das zweite wichtige Ergebnis dieses Themenbereiches ist der Nachweis, dass das Spektrum von Jacobi-Operatoren auf \foreignlanguage{UKenglish}{simple Toeplitz subshifts} stets eine Cantor-Menge vom Lebesgue-Maß Null ist. Tatsächlich zeigen wir eine noch stärkere Aussage: Wir führen das neue Konzept der \foreignlanguage{UKenglish}{leading-sequence}-Bedingung für \foreignlanguage{UKenglish}{subshifts} ein und weisen nach, dass jeder lokalkonstante \( \SL \)-Kozyklus gleichmäßig ist, wenn der zugrundeliegende \foreignlanguage{UKenglish}{subshift} diese Bedingung erfüllt. Insbesondere ist in diesem Fall der Kozyklus der (modifizierten) Transfermatrizen gleichmäßig. Wir zeigen, dass \foreignlanguage{UKenglish}{simple Toeplitz subshifts} die \foreignlanguage{UKenglish}{leading-sequence}-Bedingung erfüllen. Aufgrund ihrer Aperiodizität erhalten wir dann die oben erwähnte Aussage zum Cantor-Spektrum. Der Begriff der \foreignlanguage{UKenglish}{leading sequence} und seine Anwendung in der Untersuchung von spektralen Eigenschaften sind im Rahmen einer Zusammenarbeit mit Rostislav Grigorchuk, Daniel Lenz und Tatiana Nagnibeda entstanden.

Für \foreignlanguage{UKenglish}{simple Toeplitz subshifts} war Cantor-Spektrum bisher nur für Schrödinger-Operatoren nachgewiesen (was Liu und Qu 2011 durch eine Analyse der \foreignlanguage{UKenglish}{trace map} gelang). Die hier vorgestellte Verallgemeinerung auf den Fall von Jacobi-Operatoren ist vor allem dadurch von Interesse, dass diese eng mit Laplace-Operatoren auf den Schreier-Graphen selbstähnlicher Gruppen zusammenhängen. Dies wird im Anhang kurz erläutert.

Des Weiteren zeigen wir, dass auch Sturmsche \foreignlanguage{UKenglish}{subshifts} die \foreignlanguage{UKenglish}{leading-sequence}-Bedingung erfüllen. Unser Vorgehen hat daher den zusätzlichen Vorteil, dass die Gleichmäßigkeit der Kozyklen für zwei disjunkte Klassen von \foreignlanguage{UKenglish}{subshifts} (nämlich \foreignlanguage{UKenglish}{simple Toeplitz} und Sturmsche) auf eine einheitliche Weise nachgewiesen werden kann.
\end{otherlanguage}

%%%%%%%%%%%%%%%%%%%%%%%%
%%%%%%%%%%%%%%%%%%%%%%%%
%%%%%%%%%%%%%%%%%%%%%%%%
\chapter{Thanks and acknowledgement}

First and foremost I want to thank my thesis advisor Daniel Lenz, not only for all the knowledge about aperiodic systems that he shared with me, but also for his general guidance and assistance  in matters of academia. I am grateful for the flexibility that he gives his PhD-students, which allowed me to strengthen my ability to work independently. On the other hand I could rely on his help whenever it was needed. He always encouraged me to attend workshops or conferences and to share my ideas with others. 

I also received ample support from outside Jena. I thank Michael Baake for several invitations to workshops at the University of Bielefeld and for his advice and interest in my work. Moreover I am glad that I had the chance to spend two months of my PhD-studies in the USA. I sincerely thank David Damanik, Anton Gorodetski and everyone in their respective groups in Houston and Irvine for their time, their organisational help and their hospitality. It was a highly instructive and rewarding time. In addition I thank Tatiana Nagnibeda for her invitation to the University of Geneva, where I also met Rostislav Grigorchuk. I am grateful for the opportunity to collaborate with both of them as well as with Daniel Lenz. The results of our collaboration form a part of this thesis.

There are several other people who helped me to improve my understanding of the topics of this thesis. I especially want to mention Aitor Perez (University of Geneva), for patiently answering numerous questions about self-similar groups and Schreier graphs, as well as Siegfried Beckus and Franziska Sieron, for enlightening explanations and fruitful discussions on everything from de Bruijn graphs and trace maps to leading sequences. I also had some inspiring conversations with Till Hauser. All of them are, or used to be, at the University of Jena.

Regarding the thesis itself I kindly thank Benny, Daniel, Maike, Melchior, Ruth and Sofia. They all read parts of this text and provided valuable feedback. Their suggestions certainly made the exposition clearer and more comprehensible.

In addition to Friedrich Schiller University Jena the following institutions helped me to finance various stages and projects of my PhD. Their financial support is gratefully acknowledged:
\begin{tightitemize}
\item{the German Research Foundation (DFG), through the Research Training Group 1523 ``Quantum and Gravitational Fields'', in which I was a member from October 2015 to March 2016 and an associated member from April 2016 until the project's end in March 2018,}
\item{the federal state of Thuringia, through the Thuringian state scholarship (\foreignlanguage{ngerman}{Landesgraduiertenstipendium}) from April 2016 until March 2019,}
\item{the German Academic Exchange Service (DAAD), which supported my research visit to the USA with both travelling allowance and a temporary heightening of the aforementioned scholarship from September to November 2017.}
\end{tightitemize}
In addition I highly appreciate the financial assistance by the organisers of various conferences and workshops that I attended during the my PhD-studies.

I very much enjoyed my time in Jena and the pleasant working environment had a large part in this. For this I  want to thank everyone in the Institute of Mathematics and in particular my fellow PhD-students from the analysis groups. I also want to mention the secretaries, without whom many things would not have been possible, and our PhD Seminar which contributed a lot to the cordial atmosphere. Thanks to Siggi and Therese, who started the seminar, as well as to my co-organisers Maike and Benny for their commitment.

Last but not least I own a great deal of thanks to my friends and relatives. Many things have changed during the last four years and I am glad for our continued closeness and friendship, even though we by now all live in different places. Thus I want to conclude with a special thanks to my parents, grandparents and the whole rest of the Ge-sell-schaft, as well as to Julia, Ralf and Ruth.

%%%%%%%%%%%%%%%%%%%%%%%%
%%%%%%%%%%%%%%%%%%%%%%%%
%%%%%%%%%%%%%%%%%%%%%%%%
\cleardoublepage
\tableofcontents

%%%%%%%%%%%%%%%%%%%%%%%%
%%%%%%%%%%%%%%%%%%%%%%%%
%%%%%%%%%%%%%%%%%%%%%%%%
\cleardoublepage
\thispagestyle{empty}

\vspace*{6\baselineskip}

\begin{otherlanguage}{ngerman}
\epigraph{%
"`... dort eine andere, im grotesken Geschmack eingerichtete Grotte, wo zierliche Muscheln und weiße gewundene Schneckenhäuser, in geordneter Unordnung gefaßt, mit Stücken glänzenden Kristalls vermischt sind"'
}{%
\textsc{Miguel de Cervantes Saavedra}: \textit{Leben und Taten des scharfsinnigen Edlen Don Quixote von la Mancha}, Fünftes Buch, Neuntes Kapitel
}
\end{otherlanguage}

\clearpage
\thispagestyle{empty}

\vspace*{6\baselineskip}

\begin{otherlanguage}{ngerman}
\epigraph{%
"`Vielleicht werden im Laufe solcher Untersuchungen heute noch für uns grundlegende Begriffe und unumstößliche Gegensätze wie etwa der Gegensatz zwischen vollständiger Zufälligkeit und determinierter Kausalität brüchig und hinfällig werden. Vielleicht werden wir auf diesem Wege zu einer Ebene von Erscheinungen vordringen, auf der jene Begriffe durch völlig andere ersetzt werden müssen, welche die genannte gegenseitige Ausschließung (von Ordnung und Zufall) nicht zulassen."'
}{%
\textsc{Stanisław Lem}: \textit{Essays}, Teil III: Über außersinnliche Wahrnehmung (1974)%
}
\end{otherlanguage}

%%%%%%%%%%%%%%%%%%%%%%%%
%%%%%%%%%%%%%%%%%%%%%%%%
%%%%%%%%%%%%%%%%%%%%%%%%
\mainmatter % page numbers in arabic numerals. starts page numbering from 1. sections and above are numbered. float captions, equations, etc., are numbered per chapter

\chapter{Introduction}

The main topic of this thesis are simple Toeplitz subshifts. These are sets of sequences of symbols, with the important property that the sequences are not periodic but still highly structured. Consequently their study belongs to the field of aperiodic order. From a physics point of view this also relates them to so-called quasicrystals. In addition simple Toeplitz subshifts, as well as certain operators on them that we will consider, appear naturally in the context of self-similar groups.

\section{Organisation of this thesis and remarks on notation}

The aim of this introduction is to provide a brief overview of the aforementioned topics. We treat them roughly in the same order in which they will appear in the subsequent chapters. We give some motivating background information and mention our main results. Necessarily the details and precise definitions have to be skipped here, so the reader might want to reread the introductory section along with the corresponding main chapter to relate the precise mathematical notions to the rough outline.

Throughout this thesis we use a black square~\qedsymbol {} to denote the end of a proof and a black diamond~\( \blacklozenge \) to denote the end of an example or a remark. For the natural numbers we write \( \NN = \Set{ 1, 2, 3, \ldots } \) and \( \NN_{0} = \NN \cup \Set{0} \), respectively. We use \( \Card{A} \) to denote the \emph{cardinality} of a set \( A \), while \( \Cmplmt{A} \) denotes its \emph{complement} and \( \Close{A} \) its \emph{closure}. To emphasise that a union of two sets \( A \) and \( B \) is disjoint, we write \( A \, \DisjCup \, B \). For the most important mathematical notions, the page that contains their definition can be found in the index at the end of this thesis (page \pageref{chap:Index}).

\section{Symbolic dynamics and aperiodic order}

In general, dynamical systems describe the time evolution of points in a space. Consider for example a compact metric space \( X \) together with a homeomorphism \( f \) from \( X \) to itself (compare \cite[Definition~6.2.1]{LindMarcus_Coding}). The aim is to study various properties of the orbit \( \Set{ f^{k}( x ) : k \in \ZZ } \) which a point \( x \in X \) traces when the homeomorphism is repeatedly applied to it. However, many interesting phenomena can already be observed when we simplify the setting and partition \( X \) into finitely many regions, each of which is assigned a unique symbol. Let \( \Alphab \DefAs \Set{ a_{1}, \ldots, a_{N} } \) denote the set of all these symbols. Instead of the precise orbit of a point we only record the symbol of the region for each iteration. In this way a two-sided infinite ``string of symbols'' is generated. The study of the properties of such strings is the contents of symbolic dynamics. Usually the symbols are referred to as \emph{letters} and strings of symbols are called \emph{words}. If a collection \( \Subshift \subseteq \Alphab^{ \ZZ } \) of infinite words satisfies some basic constrains, it is called a \emph{subshift}.

Another field were subshifts play an important role is aperiodic order. To illustrate the phenomena, we first consider the two-dimensional case (we mention the three-dimensional analogues briefly in Section~\ref{sec:IntroQC}). The only rotation symmetries that can occur in a periodic tiling of \( \RR^{2} \) are twofold, threefold, fourfold or sixfold. This is known as crystallographic restriction. Clearly it is also possible to ``tile'' the plane in a completely arbitrary, non-periodic way. In his famous article from 1974 Penrose presented a ``very ordered'' tiling of \( \RR^{2} \) which uses only four different shapes, but nevertheless exhibits a fivefold symmetry (\cite{Penr_AesthInMath}). Hence it is non-periodic, but Penrose proved that it extends to the whole plane. Since a one-dimensional tiling is just a two-sided infinite ``list'' of tiles, the representation of each tile-type by a letter connects this special case to words and subshifts.

In the years following the discovery of Penrose, aperiodic order became the object of thorough mathematical analysis (for example \cite{deBr_AlgPenr1u2}, \cite{Beenk_AlgTiling}, \cite{KramNeri_NonPerTil} and many more). Also earlier works about almost periodic functions (treated for instance in \cite{Bohr_apFunc}) were reconsidered in the setting of subshifts. This concerns for instance a construction by Toeplitz (\cite{Toepl_FastperFkt}) which was originally developed for real-valued functions, but was adapted by Jacobs and Keane to the symbolic case (\cite{JacobsKeane_01Toeplitz}). These so-called Toeplitz words are the central objects of this thesis and will be presented in detail in the next section. However, the construction by Toeplitz, Jacobs and Keane is just one particular way to generate words and subshifts which combine non-periodicity with a certain amount of order. Other ways include so-called cut-and-project schemes or purely symbolic means such as substitutions. \textbf{In the first section of Chapter~\ref{chap:STSubshifts}} various notions regarding words and subshifts are made rigorous, including their definition from one-sided infinite or two-sided infinite words.

The question of how ordered an aperiodic system is, already present in the earliest articles in the field such as the seminal work by Morse and Hedlund \cite{MorseHedl_SymbDyn}, is still a central topic today. The degree of order can for example be measured by counting the number of different words of length \( L \in \NN \) in a subshift. This concept is known as \emph{complexity of the subshift}. A low complexity indicates that the system is very ordered (or, in a way, ``predictable''). For example is the number of words in a periodic element bounded by the length of the period. Conversely, if the letter at every position is chosen at random, then we should expect that all \( ( \Card{ \Alphab } )^{L} \) possible words of length \( L\) occur, which yields a very high complexity. In between these extremes there are the subshifts that are associated with aperiodic order. By the famous Morse/Hedlund theorem (\cite[Theorem~7.4]{MorseHedl_SymbDyn}, see also Proposition~\ref{prop:MHThm}) the complexity of every non-periodic infinite word grows at least like \( L+1 \). In \cite{MorseHedl_Sturmian}, Morse and Hedlund introduced the term \emph{Sturmian word} for an infinite word whose complexity is equal to \( L+1 \). In the light of the Morse/Hedlund theorem, Sturmian words can be considered as the most ordered among the non-periodic words. That makes them a particular well-studied and relatively well-understood class of examples. In this thesis, we consider different, also highly ordered subshifts which are known as simple Toeplitz subshifts. Often, our results or methods have well-established counterparts for the Sturmian case. In various places we will highlight the similarities and differences between these two types of subshifts. \textbf{In Appendix~\ref{app:Sturm}} we therefore give a brief overview over Sturmian subshifts, focused on those results that are related to this thesis.

\section{Simple Toeplitz subshifts}
\label{sec:IntroToepl}

Jacobs and Keane introduced Toeplitz words in \cite{JacobsKeane_01Toeplitz} as a particular way to generate two-sided infinite words, which in turn define subshifts. Toeplitz words are constructed via so-called ``partial words'', that is, periodic words with some undetermined positions (``holes''). The holes are then filled with other partial words. If no undetermined part remains in the limit word \( \InfWord \in \Alphab^{\ZZ} \), then every position \( \InfWord( j ) \) is repeated periodically, but the period length depends on \( j \):
\[ \forall j \in \ZZ \quad \exists p \in \NN \quad \forall m \in \ZZ \; : \;\; \InfWord( j ) = \InfWord( j + mp) \, . \]
The concept of Toeplitz words has been studied in numerous versions, for example as words on \( \NN \), \( \ZZ \) or \( \ZZ^{d} \), with either two or arbitrarily many letters (see for instance \cite{KamZamb_MaxPattCompl}, \cite{GKBYM_MaxPatternToepl}, \cite{QRWX_PatCompDDim}). Often Toeplitz words serve as (counter-)examples, for instance as minimal systems with several ergodic measures (\cite{Ox_ErgSets}, \cite{Wil_ToepNotUniqErgod}), as systems with (almost) arbitrarily prescribed dynamical pure point spectrum (\cite{DownLacr_NonRegToepl}) or, after a translation to cut-and-project schemes, for irregular model sets with zero entropy (\cite{BJL_ToeplModelSet}). In this thesis we consider so-called simple Toeplitz words in the definition of Liu and Qu (\cite{LiuQu_Simple}): for a finite alphabet \( \Alphab \) we construct a word \( \InfWord \in \Alphab^{\ZZ} \) by the hole-filling procedure described above, but every partial words consists of the repetition of a single letter and there is exactly one hole in every partial word. \textbf{In the second section of Chapter~\ref{chap:STSubshifts}}, we make this more precise and give two (equivalent) mathematical definitions of simple Toeplitz subshifts. Moreover a number of examples and some basic properties like regularity, minimality and ergodicity are discussed.

The construction of Toeplitz words from periodic partial words suggests that they are highly ordered. \textbf{In Chapter~\ref{chap:Combinat}} we study combinatorial properties of simple Toeplitz subshifts: the already mentioned complexity (which measures the number of words of a given length), palindrome complexity (the same for palindromes of a given length), repetitivity (which measures the maximal gap length between consecutive occurrences of words of a given length) and the Boshernitzan condition (which gives a bound for the average gap length). We prove that the complexity of an aperiodic simple Toeplitz subshift grows for all sufficiently large \( L \) alternately like \( \Card{ \AlphabEv } \) and \( \Card{ \AlphabEv } - 1 \) (see Corollary~\ref{cor:CompLargeL}), where \( \AlphabEv \) denotes the so-called eventual alphabet and has cardinality of at least two. Consequently the class of simple Toeplitz subshifts is disjoint from Sturmian subshifts. However, the fact that the complexity of simple Toeplitz subshifts is bounded by a linear function indicates that they are (in terms of complexity) only marginally less ordered than Sturmian systems. Similar to the complexity, also the other aforementioned combinatorial quantities measure different aspects of how much structure or randomness there is in the subshift. Many more measures of order exist that are not treated in this thesis. Pattern complexity for example counts arbitrary patterns of \( L \) letters instead of strings of \( L \) consecutive letters. For aperiodic simple Toeplitz words with two letters it is known that they have the minimal possible pattern complexity (so-called pattern Sturmian systems), see \cite[Lemma~2]{GKBYM_MaxPatternToepl}.

\section{Quasicrystals and Schrödinger operators}
\label{sec:IntroQC}
 
The characteristic property of crystals is their periodicity at atomic level. Similar to the two-dimensional case where only certain rotation symmetries are compatible with periodicity, there are restrictions in the three-dimensional case as well (given by the so-called point groups). In 1982 Shechtman conducted diffraction experiments which showed a ``forbidden'' symmetry in an alloy of aluminium and manganese, while also showing the sharp diffractions peaks that were associated with periodicity. After additional experiments together with Blech, Gratias and Cahn, the discovery of this new type of material was published in 1984 (\cite{ShechtmanBGC_QC}). Shortly after, the name \emph{quasicrystal} (from ``quasiperiodic crystal'') was suggested (\cite{LevSteinh_NameQC}). For his discovery, Schechtman was awarded the Nobel prize in Chemistry in 2011. Very roughly speaking, the apparent violation of the crystallographically allowed rotations has the same reason that allowed the Penrose tiling to have five-fold symmetry in two dimensions: the arrangement of the atoms is non-periodic, but nevertheless ``very ordered'', which causes the sharp peaks in the diffraction pattern.

Since the structure of quasicrystals contains both order and disorder, it is natural to ask if their properties also ``are in between'' the properties of periodic materials (crystals) and completely random (also: glassy) materials. Crystals for example tend to be good electric conductors, while glassy material tend to be insulators. Mathematically these properties are encoded in the spectrum of a Schrödinger operator. Focusing again on the one-dimensional case, the idea is the following: we discretise the problem and let \( \InfWord \in \Alphab^{\ZZ} \) describe the arrangement of atoms or ions in the quasicrystal (or rather: in its one-dimensional and two-sided infinite mathematical model). To each of its sites \( \InfWord( j ) \), \( j \in \ZZ \), a potential is assigned which takes the values of \( \InfWord \) around position \( i \) into account. Mathematically, the potential acts as a multiplication operator on \( \ell^{2}( \ZZ ) \). The so-called tight binding approximation of interaction of an electron with the atoms is given by the sum of this potential and a discrete Laplacian. This sum is known as the Schrödinger operator. Its spectrum encodes the allowed energies of the electron. The precise definitions are presented \textbf{in Chapter~\ref{chap:JOSimpToep}, Section~\ref{sec:JOTransMat} and~\ref{sec:SpecJO}.} We even consider a situation that is slightly more general than described above, where we allow different weights for the off-diagonal terms of the operator. Such objects are known as \emph{Jacobi operators}.

For the special case of a periodic word \( \InfWord \in \Alphab^{\ZZ} \) (which models a one-dimensional crystal) if follows from Floquet/Bloch theory that the spectrum of a Schrödinger operator has only an absolutely continuous part. For the converse case where each \( \InfWord( j ) \) is chosen randomly (which models a one-dimensional glassy material) the spectrum has only a pure point part. Note that these opposite characteristics of the spectrum correspond to the opposite conductivity in crystals and glassy materials. Since quasicrystals show order as well as disorder, one might expect that their spectrum ``is in between'' both extreme cases. Indeed it can sometimes be shown that the randomness prevents the spectrum from being absolutely continuous, while the order prevents it from being pure point. The main techniques for this are reviewed \textbf{in Section~\ref{sec:SpecJO} and~\ref{sec:JOGordon}. In Section~\ref{sec:AsNoPP}} we apply these techniques to simple Toeplitz subshifts. We show that the spectrum is purely singular continuous for almost all \( \InfWord \in \Subshift \) with respect to the unique ergodic probability measure on the subshift.

\section{Cocycles, Cantor spectrum and leading sequences}

An important tool to study the spectrum of a Schrödinger or Jacobi operator \( \Jac_{\InfWord} \) are \emph{transfer matrices,} \textbf{see Section~\ref{sec:JOTransMat}}. They describe solutions to the eigenvalue equation \( \Jac_{\InfWord} \varphi = E \varphi \) with \( E \in \RR \) and \( \varphi \colon \ZZ \to \RR \). More precisely, for every \( \InfWord \in \Subshift \) and every \( j \in \ZZ \) they define a matrix product which determines \( \varphi( j ) \), depending on \( \varphi( 0 ) \), \( \varphi( 1 ) \) and \( E \in \RR \). \textbf{In Chapter~\ref{chap:UnifCocycCantor}} we study a generalisation of this concept to more arbitrary matrix products associated to \( \InfWord \in \Subshift \) and \( j \in \ZZ \) which need not be related to the eigenvalue equation. This generalisation is know as \emph{cocycles,} \textbf{see Section~\ref{sec:CocycCantor}}. More specifically we will consider cocycles that depend locally constant on \( \InfWord \). The central question will concern the asymptotic exponential behaviour of the cocycle's norm and whether or not this behaviour is uniform with respect to \( \InfWord \).

While the question of uniform convergence of cocycles is interesting on its own, it is also connected to the study of Schrödinger and Jacobi operators. For every minimal, uniquely ergodic subshift it was shown in \cite[Theorem~3]{Lenz_SingSpec1dQC} that the spectrum of a Schrödinger operator is given by
\[ \Spec = \Set{ E \in \RR : \Lyapu{\TrMat}( E ) = 0} \DisjCup \Set{ E \in \RR : \TrMat \text{ is not uniform} } \, , \]
where \( \TrMat \) denotes the aforementioned transfer matrix cocycle and \( \Lyapu{\TrMat}( E ) \in \RR \) is a quantity known as the Lyapunov exponent (related to the asymptotic exponential behaviour of the cocycle's norm). For aperiodic subshifts, Kotani theory (\cite{Kotani_KotTheo}) implies that the set of those \( E \) for which the Lyapunov exponent vanishes, has zero Lebesgue measure. Moreover it follows in a rather general setting that the spectrum does not contain isolated points (\cite[Theorem~4]{Pastur_LyapuAndSc}). Thus we can infer from the uniformity of the transfer matrix cocycle that the spectrum is a Cantor set of Lebesgue measures zero.

It is therefore of interest to find sufficient conditions for uniform convergence. For example it was proved that a property called ``positivity of weights'' implies uniformity of the transfer matrix cocycle (\cite[Theorem~2]{Lenz_SingSpec1dQC}) and even uniformity of every locally constant \( \SL \)-cocycle (\cite[Theorem~1]{Lenz_NonUnifCocy}). Later it was shown that positivity of weights can be replaced by the weaker Boshernitzan condition (\cite[Theorem~1]{DamLenz_Boshern}). Demonstrating that this condition is satisfied has since become a standard approach to prove Cantor spectrum (see for instance \cite[Theorem~4]{DamLenz_BoshLowCompl} for Sturmian subshifts). However, for simple Toeplitz subshifts it is known that the Boshernitzan condition does not hold in general (\cite{LiuQu_Simple}). From an analysis of the so-called trace map Liu and Qu could nevertheless deduce uniformity of the transfer matrix cocycles of Schrödinger operators.

In this thesis we introduce a different approach, which stems from a collaboration with Tatiana Nagnibeda, Rostislav Grigorchuk and Daniel Lenz. \textbf{In Section~\ref{sec:LSCSubsh}} we define the \emph{leading sequence condition \LCond{}} for subshifts. It ensures the existence of finitely many elements in the subshift which are ``exhaustive'' in terms of subwords around their origin and ``well-behaved'' with respect to cocycles. We call these elements \emph{leading sequences}. All simple Toeplitz subshifts, but also all Sturmian subshifts, satisfy the leading sequence condition, \textbf{see Section~\ref{sec:LCondExmpl}.} We prove \textbf{in Section~\ref{sec:UnifCoycLCond}} that every locally constant cocycle over \LCond{}-subshift is uniform. This generalises the result of Liu and Qu from the Schrödinger case to arbitrary cocycles. In particular we obtain Cantor spectrum for Jacobi operators. An additional merit of our approach is that it provides a unified treatment of Sturmian and simple Toeplitz subshifts.

\section{Self-similar groups}

As aperiodic systems, simple Toeplitz subshifts are of interest on their own. In addition they may serve to some extend as models of one-dimensional quasicrystals. In this context, Schrödinger and Jacobi operators appear rather naturally. Another reason to study these subshifts and operators is a recently found connection to certain self-similar groups (\cite{GLN_SpectraSchreierAndSO}).

Self-similar groups are groups of graph automorphisms which act on a regular tree in a self-similar way. They are an important source of (counter-)examples. A self-similar group (Grigorchuk's group, \cite{Grig_BurnsideRuss}) was for instance the first example of a group with intermediate growth and also the first example of a group which is amenable but not elementary amenable (see \cite{Grig_DegrOfGrowthRuss}). More generally, a whole family of Grigorchuk's groups can be defined which is related to subshifts in various ways (see for instance \cite{MBon_TopoFullGr} regarding an embedding into the topological full group). For us, the similarity between simple Toeplitz subshifts and the Schreier graphs of the groups will be crucial. Roughly speaking the Schreier graphs encode the action of the group's generators on the tree. The graphs can be constructed inductively by a hole-filling procedure which is analogous to the construction of simple Toeplitz words. As a result, Laplacians on the Schreier graphs associated to the members of Grigorchuk's family of groups are unitarily equivalent to Jacobi operators on certain simple Toeplitz subshifts (\cite[Proposition~4.1]{GLN_SpectraSchreierAndSO}, \cite[Theorem~7.1]{GLNS_LeadingSeq_Arxiv}). While this is one of our main motivations to analyse the spectrum of Jacobi operators, our study of these questions is solely based on word combinatorics and methods from ergodic theory. Thus no knowledge about self-similar groups is required to follow our exposition in this thesis. For the interested reader \textbf{Appendix~\ref{app:SelfSimGr}} nevertheless provides a short introduction to self-similar groups and Schreier graphs.

%%%%%%%%%%%%%%%%%%%%%%%%
%%%%%%%%%%%%%%%%%%%%%%%%
%%%%%%%%%%%%%%%%%%%%%%%%
\chapter[Simple Toeplitz subshifts][Subshifts]{Simple Toeplitz subshifts}
\label{chap:STSubshifts}

This chapter presents the main objects of this thesis: simple Toeplitz subshifts. In Section~\ref{sec:PrelimSubsh} we recall standard notation and essential facts about subshifts in general. The second section serves as an introduction to simple Toeplitz subshifts. A first definition via one-sided infinite words is given in Subsection~\ref{subsec:STDefiPBlock}. Afterwards, we present various examples. They will be used throughout the whole thesis as (counter-)examples to illustrate the wide range of behaviour that simple Toeplitz subshifts exhibit. In Subsection~\ref{subsec:STDefiHoleFill} we recall another, probably better known definition of the subshifts, which uses a hole-filling procedure. This definition is a bit more involved, but has the advantage that it immediately yields a description of the subshift's elements. Finally, we state (and in some cases prove) several basic, well-know properties of simple Toeplitz subshifts in Subsection~\ref{subsec:STBasicProp}, which are essential for later chapters.

\section{Preliminaries on subshifts}
\label{sec:PrelimSubsh}

As discussed in the introduction, the elements of a subshift are ``infinite strings of symbols''. Each symbol represents for example a region in the underlying space of a dynamical system or a certain type of atoms in a quasicrystal. Moreover, a subshift reflects the fact that shifting the orbit in a dynamical system or the origin of a quasicrystal ``does not matter''. In this section, we formulate these ideas mathematically. We focus on basic notions which are important for our study of simple Toeplitz subshifts later on. For a more thorough treatment, the reader is referred to standard texts such as \cite{Queff_SubstDynSys}, \cite{LindMarcus_Coding} or \cite{Fogg_Substitutions} and the references therein.

By \( \Alphab = \Set{ a_{1} , \ldots , a_{\Card{\!\Alphab}} } \) we denote a finite set, the so-called \emph{alphabet}\index{alphabet}. Its elements are called \emph{letters}\index{letter} and they represent for example regions in space or different types of atoms. We will always assume \( \Card{ \Alphab } \geq 2 \). A concatenation of (finitely or infinitely many) letters is called a (finite or infinite) \emph{word}\index{word!finite}\index{word!infinite}. Formally, words are defined as maps from a (finite or infinite) set into \( \Alphab \). The elements in the subshifts that we consider will be two-sided infinite words, that is, maps \( \InfWord \colon \ZZ \to \Alphab \). When we need to mark the origin of the word, we denote it with a vertical bar between the non-positive and the positive positions:
\[ \Word{ \ldots }{ \InfWord(-2) }{ \InfWord(-1) }{ \Origin{\InfWord(0)} }{ \InfWord(1) }{ \InfWord(2) }{ \ldots } \; .\]
Similarly, one-sided infinite words are given by a map \( \RWord \colon \NN \to \Alphab \), while finite words are defined as \( u \colon \Set{ 1 , \ldots , L } \to \Alphab \) for some \( L \in \NN_{0} \). In this case, we call \( L \) the \emph{length}\index{length of a word}\index{word!length of} of the word \( u \) and we write \( \Length{u} = L \). The word \( \epsilon \colon \emptyset \to \Alphab \) of length zero is called the \emph{empty word}\index{word!empty}. A concatenation of two finite words \( u = \Word{ u(1) }{ \ldots }{ u( \Length{u} ) } \) and \( v = \Word{ v(1) }{ \ldots }{ v( \Length{v}  ) } \) is denoted by \( uv \DefAs \Word{ u(1) }{ \ldots }{ u( \Length{u}  ) }{ v(1) }{ \ldots }{ v( \Length{u}  ) } \). Conversely the restriction of a (finite or infinite) word \( \InfWord \) to a subset of its positions is defined as \( \Restr{\InfWord}{j}{j+L-1} \DefAs \Word{ \InfWord( j ) }{ \ldots }{ \InfWord( j+L-1 ) } \). In other words, \( \Restr{\InfWord}{j}{j+L-1} \) denotes the finite word of length \( L \) in \( \InfWord \) that starts at position \( j \). The set of all finite words that occur in \( \InfWord \) is known as the \emph{language of \( \InfWord \)}\index{language!of a word}. More precisely, we define
\[ \Langu{\InfWord}_{L} \DefAs \Set{ \Restr{\InfWord}{j}{j+L-1} : j \in \ZZ } \]
as the \emph{set of all words of length \( L \)}\index{set of words of length L!in a word} in \( \InfWord \in \Alphab^{\ZZ} \). The language of \( \InfWord \in \Alphab^{\ZZ} \) is then given by \( \Langu{\InfWord} \DefAs \cup_{L \in \NN_{0}} \Langu{\InfWord}_{L} \). Similarly we define the language of \( \InfWord \in \Alphab^{\NN} \). Note that the empty word is an element of every language.

We now have a mathematical way to represent an orbit in a dynamical system or to model a quasiperiodic arrangement of atoms. However, we might want to compare different orbits or approximate the quasicrystal by suitable periodic arrangements. Thus, we introduce a topology to obtain notions of closeness and convergence of words: we equip \( \Alphab \) with the discrete topology and \( \Alphab^{\ZZ} \) with the resulting product topology. The sets \( \Cylin{u}{j} \DefAs \Set{ \InfWord \in \Alphab^{\ZZ} : \Restr{\InfWord}{j}{j+\Length{u}-1} = u } \) of those words, in which a given finite word \( u \) appears at position \( j \), are called \emph{cylinder sets}\index{cylinder set} and form a base of open sets of the topology. The topological space \( \Alphab^{\ZZ} \) is metrisable and a metric is for example given by
\[ d( \InfWord_{1}, \InfWord_{2} ) \DefAs \sum_{j = - \infty}^{\infty} \frac{ \delta( \InfWord_{1}(j), \InfWord_{2}(j) ) }{ 2^{ \lvert j \rvert} } \, , \quad \text{with } \delta( a_{1}, a_{2} ) \DefAs \begin{cases} 0 & \text{ if } a_{1} = a_{2} \\ 1& \text{ if } a_{1} \neq a_{2} \end{cases} \; ,\]
for two-sided infinite words \( \InfWord_{1} , \, \InfWord_{2} \in \Alphab^{\ZZ} \) and letters \( a_{1}, a_{2} \in \Alphab \). In this topology, two words \( \InfWord , \, \varrho \in \Alphab^{\ZZ} \) are ``close'' if and only if they agree on a large interval around the origin. In particular we note that a sequence \( (\InfWord_{k}) \) converges to \( \InfWord \) if, for every interval, there exists a number \( k_{0} \) such that \( \InfWord_{k} \) agrees with \( \InfWord \) on this interval for all \( k \geq k_{0} \).

\begin{rem}
It is easy to see that \( \Alphab^{\ZZ} \) with this topology is a Cantor set: the finite union of all cylinder sets of a fixed word length and a fixed position \( j \in \ZZ \) covers \( \Alphab^{\ZZ} \), so every cylinder set is clopen. Let \( \InfWord \neq \varrho \in \Alphab^{\ZZ} \) be two distinct elements. Then there exists a cylinder set \( \Cylin{u}{j} \) that contains \( \InfWord \) but not \( \varrho \), and \( \Alphab^{\ZZ} = \Cylin{u}{j} \DisjCup \Cmplmt{ \Cylin{u}{j} } \) is a partition of \( \Alphab^{\ZZ} \) with open sets that separate \( \InfWord \) and \( \varrho \). Hence \( \Alphab^{\ZZ} \) is totally disconnected. Moreover, every finite word can be extended by all elements in \( \Alphab \). Therefore, no cylinder set contains just one element, that is, there are no isolated points in \( \Alphab^{\ZZ} \). Finally, note that \( \Alphab^{\ZZ} \) is compact as a product of finite (hence compact) sets.
\end{rem}

The above gives a proper mathematical meaning to the notion of ``strings of symbols''. Another important property is that the set of all words is invariant under movements of the origin. First, we formalise the notion of ``movement'':

\begin{defi}
The map \( \Shift \colon \Alphab^{\ZZ} \to \Alphab^{\ZZ} \), defined by \( (\Shift \InfWord) (j) \DefAs \InfWord( j+1 ) \), is called the \emph{shift}\index{shift map} (or more precisely: the \emph{left shift}) on \( \Alphab^{\ZZ} \).
\end{defi}

Since \( \Shift \) is bijective and maps cylinder sets to cylinder sets, it is a homeomorphism on \( \Alphab^{\ZZ} \). Thus, \( (\Alphab^{\ZZ}, \Shift) \) is a topological dynamical system. For \( \InfWord \in \Alphab^{\ZZ} \), the set \( \Set{ \Shift^{j} \InfWord : j \in \ZZ } \) is called the \emph{orbit of \( \InfWord \)}\index{orbit}\index{word!orbit of}. Clearly, every orbit is invariant under the shift. We are particularly interested in sets that combine invariance and closedness:

\begin{defi}
A closed subset \( \Subshift \subseteq \Alphab^{\ZZ} \) which is \( \Shift \)-invariant, is called a \emph{subshift}\index{subshift}.
\end{defi}

As for infinite words, we let \( \Langu{ \Subshift }_{L} \DefAs \cup_{ \InfWord \in \Subshift} \Langu{\InfWord}_{L} \) denote the \emph{set of all words of length \( L \)}\index{word!set of words of length L}\index{set of words of length L!in a subshift} in the subshift. Moreover, we define the \emph{language of a subshift}\index{language!of a subshift}\index{subshift!language of} \( \Langu{\Subshift} \DefAs \cup_{ \InfWord \in \Subshift} \Langu{ \InfWord }\) as the set of all finite words that occur in the subshift. The subshifts that we encounter will be defined in two different ways. First of all, a subshift can be defined as the orbit closure of a two-sided infinite word:

\begin{defi}
\label{defi:SubshiftOrbit}
The subshift associated to \( \InfWord \in \Alphab^{\ZZ} \) is defined as \( \Subshift( \InfWord ) \DefAs \Close{ \Set{ \Shift^{j} \InfWord : j \in \ZZ } } \).
\end{defi}

By definition, \( \Subshift( \InfWord ) \) is a closed set. To see that \( \Subshift( \InfWord ) \) is \( \Shift \)-invariant, note that \( \Shift \) is continuous and that every \( \varrho \in \Subshift( \InfWord ) \) can be expressed as \( \varrho = \lim_{k \to \infty} \Shift^{j_{k}} \InfWord \) for some sequence \( (j_{k}) \). Thus \( \Shift \varrho = \Shift( \lim_{k \to \infty} \Shift^{j_{k}} \InfWord ) = \lim_{k \to \infty} \Shift^{j_{k}+1} \InfWord \) is indeed an element of \( \Subshift( \InfWord ) \). Alternatively, a subshift can be defined from a one-sided infinite word through its language:

\begin{defi}
\label{defi:SubshiftLangu}
The subshift associated to \( \RWord \in \Alphab^{\NN} \) is defined as \( \Subshift( \RWord ) \DefAs \Set{ \InfWord \in \Alphab^{\ZZ} : \Langu{ \InfWord } \subseteq \Langu{ \RWord } } \).
\end{defi}

The shift-invariance of \( \Subshift( \RWord ) \) is clear, since shifted elements have the same language. To see that \( \Subshift( \RWord ) \) is closed, let \( ( \InfWord_{k} ) \) be a sequence in \( \Subshift( \RWord ) \) and let \( \varrho = \lim_{k \to \infty} \InfWord_{k} \). For every word \( u = \Restr{ \varrho }{ j }{ j+\Length{u}-1 } \) in \( \varrho \), we have \( \Restr{ \varrho }{ j }{ j+\Length{u}-1 } = \Restr{ \InfWord_{k} }{ j }{ j+\Length{u}-1 } \) for all sufficiently large \( k \), and thus \( u \in \Langu{ \InfWord_{k} } \subseteq \Langu{ \RWord } \).

We conclude this section with the definition of three important properties of subshifts: first of all, we say that an element \( \InfWord \in \Subshift \) is \emph{periodic}\index{periodic word}\index{word!periodic} if there exists a number \( P \in \NN \) with \( \Shift^{P} \InfWord = \InfWord \). A subshift is called \emph{aperiodic}\index{aperiodic}\index{subshift!aperiodic} if none of its elements is periodic. Secondly, a subshift is called \emph{minimal}\index{minimal}\index{subshift!minimal} if the orbit of every \( \InfWord \in \Subshift \) is dense in \( \Subshift \). Thirdly, a subshift is called \emph{uniquely ergodic}\index{uniquely ergodic subshift}\index{subshift!uniquely ergodic} if there exists a unique \( \Shift \)-invariant Borel probability measure on \( \Subshift \). Usually the subshifts we consider will be aperiodic, minimal and uniquely ergodic.

%%%%%%%%%%%%%%%%%%%%%%%%
\section{Simple Toeplitz subshifts}
\label{sec:SimpleToepl}

We now turn towards the special class of subshifts that are known as simple Toeplitz subshifts. Since several later chapters are devoted to their study, we discuss their definition in great detail and present numerous examples. However, before we start we briefly recall the principle on which Toeplitz words and Toeplitz subshifts rely in general. As examples of aperiodic order, Toeplitz words combine non-randomness and non-periodicity. The particular method that is used to obtain this combination goes back to Toeplitz, Jacobs and Keane. The main idea is that the value of every position is repeated periodically, but different positions may have different periods. More precisely, a word \( \InfWord \in \Alphab^{\ZZ} \) is called a \emph{Toeplitz word}\index{Toeplitz!word}\index{word!general Toeplitz} if the following holds:
\begin{equation}
\label{eqn:DefToeplitz}
\forall j \in \ZZ \quad \exists p \in \NN \quad \forall m \in \ZZ \; : \;\; \InfWord( j ) = \InfWord( j + mp) \, .
\end{equation}
According to Definition~\ref{defi:SubshiftOrbit} we can associate a subshift \( \Subshift( \InfWord ) \) to a Toeplitz word, which is then called a \emph{Toeplitz subshift}\index{Toeplitz!subshift}\index{subshift!general Toeplitz}.

One way to construct a Toeplitz word is to start with a word that has some yet undetermined positions, the so-called holes. The word, including the holes, is chosen to be periodic. The holes are then filled with the letters of another periodic word with holes and the positions that are still undetermined after that, are filled with the letters of yet another periodic word with holes, and so on. For the special case of simple Toeplitz words, this procedure is explained in detail in Subsection~\ref{subsec:STDefiHoleFill}.

% math mode not available in bookmarks, use pdfstring instead
\subsection{Definition of simple Toeplitz subshifts via subwords of \texorpdfstring{\( \PBlock{\infty} \)}{p\^{}(infinity)}}
\label{subsec:STDefiPBlock}

Now we define the special case of a simple Toeplitz subshift, which is characterised by two sequences: let \( (a_{k}) \in \Alphab^{\NN_{0}} \) be a sequence of letters and \( (n_{k}) \in ( \NN \setminus \{1\} )^{\NN_{0}} \) be a sequence of period lengths that are greater or equal to two. These sequences are called the \emph{coding sequences}\index{coding sequence} of the subshift. From the coding sequences, we define recursively a sequence \( (\PBlock{k}) \) of finite words by
\begin{equation}
\label{eqn:PBlockRecur}
\PBlock{-1} \DefAs \epsilon \qquad \text{and} \qquad \PBlock{k+1} \DefAs \Word{ \PBlock{k} }{ a_{k+1} }{ \PBlock{k} }{ a_{k+1} }{ \PBlock{k} }{ a_{k+1} }{ \ldots }{ \PBlock{k} } \, ,
\end{equation}
with \( n_{k+1} \)-many \( \PBlock{k} \)-blocks and (\( n_{k+1}-1 \))-many \( a_{k+1} \)'s. We focus on aperiodic subshifts and we will see that in this case, \( (a_{k}) \) must not be eventually constant (Proposition~\ref{prop:STAperiod}). Therefore, and since consecutive occurrences of the same letter can be expressed as a single occurrence if \(n_{k}\) is increased accordingly, we always assume \( a_{k+1} \neq a_{k} \). As we will see in Equation~(\ref{eqn:DecompPBlock}), every element of the subshift can be decomposed into single letters and ``blocks of letters'', which are precisely the words \( \PBlock{k} \). Since in addition all \( \PBlock{k} \) are palindromes (see Example~\ref{exmpl:PBlockPalindr}), we refer to them as \emph{\( \PBlock{k} \)-blocks}\index{p-blocks@\( \PBlock{k} \)-blocks} or \emph{palindromic blocks}\index{palindromic blocks}. Their length is given by a simple recursion:

\begin{prop}
For all \( k \geq -1 \), we have \( \Length{\PBlock{k+1}} + 1= n_{k+1} (\Length{\PBlock{k}} + 1) \). In particular this yields \( \Length{\PBlock{k+1}} + 1 = \prod_{j=0}^{k+1} n_{j} \). 
\end{prop}

Note that \( \PBlock{k} \) is a prefix of \( \PBlock{k+1} \) for every \( k \) (it is also a suffix of \( \PBlock{k+1} \), which we will use later). After extending the finite words \( \PBlock{k} \) arbitrarily to one-sided infinite words, the sequence will therefore converge in \( \Alphab^{\NN} \). We denote the limit by \( \PBlock{\infty} \in \Alphab^{\NN} \). 

\begin{defi}
\label{defi:SToeplPInfty}
The subshift \( \Subshift( \PBlock{\infty} ) \), defined according to Definition~\ref{defi:SubshiftLangu} by \( \PBlock{\infty} \DefAs \lim_{k \to \infty} \PBlock{k} \), is called the \emph{simple Toeplitz subshift}\index{simple Toeplitz!subshift}\index{subshift!simple Toeplitz} with coding sequences \( (a_{k}) \) and \( (n_{k}) \).
\end{defi}

\begin{rem}
It is not obvious that a simple Toeplitz subshift is actually a Toeplitz subshift in the sense of Equation~(\ref{eqn:DefToeplitz}). In Subsection~\ref{subsec:STDefiHoleFill} we will see an alternative definition of simple Toeplitz subshifts via hole-filling, which is more involved, but highlights the Toeplitz character of the subshift.
\end{rem}

The properties of a simple Toeplitz subshift depend on its coding sequences. Similar to \cite{LiuQu_Simple}, we introduce the following notions to describe \( (a_{k}) \) in more detail: the set of all letters which appear at position \( k \) or later in the coding sequence, is denoted by \( \Alphab_{k} \DefAs \Set{ a_{j} : j \geq k } \). The set of letters which appear infinitely often is denoted by \( \AlphabEv \DefAs \cap_{k \geq 0} \Alphab_{k} \) and called the \emph{eventual alphabet}\index{eventual alphabet}\index{alphabet!eventual} or the \emph{set of recurrent letters}\index{recurrent letters}\index{letter!set of recurrent letters}. Since \( \Alphab \) is finite, there exists a number \( \EventNr \) with \( a_{k} \in \AlphabEv \) and \( \Alphab_{k} = \AlphabEv \) for all \( k \geq \EventNr \).

\begin{exmpl}[leading words]
\label{exmpl:SpWBlocks}
To conclude this subsection, we describe particular elements of simple Toeplitz subshifts, which we call the \emph{leading words}\index{leading word!of a simple Toeplitz subshift}. For every \( \AEv \in \AlphabEv \), the leading word \( \SpW{ \AEv } \) is defined as the two-sided limit
\[ \SpW{ \AEv } \DefAs \lim_{k \to \infty} \Word{ \PBlock{k} }{ \Origin{ \AEv } }{ \PBlock{k} } \, \in \Alphab^{ \ZZ } \, . \]
To define this limit, we use that \( \PBlock{k} \) is a suffix as well as prefix of  \( \PBlock{k+1} \) for every \( k \). To prove \( \SpW{ \AEv } \in \Subshift( \PBlock{\infty} ) \), it remains to show that, for arbitrary \( k \), all finite words in \( \Word{ \PBlock{k} }{ \AEv }{ \PBlock{k} } \) occur in \( \PBlock{\infty} \) as well: because of \( \AEv \in \AlphabEv \), there exists a number \( l > k \) with \( a_{l} = \AEv \). And since \( \PBlock{\infty} \) contains \( \PBlock{l} = \Word{ \PBlock{l-1} }{ a_{l} }{ \PBlock{l-1} }{ \ldots }{ \PBlock{l-1} } \) and since \( \PBlock{l-1} \) begins and ends with \( \PBlock{k} \), all subwords of \( \Word{ \PBlock{k} }{ \AEv }{ \PBlock{k} } \) occur in \( \PBlock{\infty} \).
\end{exmpl}

\begin{rem}
As the name suggests, the leading words determine important properties of the subshift. Roughly speaking, this is possible because every finite word occurs close to the origin of a leading word. Moreover they have certain symmetries and ``enough structure'' to prevent exponential decay of the norm of certain matrix products as we go to plus/minus infinity in \( \SpW{ \AEv } \). Later we will generalise the notion of leading words to the much broader class of so-called \LCond{}-subshifts (Definition~\ref{defi:LCond}).
\end{rem}

\subsection{Examples of simple Toeplitz subshifts}
\label{subsec:STExmpl}

To illustrate the notion of simple Toeplitz subshifts, we now discuss a number of examples. Throughout the whole thesis we will come back to them to demonstrate the concepts we encounter. Our first example is the well-known \emph{period doubling subshift}\index{subshift!period doubling}\index{period doubling subshift}:

\begin{exmpl}[period doubling]
\label{exmpl:PD}
Consider the alphabet \( \Alphab = \Set{ a, b } \), the alternating sequence \( (a_{k}) = ( a , b , a , b, \ldots ) \) and the constant sequence \( (n_{k}) = ( 2 , 2 , 2 , \ldots ) \). This yields the blocks
\begin{alignat*}{3}
\PBlock{0} & \, = \, a_{0}^{n_{0}-1} && \, = \, a \\
\PBlock{1} & \, = \, \Word{ \PBlock{0} }{ a_{1} }{ \PBlock{0} } && \, = \, \Word{ a }{ b }{ a } \\
\PBlock{2} & \, = \, \Word{ \PBlock{1} }{ a_{2} }{ \PBlock{1} } && \, = \, \Word{ a }{ b }{ a }{ a }{ a }{ b }{ a } \\
\PBlock{3} & \, = \, \Word{ \PBlock{2} }{ a_{3} }{ \PBlock{2} } && \, = \, \Word{ a }{ b }{ a }{ a }{ a }{ b }{ a }{ b }{ a }{ b }{ a }{ a }{ a }{ b }{ a } \\[-1ex]
\vdots\hspace{0.5em}
\end{alignat*}
In a way, this defines the simplest of the aperiodic simple Toeplitz subshifts: by definition, \( n_{k} \) has to be at least two for every \( k \). For aperiodicity we need an alphabet with at least two letters, as well as a sequence \( (a_{k}) \) which is not eventually constant (see Proposition~\ref{prop:STAperiod}). For a two-letter alphabet \( \Set{ a, b } \) our assumption \( a_{k} \neq a_{k+1} \) enforces an alternating sequence of \( a \)'s and \( b \)'s.
\end{exmpl}

\begin{rem}
\label{rem:SubstPD}
Alternatively, the subshift above can be defined by the period doubling substitution \( \sigma_{\text{pd}} \colon a \mapsto \Word{ a }{ b } \, , \; b \mapsto \Word{ a }{ a } \). For details on substitution systems the reader is referred to \cite{Fogg_Substitutions} or \cite[Chapter~4]{BaakeGrimm_Aperio}. Here we only prove that the two subshifts, defined by the one-sided infinite words \( \PBlock{\infty} \) and \( \lim_{k \to \infty} \sigma_{\text{pd}}^{k}( a ) \), are actually equal. First note that \( \sigma_{\text{pd}}^{k}( a ) \) and \( \sigma_{\text{pd}}^{k}( b ) \) differ exactly in the last letter, which can easily be shown by induction. The claim now follows from a second induction, which shows that \( \Word{ \PBlock{k} }{ a_{k+1} } \) and \( \sigma_{\text{pd}}^{k+1}( a ) \) are equal: for \( k = 0 \), this is obvious. Now assume that \( \Word{ \PBlock{k} }{ a_{k+1} } = \sigma_{\text{pd}}^{k+1}( a ) \) holds for some \( k \). By changing the last letter on both sides, we obtain \( \Word{ \PBlock{k} }{ a_{k+2} } = \sigma_{\text{pd}}^{k+1}( b ) \), where we used that the sequence \( (a_{k}) \) alternates between \( a \) and \( b \). This yields
\[ \Word{ \PBlock{k+1} }{ a_{k+2} } =  \Word{ \PBlock{k} }{ a_{k+1} }{ \PBlock{k} }{ a_{k+2} } = \Word{ \sigma_{\text{pd}}^{k+1}( a ) }{\sigma_{\text{pd}}^{k+1}( b )  } = \sigma_{\text{pd}}^{k+2}( a ) \, . \qedhere \]
\end{rem}
 
As mentioned in the introduction, certain simple Toeplitz subshifts are related to self-similar groups. In particular it turns out that the subshift that is defined in the next example is connected to Grigorchuk's group (\cite{Grig_BurnsideRuss}, for the connection see \cite{GLN_SpectraSchreierAndSO}). Thus we call it the \emph{Grigorchuk subshift}\index{subshift!Grigorchuk}\index{Grigorchuk!Grigorchuk subshift}, which is, however, not standard terminology. Here we only define the subshift. Some information on its connection to Grigorchuk's group can be found in Appendix~\ref{app:SelfSimGr}.

\begin{exmpl}[Grigorchuk subshift]
\label{exmpl:GrigSubshiftDef}
Consider the alphabet \( \Alphab = \Set{ a, b, c, d } \). The subshift is defined by \( (a_{k})_{k \in \NN_{0}} = ( a, b, c, d, b, c, d, \ldots ) \), which is three-periodic from \( a_{1} \) on, and the constant sequence \( (n_{k})_{k \in \NN_{0}} = (2, 2, 2, \ldots ) \). In particular, we have \( \AlphabEv = \Set{ b , c , d } \) and \( \Alphab_{k} = \AlphabEv \) for all \( k \geq 1 \). Moreover, the length of the \( \PBlock{k} \)-blocks is given by \( \Length{ \PBlock{k} } = 2^{k+1} - 1 \). 
\end{exmpl}

\begin{rem}
\label{rem:SubstGrig}
Just as the period doubling subshift (see Remark~\ref{rem:SubstPD}), the Grigorchuk subshift can also be defined by a substitution. It is given by
\[ \sigma_{\text{Grig}} \, \colon \; a \mapsto \Word{ a }{ b }{ a } \; , \;\; b \mapsto c  \; , \;\; c \mapsto d \; , \;\; d \mapsto b \, . \]
In fact, the same substitution appears in the description of Grigorchuk's group (\cite{Lys_RelForGrigGr}). The associated subshift and its properties are discussed in \cite[Section~2]{GLN_SpectraSchreierAndSO}. In particular, it is shown that the subshift is defined by the same word \( \PBlock{\infty} \in \Alphab^{\NN} \) as in our definition above. Hence the subshift is indeed the Grigorchuk subshift. Alternatively, it can also be defined by the primitive substitution 
\[ \widetilde{\sigma}_{\text{Grig}} \, \colon \; a \mapsto \Word{ a }{ b } \; , \;\; b \mapsto \Word{ a }{ c } \; , \;\; c \mapsto \Word{ a }{ d } \; , \;\; d \mapsto \Word{ a }{ b } \, , \]
which is due to Fabien Durand. For details see \cite[Subsection~2.5]{GLN_SpectraSchreierAndSO}, where also the equality of both substitution subshifts is proved.
\end{rem}
 
In generalisation of the aforementioned group, there is a whole family of Grigorchuk's groups (\cite{Grig_DegrOfGrowthRuss}, see also Definition~\ref{defi:DefFamGrigGr} in the appendix). The subshifts associated to these groups are discussed next. We call them \emph{generalised Grigorchuk subshifts}\index{Grigorchuk!generalised Grigorchuk subshift}\index{subshift!generalised Grigorchuk}\index{generalised Grigorchuk subshift}, which is not standard terminology either.

\begin{exmpl}[generalised Grigorchuk subshift]
\label{exmpl:defGenGrigSubsh}
Consider the alphabet \( \Alphab = \Set{ a, b , c , d } \) and the constant sequence \( (n_{k}) = (2, 2, 2, \ldots ) \). The subshift is determined by a sequence \( (b_{k}) \in \Alphab^{\NN_{0}} \) of letters with \( b_{0} = a \) and \( b_{k} \in \Set{ b ,  c , d } \) for \( k \geq 1 \). Here \( b_{k} = b_{k+1} \) is allowed. For \( b_{k} \neq b_{k+1} = \ldots = b_{k+l} \neq b_{k+l+1} \), we can replace the repetitions of \( b_{k+1} \) by a single letter \( a_{m} \) with period length \( n_{m } = 2^{l} \) since
\begin{alignat*}{8}
\PBlock{k+l} & {} = {}&&& \PBlock{k+l-1} & b_{k+l} && \PBlock{k+l-1} & \\
&{} = {} && \PBlock{k+l-2} \, b_{k+l-1} \, & \PBlock{k+l-2} \, & b_{k+l} & & \, \PBlock{k+l-2} \, & b_{k+l-1} \, \PBlock{k+l-2} \\
&{} = {} && \PBlock{k} \, b_{k+1} \, \PBlock{k} \, & \ldots \;\;\, & \ldots & & \, \ldots \, & \PBlock{k} \, b_{k+1} \, \PBlock{k}
\end{alignat*}
holds with \( 2^{l} \)-many \(\PBlock{k}\)-blocks. Hence, a generalised Grigorchuk subshift is a simple Toeplitz subshift with \( \Alphab = \Set{ a, b, c, d } \), \( a_{0} = a \) and \( a_{k} \in \Set{ b , c , d } \) for \( k \geq 1\), where \( n_{k} \) has the form \( n_{k} = 2^{l_{k}} \) with \( l_{k} \in \NN \). Consequently, the block length is given by
\[ \Length{ \PBlock{k} } + 1 = ( \Length{ \PBlock{k-1} } + 1 ) \cdot 2^{l_{k}} = \ldots = 2^{ l_{0} + l_{1} + \hdots + l_{k} } \, .\]
When we restrict ourself to subshifts with \( \Card{ \AlphabEv } \geq 2 \), there are two possibilities: in the case \( \Card{ \AlphabEv } = 3 \) we have \( \AlphabEv = \Set{ b , c , d } \) and \( \Alphab_{k} = \AlphabEv \) for all \( k \geq 1 \). In the case \( \Card{ \AlphabEv } = 2 \), one of the letters \( b \), \( c \), \( d \) appears in \( (a_{k}) \) only finitely many times (possibly zero times) and consequently \( (a_{k}) \) alternates between the two remaining letters for all sufficiently large \( k \).
\end{exmpl}

\begin{rem}
Our notion of generalised Grigorchuk subshifts includes the so-called \emph{\( l \)-Grigorchuk subshifts}\index{l-Grigorchuk subshift@\( l \)-Grigorchuk subshift}\index{Grigorchuk!l-Grigorchuk@\( l \)-Grigorchuk subshift}\index{subshift!l-Grigorchuk@\(l\)-Grigorchuk} that were studied in \cite{DKMSS_Regul-Article}. They are given by \( \Alphab = \Set{ a , b , c , d } \), \( (a_{k}) = (a, b , c , d , b, c, d , \ldots) \) and \( (n_{k}) = (2, 2^{l_{1}}, 2^{l_{2}}, 2^{l_{3}} \ldots) \) with \( l_{k} \in \NN \). 
\end{rem}

Finally, we include an example from \cite{LiuQu_Simple}. It will serve as our main counterexample for notions related to repetitivity in general and the Boshernitzan condition~(\ref{eqn:BoshCond}) in particular (see Section~\ref{sec:Repe} and~\ref{sec:BoshCond}). This is also its original context in \cite{LiuQu_Simple} and the reason that we call it the non-(B) example.

\begin{exmpl}[non-(B) example, {\cite[Section~4]{LiuQu_Simple}}]
\label{exmpl:BoshSubshiftDef}
Consider \( \Alphab = \Set{ a, b, c, d } \) and the constant sequence \( (n_{k}) = (2, 2, 2, \ldots) \). The second coding sequence is given by
\[ (a_{k}) = (\hspace{-0.3em} \underbrace{a, b}_{2 \text{ letters}} \hspace{-0.5em}, c, \underbrace{a, b, a, b}_{4 \text{ letters}}, d, \ldots , c, \underbrace{a, b, \ldots, a, b}_{4l \text{ letters}}, d, \underbrace{a, b, \ldots, a, b}_{4l+2 \text{ letters}}, c, \ldots ) \, , \]
where increasing blocks of \( a \)'s and \(b \)'s are separated alternately by \( c \) and \( d \). Obviously \( \Alphab_{k} = \Alphab = \AlphabEv\) holds for all \( k \in \NN_{0} \). Moreover, we have once again \( \Length{ \PBlock{k} } = 2^{k+1} - 1 \). 
\end{exmpl}

\subsection{Definition of simple Toeplitz subshifts via hole-filling}
\label{subsec:STDefiHoleFill}

Now we discuss an alternative definition of simple Toeplitz subshifts, which was used for example in \cite{LiuQu_Simple}. It is more involved than the one given in Subsection~\ref{subsec:STDefiPBlock}, but shows more explicitly the structure of the subshift's elements. The main idea is to consider periodic words with ``holes'', fill these holes successively with other periodic words with holes, and take the limit. More precisely, let \( \Hole \notin \Alphab\) denote an additional letter (the ``hole''). The positions of the holes in a two-sided infinite word \( \InfWord \) are called the \emph{undetermined part of \( \InfWord \)}\index{undetermined part}\index{word!undetermined part of} and are denoted by \( U_{\InfWord} \DefAs \Set{ j \in \ZZ : \InfWord( j ) = \, \Hole } \). As earlier, we have the sequences \( (a_{k}) \in \Alphab^{\NN_{0}} \), with \( a_{k} \neq a_{k+1} \), and \( (n_{k}) \in ( \NN \setminus \{1\} )^{\NN_{0}} \). In addition, we now define a third coding sequence \( (r_{k})_{k \in \NN_{0}} \) of integers with \( 0 \leq r_{k} < n_{k} \), which encodes the positions of the holes. From these sequences we define the periodic words
\[ ( \Word{ a_{k}^{n_{k}-1} }{ \Hole } )^{\infty} \DefAs \Word{ {} \ldots }{ a_{k} }{ \ldots }{ a_{k} }{ \Hole } \, \underbrace{\Word{ a_{k} }{ \ldots }{ a_{k} }}_{n_{k}-1 \text{ times}} \, \Word{ \Hole }{ a_{k} }{ \ldots }{ a_{k} }{ \Hole }{ \ldots }  \;\; \in \; ( \Alphab \DisjCup \Set{ ? } )^{\ZZ} \, ,\]
with period length \( n_{k} \) and undetermined part \( n_{k} \ZZ + r_{k} \). Following \cite{LiuQu_Simple}, we define the filling of the holes in the following way:

\begin{defi}
Let \( \varrho_{1} \in (\Alphab \DisjCup \Set{ \Hole })^{\ZZ} \) be periodic with period $n$ and undetermined part \( n \ZZ + r \). Inserting the letters of \( \varrho_{2} \in (\Alphab \DisjCup \Set{ \Hole })^{\ZZ} \) into the holes of \( \varrho_{1} \) yields a word \( \varrho_{1} \triangleleft \varrho_{2} \), which is defined by
\[ (\varrho_{1} \triangleleft \varrho_{2})(j) \DefAs \begin{cases}
\varrho_{1}(j) & \text{for } j \notin n \ZZ + r\\
\varrho_{2}( \frac{j-r}{n} ) & \text{for } j \in n \ZZ + r
\end{cases} \, . \]
\end{defi}

In other words, we insert the letter \( \varrho_{2}( 0 ) \) into the first hole at a non-negative position in \( \varrho_{1} \), we insert \( \varrho_{2}( 1 ) \) into the second hole at a non-negative position in \( \varrho_{1} \), and so on, and similarly for the negative direction as well.

This procedure is now applied to \( ( \Word{ a_{k}^{n_{k}-1} }{ \Hole } )^{\infty} \): first, we fill the holes of \( ( \Word{ a_{0}^{n_{0}-1} }{ \Hole })^{\infty} \) with the letters from \( ( \Word{ a_{1}^{n_{1}-1} }{ \Hole })^{\infty} \). The result is another periodic word with holes, which are then filled with the letters from \( ( \Word{ a_{2}^{n_{2}-1} }{ \Hole })^{\infty} \). This defines a sequence \( ( \InfWord_{k} )_{k \in \NN_{0}} \) of two-sided infinite words, given by
\[ \InfWord_{k} \DefAs ( \Word{ a_{0}^{n_{0}-1} }{ \Hole })^{\infty} \triangleleft ( \Word{ a_{1}^{n_{1}-1} }{\Hole })^{\infty} \triangleleft ( \Word{ a_{2}^{n_{2}-1} }{ \Hole })^{\infty} \triangleleft \ldots  \triangleleft ( \Word{ a_{k}^{n_{k}-1} }{ \Hole })^{\infty} \, . \]

\begin{prop}
\label{prop:PeriodWord}
The word \( \InfWord_{k} \) is periodic with period length \( \Length{ \PBlock{k} } + 1 = \prod_{l = 0}^{k} n_{l} \) and has undetermined part \( U_{k} = ( n_{0} \cdot \ldots \cdot n_{k} ) \ZZ + [ r_{0} + \sum_{j=1}^{k} r_{j} \cdot n_{0} \cdot \ldots \cdot n_{j-1} ] \) . The period of \( \InfWord_{k} \) is the finite word \( \Word{ \PBlock{k} }{ \Hole } \). In particular, there is exactly one undetermined position per period.
\end{prop}

\begin{proof}
We prove the formula for the undetermined part and show that the period length of \( \InfWord_{k} \) is \( \Length{ \PBlock{k} } + 1 \). This implies that there is exactly one hole per period, since \( U_{k} \) and \( \InfWord_{k} \) have the same period length. We proceed by induction: for \( \InfWord_{0} = ( \Word{ a_{0}^{n_{0}-1} }{ \Hole } )^{\infty} \) the statements are true by definition. Now assume that the statements holds for \( \InfWord_{k} \) and recall that \( ( \Word{ a_{k+1}^{n_{k+1}-1} }{ \Hole })^{\infty} \) has undetermined part \( n_{k+1} \ZZ + r_{k+1} \). Thus, the undetermined part of \( \InfWord_{k+1} = \InfWord_{k} \triangleleft ( \Word{ a_{k+1}^{n_{k+1}-1} }{ \Hole } )^{\infty} \) is given by
\begin{align*}
& n_{0} \cdot \hdots \cdot n_{k} \cdot \big( n_{k+1} \ZZ + r_{k+1}  \big) + \big[ r_{0} + \sum_{j=1}^{k} r_{j} \cdot n_{0} \cdot \hdots \cdot n_{j-1} \big] \\[-1ex]
={} & \big( n_{0} \cdot \hdots \cdot n_{k+1} \big) \, \ZZ + \big[ r_{0} + \sum_{j=1}^{k+1} r_{j} n_{0} \cdot \hdots \cdot n_{j-1} \big] \, .
\end{align*}
Clearly, the period length of \( \InfWord_{k+1} \) is at least as large as the distance between two undetermined positions, which is \( \prod_{l = 0}^{k+1} n_{l} \). Moreover, we obtained \( \InfWord_{k+1} \) by inserting a word with period length \( n_{k+1} \) into a word with period length \( \prod_{l = 1}^{k} n_{l} \) and exactly one hole per period. Thus, the period length of \( \InfWord_{k+1} \) is at most \( n_{k+1} \cdot \prod_{l = 1}^{k} n_{l} \), which finishes the first part of the proof. Finally, it is easy to see that the period of \( \InfWord_{k} \) and the word \( \Word{ \PBlock{k} }{ ? } \) satisfy the same recurrence relation, see Equation~(\ref{eqn:PBlockRecur}), with the same initial words for \( k = 0 \).
\end{proof}

For every position \( j \in \ZZ \), the letter \( \InfWord_{k}( j ) \) is eventually constant as \( k \) tends to infinity. Thus the limit \( \InfWord_{\infty} \DefAs \lim_{k \to \infty} \InfWord_{k} \) exists in \( ( \Alphab \cup \Set{ \Hole } )^{\ZZ} \). Since \( ( U_{k} ) \) is a decreasing sequence of sets and the period length tends to infinity, the undetermined part \( U_{\infty} \DefAs \cap_{k \geq 0} U_{k} \) of \( \InfWord_{\infty} \) is either empty or a single position. In resemblance to \cite{LiuQu_Simple}, we define the following notions:

\begin{defi}
\label{defi:SToepHoleFill}
In the case \( U_{\infty} = \emptyset \), we call \( \InfWord_{\infty} \) a \emph{normal Toeplitz word}\index{normal Toeplitz word}\index{word!normal Toeplitz}. In the case \( U_{\infty} \neq \emptyset \), we insert an arbitrary letter \( \AEv \in \AlphabEv \) into the undetermined position \( U_{\infty} \). The resulting word \( \InfWord_{\infty}^{(\AEv)} \) is called an \emph{extended Toeplitz word}\index{extended Toeplitz word}\index{word!extended Toeplitz}. A word is called a \emph{simple Toeplitz word}\index{simple Toeplitz!word}\index{word!simple Toeplitz} if it is either a normal Toeplitz word or an extended Toeplitz word. The orbit closure \( \Subshift( \InfWord ) \DefAs \Close{ \Set{ \Shift^{j} \InfWord : j \in \ZZ } } \) of a simple Toeplitz word \( \InfWord \) is called a \emph{simple Toeplitz subshift}\index{simple Toeplitz!subshift}\index{subshift!simple Toeplitz}. 
\end{defi}

We will show in Proposition~\ref{prop:SToeplDefSame} that the two definitions of a simple Toeplitz subshift (Definition~\ref{defi:SToeplPInfty} and Definition~\ref{defi:SToepHoleFill}) really agree.

\begin{rem}
In general, an extended Toeplitz word is not actually a Toeplitz word in the sense of Equation~(\ref{eqn:DefToeplitz}), since there is no period for the position \( U_{\infty} \). However, since the general definition of Toeplitz words will not play any role in this thesis, no confusion should arise from this inconsistency.
 
The reader is also warned that there are many different notions of simple Toeplitz words in the literature. For example in \cite{LiuQu_Simple}, where our definitions are essentially taken from, simple Toeplitz words must not be periodic. We have omitted this requirement to keep the notion consistent with Definition~\ref{defi:SToeplPInfty}. This will have no practical consequences since, from the next chapter on, we will only deal with aperiodic simple Toeplitz subshifts. Other authors use the term only for Toeplitz words in the sense of  Equation~(\ref{eqn:DefToeplitz}) (see for example \cite{KamZamb_MaxPattCompl}), only for words with two letters (\cite{DamLiuQu_PatternSturm}), for words over \( \NN \) (\cite{GKBYM_MaxPatternToepl}) or for words over \( \ZZ^{d} \) (\cite{QRWX_PatCompDDim}).
\end{rem}

Next we discuss some examples to illustrate the hole-filling procedure and the structure of simple Toeplitz words. For the first example, recall our notions of leading words (Example~\ref{exmpl:SpWBlocks}) and of the Grigorchuk subshift (Example~\ref{exmpl:GrigSubshiftDef}).

\begin{exmpl}[leading words in the Grigorchuk subshift]
We consider the coding sequences \( (a_{k}) = ( a , b , c , d, \ldots ) \), \( (n_{k}) = (2, 2, 2, \ldots ) \) and \( (r_{k}) = ( 0, 0, 0, \ldots ) \). They define the sequence \( ( \InfWord_{k} ) \), whose first elements are shown below. Recall that the origin is denoted by a bar between the positions 0 and 1.
\begingroup
\setlength{\arraycolsep}{1pt}
\[ \begin{array}{r *{42}{c}}
\InfWord_{0} = & \ldots & a & \Hole &  a & \Hole & a & \Hole & a & \Hole & a & \Hole & a & \Hole & a & \Hole & a & \Hole & a & \Origin{\Hole} & a & \Hole & a & \Hole & a & \Hole & a & \Hole & a & \Hole& a & \Hole & a & \Hole & a & \Hole & a & \Hole & a & \Hole & a & \Hole \\
\InfWord_{1} = & \ldots & a & \Hole & a & b & a & \Hole & a & b & a & \Hole & a & b & a & \Hole & a & b & a & \Origin{\Hole} & a & b & a & \Hole & a & b & a & \Hole & a & b & a & \Hole & a & b & a & \Hole & a & b & a & \Hole & a & b \\
\InfWord_{2} =& \ldots & a & \Hole & a & b & a & c & a & b & a & \Hole & a & b & a & c & a & b & a & \Origin{\Hole} & a & b & a & c & a & b & a & \Hole & a & b & a & c & a & b & a & \Hole & a & b & a & c & a & b \\
\InfWord_{3} = & \ldots & a & \Hole & a & b & a & c & a & b & a & d & a & b & a & c & a & b & a & \Origin{\Hole} & a & b & a & c & a & b & a & d & a & b & a & c & a & b & a & \Hole & a & b & a & c & a & b
\end{array} \]
\endgroup
By Proposition~\ref{prop:PeriodWord} all \( \InfWord_{k} \) have the form \( \InfWord_{k} = \Word{ \ldots }{ \PBlock{k} }{ \Origin{?} }{ \PBlock{k} }{ \ldots } \) since their undetermined part is given by \( U_{k} = 2^{k+1} \ZZ \). Thus, the sequence \( ( \InfWord_{k} ) \) converges to \( \InfWord_{\infty} = \Word{ \PBlock{\infty} }{ \Origin{?} }{ \PBlock{\infty} } \). When the remaining hole is filled with a letter \( \AEv \in \AlphabEv \), the resulting element \( \SpW{ \AEv }_{\infty} = \Word{ \PBlock{\infty} }{ \Origin{\AEv} }{ \PBlock{\infty} } \) is precisely one of the leading words of the Grigorchuk subshift. Note however that we have not yet shown that this subshift, defined as the orbit closure, is really the same as the Grigorchuk subshift defined from the language in Example~\ref{exmpl:GrigSubshiftDef}.
\end{exmpl}

\begin{exmpl}[alternating word in the Grigorchuk subshift]
\label{exmpl:DefAlterWord}
For the same sequences \( (a_{k}) \) and \( (n_{k}) \) as in the previous example, we now consider \( (r_{k}) = ( 0, 1, 0, 1 \ldots ) \). The holes are filled alternately left and right of the origin, as illustrated by the first \( \InfWord_{k} \) below. Thus, we call this element an \emph{alternating word}\index{alternating word}\index{word!alternating}. 
\begingroup
\setlength{\arraycolsep}{1pt}
\[ \begin{array}{r *{45}{c}}
\InfWord_{0} = & \ldots & a & \Hole & a & \Hole & a & \Hole & a & \Hole & a & \Hole & a & \Hole & a & \Hole & a & \Hole & a & \Origin{\Hole} & a & \Hole & a & \Hole & a & \Hole & a & \Hole & a & \Hole & a & \Hole & a & \Hole & a & \Hole & a & \Hole & a & \Hole & a & \Hole & \ldots \\
\InfWord_{1} = & \ldots & a & b & a & \Hole & a & b & a & \Hole & a & b & a & \Hole & a & b & a & \Hole & a & \Origin{b} & a & \Hole & a & b & a & \Hole & a & b & a & \Hole & a & b & a & \Hole & a & b & a & \Hole & a & b & a & \Hole & \ldots \\
\InfWord_{2} = & \ldots & a & b & a & c & a & b & a & \Hole & a & b & a & c & a & b & a & \Hole & a & \Origin{b} & a & c & a & b & a & \Hole & a & b & a & c & a & b & a & \Hole & a & b & a & c & a & b & a & \Hole & \ldots \\
\InfWord_{3} = & \ldots & a & b & a & c & a & b & a & \Hole & a & b & a & c & a & b & a & d & a & \Origin{b} & a & c & a & b & a & \Hole & a & b & a & c & a & b & a & d & a & b & a & c & a & b & a & \Hole & \ldots
\end{array} \]
\endgroup
For every \( k \) there is a \( \PBlock{k} \)-block that contains the origin. Asymptotically, the origin is placed alternately at one third and at two thirds of the block, as a short computation shows: by Proposition~\ref{prop:PeriodWord}, the undetermined part of \( \InfWord_{k} \), with \(  k = 2l \) or \( k = 2l-1 \), is given by
\begin{align*}
& 2^{k+1} \ZZ + \Big[ 0 + \sum_{j=1}^{k} ( j \bmod 2 ) \cdot 2^{j} \Big] \\
= {} & 2^{k+1} \ZZ + 2 [ 4^{0} + 4^{1} + \ldots + 4^{l-1} ] \\
= {} & 2^{k+1} \ZZ + \begin{cases}
\frac{ 2 }{ 3 } (2^{k+1} - 1) & \text{for } k = 2l-1 \\
\frac{ 1 }{ 3 } (2^{k+1} - 2) & \text{for } k = 2l 
\end{cases} \, . \qedhere
\end{align*} 
\end{exmpl}

\begin{rem}
Alternating words and leading words have in some sense opposite hole-filling patterns. We will see in Example~\ref{exmpl:AbsenSpWords} and Example~\ref{exmpl:NoSquaresUnSpW} that they also show opposite behaviour regarding the occurrence of squares of words.
\end{rem}

\subsection{Basic properties of simple Toeplitz subshifts}
\label{subsec:STBasicProp}

In the previous subsections we gave two definitions of simple Toeplitz subshifts. Now we show that they really define the same object. Before we do so, recall that Definition~\ref{defi:SToepHoleFill} describes a simple Toeplitz subshift as the orbit closure of a simple Toeplitz word. This raises the natural question whether taking the closure introduces an element in the subshift that is not a simple Toeplitz word itself. In \cite{LiuQu_Simple} it was shown that this is not the case:

\begin{prop}[{\cite[Proposition~2.4]{LiuQu_Simple}}]
\label{prop:AllAreST}
Let \( \InfWord \) be a simple Toeplitz word with coding sequences \( (a_{k}) \), \( (n_{k}) \) and \( (r_{k}) \) and let \( \Subshift( \InfWord ) \) be the associated simple Toeplitz subshift. Then every element \( \widetilde{\InfWord} \in \Subshift( \InfWord ) \) is a simple Toeplitz word with coding sequences \( (a_{k}) \), \( (n_{k}) \) and \( (\widetilde{r}_{k}) \).
\end{prop}

In particular, there exists a sequence \( ( \PBlock{k} )_{k} \) of finite words such that every \( \widetilde{\InfWord} \in \Subshift( \InfWord ) \) can be written as
\begin{equation}
\label{eqn:DecompPBlock}
\widetilde{\InfWord} = \Word{ {} \ldots }{ \PBlock{k} }{ \star }{ \PBlock{k} }{ \star }{ \PBlock{k} }{ \star }{ \PBlock{k} }{ \ldots}
\end{equation}
for every \( k \in \NN_{0} \). Here, the words \( \PBlock{k} \) are defined as in Equation~(\ref{eqn:PBlockRecur}) and each \( \star \) denotes a position that is filled with one of the elements of \( \Alphab_{k+1} = \Set{ a_{j} : j \geq k+1 } \). This decomposition yields the equivalence of our two definitions of simple Toeplitz subshifts:

\begin{prop}
\label{prop:SToeplDefSame}
Let \( (a_{k}) \), \( (n_{k}) \) and \( (r_{k}) \) be coding sequences that define the palindromic blocks \( \PBlock{k} \) and the simple Toeplitz word \( \InfWord \). Let \( \Subshift( \PBlock{\infty} ) \) and \( \Subshift( \InfWord ) \) be the associated subshifts, given by Definition~\ref{defi:SToeplPInfty} and~\ref{defi:SToepHoleFill} respectively. Then \( \Subshift( \PBlock{\infty} ) = \Subshift( \InfWord ) \) holds.
\end{prop}

\begin{proof}
First we show the inclusion \( \Subshift( \PBlock{\infty} ) \subseteq \Subshift( \InfWord ) \): let \( \varrho \in \Subshift( \PBlock{\infty} ) \) and \( J \in \NN \). By definition, the word \( \Restr{\varrho}{-J}{J} \) occurs in \( \PBlock{\infty} \) and thus in \( \PBlock{k} \) for every sufficiently large \( k \). Because of the decomposition \( \InfWord = \Word{ {} \ldots }{ \star }{ \PBlock{k} }{ \star }{ \PBlock{k} }{ \star }{ \ldots {}} \) (see Equation~(\ref{eqn:DecompPBlock})), there exists a number \( k_{J} \in \ZZ \) with \( \Shift^{k_{J}} \Restr{\InfWord}{-J}{J} = \Restr{\varrho}{-J}{J} \). Hence there is a sequence \( (k_{J}) \) with \( \varrho = \lim_{J \to \infty} \Shift^{k_{J}} \InfWord \), which implies \( \varrho \in \Subshift( \InfWord ) \).

For the converse inclusion, we first consider the case that \( \InfWord \) is a normal Toeplitz word. Recall that the periodic words \(  \InfWord_{k}\) tend to \( \InfWord \) as \( k \) tends to infinity. Since \( \InfWord \) is normal, for every \( J \in \NN \) there exists a number \( k_{J} \in \NN \) such that all positions in \( \Restr{\InfWord}{-J}{J} \) are determined in \( \InfWord_{k_{J}} \). Thus, \( \Restr{\InfWord}{-J}{J} \) is contained in \( \PBlock{k_{J}} \), which is a subword of \( \PBlock{\infty} \). This yields \( \InfWord \in \Subshift( \PBlock{\infty} ) \). Now consider the second case that \( \InfWord = \SpW{ \AEv }\) is an extended Toeplitz word. Then, for every sufficiently large \( J \) there is a number \( k_{J} \) with \( a_{k_{J}} = \AEv \) such that \( \Restr{\InfWord}{-J}{J} \) is contained in \( \Word{ \PBlock{ k_{J}-1 } }{ \AEv }{ \PBlock{ k_{J}-1 } } \). Since the latter is contained in \( \PBlock{k_{J}} \) and hence in \( \PBlock{\infty} \), we again obtain \( \InfWord \in \Subshift( \PBlock{\infty} ) \).
\end{proof}

Note that \( \PBlock{\infty} \) depends only on the coding sequences \( (a_{k}) \) and \( ( n_{k}) \), but not on \( (r_{k}) \). Consequently, a simple Toeplitz subshift does not depend on \( (r_{k}) \) either:

\begin{cor}
\label{cor:SameSubsh}
If two simple Toeplitz words \( \InfWord, \widetilde{\InfWord} \) have the coding sequences \( (a_{k}) \), \( (n_{k}) \), \( (r_{k}) \) and \( (a_{k}) \), \( (n_{k}) \), \( (\widetilde{r}_{k}) \) respectively, then \( \Subshift( \InfWord ) = \Subshift( \widetilde{\InfWord} ) \) holds.
\end{cor}

\begin{rem}
The sequence \( (r_{k}) \) describes where the origin is mapped in each step of the hole-filling procedure. Another way to think about this result is therefore that changing \( r_{k} \) shifts the word, which still yields the same orbit closure. Moreover, note that the corollary above is essentially the converse of Proposition~\ref{prop:AllAreST}. By different arguments, it was also obtained in \cite[Proposition~2.3]{LiuQu_Simple}.
\end{rem}

Finally, we comment on the three important properties that we mentioned at the end of Section~\ref{sec:PrelimSubsh}: aperiodicity, minimality and unique ergodicity. As indicated earlier, a coding sequence \( (a_{k} ) \) which is not eventually constant is equivalent to aperiodicity:

\begin{prop}
\label{prop:STAperiod}
A simple Toeplitz subshift \( \Subshift \) is aperiodic if and only if \( \Card{ \AlphabEv } \geq 2 \) holds.
\end{prop}

\begin{proof}
First we assume that \( \Alphab_{k} = \Set{ \AEv } \) holds for all sufficiently large \( k \). According to Equation~(\ref{eqn:DecompPBlock}) every \( \InfWord\in \Subshift \) can be written as \( \Word{ {} \ldots }{ \PBlock{k} }{ \star }{ \PBlock{k} }{ \star }{ \PBlock{k} }{ \ldots} \), with \( \star \in \Alphab_{k+1} = \Set{ \AEv } \). Thus, every \( \InfWord\in \Subshift \) is periodic.

To prove the converse, we use a similar idea as in the proof of \cite[Proposition~6.2]{QRWX_PatCompDDim}: assume that there exists a periodic element \( \InfWord \in \Subshift \) with period length \( r\). Choose \( k_{0} \in \NN \) large enough such that \( r < \Length{ \PBlock{k_{0}} } + 1 \) holds. Note that for every \( k \), the word \( \PBlock{k} = \Word{ \PBlock{k-1} }{ a_{k} }{ \PBlock{k-1} } \) is contained in \( \InfWord \). Moreover, \( \PBlock{k} \) is a prefix of \( \PBlock{k+1} \) for every \( k \). Thus, all words \( \Word{ a_{k} }{ \PBlock{k_{0}} } \) with \( k \geq k_{0}+1 \) are contained in \( \InfWord \). Now \( r \)-periodicity of \( \InfWord \) yields \( a_{k} = \PBlock{k_{0}}(r) \) for all \( k \geq k_{0}+1 \) and thus \( \Card{ \AlphabEv } = 1 \).
\end{proof}

Before we discuss minimality and unique ergodicity, we introduce one more notion: recall that a (not necessarily simple) Toeplitz word \( \InfWord \) can be defined as the limit of periodic words \( \InfWord_{k} \) by a hole-filling procedure. Assume that every position gets eventually filled in this process, that is, for every \( j \in \ZZ \) there exists a number \( k \in \NN \) with \( \InfWord_{k}(j) \in \Alphab \). If additionally 
\[ \frac{\text{number of undetermined positions in } \InfWord_{k} \text{ per period}}{\text{length of the period of } \InfWord_{k}} \xrightarrow{k \to \infty} 0 \]
holds, then the Toeplitz word is called \emph{regular}\index{regular Toeplitz word}\index{word! regular Toeplitz}. The subshift associated to a regular Toeplitz word is always minimal and uniquely ergodic. For Toeplitz words over two letters, this was shown in \cite[corollary to Theorem~5]{JacobsKeane_01Toeplitz}. For Toeplitz words over arbitrary finite alphabets, see for example \cite[Theorem~13.1]{Downa_OdomToepl}. An explicitly constructed example of an irregular Toeplitz word can for instance be found in \cite[Example~4.2]{KasKelLem_BFreeWindArtic}. However, all Toeplitz words that we consider in this thesis will be regular.

\begin{exmpl}
By definition, every position in a normal Toeplitz word is filled eventually. Moreover, we have seen in Proposition~\ref{prop:PeriodWord} that there is exactly one hole per period and that the period length tends to infinity. Thus, every normal Toeplitz word is regular.
\end{exmpl}

\begin{cor}[{\cite[Corollary~2.1]{LiuQu_Simple}}]
\label{cor:STStrictErgod}
Every simple Toeplitz subshift is minimal and uniquely ergodic.
\end{cor}

\begin{proof}
Clearly, a simple Toeplitz subshift \( \Subshift( \InfWord ) \) associated to a normal Toeplitz word \(\InfWord \) is minimal and uniquely ergodic. If \( \InfWord \) is an extended Toeplitz word, then we change the coding sequence \( (r_{k}) \) into a sequence \( (\widetilde{r}_{k}) \), such that the modified simple Toeplitz word \( \widetilde{\InfWord} \) is normal. This can always be done, since every sequence \( (\widetilde{r}_{k}) \) that is not eventually equal to \( 0 \) or eventually equal to \( n_{k}-1 \), yields a normal Toeplitz word. Since \( \InfWord \) and \( \widetilde{\InfWord} \) define the same subshift, this finishes the proof.
\end{proof}

We conclude this chapter with some comments on the relation between simple Toeplitz words and self-similar groups. Roughly speaking, it is the hole-filling procedure that links the two concepts. To be a bit more precise, a self-similar group defines a so-called Schreier graph. This graph can be constructed by a hole-filling process similar to the one for simple Toeplitz words. This yields a connection between Laplacians on these graphs and Jacobi operators on simple Toeplitz words. More details on self-similar groups, Schreier graphs and their relation to subshifts can be found in Appendix~\ref{app:SelfSimGr}. The interested reader is encouraged to pause here and read or skim the appendix for motivating background information. For instance, the connection between the (generalised) Grigorchuk subshift and self-similar groups highlights why the coding sequences were chosen as they were. However, it should be stressed that the content of Appendix~\ref{app:SelfSimGr} is purely supplementary and not necessary for out treatment of simple Toeplitz subshifts. The reader can just as well proceed with Chapter~\ref{chap:Combinat},~\ref{chap:JOSimpToep} and~\ref{chap:UnifCocycCantor} and postpone the reading of Appendix~\ref{app:SelfSimGr} or skip it altogether.

%%%%%%%%%%%%%%%%%%%%%%%%
%%%%%%%%%%%%%%%%%%%%%%%%
%%%%%%%%%%%%%%%%%%%%%%%%
\chapter[Combinatorial properties of simple Toeplitz subshifts][Combinatorial properties]{Combinatorial properties of simple Toeplitz subshifts}
\label{chap:Combinat}

There are numerous notions related to the ``amount of order'' in a subshift. Some of them, all related to the subshift's language, are studied in this chapter. Unless explicitly stated otherwise, we consider a simple Toeplitz subshift \( \Subshift \) which is defined by coding sequences \( (a_{k})_{k \geq 0} \) and \( (n_{k})_{k \geq 0} \). We always assume \( \Card{ \AlphabEv } \geq 2 \), which makes the subshift aperiodic, minimal and uniquely ergodic.

First we discuss the subword complexity, that is, the number of words of a given length. An explicit formula for the complexity is derived, see Proposition~\ref{prop:ComplexI}, Theorem~\ref{thm:ComplexII} and Theorem~\ref{thm:ComplexIII}. Secondly, in Section~\ref{sec:DeBruijn} we describe in detail the de Bruijn graphs (also known as Rauzy graphs) of simple Toeplitz subshifts. They yield an explicit formula for the palindrome complexity (Corollary~\ref{cor:PaliComp}). Afterwards we study repetitivity, that is, the maximal gap length between consecutive occurrences of a word. Again an explicit formula is proved (Theorem~\ref{thm:Repe}). After a brief discussion of \( \alpha \)-repetitivity, we conclude the chapter with a characterisation of the Boshernitzan condition. This result will also play a role in Chapter~\ref{chap:JOSimpToep}, since the Boshernitzan condition determines certain spectral properties of Jacobi operators.

Most of the results in this chapter were published by the author of this thesis in \cite{Sell_SimpToepCombETDS}. In large parts, the text below follows the structure and contents of that article.

%%%%%%%%%%%%%%%%%%%%%%%%
\section{Subword complexity}
\label{sec:Compl}

There are several classes of Toeplitz words for which the complexity has already been studied, for example in \cite{CassKar_ToeplWords} and \cite{Kosk_ComplSuites}. The Toeplitz words in \cite{CassKar_ToeplWords} are obtained from a single word with holes, which is repeatedly inserted into itself. This class contains for example simple Toeplitz words with periodic coding. For them \cite[Theorem~5]{CassKar_ToeplWords} implies that the complexity is dominated by a linear function. For more general Toeplitz constructions it is shown in \cite{Kosk_ComplSuites} that, for every \( r \in \QQ \), there is a Toeplitz word whose complexity grows like \( L^{r} \). Moreover we infer from \cite[Theorem~9 ]{Kosk_ComplSuites} that the complexity of simple Toeplitz words with a bounded coding sequence \( ( n_{k} ) \) is again dominated by a linear function. The aim of this section is to provide an explicit formula for the complexity of a general simple Toeplitz subshift. It will imply that the complexity is dominated by a linear function also in the general case (see Proposition~\ref{prop:ComplQuotient}).

Our main strategy is similar to the one that was used in \cite{GLN_SpectraSchreierAndSO} for the special case of the Grigorchuk subshift: first we prove an upper bound for the complexity at certain lengths. Then we prove a lower bound for the growth rate of the complexity. Together these inequalities determine the complexity. The same technique was also employed in \cite{DKMSS_Regul-Article} for \( l \)-Grigorchuk subshifts.

Recall that \( \Langu{\Subshift} \) denotes the language of a subshift \( \Subshift \), that is, the set of all finite words which occur in elements of \( \Subshift \). The subset of words of length \( L \) is denoted by \( \Langu{\Subshift}_{L} \). The \emph{complexity function}\index{complexity} (or short: \emph{the complexity}) is defined as
\[ \Comp \colon \NN_{0} \rightarrow \NN \;\; , \;\;\; L \mapsto \Card{ \Langu{\Subshift}_{L} } =  \text{``number of words of length } L \text{''} \, .\]
Note that we always have \( \Comp( 0 ) = 1\), since every subshift has exactly one word of length zero (namely the empty word). We define the \emph{growth rate}\index{growth rate of complexity}\index{complexity!growth rate of} (or \emph{difference function}\index{difference function of complexity}) of the complexity by \( \Growth( L ) \DefAs \Comp( L+1 ) - \Comp( L ) \) for \( L \in \NN_{0} \). When counting subwords, it actually suffices to consider subwords of \( \Word{ \PBlock{k} }{ a }{ \PBlock{k} } \):

\begin{prop}
\label{prop:WordsContained}
For every word \( u \in \Langu{ \Subshift } \) of length \( \Length{ u } \leq \Length{ \PBlock{k} }+1 \), there exists a letter \( a \in \Alphab_{k+1} \) such that \( u \) is a subword of \( \Word{ \PBlock{k} }{ a }{ \PBlock{k} } \).
\end{prop}

\begin{proof}
This follows immediately from the decomposition \( \InfWord = \Word{ \ldots }{ \PBlock{k} }{ \star }{ \PBlock{k} }{ \star }{ \PBlock{k} }{ \star }{ \PBlock{k} }{ \ldots } \) with letters \( \star \in \Alphab_{k+1} \), which exists for all \( \InfWord \in \Subshift \) and all \( k \geq 0 \) (see Equation~(\ref{eqn:DecompPBlock})).
\end{proof}

\subsection{Inequalities for the complexity and its growth}
\label{subsec:IneqCompGrow}

In this subsection we derive the necessary combinatorial results for the computation of the complexity, which we postpone to Subsection~\ref{subsec:Complexity}. First we give an upper bound for \( \Comp( \Length{ \PBlock{k} } +1 ) \). Then we prove lower bounds for \( \Growth( L ) \), which yield a lower bound for \( \Comp( \Length{ \PBlock{k} } +1 ) \) through a telescoping sum. Since it agrees with the upper bound, we get an exact value for \( \Comp( \Length{ \PBlock{k} } +1 ) \). We also get exact values for \( \Growth( L ) \), since any growth which is faster than the lower bound would contradict the value that we obtain for \( \Comp( \Length{ \PBlock{k} } +1 ) \).

\begin{prop}
\label{prop:UpBdCompBlock}
For all \( k \geq 0 \) the following inequality holds:
\[ \Comp( \Length{ \PBlock{k} } +1 ) \leq ( \Card{ \Alphab_{k} } - 1 ) \cdot (\Length{ \PBlock{k} }+1) + \CharFkt{ \Alphab_{k+1} }( a_{k} ) \cdot ( \Length{ \PBlock{k-1} }+1 ) \, . \]
\end{prop}

\begin{proof}
First we assume \( k \geq 1 \) (the case \( k = 0 \) is similar, but due to technical details we treat it separately). By Proposition~\ref{prop:WordsContained} it is sufficient to study subwords of length \( \Length{ \PBlock{k} }+1 \) of \( \Word{ \PBlock{k} }{ a }{ \PBlock{k} } \), with \( a \in \Alphab_{k+1} \). For \( a \neq a_{k} \) there are at most \( \Length{ \PBlock{k} } + 1 \) different subwords of length \( \Length{ \PBlock{k} } + 1 \). If \( a_{k} \in \Alphab_{k+1} \) holds, then \( a = a_{k} \) results in at most \( \Length{ \PBlock{k-1} } + 1 \) additional subwords, since the words in
\[ \Word{ \PBlock{k} }{ a_{k} }{ \PBlock{k} } = \big[ \Word{ \PBlock{k-1} }{ a_{k} }{ \PBlock{k-1} }{ \ldots }{ \PBlock{k-1} } \big] \, a_{k} \, \big[ \Word{ \PBlock{k-1} }{ a_{k} }{ \PBlock{k-1} }{ \ldots }{ \PBlock{k-1} } \big] \]
repeat when the second \( \PBlock{k-1} \) is reached. We obtain for \( k \geq 1 \):
\[ \Comp( \Length{ \PBlock{k} }+1 ) \leq \Card{ ( \Alphab_{k+1} \setminus \Set{ a_{k} } ) } \cdot ( \Length{ \PBlock{k} }+1) + \CharFkt{ \Alphab_{k+1}}( a_{k} ) \cdot (\Length{ \PBlock{k-1} } + 1) \, . \]

The case \( k = 0 \) is similar: for every \( a \in \Alphab_{1} \setminus \Set{ a_{0} } \) we have at most \( \Length{ \PBlock{0} } + 1 \) subwords of length \( \Length{ \PBlock{0} } + 1 \) in \( \Word{ \PBlock{0} }{ a }{ \PBlock{0} } \). Moreover, there is exactly one subword of length \( \Length{ \PBlock{0} } + 1 \) in \( \Word{ \PBlock{0} }{ a_{0} }{ \PBlock{0} } \), namely \( a_{0}^{n_{0}} = \PBlock{0} a_{0} \). With the convention \( \Length{ \PBlock{-1} } = 0 \) this yields
\[ \Comp( \Length{ \PBlock{0} }+1 ) \leq \Card{ ( \Alphab_{1} \setminus \Set{ a_{0} } ) } \cdot (\Length{ \PBlock{0} } + 1) + \CharFkt{ \Alphab_{1} }( a_{0} ) \cdot (\Length{ \PBlock{-1} } + 1) \, . \]

Finally we note that \( \Card{ ( \Alphab_{k+1} \setminus \Set{ a_{k} } ) } = \Card{ \Alphab_{k} } - 1 \) holds for all \( k \geq 0 \): if \( a_{k} \in \Alphab_{k+1} \) holds, then we have \( \Alphab_{k+1} = \Alphab_{k} \), which implies \( \Card{ ( \Alphab_{k+1} \setminus \Set{ a_{k} } ) } = \Card{ \Alphab_{k+1} } - 1 =  \Card{ \Alphab_{k} } - 1 \). If \( a_{k} \notin \Alphab_{k+1} \) holds, then we have \( \Alphab_{k+1} = \Alphab_{k} - 1 \), which implies \( \Card{ ( \Alphab_{k+1} \setminus \Set{ a_{k} } ) } = \Card{ \Alphab_{k+1} } =  \Card{ \Alphab_{k} } - 1 \).
\end{proof}

In the following we establish lower bounds on the growth of the complexity by counting the number of possible extensions of finite words. In Proposition~\ref{prop:LowBdCompBlock} we use them to prove a lower bound for the complexity.

\begin{prop}
\label{prop:CompGrow1}
For all \( k \geq 0 \) and \( 0 \leq L \leq \Length{ \PBlock{k} } - \Length{ \PBlock{k-1} } - 1 \) the inequality \( \Growth( L ) = \Comp( L+1 ) - \Comp( L ) \geq \Card{ \Alphab_{k} } - 1 \) holds.
\end{prop}

\begin{proof}
We consider the suffix \( v_{1} \) of \( \PBlock{k} \) that consists of the last \( L \) letters and show that it has at least \( \Card{ \Alphab_{k} } \) different extensions to the right, which yields the claim. First we prove that \( v_{1} \) can be extended by all letters in \( \Alphab_{k+1} \): for every \( a \in \Alphab_{k+1} \) there exists an index \( l \geq k+1 \) with \( a_{l} = a \). Since \( \PBlock{l} \) contains the word \( \Word{ \PBlock{k} }{ a }{ \PBlock{k} } \), the suffix \( v_{1} \) of \( \PBlock{k} \) can be extended by \( a \).

Now there are two cases: for \( a_{k} \in \Alphab_{k+1} \) we obtain \( \Card{ \Alphab_{k+1} } = \Card{ \Alphab_{k} } \) and we are done. For \( a_{k} \notin \Alphab_{k+1} \) we show that \( v_{1} \) can nevertheless be extended by \( a_{k} \): first we assume \( k \geq 1 \). Then \( v_{1} \) is a suffix of the first \( (n_{k} - 1) \cdot ( \Length{ \PBlock{k-1} } + 1 ) - 1 \) letters in \( \PBlock{k} \), since \( \PBlock{k} = \Word{ \PBlock{k-1} }{ \ldots }{ \PBlock{k-1} }{ a_{k} }{ \PBlock{k-1} } \) and \( L \leq \Length{ \PBlock{k} } - ( \Length{ \PBlock{k-1}} + 1) \) hold. Thus \( v_{1} \) can be extended by \( a_{k} \). Now we assume \( k = 0 \) and \( n_{0} > 2 \). In this case, \( v_{1} \) is given by \( v_{1} = a_{0}^{L} \) with \( L \leq n_{0} - 2 \). As \( \PBlock{0} = a_{0}^{n_{0}-1} \) shows, \( v_{1} \) can be extended by \(a_{0} \). Finally we assume \( k = 0 \) and \( n_{0} = 2 \). Then the claim follows immediately from \( \Comp(0) = 1 \) and \( \Comp(1) = \Card{ \Alphab } = \Card{ \Alphab_{0} } \).
\end{proof}

A word which has more than one extension to the right (such as \( v_{1} \) from the proof above), is called a \emph{right special word}\index{right special word}\index{word!right special}. Under certain conditions there exists another right special word \( v_{2} \neq v_{1} \), which further increases the growth. Since we need the inequalities to be sharp to deduce the complexity, we now establish when this is the case.

\begin{prop}
\label{prop:GrowGreater}
Let \( k \geq 1 \) and \( \Length{ \PBlock{k-1} } + 1 \leq L \leq 2 \Length{ \PBlock{k-1} } - \Length{ \PBlock{k-2} } \). If \( a_{k-1} \in \Alphab_{k+1} \) holds, then \( \Growth( L ) \) is at least by one greater than stated in Proposition~\ref{prop:CompGrow1}.
\end{prop}

\begin{proof}
In addition to the right special word \( v_{1} \) from the previous proof, we show the existence of another right special word \( v_{2} \neq v_{1} \) of length \( L \). First note that \( \Word{ \PBlock{k} }{ a_{k-1} }{ \PBlock{k} } \) occurs in the subshift since \( a_{k-1} \in \Alphab_{k+1} \) holds. We define \( v_{2} \) as the suffix of length \( L \) of \( \Word{ \PBlock{k-1} }{ a_{k-1} }{ \PBlock{k-1} } \). Because of \( L \geq \Length{ \PBlock{k-1} } + 1 \), the word \( v_{2} \) ends with \( \Word{ a_{k-1} }{ \PBlock{k-1} } \). Hence it is different from \( v_{1} \), which ends with \( \Word{ a_{k} }{ \PBlock{k-1} } \). We show that \( v_{2} \) can be extended to the right with both, \( a_{k} \) and \( a_{k-1} \): for \( k = 1 \) this follows from \( v_{2} = a_{0}^{L} \), with \( n_{0} \leq L \leq 2 \cdot (n_{0} - 1) \), combined with the decomposition
\[ \Word{ \PBlock{1} }{ a_{0} }{ \PBlock{1} } = \big[ \Word{ a_{0}^{n_{0}-1} }{ a_{1} }{ \ldots }{ a_{1} }{ a_{0}^{n_{0}-1} } \, \big] \, a_{0} \, \big[ \, \Word{ a_{0}^{n_{0}-1} }{ a_{1} }{ \ldots }{ a_{1} }{ a_{0}^{n_{0}-1} } \big] \, . \]
For \( k \geq 2 \) the same result follows from a decomposition of \( \Word{ \PBlock{k-1} }{ a_{k-1} }{ \PBlock{k-1} } \) into \( \PBlock{k-1} \)-blocks, \( \PBlock{k-2} \)-blocks and the letters \( a_{k-1} \) and \( a_{k} \), see Figure~\ref{fig:GrowPlusOne}. 
\end{proof}

\begin{figure}
\centering
\footnotesize
% set parameters for the image
\pgfmathsetmacro{\TikzNumberPkkk}{2} % number of pkkk in image (separated by single letter)
\pgfmathsetmacro{\TikzCodSeqNkkk}{2} % number of pkk in pkkk (has to be >= 2)
\pgfmathsetmacro{\TikzCodSeqNkk}{4} % number of pk in pkk (has to be >= 2)
\pgfmathsetmacro{\TikzDistBlocks}{3} % distance between two blocks in image in pt
\pgfmathsetmacro{\TikzBuchstLength}{9} % length of a single-letter-block in pt
% compute block length from parameters; computation assumes that labels are sufficiently small !!!
\pgfmathsetmacro{\TikzLengthPkkk}{(0.95*\linewidth - (\TikzNumberPkkk - 1) * ( \TikzBuchstLength + 2 * \TikzDistBlocks) ) / \TikzNumberPkkk} % length of pkkk block in pt
\pgfmathsetmacro{\TikzLengthPkk}{(\TikzLengthPkkk - (\TikzCodSeqNkkk - 1) * ( \TikzBuchstLength + 2 * \TikzDistBlocks) ) / \TikzCodSeqNkkk} % length of pkk block in pt
\pgfmathsetmacro{\TikzLengthPk}{(\TikzLengthPkk - (\TikzCodSeqNkk - 1) * ( \TikzBuchstLength + 2 * \TikzDistBlocks) ) / \TikzCodSeqNkk} % length of pk block in pt
\begin{tikzpicture}
[every node/.style ={rectangle, outer sep=0pt, inner sep=0pt, minimum height=0.6cm, draw},
buchst/.style={minimum width=\TikzBuchstLength pt},
pkkk/.style={minimum width=\TikzLengthPkkk pt},
pkk/.style={minimum width=\TikzLengthPkk pt},
pk/.style={minimum width=\TikzLengthPk pt}]
\node [right] at (0,1) [pkk] (A) {\( \PBlock{k-1} \)};
\ifnum 2<\TikzCodSeqNkkk % if first number in foreach-set is greater last number, don't use loop (foreach sets increment to -1)
	\foreach \x in {3,...,\TikzCodSeqNkkk}{
		\node [buchst] (A) [right=\TikzDistBlocks pt of A] {\( c \)};
		\node [pkk] (A) [right=\TikzDistBlocks pt of A] {\( \PBlock{k-1} \)};}
\fi
\node [buchst] (A) [right=\TikzDistBlocks pt of A] {\( c \)};
\foreach \x in {1,...,\TikzCodSeqNkk}{
	\node[pk] (A) [right=\TikzDistBlocks pt of A] {\( \)};
	\node [buchst] (A) [right=\TikzDistBlocks pt of A] {\( b \)};}
\node [pkk] (A) [right=\TikzDistBlocks pt of A] {\( \PBlock{k-1} \)};
\node [buchst] (Fix1) [right=\TikzDistBlocks pt of A] {\( c \)};
\node [pkk] (A) [right=\TikzDistBlocks pt of Fix1] {\( \PBlock{k-1} \)};
\ifnum 2<\TikzCodSeqNkkk
	\foreach \x in {3,...,\TikzCodSeqNkkk}{
		\node [buchst] (A) [right=\TikzDistBlocks pt of A] {\( c \)};
		\node [pkk] (A) [right=\TikzDistBlocks pt of A] {\( \PBlock{k-1} \)};}
\fi
\node [right] at (0,0) [pkkk] (A) {\( \PBlock{k} \)};
\node [buchst] (A) [right=\TikzDistBlocks pt of A] {\( b \)};
\node [pkkk] (A) [right=\TikzDistBlocks pt of A] {\( \PBlock{k} \)};
\node [right] at (0,-1) [pkk] (A) {\( \PBlock{k-1} \)};
\ifnum 2<\TikzCodSeqNkkk
	\foreach \x in {3,...,\TikzCodSeqNkkk}{
		\node [buchst] (A) [right=\TikzDistBlocks pt of A] {\( c \)};
		\node [pkk] (A) [right=\TikzDistBlocks pt of A] {\( \PBlock{k-1} \)};}
\fi
\node [buchst] (A) [right=\TikzDistBlocks pt of A] {\( c \)};
\foreach \x in {1,...,\TikzCodSeqNkk}{
	\node[pk] (A) [right=\TikzDistBlocks pt of A] {\( \)};
	\node [buchst] (A) [right=\TikzDistBlocks pt of A] {\( b \)};}
\node [pk] (A) [right=\TikzDistBlocks pt of A] {\( \)};
\foreach \x in {2,...,\TikzCodSeqNkk}{
	\node [buchst] (A) [right=\TikzDistBlocks pt of A] {\( b \)};
	\node[pk] (A) [right=\TikzDistBlocks pt of A] {\( \)};}
\foreach \x in {2,...,\TikzCodSeqNkkk}{
	\node [buchst] (A) [right=\TikzDistBlocks pt of A] {\( c \)};
	\node[pkk] (A) [right=\TikzDistBlocks pt of A] {\( \PBlock{k-1} \)};}
\node [right] at (0,-2) [pkk] (A) {\( \PBlock{k-1} \)};
\ifnum 2<\TikzCodSeqNkkk
	\foreach \x in {3,...,\TikzCodSeqNkkk}{
		\node [buchst] (A) [right=\TikzDistBlocks pt of A] {\( c \)};
		\node [pkk] (A) [right=\TikzDistBlocks pt of A] {\( \PBlock{k-1} \)};}
\fi
\node [buchst] (A) [right=\TikzDistBlocks pt of A] {\( c \)};
\foreach \x in {2,...,\TikzCodSeqNkk}{
	\node[pk] (A) [right=\TikzDistBlocks pt of A] {\( \)};
	\node [buchst] (A) [right=\TikzDistBlocks pt of A] {\( b \)};}
\node[pkk] (A)  [right=\TikzDistBlocks pt of A] {\( \PBlock{k-1} \)};
\node [buchst] (Fix2) [right=\TikzDistBlocks pt of A] {\( b \)};
\node[pk] (A) [right=\TikzDistBlocks pt of Fix2] {\( \)};
\foreach \x in {2,...,\TikzCodSeqNkkk}{
	\node [buchst] (A) [right=\TikzDistBlocks pt of A] {\( c \)};
	\node [pkk] (A) [right=\TikzDistBlocks pt of A] {\( \PBlock{k-1} \)};}
\node [buchst] (A) [above=0.4 of Fix1] {\( c \)};
\node [left=\TikzDistBlocks pt of A, minimum width=\TikzBuchstLength + \TikzLengthPk + \TikzLengthPkk + 2*\TikzDistBlocks pt] {\( v_{2} \)};
\node [buchst] (A) [below=0.4 of Fix2] {\( b \)};
\node [left=\TikzDistBlocks pt of A, minimum width=\TikzBuchstLength + \TikzLengthPk + \TikzLengthPkk + 2*\TikzDistBlocks pt] {\( v_{2} \)};
\end{tikzpicture}
\normalsize
\caption{The right special word \( v_{2} \) with two possible extensions. To shorten notation we use \( b \DefAs a_{k-1} \) and \( c \DefAs a_{k} \). Unlabelled rectangles denote the block \( \PBlock{k-2} \). \label{fig:GrowPlusOne}}
\end{figure}
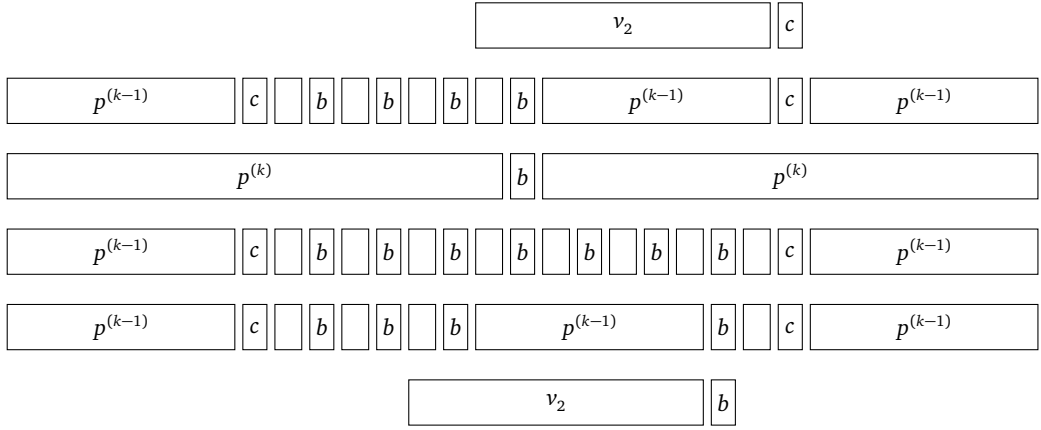

\begin{rem}
\label{rem:An+1An}
Note that \( a_{k-1} \in \Alphab_{k+1} \) holds if and only if \( a_{k-1} \in \Alphab_{k} \) holds, since we always assume \( a_{k} \neq a_{k+1} \).
\end{rem}

Since the initial value \( \Comp( 0 ) = 1 \) is known, the lower bound for the growth yields a lower bound for the complexity in the obvious way. It turns out that it agrees with the upper bound from Proposition~\ref{prop:UpBdCompBlock}:

\begin{prop}
\label{prop:LowBdCompBlock}
For all \( k \geq 0 \) the following inequality holds:
\[ \Comp( \Length{ \PBlock{k} } +1 ) \geq ( \Card{ \Alphab_{k} } - 1 ) \cdot (\Length{ \PBlock{k} }+1) + \CharFkt{ \Alphab_{k+1} } (a_{k}) \cdot ( \Length{ \PBlock{k-1} }+1 ) \, . \]
\end{prop}

\begin{proof}
We proceed by induction and express the complexity as a telescoping sum of the growth. For \( k = 0 \) the growth is bounded from below by Proposition~\ref{prop:CompGrow1}. We obtain
\begin{align*}
\Comp( \Length{ \PBlock{0} }+1 ) &= \Growth( \Length{ \PBlock{0} } ) + \sum \nolimits_{L=0}^{\Length{ \PBlock{0} }-1} \Growth( L ) + \Comp( 0 ) \\
& \geq \Card{ \Alphab_{1} } - 1 + ( \Card{ \Alphab_{0} } - 1 ) \cdot \Length{ \PBlock{0} } + 1 \\
&= (\Card{ \Alphab_{0} } - 1) \cdot ( \Length{ \PBlock{0} } + 1 ) + \CharFkt{ \Alphab_{1} }( a_{0} ) \cdot( \Length{ \PBlock{-1} } + 1 ) \, ,
\end{align*}
where we used \( 1 - \Card{ \Alphab_{0} } + \Card{ \Alphab_{1} } = \CharFkt{ \Alphab_{1} }( a_{0} ) \) and \( \Length{ \PBlock{-1} } = 0 \) in the last line. Now assume that the claim is true for \( k-1 \geq 0 \). The bounds from Proposition~\ref{prop:CompGrow1} and~\ref{prop:GrowGreater} yield
\begin{align*}
\Comp( \Length{ \PBlock{k} }+1 ) &= \sum \nolimits_{ L = \Length{ \PBlock{k-1} }+1 }^{ \Length{ \PBlock{k} } - \Length{ \PBlock{k-1} } -1 }{ \Growth( L ) } + \sum \nolimits_{ L = \Length{ \PBlock{k} } - \Length{ \PBlock{k-1} } }^{ \Length{ \PBlock{k} } }{ \Growth( L ) } + \Comp( \Length{ \PBlock{k-1} }+1) \\
& \geq  (\Card{ \Alphab_{k} } - 1) \cdot ( \Length{ \PBlock{k} } - 2 \Length{ \PBlock{k-1} } -1 ) + (\Card{ \Alphab_{k+1} } - 1) \cdot ( \Length{ \PBlock{k-1}} + 1 ) \\
& \qquad + \CharFkt{ \Alphab_{k+1} }( a_{k-1} ) \cdot (\Length{ \PBlock{k-1} } - \Length{ \PBlock{k-2} }) + \Comp( \Length{ \PBlock{k-1} }+1) \\
& = (\Card{ \Alphab_{k} } - 1) \cdot ( \Length{ \PBlock{k} } + 1 - 2 \Length{ \PBlock{k-1} } -2 ) + (\Card{ \Alphab_{k} } - 1 + \CharFkt{ \Alphab_{k+1} }( a_{k} ) - 1) \cdot \\
& \qquad ( \Length{ \PBlock{k-1} } + 1 ) + \CharFkt{ \Alphab_{k+1} }( a_{k-1} ) \cdot (\Length{ \PBlock{k-1} } - \Length{ \PBlock{k-2} }) + \Comp( \Length{ \PBlock{k-1} }+1) \\
& = (\Card{ \Alphab_{k} } - 1) \cdot ( \Length{ \PBlock{k} } +1 ) + \CharFkt{ \Alphab_{k+1} }( a_{k} ) \cdot ( \Length{ \PBlock{k-1} } + 1 ) \\
& \qquad - \big[ ( \Card{ \Alphab_{k-1} } - 1 ) \cdot ( \Length{ \PBlock{k-1} } + 1 ) + \CharFkt{ \Alphab_{k} }( a_{k-1} ) \cdot ( \Length{ \PBlock{k-2} } + 1) \big] \\
& \qquad + \Comp( \Length{ \PBlock{k-1} }+1 ) \\
& \geq ( \Card{ \Alphab_{k} } - 1 ) \cdot ( \Length{ \PBlock{k} } +1 ) + \CharFkt{ \Alphab_{k+1} }( a_{k} ) \cdot ( \Length{ \PBlock{k-1} } + 1 ) \, ,
\end{align*}
where we used \( \Card{ \Alphab_{k} } + 1 - \CharFkt{ \Alphab_{k} }( a_{k-1} ) =  \Card{ \Alphab_{k-1} } \) in the next to last line and the induction hypothesis in the last line.
\end{proof}

Thus we actually have equality in Proposition~\ref{prop:UpBdCompBlock} and~\ref{prop:LowBdCompBlock}. In particular the subwords of \( \Word{ \PBlock{k} }{ a }{ \PBlock{k} } \), with \( a \in \Alphab_{k+1} \), that were counted in the proof of Proposition~\ref{prop:UpBdCompBlock} are pairwise different. Moreover, the complexity grows exactly by the values given in Proposition~\ref{prop:CompGrow1} and~\ref{prop:GrowGreater}. If the growth was faster, then the value at \( \Length{ \PBlock{k} } + 1 \) would be exceeded. For convenience we collect our results in the following corollary. We use that \( \Card{ \Alphab_{k} } - \Card{ \Alphab_{k+1} } = \CharFkt{ \Cmplmt{ \Alphab_{k+1} } }(a_{k}) \) holds for all \( k \).

\begin{cor}
\label{cor:CompGrowth}
For all \( k \geq 0 \) the complexity at length \( \Length{ \PBlock{k} } + 1 \) is given by
\[ \Comp( \Length{ \PBlock{k} } +1 ) = ( \Card{ \Alphab_{k} } - 1 ) \cdot (\Length{ \PBlock{k} }+1) + \CharFkt{ \Alphab_{k+1} } (a_{k}) \cdot ( \Length{ \PBlock{k-1} }+1 ) \, .\]
The growth rate of the complexity function has the following values:
\begin{align*}
\Growth( L ) = {} & \Card{ \Alphab_{0} } - 1 \quad \text{for } 0 \leq L \leq \Length{ \PBlock{0} } - 1 \, , \\
\Growth( L ) = {} & \Card{ \Alphab_{1} } - 1 \quad \text{for } L = \Length{ \PBlock{0} }  \\
\text{and} \quad \Growth( L ) = {} & \Card{ \Alphab_{k} } - 1\\
& - \begin{cases}
0 & \text{if } \Length{ \PBlock{k-1} }+1 \leq L \leq \Length{ \PBlock{k} } - \Length{ \PBlock{k-1} }-1 \\
\Card{ \Alphab_{k} } - \Card{ \Alphab_{k+1} } & \text{if } \Length{ \PBlock{k} } - \Length{ \PBlock{k-1} } \leq L \leq \Length{ \PBlock{k} } 
\end{cases} \\ & + \begin{cases}
\CharFkt{ \Alphab_{k} }( a_{k-1} ) &  \text{if } \Length{ \PBlock{k-1} } + 1 \leq L \leq 2 \Length{ \PBlock{k-1} } - \Length{ \PBlock{k-2} } \\
0 & \text{if } 2 \Length{ \PBlock{k-1} } - \Length{ \PBlock{k-2} } + 1 \leq L \leq \Length{ \PBlock{k} }
\end{cases}
\end{align*}
for \( k \geq 1 \) and \( \Length{ \PBlock{k-1} } + 1 \leq L \leq \Length{ \PBlock{k} } \).
\end{cor}

\subsection{The complexity of simple Toeplitz subshifts}
\label{subsec:Complexity}

In the following we compute the complexity from the growth rate in Corollary~\ref{cor:CompGrowth}. Since \( \Growth(L ) \) is defined piecewise, we have to distinguish three cases. Accordingly we split the result into three statements. First of all we treat the case \( L \leq \Length{ \PBlock{0} } + 1 \), where the additional increase of Proposition~\ref{prop:GrowGreater} does not occur.

\begin{prop}[complexity function I]
\label{prop:ComplexI}
The first values of the complexity function are
\begin{align*}
\Comp( L ) &= ( \Card{ \Alphab_{0} } - 1 ) L + 1 & & \hspace{-4em} \text{for } 0 \leq L \leq \Length{ \PBlock{0} } \\
\text{and} \quad \Comp( L ) &= ( \Card{ \Alphab_{0} } - 1 ) L + \CharFkt{ \Alphab_{1} }( a_{0} ) & & \hspace{-4em} \text{for } L = \Length{ \PBlock{0} } + 1 \, .
\end{align*}
\end{prop}

\begin{proof}
By Corollary~\ref{cor:CompGrowth} we have \( \Growth( L ) = \Alphab_{0} - 1 \) for \( 0 \leq L \leq \Length{ \PBlock{0} } - 1 \). Now \( \Comp( 0)  = 1 \) yields
\[ \Comp( L ) = \sum \nolimits_{l = 0}^{L-1} \Growth( l ) + \Comp( 0 ) = L ( \Alphab_{0} - 1 ) + 1\]
for \( 0 \leq L \leq \Length{ \PBlock{0} } \). For \( L = \Length{ \PBlock{0} }+1 \) the claim is already contained in Corollary~\ref{cor:CompGrowth}.
\end{proof}

Now we treat the case \( L \geq \Length{ \PBlock{0} }+2 \). Note that the second and the third summand in the growth function are defined piecewise. If \( L \) is close to \( \Length{ \PBlock{k-1} } + 1 \), then the first case applies in both summands. If \( L \) is close to \( \Length{ \PBlock{k} } \), then the second case applies in both summands. For intermediate \( L \) it depends on \( n_{k} \) in which summand the case changes first. This causes the growth function to be different for \( n_{k} = 2 \) and \( n_{k} > 2 \). First, we consider \( n_{k} = 2 \):

\begin{thm}[complexity function II]
\label{thm:ComplexII}
For \( k \geq 1 \) with \( n_{k}  = 2 \), the complexity function in the range \( \Length{ \PBlock{k-1} } + 2 \leq L \leq \Length{ \PBlock{k} } + 1 \) is given by
\begin{align*}
\Comp( L ) &= ( \Card{ \Alphab_{k+1} } - 1 ) L + ( \Card{ \Alphab_{k-1} } - \Card{ \Alphab_{k+1} } ) ( \Length{ \PBlock{k-1} } + 1) \\
& \quad + \CharFkt{ \Alphab_{k} }( a_{k-1} ) \cdot \begin{cases}
- \Length{ \PBlock{k-1} } + \Length{ \PBlock{k-2} } + L & \text{if } \Length{ \PBlock{k-1} } + 2 \leq L \leq \Length{ \PBlock{k} } - \Length{ \PBlock{k-2}}  \\
\Length{ \PBlock{k-1} } + 1 & \text{if } \Length{ \PBlock{k} } - \Length{ \PBlock{k-2} } + 1 \leq L \leq \Length{ \PBlock{k} } + 1
\end{cases} \, .
\end{align*}
\end{thm}

\begin{proof}
Because of \( \Length{ \PBlock{k} } = 2 \Length{ \PBlock{k-1} } + 1 \), the growth from Corollary~\ref{cor:CompGrowth} simplifies to
\[ \Growth( L ) = \Card{ \Alphab_{k+1} } - 1 + \begin{cases}
\CharFkt{ \Alphab_{k} }( a_{k-1} ) & \text{if } \Length{ \PBlock{k-1} } + 1 \leq L \leq \Length{ \PBlock{k} } - \Length{ \PBlock{k-2} } - 1 \\
0 & \text{if } \Length{ \PBlock{k} } - \Length{ \PBlock{k-2} } \leq L \leq \Length{ \PBlock{k} }
\end{cases} \, .\]
Together, the value of \( \Comp( \Length{ \PBlock{k-1} } + 1 ) \) and the growth yield the complexity. First we consider the range \( \Length{ \PBlock{k-1} } + 2 \leq L \leq \Length{ \PBlock{k} } - \Length{ \PBlock{k-2} } \):
\begin{align*}
\Comp( L ) &= \Comp( \Length{ \PBlock{k-1} } + 1 ) + \sum \nolimits_{l = \Length{ \PBlock{k-1} }+1}^{L-1} \Growth( l ) \\
&= ( \Card{ \Alphab_{k-1} } - 1 ) \cdot (\Length{ \PBlock{k-1} }+1) + \CharFkt{ \Alphab_{k} }( a_{k-1} ) \cdot ( \Length{ \PBlock{k-2} }+1 ) \\
& \qquad + (L - \Length{ \PBlock{k-1} } - 1) \cdot ( \Card{ \Alphab_{k+1} } - 1 + \CharFkt{ \Alphab_{k} }( a_{k-1}) ) \\
&= ( \Card{ \Alphab_{k-1} } - \Card{ \Alphab_{k+1} } ) ( \Length{ \PBlock{k-1} } + 1 ) - \CharFkt{ \Alphab_{k} }( a_{k-1} ) \cdot ( \Length{ \PBlock{k-1} } - \Length{ \PBlock{k-2} } ) \\
& \qquad + ( \Card{ \Alphab_{k+1} } - 1 + \CharFkt{ \Alphab_{k} }( a_{k-1} ) ) L \, .
\end{align*}
Similarly we obtain for \( \Length{ \PBlock{k} } - \Length{ \PBlock{k-2} } + 1 \leq L \leq \Length{ \PBlock{k} } + 1 \):
\begin{align*}
\Comp( L ) &= \Comp( \Length{ \PBlock{k-1} } + 1 ) + \sum \nolimits_{ j = \Length{ \PBlock{k-1} }+1 }^{L-1} \Growth( j ) \\
&= ( \Card{ \Alphab_{k-1} } - 1 ) \cdot ( \Length{ \PBlock{k-1} } + 1 ) + \CharFkt{ \Alphab_{k} } ( a_{k-1} ) \cdot ( \Length{ \PBlock{k-2} } + 1 ) \\
& \qquad + ( \Card{ \Alphab_{k+1} } - 1 ) ( L - 1 - \Length{ \PBlock{k-1} } ) + \CharFkt{ \Alphab_{k} }( a_{k-1} ) ( \Length{ \PBlock{k-1} } - \Length{ \PBlock{k-2} } ) \\
&= ( \Card{ \Alphab_{k-1} } - \Card{ \Alphab_{k+1} } ) ( \Length{ \PBlock{k-1} } + 1) + \CharFkt{ \Alphab_{k} }( a_{k-1} ) \cdot ( \Length{ \PBlock{k-1} } + 1 ) \\
& \qquad + ( \Card{ \Alphab_{k+1} } - 1 ) L \, . \qedhere
\end{align*} 
\end{proof}

In the next theorem we treat the remaining case \( \Length{ \PBlock{k-1} } + 2 \leq L \leq \Length{ \PBlock{k} } + 1 \) with \( n_{k} > 2 \). Except for different values of the growth function, we proceed exactly as for \( n_{k} = 2 \). There are three different intervals to be considered. To shorten notation we denote them as follows:
\begin{align*}
I_{1} &\DefAs \ZZ \cap \big[ \Length{ \PBlock{k-1} } + 2 \, , \, 2 \Length{ \PBlock{k-1} } - \Length{ \PBlock{k-2} } + 1 \big] \, ,\\
I_{2} &\DefAs\ZZ \cap \big[ 2 \Length{ \PBlock{k-1} } - \Length{ \PBlock{k-2} } + 2 \, , \, \Length{ \PBlock{k} } - \Length{ \PBlock{k-1} } \big] \, ,\\
I_{3} &\DefAs \ZZ \cap \big[ \Length{ \PBlock{k} } - \Length{ \PBlock{k-1} } + 1 \, , \, \Length{ \PBlock{k} } +1 \big] \, .
\end{align*}

\begin{thm}[complexity function III]
\label{thm:ComplexIII}
For \( k \geq 1 \) with \( n_{k} > 2 \), the complexity function in the range \( \Length{ \PBlock{k-1} } + 2 \leq L \leq \Length{ \PBlock{k} } + 1 \) is given by
\begin{align*}
\Comp( L ) = {} & ( \Length{ \PBlock{k-1} } + 1 )  + ( \Card{ \Alphab_{k} } - 1 ) L \\
& {} + \begin{cases}
\CharFkt{ \Alphab_{k} }( a_{k-1} )( L - 2 \Length{ \PBlock{k-1} } + \Length{ \PBlock{k-2} } - 1 ) & \text{if } L \in I_{1} \\
0 & \text{if } L \in I_{2} \\
- \CharFkt{ \Cmplmt{ \Alphab_{k+1} } }( a_{k} ) ( L - \Length{ \PBlock{k} } + \Length{ \PBlock{k-1} } ) & \text{if } L \in I_{3}
\end{cases} \, .
\end{align*}
\end{thm}

\begin{proof}
The assumption \( n_{k} > 2 \) yields \( \Length{ \PBlock{ k } } - \Length{ \PBlock{k-1} } = ( n_{k} - 1 ) ( \Length{ \PBlock{k-1} } + 1 ) > 2 \Length{ \PBlock{k-1} } - \Length{ \PBlock{ k-2 } } \). Thus the growth from Corollary~\ref{cor:CompGrowth} simplifies to
\[ \Growth( L ) = \Card{ \Alphab_{k} } - 1 + \begin{cases}
\CharFkt{ \Alphab_{k} }( a_{k-1} ) & \text{if } L+1 \in I_{1} \\
0 & \text{if } L+1 \in I_{2} \\
- ( \Card{ \Alphab_{k} } - \Card{ \Alphab_{k+1} }  ) & \text{if } L+1 \in I_{3}
\end{cases} \, .\]
We compute the complexity from the growth and the value of \( \Comp( \Length{ \PBlock{k-1} } + 1 ) \):
\begin{align*}
& \Comp( L ) = \Comp( \Length{ \PBlock{k-1} } + 1 ) + \sum \nolimits_{l = \Length{ \PBlock{k-1} } + 1 }^{L-1} \Growth( l ) \\
&= ( \Card{ \Alphab_{k-1} } - 1 ) ( \Length{ \PBlock{k-1} } + 1 ) + \CharFkt{ \Alphab_{k} }( a_{k-1} ) ( \Length{ \PBlock{k-2} } + 1 ) + ( \Card{ \Alphab_{k} } - 1 ) ( L - \Length{ \PBlock{k-1} } - 1 ) \\
& \quad + \begin{cases}
\CharFkt{ \Alphab_{k} }( a_{k-1} ) ( L - \Length{ \PBlock{k-1} } - 1 ) & \text{if } L \in I_{1} \\
\CharFkt{ \Alphab_{k} }( a_{k-1} ) ( \Length{ \PBlock{k-1} } - \Length{ \PBlock{k-2} } ) & \text{if } L \in I_{2} \\
\CharFkt{ \Alphab_{k} }( a_{k-1} ) ( \Length{ \PBlock{k-1} } - \Length{ \PBlock{k-2} } ) - ( \Card{ \Alphab_{k} } - \Card{ \Alphab_{k+1} } ) ( L - \Length{ \PBlock{k} } + \Length{ \PBlock{k-1} } ) & \text{if } L \in I_{3}
\end{cases}\\
&= ( \Card{ \Alphab_{k-1} } - \Card{ \Alphab_{k} } ) ( \Length{ \PBlock{k-1} } + 1 ) + \CharFkt{ \Alphab_{k} }( a_{k-1} ) ( \Length{ \PBlock{k-1} } + 1 ) + ( \Card{ \Alphab_{k} } - 1) L \\
& \quad + \begin{cases}
\CharFkt{ \Alphab_{k} }( a_{k-1} )( L - 2 \Length{ \PBlock{k-1} } + \Length{ \PBlock{k-2} } - 1 ) & \text{if } L \in I_{1} \\
0 & \text{if } L \in I_{2} \\
- ( \Card{ \Alphab_{k} } - \Card{ \Alphab_{k+1} } ) ( L - \Length{ \PBlock{k} } + \Length{ \PBlock{k-1} } ) & \text{if } L \in I_{3}
\end{cases} \\
&= ( \Length{ \PBlock{k-1} } + 1 )  + ( \Card{ \Alphab_{k} } - 1) L + \begin{cases}
\CharFkt{ \Alphab_{k} }( a_{k-1} )( L - 2 \Length{ \PBlock{k-1} } + \Length{ \PBlock{k-2} } - 1 ) & \text{if } L \in I_{1} \\
0 & \text{if } L \in I_{2} \\
- ( \Card{ \Alphab_{k} } - \Card{ \Alphab_{k+1} } ) ( L - \Length{ \PBlock{k} } + \Length{ \PBlock{k-1} } ) & \text{if } L \in I_{3}
\end{cases} \, .
\end{align*}
Now the relation \( \Card{ \Alphab_{k} } - \Card{ \Alphab_{k+1} } = \CharFkt{ \Cmplmt{ \Alphab_{k+1} } }( a_{k} ) \) yields the claim.
\end{proof}

The previous results completely describe the complexity function. As we can see, its value depends on whether or not \( a_{k-1} \) is in \( \Alphab_{k} \). Since the alphabet is finite, we have \( \Alphab_{k} = \AlphabEv \) for all sufficiently large \( k \). In this case, the complexity function simplifies and even the difference between \( n_{k} = 2 \) and \( n_{k} > 2 \) vanishes:

\begin{cor}
\label{cor:CompLargeL}
Let \( k \) be large enough such that \( \Alphab_{k-1} = \AlphabEv \) holds. Then the complexity function in the range \( \Length{ \PBlock{k-1} } + 2 \leq L \leq \Length{ \PBlock{k} } + 1 \) is given by
\begin{align*}
\Comp( L ) &= \begin{cases}
\Card{ \AlphabEv } \cdot L - \Length{ \PBlock{k-1} } + \Length{ \PBlock{k-2} } & \text{if } \Length{ \PBlock{k-1} } + 2 \leq L \leq 2 \Length{ \PBlock{k-1} } - \Length{ \PBlock{k-2} } + 1 \\
( \Card{ \AlphabEv } - 1 ) \cdot L + \Length{ \PBlock{k-1} } + 1 & \text{if } 2 \Length{ \PBlock{k-1} } - \Length{ \PBlock{k-2} } + 2 \leq L \leq \Length{ \PBlock{k} } + 1
\end{cases} \, .
\end{align*}
\end{cor}

\begin{rem}
A different extensively studied class of subshifts are \emph{Sturmian subshifts}\index{Sturmian!subshift}\index{subshift!Sturmian}. Their complexity is given by \( \Comp( L ) = L+1 \) for all \( L \geq 0 \), which indicates the high amount of order in Sturmian subshifts. In fact, by the famous Morse/Hedlund theorem (\cite[Theorem~7.4]{MorseHedl_SymbDyn}) this is the lowest possible complexity for an aperiodic subshift. As the previous corollary shows, the complexity of simple Toeplitz subshifts grows linearly with a slope that alternates between \( \Card{ \AlphabEv } \) and  \( \Card{ \AlphabEv } - 1 \). In particular this proves that simple Toeplitz subshifts and Sturmian subshifts are disjoint classes. In Chapter~\ref{chap:UnifCocycCantor} we will see that both, simple Toeplitz subshifts and Sturmian subshifts, are examples of so-called \LCond{}-subshifts. For the interested reader, Appendix~\ref{app:Sturm} provides further details and references about Sturmian subshifts.
\end{rem}

We illustrate the complexity of simple Toeplitz subshifts through two example cases, namely the Grigorchuk subshift and the non-(B) example. For their definitions, see Example~\ref{exmpl:GrigSubshiftDef} and~\ref{exmpl:BoshSubshiftDef} respectively.

\begin{exmpl}[Grigorchuk subshift]
We have \( \Length{ \PBlock{k} } = 2^{k+1} - 1 \), \( \Card{ \Alphab_{0} } = 4 \), \( \Card{ \AlphabEv } = 3 \) and \( \Alphab_{k} = \AlphabEv \) for all \( k \geq 1 \). In particular \( \CharFkt{ \Alphab_{1} }( a_{0} ) = 0 \) holds. From Proposition~\ref{prop:ComplexI}, Theorem~\ref{thm:ComplexII} and Corollary~\ref{cor:CompLargeL} we recover precisely the complexity function that was given in \cite[Theorem~1]{GLN_Combinat}:
\begin{align*}
\Comp( L ) &= 3L + 1 \quad \text{for } 0 \leq L \leq 1 \, ,\\
\Comp( L ) &= 3 L \hphantom{{} + 1} \quad \text{for } L = 2 \, , \\
\Comp( L ) &= 2 L + 2 \quad \text{for } 3 \leq L \leq 4 \\
\text{and} \quad \Comp( L ) &= \begin{cases}
3 \cdot L - 2^{k} + 2^{k-1} & \text{if } 2^{k} + 1 \leq L \leq 2^{k+1} -  2^{k-1} \\
2 \cdot L + 2^{k} & \text{if } 2^{k+1} - 2^{k-1} + 1 \leq L \leq 2^{k+1}
\end{cases} \quad \text{for } k \geq 2 \, .
\end{align*}
 A plot of the complexity function, with its typical pattern of alternating slope, is shown in Figure~\ref{fig:GrigCompl}.
\end{exmpl}

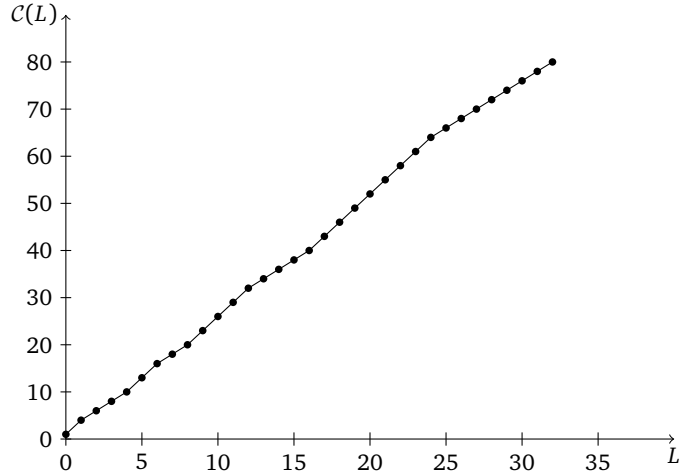
\begin{figure}
\centering
\footnotesize
\begin{tikzpicture}
[xscale=1/5, yscale=1/16, data/.style={circle, outer sep=0pt, inner sep=1 pt, fill=black}]
\node at ( 0 , 1 )[data]{};
\node at ( 1 , 4 )[data]{};
\node at ( 2 , 6 )[data]{};
\node at ( 3 , 8 )[data]{};
\node at ( 4 , 10 )[data]{};
\draw ( 0 , 1 ) -- ( 1 , 4 );
\draw ( 1, 4 ) -- ( 4, 10 );
\foreach \k in {2,...,4}{
	\pgfmathsetmacro{\xOneStart}{2^\k + 1}
	\pgfmathsetmacro{\xOneEnd}{2^(\k+1) - 2^(\k-1)}
	\foreach \L in {\xOneStart,...,\xOneEnd}{
		\pgfmathsetmacro{\Comp}{3*\L - 2^\k + 2^(\k-1)}
		\node at ( \L , \Comp )[data]{};
		\draw ( \L-1, \Comp-3) --(\L , \Comp);
	}
	\pgfmathsetmacro{\xTwoStart}{2^(\k+1)-2^(\k-1)+1}
	\pgfmathsetmacro{\xTwoEnd}{2^(\k+1)}
	\foreach \L in {\xTwoStart,...,\xTwoEnd}{
		\pgfmathsetmacro{\Comp}{2*\L + 2^\k}
		\node at ( \L , \Comp )[data]{};
		\draw ( \L-1, \Comp-2) --(\L , \Comp);
	}
}
\draw[->] ( 0 , 0 ) -- ( 40 , 0 ) node[below] {\( L \)};
\draw[->] ( 0 , 0 ) -- ( 0 , 90 )node[left] {\( \Comp( L ) \)};
\foreach \x in {0,5,...,35}
\draw ( \x , 40pt ) -- ( \x , -40pt ) node[below] {$\x$};
\foreach \y in {0,10,...,80}
\draw ( 10pt , \y ) -- ( -10pt , \y ) node[left] {$\y$};
\end{tikzpicture}
\normalsize
\caption{Plot of the complexity function of the Grigorchuk subshift.\label{fig:GrigCompl}}
\end{figure}

\vspace{-2ex} % Figure adds space between examples, even when LaTeX floats it somewhere else
\begin{exmpl}[non-(B) example]
We have \( \Length{ \PBlock{k} } = 2^{k+1} - 1 \), \( \Alphab_{k} = \AlphabEv \) for all \( k \) and \( \Card{ \Alphab_{0} } = \Card{ \AlphabEv } = 4 \). Thus Proposition~\ref{prop:ComplexI} and Corollary~\ref{cor:CompLargeL} yield
\begin{align*}
\Comp( L ) &= 3 L + 1 \quad \text{for } 0 \leq L \leq 2 \\
\text{and} \quad \Comp( L ) &= \begin{cases}
4 \cdot L - 2^{k} + 2^{k-1} & \text{if } 2^{k} + 1 \leq L \leq 2^{k+1} - 2^{k-1} \\
3 \cdot L + 2^{k} & \text{if } 2^{k+1} - 2^{k-1} +1 \leq L \leq 2^{k+1}
\end{cases} \quad \text{for } k \geq 1 \, . \qedhere
\end{align*}
\end{exmpl}

As we saw above, the complexity function grows linearly with alternating slope for all sufficiently large \( L \). More precisely, it is bounded from below and above by the linear functions \(  ( \Card{ \AlphabEv } - 1 ) L \) and \( ( \Card{ \AlphabEv } - \frac{ 1 }{ 3 } ) L \,\):

\begin{prop}
\label{prop:ComplQuotient}
Let \( k \) be large enough that \( \Alphab_{k-1} = \AlphabEv \) holds. For \( \Length{ \PBlock{k-1} } + 2 \leq L \leq \Length{ \PBlock{k} }+1 \) the quotient \( \frac{ \Comp( L ) }{ L } \) lies between
\[ \Card{ \AlphabEv } - 1 < \min \Big\{ \Card{ \AlphabEv - \frac{ n_{k} - 1 }{ n_{k} } \, , \, \AlphabEv - \frac{ n_{k-1} - 1 }{ n_{k-1} } } \Big\} \leq \min_{ \Length{ \PBlock{k-1} } + 2 \leq L \leq \Length{ \PBlock{k} } + 1 } \frac{ \Comp( L ) }{ L } \]
and
\[ \max_{ \Length{ \PBlock{k-1} } + 2 \leq L \leq \Length{ \PBlock{k} } + 1 } \frac{ \Comp( L ) }{ L } = \Card{ \AlphabEv } - \frac{ n_{k-1} - 1 }{ 2 n_{k-1} - 1 } \leq \Card{ \AlphabEv } - \frac{ 1 }{ 3 } \, . \]
The maximum is attained at \( L = 2 \Length{ \PBlock{k-1} } - \Length{ \PBlock{k-2} } + 1 \).
\end{prop}

\begin{proof}
For \( \Length{ \PBlock{k-1} } + 2 \leq L \leq \Length{ \PBlock{k} }+1 \), Corollary~\ref{cor:CompLargeL} yields
\begin{align*}
\frac{ \Comp( L ) }{ L } &= \begin{cases}
\Card{ \AlphabEv } - \frac{ \Length{ \PBlock{k-1} } - \Length{ \PBlock{k-2} } }{ L } & \text{if } \Length{ \PBlock{k-1} } + 2 \leq L \leq 2 \Length{ \PBlock{k-1} } - \Length{ \PBlock{k-2} } + 1 \\
\Card{ \AlphabEv } - 1 + \frac{ \Length{ \PBlock{k-1} } + 1 }{ L } & \text{if } 2 \Length{ \PBlock{k-1} } - \Length{ \PBlock{k-2} } + 2 \leq L \leq \Length{ \PBlock{k} } + 1
\end{cases} \, .
\end{align*}
Thus the maximum is attained either at \( L_{1} \DefAs 2 \Length{ \PBlock{k-1} } - \Length{ \PBlock{k-2} } + 1 \) or at \( L_{2} \DefAs 2 \Length{ \PBlock{k-1} } - \Length{ \PBlock{k-2} } + 2 \). A short computation shows that the value at \( L_{1} \) is greater:
\[ \frac{ \Comp( L_{1} ) }{ L_{1} } = \Card{ \AlphabEv } - \frac{ \Length{ \PBlock{k-1} } - \Length{ \PBlock{k-2} } }{ 2 \Length{ \PBlock{k-1} } - \Length{ \PBlock{k-2} } + 1 } = \Card{ \AlphabEv } - 1 + \frac{ \Length{ \PBlock{k-1} } + 1 }{ 2 \Length{ \PBlock{k-1} } - \Length{ \PBlock{k-2} } + 1 } >  \frac{ \Comp( L_{2} ) }{ L_{2} } \, . \]
Now reducing the fraction on the left hand side by \( \Length{ \PBlock{k-2} } + 1 \) yields the claim. Similarly the minimum is attained either at \( L_{3} \DefAs \Length{ \PBlock{k-1} } + 2 \) or at \( L_{4} \DefAs \Length{ \PBlock{k} } + 1 \). The claim follows from another direct computation:
\begin{align*}
\frac{ \Comp( L_{3} ) }{ L_{3} } &= \Card{ \AlphabEv } - \frac{ \Length{ \PBlock{k-1} } - \Length{ \PBlock{k-2} } }{ \Length{ \PBlock{k-1} } + 2 } > \Card{ \AlphabEv } - \frac{ ( \Length{ \PBlock{k-1} } +1 )( 1 - \frac{ 1 }{ n_{k-1} } ) }{ \Length{ \PBlock{k-1} } + 1 } = \Card{ \AlphabEv } - \frac{ n_{k-1} - 1 }{ n_{k-1} } \, , \\
\frac{ \Comp( L_{4} ) }{ L_{4} } &= \Card{ \AlphabEv } - 1 + \frac{ \Length{ \PBlock{k-1} } + 1 }{ \Length{ \PBlock{k} } + 1 } = \Card{ \AlphabEv } - 1 + \frac{ 1 }{ n_{k} } = \Card{ \AlphabEv } - \frac{ n_{k} - 1 }{ n_{k} } \, . \qedhere
\end{align*}
\end{proof}

\begin{rem}
\label{rem:STCompUniErg}
We have already seen that simple Toeplitz subshifts are uniquely ergodic due to their regularity (Corollary~\ref{cor:STStrictErgod}). If \( \Card{ \AlphabEv } \leq 3 \) holds, we can now give an alternative proof: clearly the previous proposition yields
\[ \Card{ \AlphabEv } - 1 \leq \liminf_{ L \to \infty } \frac{ \Comp( L ) }{ L } \leq \limsup_{ L \to \infty } \frac{ \Comp( L ) }{ L } <  \Card{ \AlphabEv } \, . \]
 By \cite[Theorem~1.5]{Boshernitzan_UniErgodic} every minimal subshift with \( \limsup_{ L \to \infty } \frac{ \Comp( L ) }{ L } < 3 \) is uniquely ergodic.
\end{rem}

%%%%%%%%%%%%%%%%%%%%%%%%
\section{De Bruijn graphs and palindrome complexity}
\label{sec:DeBruijn}

In this section we investigate a sequence of graphs which encodes the language of the subshift. They are called \emph{de Bruijn graphs}\index{de Bruijn graph}\index{graph!de Bruijn graph} and were introduced by de~Bruijn (\cite{deBruijn_Graphs}) and Good (\cite{Good_deBrGraphs}) to illustrate how certain finite words can be concatenated. Since they were used by Rauzy to analyse combinatorial properties of infinite words (\cite{Rauzy_deBrGraphs}), they are also known as \emph{Rauzy graphs}\index{Rauzy graph}\index{graph!Rauzy graph}. Each graph in the sequence represents all words of a certain length, together with their extensions to the left and to the right. More precisely, the \( L \)-th de Bruijn graph \( \Debruijn{L} = ( \Vertices{L}, \Edges{L} ) \) of a subshift \( \Subshift \) is a directed graph with vertices \( \Vertices{L} \DefAs \Langu{ \Subshift }_{L} \). There is an edge \( (u, v) \in \Edges{L} \subseteq \Vertices{L} \times \Vertices{L} \) from \( u = \Word{ u(1) }{ \ldots }{ u(L) }\) to \( v = \Word{ v(1) }{ \ldots }{ v(L) } \) if and only if the words \( u \) and \( v \) satisfy
\[ \Word{ u(2) }{ \ldots }{ u(L) } = \Word{ v(1) }{ \ldots }{ v(L-1) } \quad  \text{and} \quad \Word{ u }{ v(L) } = \Word{ u(1) }{ v } \in \Vertices{L+1} \, .\]
 
In the previous section we analysed the right special words of simple Toeplitz subshifts. In the de Bruijn graph they correspond to branching points. Thus the results from Section~\ref{sec:Compl} allow us to deduce a detailed description of the graphs, see Subsection~\ref{subsec:deBruijn}. It turns out that a reflection of the de Bruijn graphs corresponds to a reflection of finite words. Therefore the graphs yield an explicit formula for the palindrome complexity. This is discussed in the second subsection.

\subsection{De Bruijn graphs of simple Toeplitz subshifts}
\label{subsec:deBruijn}

As in Proposition~\ref{prop:CompGrow1}, let \( v_{1} \) denote the suffix of length \( L \)\( \PBlock{k} \). For \( L \) with \( L \leq \Length{ \PBlock{k} } - \Length{ \PBlock{k-1} } - 1 \) we know that \( v_{1} \) can be extended by all letters in \( \Alphab_{k} \).

First we treat the case \( 1 \leq L \leq \Length{ \PBlock{ 0 } } = n_{0} - 1 \). The extension of \( v_{1} \) by a letter \( b \in \Alphab_{0} \setminus \Set{ a_{0} } \) yields an edge from \( v_{1} = a_{0}^{L} \) to \( \Word{ a_{0}^{L-1} }{ b } \). Since \( v_{1} \) is a prefix of \( \PBlock{0} \) as well, the word \( \Word{ v_{1} }{ b }{ v_{1} } \) exists and we obtain a loop with \( L+1 \) edges from \( v_{1} \) through \( \Word{ a_{0}^{L-1} }{ b } \) back to \( v_{1} \). Clearly the words in this loop are pairwise different and different from the words in all similar loops for letters \( \widetilde{b} \neq b \). For \( L \leq \Length{ \PBlock{0} } - \Length{ \PBlock{-1} } - 1 = n_{0} - 2  \) the extension of \( v_{1} \) by \( a_{0} \) exists and yields an edge from \( v_{1} \) to \( v_{1} \). For \( L = n_{0} - 1 \), the extension by \( a_{0} \) exists if and only if \( a_{0} \in \Alphab_{1} \) holds. Combined, these considerations yield the graph in Figure~\ref{fig:deBruijn1ToL0}.

\begin{figure}
\centering
\footnotesize
\begin{tikzpicture}
[ every path/.style = {shorten <=1pt, shorten >=1pt, >=stealth},
  vertex/.style = {circle, minimum size=17pt, inner sep=2pt, draw}]
\node (dotstop) {\( \ldots \)};
\node [vertex] (vtopleft)  [left = of dotstop] {};
\node [vertex] (vtopright) [right = of dotstop] {};
\node (dotsmiddle) [below = of dotstop] {\( \vdots \)};
\node (dotsbottom) [below = of dotsmiddle] {\( \ldots \)};
\node [vertex] (vbottomleft) [left = of dotsbottom] {};
\node [vertex] (vbottomright) [right = of dotsbottom] {};
\node [vertex] (w0w0) [below = of dotsbottom] {\( v_{1} \)};
\graph{
(w0w0) -> [controls ={ +(10em, -5ex) and + (8em, 0)}] (vtopright.east);
(vtopright.west)  -> (dotstop) -> (vtopleft.east);
(w0w0) <-  [controls ={ +(-10em, -5ex) and + (-8em, 0ex)}] (vtopleft.west);
(w0w0) -> [controls ={ +(10em, 0) and + (10em, 0)}, edge node ={ node [at end, above, xshift = -1em, yshift = 1.5ex] {\( L + 1 \) edges on each arc} }] (dotsmiddle.north east);
(w0w0) <- [controls ={ +(-10em, 0) and + (-10em, 0)}]  (dotsmiddle.north west);
(w0w0) -> [controls ={ +(8.5em, 5ex) and + (8.5em, 0)}]  (dotsmiddle.south east);
(w0w0) <- [controls ={ +(-8.5em, 5ex) and + (-8.5em, 0)}] (dotsmiddle.south west);
(w0w0) -> [bend right = 15, edge node= {node [right, yshift=-1ex, fill = white, inner sep = 0] {{\( \underbrace{\hspace{1.8cm}}_{\Card{ \Alphab_{0} } - 1 } \)} } } ] (vbottomright) -> (dotsbottom) -> (vbottomleft)  -> [bend right = 15] (w0w0);
};
\draw [->] (w0w0.south east) .. controls +(1.6cm, -1.8cm) and +(-1.6cm, -1.8cm) .. node [right, near start, xshift=0.5em, align = left]{for \( L = \Length{\PBlock{ 0 }} \) this edge exists\\ if and only if \( a_{0} \in \Alphab_{1} \) holds} (w0w0.south west);
\end{tikzpicture}
\normalsize
\caption{The de Bruijn graph \( \Debruijn{L} \) for \( 1 \leq L \leq \Length{ \PBlock{ 0 }} \).\label{fig:deBruijn1ToL0}}
\end{figure}

Now we discuss \( \Debruijn{L} \) for \( \Length{ \PBlock{ k-1 } } + 1 \leq L \leq \Length{ \PBlock{ k } } \) and \( k \geq 1 \). Let \( u_{1} \) denote the prefix of length \( L \) of \( \PBlock{k} \) and define \( r \DefAs L \bmod ( \Length{ \PBlock{ k-1 } } + 1 ) \). The decomposition \( \PBlock{k} = \Word{ \PBlock{k-1} }{ a_{k} }{ \ldots }{ \PBlock{k-1} } \) implies that we obtain \( v_{1} \) if we shift \( u_{1} \) by \( \Length{ \PBlock{ k-1 } } - r \) positions to the right, see Figure~\ref{fig:ShiftU1V1}.

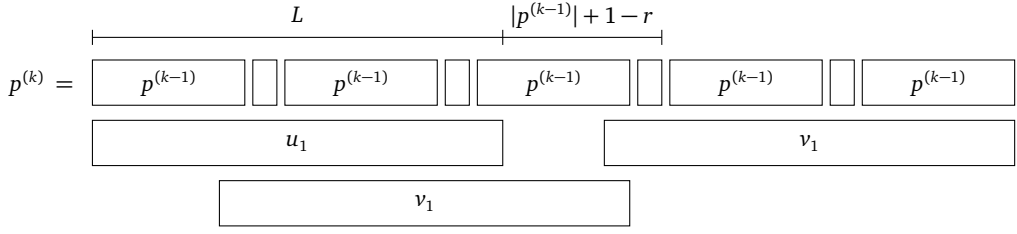
\begin{figure}
\centering
\footnotesize
% set parameters for the image
\pgfmathsetmacro{\TikzNumberPkkk}{5} % number of pkkk in image (separated by single letter)
\pgfmathsetmacro{\TikzDistBlocks}{3} % distance between two blocks in image in pt
\pgfmathsetmacro{\TikzBuchstLength}{9} % length of a single-letter-block in pt
\pgfmathsetmacro{\TikzNumberPkkkL}{2} % number of whole [pkkk + buchst] in L
% compute block length from parameters
\pgfmathsetmacro{\TikzLengthPkkk}{(0.85*\linewidth - (\TikzNumberPkkk - 1) * ( \TikzBuchstLength + 2 * \TikzDistBlocks) ) / \TikzNumberPkkk} % length of pkkk block in pt
\pgfmathsetmacro{\TikzLengthTotal}{\TikzNumberPkkk*\TikzLengthPkkk + (\TikzNumberPkkk - 1)*(\TikzDistBlocks + \TikzDistBlocks + \TikzBuchstLength) } % total length of all pkkk blocks
\pgfmathsetmacro{\TikzLengthR}{\TikzLengthPkkk/6} % set length of r
\pgfmathsetmacro{\TikzLengthL}{(\TikzNumberPkkkL*(\TikzLengthPkkk + \TikzBuchstLength + 2*\TikzDistBlocks) +\TikzLengthR} % length of L
\begin{tikzpicture}
[every node/.style ={rectangle, outer sep=0pt, inner sep=0pt, minimum height=0.6cm},
buchst/.style={draw, minimum width=\TikzBuchstLength pt},
pkkk/.style={draw, minimum width=\TikzLengthPkkk pt}]
\draw (0, 0.5) -- (0, 0.7);
\draw (\TikzLengthL pt, 0.5) -- (\TikzLengthL pt, 0.7);
\draw (\TikzLengthL - \TikzLengthR + \TikzLengthPkkk + \TikzDistBlocks + \TikzBuchstLength pt, 0.5) -- (\TikzLengthL - \TikzLengthR + \TikzLengthPkkk + \TikzDistBlocks + \TikzBuchstLength pt, 0.7);
\draw (0, 0.6) -- node [midway, above] {\( L \)} (\TikzLengthL pt, 0.6) -- node [midway, above] {\( \Length{ \PBlock{ k-1 } } + 1 - r \)} (\TikzLengthL - \TikzLengthR + \TikzLengthPkkk + \TikzDistBlocks + \TikzBuchstLength pt, 0.6);
\node [right] at (0,0) [pkkk] (A) {\( \PBlock{k-1} \)};
\node [left=\TikzDistBlocks pt of A] {\( \PBlock{k} \; = \;\; \)};
\foreach \x in {2,...,\TikzNumberPkkk}{
	\node [buchst] (A) [right=\TikzDistBlocks pt of A] {};
	\node [pkkk] (A) [right=\TikzDistBlocks pt of A] {\( \PBlock{k-1} \)};}
\draw (0, -1.1) rectangle (\TikzLengthL pt, -0.5) node[midway]{\( u_{1} \)};
\draw (\TikzLengthTotal - \TikzLengthL pt, -1.1) rectangle (\TikzLengthTotal pt, -0.5) node[midway]{\( v_{1} \)};
\draw (\TikzLengthPkkk - \TikzLengthR pt, -1.9) rectangle (\TikzLengthPkkk - \TikzLengthR + \TikzLengthL pt, -1.3) node[midway]{\( v_{1} \)};
\end{tikzpicture}
\normalsize
\caption{After \( \Length{ \PBlock{ k-1 } } - r \) shifts the prefix \( u_{1} \) becomes the suffix \( v_{1} \).\label{fig:ShiftU1V1}}
\end{figure}

Recall that \( v_{1} \) can be extended to the right by all letters in \( \Alphab_{k} \setminus \Set{a_{k}} \). An extension by \( a_{k} \) is possible if and only if either \( a_{k} \in \Alphab_{k+1} \) or \( L \leq \Length{ \PBlock{ k } } - \Length{ \PBlock{ k-1 } } - 1 \) holds. Starting from \( v_{1} \), we obtain \( u_{1} \) by \( L+1 \) shifts along an extension with a letter in \( \Alphab_{k} \setminus \Set{ a_{k} } \) or by \( r+1 \) shifts along the extension with \( a_{k} \). This yields the graph in Figure~\ref{fig:deBruijnAllg1}. It is the complete de Bruijn graph of the subshift if there is no other right special word, that is, if Proposition~\ref{prop:GrowGreater} does not apply.

\begin{figure}
\centering
\footnotesize
\begin{tikzpicture}
[ every path/.style = {shorten <=1pt, shorten >=1pt, >=stealth, absolute, draw},
  vertex/.style = {circle, minimum size=14pt, inner sep=1pt, draw}]
% define position of nodes
\matrix[row sep = 3ex, column sep = 3em]{
& \node [vertex] (zu) {}; & \node (dotstopz) {\( \cdots \)}; & \node [vertex] (vz) {}; & \\
& & \node (dotstopy2){\( \cdots \)}; & &\\
& & \node (dotstopy1) {\( \cdots \)}; & &\\
& \node [vertex] (xu) {}; & \node (dotstopx) {\( \cdots \)}; & \node [vertex] (vx) {}; & \\
\node [vertex] (u) {\( u_{1} \)}; & \node [vertex] (ux) {}; & \node (dotsmiddle) {\( \cdots \)}; & \node [vertex] (xv) {}; & \node [vertex] (v) {\( v_{1} \)}; \\
& \node [vertex] (tu) {}; &  \node (dotsbottom) {\( \cdots \)}; & \node [vertex] (vt) {}; & \\
};
% define edges
% for a node, [>...] resp. [<...] specify style of ingoing (traget edge) resp. outgoing (source edge) edge
% ''in'' and ''out'' specify angle wrt. page (due to ''absolue'' in ''every path'');  x-axis = 0 degrees
\graph [use existing nodes]{
	% path splits at -> { 
u -> ux ->[edge node= {node [above, xshift = 2.5em] {\( \Length{ \PBlock{ k-1 } } - r \) edges} }]  dotsmiddle -> xv -> v -> {
	% bottom arc
	vt [> out = 240, > in = 0,  target edge node = { node [right, xshift = 1em, yshift = -1.5ex, align = left] {for \( L \geq \Length{ \PBlock{ k } } - \Length{ \PBlock{ k-1 } } \) this arc exists\\if and only if \( a_{k} \in \Alphab_{k+1} \) holds }}] -> dotsbottom [target edge node= {node [below, xshift = 2em, yshift=-1.5ex] {\( \Word{ \Restr{ v_{1} }{ 2 }{ L } }{ a_{k} } \)} }, source edge node= {node [below, xshift = 1.5em] {$ r+1 $ edges} }] -> tu [< out = 180, < in = 300],
	vx [> out = 120, > in = 0] -> dotstopx [target edge node= {node [above, xshift = 1.5em, yshift=1.5ex] {\( \Word{ \Restr{ v_{1} }{ 2 }{ L } }{ a } \) } }] -> xu [< out = 180, < in = 60],	
	dotstopy1 [> out = 90, > in = 0, target edge node = {node [above, near end] {\( \vdots \)} } , < out = 180, < in = 90, source edge node = {node [above, near start] {\( \vdots \)} }],	
	dotstopy2 [> out = 60, > in = 0, < out = 180, < in = 120],	
	% top arc
	vz [> out =30, > in = 340, target edge node = {node [xshift = -1em, yshift = -8ex, fill = white ,inner sep = 1pt] { \( \overbrace{ \hspace{1.4cm} }^{ \Card{ \Alphab_{k} } -1 } \) } }] -> dotstopz [target edge node= {node [above, xshift = 6.5em, yshift=1.5ex] {\( \Word{ \Restr{ v_{1} }{ 2 }{ L } }{ a } \) \quad with \( a \in \Alphab_{k} \setminus \Set{ a_{k} } \) } }]  -> [edge node= {node [below, xshift = 2.5em, yshift=-2ex] {\( L + 1 \) edges on each arc} }] zu [< out = 200, < in = 150]
	% arcs unite again
} -> u;
};
\end{tikzpicture}
\normalsize
\caption{The de Bruijn graph \( \Debruijn{L} \) for \( k \geq 1 \) and \( \Length{ \PBlock{ k-1 } } + 1 \leq L \leq \Length{ \PBlock{ k } } \), when either \( a_{k-1} \notin \Alphab_{k} \) or \( L > 2 \Length{ \PBlock{k-1} } - \Length{ \PBlock{k-2} } \) holds. The number of edges refers to the distance between \( u_{1} \) and \( v_{1} \).\label{fig:deBruijnAllg1}}
\end{figure}

If Proposition~\ref{prop:GrowGreater} applies, then the de Bruijn graph changes. In this case, the suffix \( v_{2} \) of length \( L \) of \( \Word{ \PBlock{k-1} }{ a_{k-1} }{ \PBlock{k-1} } \) is another right special word and can be extended with both, \( a_{k-1} \) and \( a_{k} \). We define \( \widetilde{r} \DefAs L \bmod ( \Length{ \PBlock{ k-2 } } + 1 ) \). Clearly \( v_{2} \) is contained in \( \Word{ \PBlock{k} }{ a_{k-1} }{ \PBlock{k} } \). Therefore \( v_{2} \) lies on the arc that represents the extension of \( v_{1} \) by \( a_{k-1} \). It is reached after \( r + 1 + \Length{ \PBlock{ k-2 } } - \widetilde{r} \) shifts from \( v_{1} \), which can be seen as follows: the end of \( v_{2} \) has to align with the end of a \( \PBlock{k-2} \)-block. When we reach \( v_{2} \) for the first time after \( v_{1} \), the beginning of \( v_{2} \) has to be in the leftmost \( \Word{ \PBlock{k-2} }{ a_{k-1} } \) of the rightmost \( \PBlock{k-1} \)-block of the left \( \PBlock{k} \)-block, see Figure~\ref{fig:LocV2}.

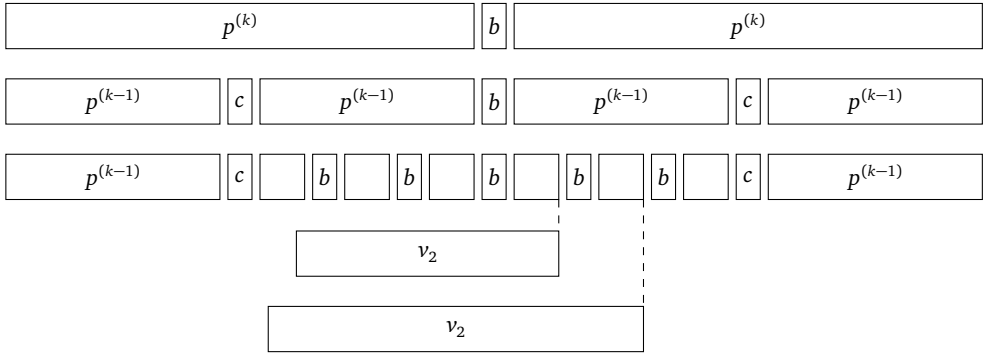
\begin{figure}
\centering
\footnotesize
% set parameters for the image
\pgfmathsetmacro{\TikzDistBlocks}{3} % distance between two blocks in image in pt
\pgfmathsetmacro{\TikzBuchstLength}{9} % length of a single-letter-block in pt
% compute block length from parameters
\pgfmathsetmacro{\TikzLengthPkkk}{(0.9*\linewidth - ( \TikzBuchstLength + 2 * \TikzDistBlocks) ) / 2} % length of pkkk block in pt
\pgfmathsetmacro{\TikzLengthPkk}{(\TikzLengthPkkk - ( \TikzBuchstLength + 2 * \TikzDistBlocks) ) / 2} % length of pkk block in pt
\pgfmathsetmacro{\TikzLengthPk}{(\TikzLengthPkk - 2 * ( \TikzBuchstLength + 2 * \TikzDistBlocks) ) / 3} % length of pk block in pt
\begin{tikzpicture}
[every node/.style ={rectangle, outer sep=0pt, inner sep=0pt},
buchst/.style={minimum width=\TikzBuchstLength pt, minimum height=0.6cm, draw},
pkkk/.style={minimum width=\TikzLengthPkkk pt, minimum height=0.6cm, draw},
pkk/.style={minimum width=\TikzLengthPkk pt, minimum height=0.6cm, draw},
pk/.style={minimum width=\TikzLengthPk pt, minimum height=0.6cm, draw}]
\node [right] at (0,0) [pkkk] (A) {\( \PBlock{k} \)};
\node [buchst] (A) [right=\TikzDistBlocks pt of A] {\( b \)};
\node [pkkk] (A) [right=\TikzDistBlocks pt of A] {\( \PBlock{k} \)};
\node [right] at (0,-1) [pkk] (A) {\( \PBlock{k-1} \)};
\node [buchst] (A) [right=\TikzDistBlocks pt of A] {\( c \)};
\node [pkk] (A) [right=\TikzDistBlocks pt of A] {\( \PBlock{k-1} \)};
\node [buchst] (A) [right=\TikzDistBlocks pt of A] {\( b \)};
\node [pkk] (A) [right=\TikzDistBlocks pt of A] {\( \PBlock{k-1} \)};
\node [buchst] (A) [right=\TikzDistBlocks pt of A] {\( c \)};
\node [pkk] (A) [right=\TikzDistBlocks pt of A] {\( \PBlock{k-1} \)};
\node [right] at (0,-2) [pkk] (A) {\( \PBlock{k-1} \)};
\node [buchst] (A) [right=\TikzDistBlocks pt of A] {\( c \)};
\node [pk] (Fix1) [right=\TikzDistBlocks pt of A] {};
\node [buchst] (A) [right=\TikzDistBlocks pt of Fix1] {\( b \)};
\node [pk] (A) [right=\TikzDistBlocks pt of A] {};
\node [buchst] (A) [right=\TikzDistBlocks pt of A] {\( b \)};
\node [pk] (A) [right=\TikzDistBlocks pt of A] {};
\node [buchst] (A) [right=\TikzDistBlocks pt of A] {\( b \)};
\node [pk] (Fix2) [right=\TikzDistBlocks pt of A] {};
\node [buchst] (A) [right=\TikzDistBlocks pt of Fix2] {\( b \)};
\node [pk] (Fix3) [right=\TikzDistBlocks pt of A] {};
\node [buchst] (A) [right=\TikzDistBlocks pt of Fix3] {\( b \)};
\node [pk] (A) [right=\TikzDistBlocks pt of A] {};
\node [buchst] (A) [right=\TikzDistBlocks pt of A] {\( c \)};
\node [pkk] (A) [right=\TikzDistBlocks pt of A] {\( \PBlock{k-1} \)};
\node (C1) [below=0.4 of Fix2.south east] {};
\node (C2) [below=1.4 of Fix3.south east]{};
\node (B1) [below=1 of Fix1.south] {};
\node (B2) [below=2 of Fix1.south] {};
\node (D1) [right=0.15 of B1]{};
\node (D2) [left=0.15 of B2]{};
\draw (D1) rectangle (C1) node[midway]{\( v_{2} \)};
\draw (D2) rectangle (C2) node[midway]{\( v_{2} \)};
\draw[dashed] (Fix2.south east) -- (C1);
\draw[dashed] (Fix3.south east) -- (C2);
\end{tikzpicture}
\normalsize
\caption{Two examples of the location of \( v_{2} \) in \( \PBlock{k} \, a_{k-1} \, \PBlock{k} \), for different word length \( L \). To shorten notation we write \( b \DefAs a_{k-1} \) and \( c \DefAs a_{k} \).\label{fig:LocV2}}
\end{figure}

After \( r + 1 \) shifts from \( v_{1} \) along the path to \( v_{2} \), we encounter the prefix of length \( L \) of \( \Word{ \PBlock{ k-1 } }{ a_{k-1} }{ \PBlock{ k-1 } } \). We denote it by \( u_{2} \). When we shift further, we reach \( v_{2} \), where the path splits: one path is the extension of \( v_{2} \) with \( a_{k-1} \) and leads back to \( u_{2} \) after \( \widetilde{r} + 1  \) shifts, the other path is the extension of \( v_{2} \) with \( a_{k} \) and leads to \( u_{1} \) after \( L - \Length{  \PBlock{ k-1 } } = r + 1 \) shifts. This yields the de Bruijn graph that is shown in Figure~\ref{fig:deBruijnAllg2}.

\begin{figure}
\centering
\footnotesize
\begin{tikzpicture}
[ every path/.style = {shorten <=1pt, shorten >=1pt, >=stealth, absolute, draw},
  vertex/.style = {circle, minimum size=14pt, inner sep=1pt, draw}]
% define nodes
\matrix[row sep = 3ex, column sep = 1.5em]{
& \node [vertex] (zu) {}; & & \node (dotstopz) {\( \cdots \)}; & & \node [vertex] (vz) {}; & \\
& & & \node (dotstopy2){\( \cdots \)}; & & &\\
& & & \node (dotstopy1) {\( \cdots \)}; & & &\\
& \node [vertex] (xu) {}; & & \node (dotstopx) {\( \cdots \)}; & & \node [vertex] (vx) {}; & \\
\node [vertex] (u) {\( u_{1} \)}; & \node [vertex] (ux) {}; & & \node (dotsmiddle) {\( \cdots \)}; & & \node [vertex] (xv) {}; & \node [vertex] (v) {\( v_{1} \)}; \\
& \node (dotsbottoms1) {\( \cdots \)}; & \node [vertex] (vtilde) {\( v_{2} \)}; & \node (dotsbottoms2) {\( \cdots \)}; & \node [vertex] (utilde) {\( u_{2} \)}; & \node (dotsbottoms3) {\( \cdots \)}; & \\
& & & \node (dotsbottomr){\( \cdots \)}; & & &\\
& \node [vertex] (tu) {}; & & \node (dotsbottomt) {\( \cdots \)}; & & \node [vertex] (vt) {}; & \\
};
% define edges
\graph [use existing nodes]{
	% split path
u -> ux ->[edge node= {node [above, xshift = 3em] {\( \Length{ \PBlock{ k-1 } } - r \) edges} }]  dotsmiddle -> xv -> v -> {
	% bottom arc 
	vt [> out = 270, > in = 0, target edge node = {node [xshift = 8.5em, yshift = -2.5ex, align = left, fill=white] { for \( L \geq \Length{ \PBlock{ k } } - \Length{ \PBlock{ k-1 } } \) this arc exists\\ if and only if \( a_{k} \in \Alphab_{k+1} \) holds} }] -> dotsbottomt [target edge node= {node [xshift = 3.5em, yshift=-3.5ex] {\( \Word{ \Restr{ v_{1} }{ 2 }{ L } }{ a_{k} } \)} }, source edge node= {node [below, xshift = 3em] {\( r + 1 \) edges}}] -> tu [< out = 180, < in = 270],	
	% arc from u_{2} to v_{2}, except arc back to u_{2}
	dotsbottoms3 [> out = 240, > in = 0, target edge node = {node [xshift = 3em, yshift=-4ex, inner sep = 0.5ex, fill = white] {\( r+ 1 \) edges between \( v_{1} \) and \( u_{2} \)}}] -> utilde -> dotsbottoms2 -> vtilde [target edge node = {node [xshift = 1.7em, yshift = 3ex] {\( \Length{ \PBlock{ k-2 } } - \widetilde{r} \) edges between \( u_{2} \) and \( v_{2} \)} }] -> dotsbottoms1 [< out = 180, < in = 300, source edge node = { node [xshift = -3em, yshift = -4ex, inner sep=0.5ex, fill = white] {\( r + 1 \) edges between \( v_{2} \) and \( u_{1} \)} }],
	vx [> out = 120, > in = 0] -> dotstopx -> xu [< out = 180, < in = 60],	
	dotstopy1 [> out = 90, > in = 0, target edge node = {node [above, near end] {\( \vdots \)} } , < out = 180, < in = 90, source edge node = {node [above, near start] {\( \vdots \)} }],	
	dotstopy2 [> out = 60, > in = 0, < out = 180, < in = 120],	
	% top arc
	vz [> out =30, > in = 340, target edge node = {node [xshift = -1.3em, yshift = -8.2ex, fill = white, inner sep = 0.5pt] { \( \overbrace{ \hspace{1.2cm} }^{ \Card{ \Alphab_{k} } - 2 }  \)} }] -> dotstopz [target edge node= {node [xshift = 8.5em, yshift=3.5ex] {\( \Word{ \Restr{ v_{1} }{ 2 }{ L } }{ a } \) \quad with \( a \in  \Alphab_{k} \setminus \Set{ a_{k-1}, a_{k} } \)} }] -> [edge node= {node [xshift = 3.5em, yshift=-3.5ex] {\( L + 1 \) edges on each arc} }] zu [< out = 200, < in = 150]
	% arcs unite again
} -> u;
% arc from v_{2} back to u_{2}
vtilde ->[out=270, in= 180, edge node= {node [at end, below, xshift = 1em]{\( \widetilde{r} + 1 \) edges between \( v_{2} \) and \( u_{2} \)} }] dotsbottomr ->[out = 0, in = 270] utilde;
};
\end{tikzpicture}
\normalsize
\caption{The de Bruijn graph \( \Debruijn{L} \) for \( k \geq 1 \) and \( \Length{ \PBlock{ k-1 } } + 1 \leq L \leq 2 \Length{ \PBlock{k-1} } - \Length{ \PBlock{k-2} } \) when \( a_{k-1} \in \Alphab_{k} \) holds. Unless stated otherwise, the number of edges refers to the distance between \( u_{1} \) and \( v_{1} \).}
\label{fig:deBruijnAllg2}
\end{figure}
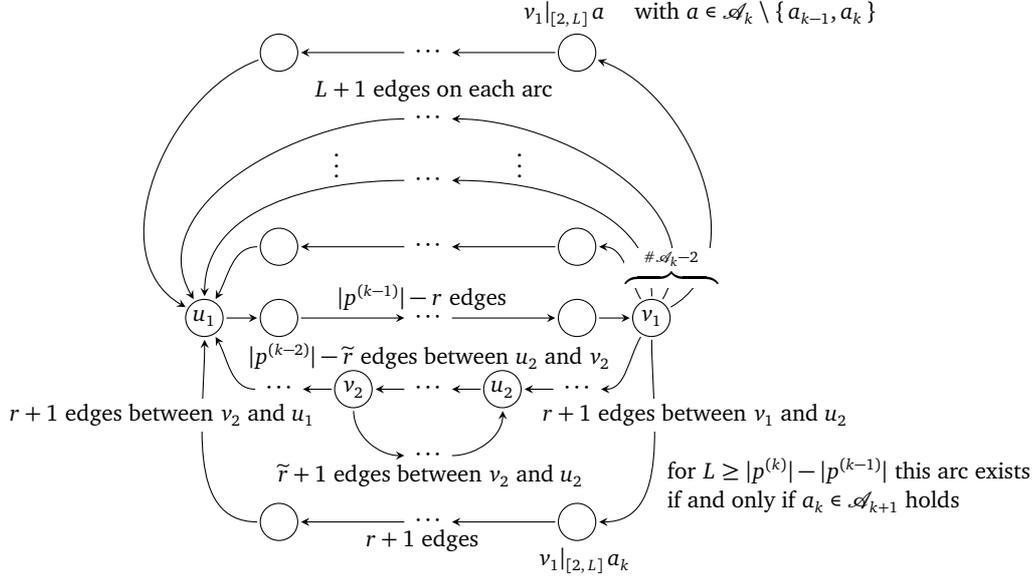

Our results from Section~\ref{sec:Compl} imply that this is the whole de Bruijn graph: by Corollary~\ref{cor:CompGrowth} the only branching points are \( v_{1} \) and \( v_{2} \) and these have only the extensions that were discussed in Proposition~\ref{prop:CompGrow1} and~\ref{prop:GrowGreater}. Moreover, every word of length \( L \leq \Length{ \PBlock{ k } } \) is contained in \( \Word{ \PBlock{k} }{ a }{ \PBlock{k} } \) for some \( a \in \Alphab_{k+1} \) by Proposition~\ref{prop:WordsContained}. Since all such extensions of \( v_{1} \) are included in the graph, together with all the words that emerge when we shift further, neither edges nor vertices are missing.

\begin{exmpl}[Grigorchuk subshift]
From Figure~\ref{fig:deBruijn1ToL0} we obtain the graph for \( \Length{ \PBlock{ 0 } } \), see Figure~\ref{subfig:deBruGrigo1}. Because of \( a_{0} \notin \Alphab_{1} \), it is Figure~\ref{fig:deBruijnAllg1} which describes the graphs for \( k = 1 \). The result is shown in Figure~\ref{subfig:deBruGrigo2} and~\ref{subfig:deBruGrigo3}. For \( k \geq 2 \) the relation \( a_{k-1} \in \Alphab_{k} \) implies that \( \Debruijn{L} \) is given by Figure~\ref{fig:deBruijnAllg2} for \( \Length{ \PBlock{ k-1 } } + 1 \leq L \leq \Length{ \PBlock{k} } - \Length{ \PBlock{k-2} } - 1 \) and by Figure~\ref{fig:deBruijnAllg1} for \( \Length{ \PBlock{k} } - \Length{ \PBlock{k-2} } \leq L \leq \Length{ \PBlock{k} } \). With \( \Alphab_{k} = \Set{ a_{k-1}, a_{k}, a_{k+1} } \) and \( \Length{ \PBlock{ k } } = 2^{k+1} - 1 \) we obtain the graphs in Figure~\ref{fig:deBruGrigoAllg1} and~\ref{fig:deBruGrigoAllg2}.
\end{exmpl}

\begin{rem}
For \( L = \Length{\PBlock{ k }} \) the de Bruijn graphs have a particularly interesting form: for  \( k\geq 1 \) we note that \( r = \Length{\PBlock{ k-1 }} \) holds in Figure~\ref{fig:deBruijnAllg1}. Hence the vertices \( u_{1} \) and \( v_{1} \) coincide. In other words, the de Bruijn graphs \( \Debruijn{ \Length{\PBlock{ k }} } \), \( k \in \NN \), has exactly one branching point and a number of loops that start and terminate at this vertex. As Figure~\ref{fig:deBruijn1ToL0} shows this is true for \( k = 0 \) as well.
\end{rem}

\begin{figure}[p] % place on float page
% don't use \Word{}{} since graph becomes to large because of additional spaces
\centering
\footnotesize
%
% G_{1}
\begin{minipage}[b]{0.2\textwidth}
\centering
\begin{tikzpicture}
[ every path/.style = {shorten <=1pt, shorten >=1pt, >=stealth},
  vertex/.style = {circle, minimum size=17pt, inner sep=2pt, draw}]
\node [vertex] (b) {\( b \)};
\node [vertex] (c)  [below = 2ex of b] {\( c \)};
\node [vertex] (d) [below = 2ex of c] {\( d \)};
\node [vertex] (a) [below = 2ex of d] {\( a \)};
\graph{
(a) -> [bend right = 90] (b) -> [bend right = 90] (a);
(a) -> [bend right = 60] (c) -> [bend right = 60] (a);
(a) -> [bend right = 30] (d) -> [bend right = 30] (a);
};
\node[draw=none, below = 2ex of a]{}; % to adjust spacing
\end{tikzpicture}
\subcaption{The graph \( \Debruijn{1} \). We have \( u_{1} = v_{1} = a \).\label{subfig:deBruGrigo1}}
\end{minipage}
\hfill
%
% G_{2}
\begin{minipage}[b]{0.35\textwidth}
\centering
\begin{tikzpicture}
[ every path/.style = {shorten <=1pt, shorten >=1pt, >=stealth, absolute, draw},
  vertex/.style = {circle, minimum size=17pt, inner sep=1pt, draw}]
% define nodes
\matrix[row sep = 2ex, column sep = 2em]{
& \node [vertex] (ca) {\( c a \)}; & \node [vertex] (ac) {\( a c \)}; & \\
& \node [vertex] (da) {\( d a \)}; & \node [vertex] (ad) {\( a d \)}; & \\
\node [vertex] (u) {\( a b \)}; & & & \node [vertex] (v) {\( b a \)}; \\
};
% define edges
\graph [use existing nodes]{
u  -> v -> {
	% bottom arc defined later
	ad [> out = 120, > in = 0] -> da [< out = 180, < in = 60],	
	ac [> out =90, > in = 0] -> ca [< out = 180, < in = 90]
} -> u;
% bottom arc
v [< out = 270, < in = 270] -> u;
};
\end{tikzpicture}
\subcaption{The graph \( \Debruijn{2} \). We have \( u_{1}= a \, b  \), \( v_{1} = b \, a  \) and \( r = 0 \).\label{subfig:deBruGrigo2}}
\end{minipage}
\hfill
%
% G_{3}
\begin{minipage}[b]{0.35\textwidth}
\centering
\begin{tikzpicture}
[ every path/.style = {shorten <=1pt, shorten >=1pt, >=stealth, absolute, draw},
  vertex/.style = {circle, minimum size=17pt, inner sep=1pt, draw}]
% define nodes
\matrix[row sep = 2ex, column sep = 2em]{
\node [vertex] (cab) {\( c a b \)}; & \node [vertex](aca) {\( a c a \)}; & \node [vertex] (bac) {\( b a c \)}; \\
\node [vertex] (dab) {\( d a b \)}; & \node [vertex] (ada) {\( a d a \)}; & \node [vertex] (bad) {\( b a d \)}; \\
& \node [vertex] (aba) {\( a b a \)}; & \\
&  \node [vertex] (bab) {\( b a b \)}; & \\
};
% define edges
\graph [use existing nodes]{
aba -> {
	bab [> bend left=30, < bend left = 30],	
	bad [> bend right = 30] -> ada -> dab [< bend right = 30],
	bac [> bend right = 90] -> aca -> cab [< bend right = 90]
} -> aba;
};
\end{tikzpicture}
\subcaption{The graph \( \Debruijn{3} \). We have \( u_{1} = v_{1} =  a \, b \, a  \) and \( r = 1 \).\label{subfig:deBruGrigo3}}
\end{minipage}
\caption{The first de Bruijn graphs for the Grigorchuk subshift.\label{fig:deBruGrigo123}}
\vspace{8ex}
% general graph one
\begin{tikzpicture}
[ every path/.style = {shorten <=1pt, shorten >=1pt, >=stealth, absolute, draw},
  vertex/.style = {circle, minimum size=14pt, inner sep=1pt, draw}]
% define nodes
\matrix[row sep = 3ex, column sep = 1.5em]{
& \node [vertex] (zu) {}; & & \node (dotstopz) {\( \cdots \)}; & & \node [vertex] (vz) {}; & \\
\node [vertex] (u) {\( u_{1} \)}; & \node [vertex] (ux) {}; & & \node (dotsmiddle) {\( \cdots \)}; & & \node [vertex] (xv) {}; & \node [vertex] (v) {\( v_{1} \)}; \\
& \node (dotsbottoms1) {\( \cdots \)}; & \node [vertex] (vtilde) {\( v_{2} \)}; & \node (dotsbottoms2) {\( \cdots \)}; & \node [vertex] (utilde) {\( u_{2} \)}; & \node (dotsbottoms3) {\( \cdots \)}; & \\
& & & \node (dotsbottomr){\( \cdots \)}; & & &\\
& \node [vertex] (tu) {}; & & \node (dotsbottomt) {\( \cdots \)}; & & \node [vertex] (vt) {}; & \\
};
% define edges
\graph [use existing nodes]{
u -> ux ->[edge node= {node [above, xshift = 3em] {\( 2^{k+1} -  L - 1 \) edges} }]  dotsmiddle -> xv -> v -> {
	% bottom arc 
	vt [> out = 270, > in = 0] -> dotsbottomt [target edge node= {node [xshift = 3.5em, yshift=-3.5ex] {\( \Word{ \Restr{ v_{1} }{ 2 }{ L } a_{k} } \)} }, source edge node= {node [xshift = 3em, yshift=-2.3ex] {\( L - 2^{k} + 1 \) edges}}] -> tu [< out = 180, < in = 270],	
	% arc from u_{2} to v_{2} except arc back to u_{2}
	dotsbottoms3 [> out = 240, > in = 0, target edge node = {node [xshift = 4em, yshift=-4ex, fill = white, inner sep = 0.5ex] {\( L - 2^{k} + 1 \) edges between \( v_{1} \) and \( u_{2} \)}}] -> utilde -> dotsbottoms2 -> vtilde [target edge node = {node [above, xshift = 1.7em, yshift = 1ex] {\( 2^{k-1} - 1 - \widetilde{r} \) edges between \( u_{2} \) and \( v_{2} \)} }] -> dotsbottoms1 [< out = 180, < in = 300, source edge node = { node [xshift = -4em, yshift = -4ex, fill = white, inner sep = 0.5ex] {\( L - 2^{k} + 1 \) edges between \( v_{2} \) and \( u_{1} \)} }],
	% top arc
	vz [> out =90, > in = 0] -> dotstopz [target edge node= {node [xshift = 3.5em, yshift=3.5ex] {\( \Word{ \Restr{ v_{1} }{ 2 }{ L} }{ a_{k+1} } \)} }] -> [edge node= {node [xshift = 3.5em, yshift=2ex] {\( L + 1 \) edges} }] zu [< out = 180, < in = 90]
} -> u;
% arc from v_{2} back to u_{2}
vtilde ->[out=270, in= 180, edge node= {node [at end, below, xshift = 1em]{\( \widetilde{r} + 1 \) edges between \( v_{2} \) and \( u_{2} \)} }] dotsbottomr ->[out = 0, in = 270] utilde;
};
\end{tikzpicture}
\caption{The de Bruijn graph for the Grigorchuk subshift for \( k \geq 2 \) and \( 2^{k} \leq L \leq 2^{k+1} - 2^{k-1} - 1 \). We have \( r = L - 2^{k} \). Unless stated otherwise, the number of edges refers to the distance between \( u_{1} \) and \( v_{1} \).\label{fig:deBruGrigoAllg1}}
\vspace{8ex}
% general graph two
\begin{tikzpicture}
[ every path/.style = {shorten <=1pt, shorten >=1pt, >=stealth, absolute, draw},
  vertex/.style = {circle, minimum size=14pt, inner sep=1pt, draw}]
% define nodes
\matrix[row sep = 3ex, column sep = 3em]{
& \node [vertex] (zu) {}; & \node (dotstopz) {\( \cdots \)}; & \node [vertex] (vz) {}; & \\
& \node [vertex] (xu) {}; & \node (dotstopx) {\( \cdots \)}; & \node [vertex] (vx) {}; & \\
\node [vertex] (u) {\( u_{1} \)}; & \node [vertex] (ux) {}; & \node (dotsmiddle) {\( \cdots \)}; & \node [vertex] (xv) {}; & \node [vertex] (v) {\( v_{1} \)}; \\
& \node [vertex] (tu) {}; &  \node (dotsbottom) {\( \cdots \)}; & \node [vertex] (vt) {}; & \\
};
% define edges
\graph [use existing nodes]{
u -> ux ->[edge node= {node [above, xshift = 2.5em] {\( 2^{k+1} - L - 1 \) edges} }]  dotsmiddle -> xv -> v -> {
	% bottom arc
	vt [> out = 240, > in = 0] -> dotsbottom [target edge node= {node [below, xshift = 2em, yshift=-1.5ex] {\( \Word{ \Restr{ v_{1} }{ 2 }{ L } }{ a_{k} } \)} }, source edge node= {node [below, xshift = 2em] {\(  L - 2^{k} + 1 \) edges} }] -> tu [< out = 180, < in = 300],
	vx [> out = 120, > in = 0] -> dotstopx [target edge node= {node [below, xshift = 1.5em, yshift=-0.5ex] {\( \Word{ \Restr{ v_{1} }{ 2 }{ L } }{ a_{k+1} } \)} }] -> xu [< out = 180, < in = 60],	
	% top arc
	vz [> out =90, > in = 0] -> dotstopz [target edge node= {node [above, xshift = 2em, yshift=1.5ex] {\( \Word{ \Restr{ v_{1} }{ 2 }{ L } }{ a_{k-1} } \)} }]  -> [edge node= {node [below, xshift = 2.5em, yshift=-2ex] {\( L + 1 \) edges on each arc} }] zu [< out = 180, < in = 90]
} -> u;
};
\end{tikzpicture}
\caption{The de Bruijn graph for the Grigorchuk subshift for \( k \geq 2 \) and \( 2^{k+1} - 2^{k-1} \leq L \leq 2^{k+1} - 1 \). We have \( r = L -2^{k} \). The number of edges refers to the distance between \( u_{1} \) and \( v_{1}  \).\label{fig:deBruGrigoAllg2}}
\normalsize
\end{figure}

\subsection{Palindrome complexity of simple Toeplitz subshifts}
\label{subsec:PalComp}

A palindrome is a word that remains the same when read backwards. More precisely, a finite word \( u \in \Langu{ \Subshift } \) is called a \emph{palindrome}\index{palindrome} if
\[ u = \Word{ u(1) }{ \ldots }{ u(L) } = \Word{ u(L) }{ \ldots }{ u(1) } \AsDef \Rev{ u } \]
holds, where \( \Rev{ u } \) denotes the \emph{reflection of the word \( u \) at its midpoint}\index{reflection! of a finite word}\index{word!reflection of}. We also consider the empty word as a palindrome. Note that \( \Rev{ \boldsymbol{\cdot} } \colon \Langu{ \Subshift } \to \Langu{ \Subshift } \) defines an involution.

\begin{exmpl}
\label{exmpl:PBlockPalindr}
For every \( k \in \NN_{0} \) the word \( \PBlock{k} \) is a palindrome: for \( \PBlock{0} = a_{0}^{n_{0}-1} \) this is obvious. For \( k \geq 1 \) the claim follows by induction from
\[ \Rev{ \PBlock{k+1} } = \Word{ \Rev{ \PBlock{k} } }{ a_{k+1} }{ \Rev{ \PBlock{k} } }{ \ldots }{ \Rev{ \PBlock{k}} } = \Word{ \PBlock{k} }{ a_{k+1} }{ \PBlock{k} }{ \ldots }{ \PBlock{k} } = \PBlock{k+1} \, .\]
In particular we have \( \PBlock{k}( j ) = \PBlock{k}( \Length{ \PBlock{ k } } - j + 1 ) \).
\end{exmpl}

Similar to the subword complexity, we define the \emph{palindrome complexity}\index{palindrome complexity}\index{complexity!palindrome complexity} \( \Pali \colon \NN_{0} \rightarrow \NN \) by
\[ \Pali( L ) \DefAs \Card{ \Set{ u \in \Langu{ \Subshift } : \Length{ u } = L \, , \; \Rev{ u } = u } } = \text{ ``number of palindromes of length } L \text{''} \, . \]
It turns out that reflecting the words and reflecting the de Bruijn graph is essentially the same:

\begin{prop}
Reflection of a word at its midpoint corresponds in Figure~\ref{fig:deBruijn1ToL0}, \ref{fig:deBruijnAllg1} and~\ref{fig:deBruijnAllg2} to reflection of the graph at a vertical axis through its middle.
\end{prop}

\begin{proof}
Recall that \( u_{1} \) and \( v_{1} \) denote respectively the prefix and the suffix of length \( L \) of \( \PBlock{k} \). First we consider the arcs from \( v_{1} \) to \( u_{1} \). For \( L \leq \Length{ \PBlock{0} } \) there is no such arc since \( u_{1} \) and \( v_{1} \) agree, see Figure~\ref{fig:deBruijn1ToL0}. For \( L \geq \Length{ \PBlock{0} } + 1 \) the vertices on such an arc are the subwords of \( \Word{ \PBlock{k} }{ a }{ \PBlock{k} } \), with \( a \in \Alphab_{k} \), which occur between the suffix of the first \( \PBlock{k} \)-block and the prefix of the second \( \PBlock{k} \)-block. The \( j \)-th vertex on such an arc, counted from \( v_{1} \), is the word \( \Word{ \Restr{ \PBlock{k} }{ \Length{ \PBlock{ k } } - L + 1 + j }{ \Length{ \PBlock{ k } } } }{ a }{ \Restr{ \PBlock{k} }{ 1 }{ j - 1 } } \). An empty interval means that no letter of this \( \PBlock{k} \) occurs. The reversed word is
\[ \Word{ \Rev{ \Restr{ \PBlock{k} }{ 1 }{ j - 1 } }  }{ a }{ \Rev{ \Restr{ \PBlock{k} }{ \Length{ \PBlock{ k } } - L + 1 + j }{ \Length{ \PBlock{ k } } } } } = \Word{ \Restr{ \PBlock{k} }{ \Length{ \PBlock{ k } } - j + 2 }{ \Length{ \PBlock{ k } } } }{ a }{ \Restr{ \PBlock{k} }{ 1 }{ L - j } } \, . \]
As claimed, this is precisely the \( j \)-th vertex on the same arc, but counted from \( u_{1} \).

Now we consider the path from \( u_{1} \) to \( v_{1} \). The \(j\)-th vertex, counted from \( u_{1} \), is \( u \DefAs \Restr{ \PBlock{k} }{ j }{ j + L - 1 } \), which is the subword of \( \PBlock{k} \) that starts at the \( j \)-th letter from the left. Its reflection is \( \Rev{ u } = \Restr{ \PBlock{k} }{ \Length{ \PBlock{k} } - j - L + 2 }{ \Length{ \PBlock{k} } - j + 1 } \), which is the subword of \( \PBlock{k} \) that ends at the \( j \)-th letter from the right. It corresponds to the \(j\)-th vertex, counted from \( v_{1} \), which proves the claim also for this path. 

Finally we consider the arc from \( v_{2} \) to \( u_{2} \) in Figure~ \ref{fig:deBruijnAllg2}. We define \( \widetilde{r} \DefAs \! L \bmod{( \Length{ \PBlock{ k-2 } } + 1 )} \). By the assumptions of Figure~\ref{fig:deBruijnAllg2} we have \( L \leq 2 \Length{ \PBlock{k-1} } - \Length{ \PBlock{k-2} } \). Thus there is a copy of \( u_{2} \) that begins at the start of the first \( \PBlock{k-2} \)-block in \( \Word{ \PBlock{k-1} }{ a_{k-1} }{ \PBlock{k-1} } \). Moreover, there is a copy of \( v_{2} \) that begins in the first \( \PBlock{k-2} \)-block in \( \Word{ \PBlock{k-1} }{ a_{k-1} }{ \PBlock{k-1} } \). The path from \( v_{2} \) to \( u_{2} \) corresponds to the words between them, see Figure~\ref{figV2ToU2}.

\begin{figure}
\centering
\footnotesize
% parameters for the image
\pgfmathsetmacro{\TikzDistBlocks}{3} % distance between two blocks in image in pt
\pgfmathsetmacro{\TikzBuchstLength}{9} % length of a single-letter-block in pt
\pgfmathsetmacro{\TikzLengthPkkk}{(0.95*\linewidth - ( \TikzBuchstLength + 2 * \TikzDistBlocks) ) / 2} % length of pkkk block in pt
\pgfmathsetmacro{\TikzLengthPkk}{(\TikzLengthPkkk - 3*( \TikzBuchstLength + 2 * \TikzDistBlocks) ) / 4} % length of pkk block in pt
\pgfmathsetmacro{\TikzUVLength}{4*(\TikzLengthPkk + 2 * \TikzDistBlocks + \TikzBuchstLength) + \TikzLengthPkk / 3} % length of Prefix / Suffix
\begin{tikzpicture}
[every node/.style ={rectangle, outer sep=0pt, inner sep=0pt},
buchst/.style={minimum width=\TikzBuchstLength pt, minimum height=0.6cm, draw},
pkkk/.style={minimum width=\TikzLengthPkkk pt, minimum height=0.6cm, draw},
pkk/.style={minimum width=\TikzLengthPkk pt, minimum height=0.6cm, draw}]
\node [right] at (0,0) [pkkk] (A) {\( \PBlock{k-1} \)};
\node [buchst] (A) [right=\TikzDistBlocks pt of A] {\( b \)};
\node [pkkk] (A) [right=\TikzDistBlocks pt of A] {\( \PBlock{k-1} \)};
\node [right] at (0,-1) [pkk] (A) {\( \PBlock{k-2} \)};
\node [buchst] (A) [right=\TikzDistBlocks pt of A] {\( b \)};
\node [pkk] (Fix1) [right=\TikzDistBlocks pt of A] {\( \PBlock{k-2} \)};
\node [buchst] (A) [right=\TikzDistBlocks pt of Fix1] {\( b \)};
\foreach \x in {1,...,2}{
	\node [pkk] (A) [right=\TikzDistBlocks pt of A] {\( \PBlock{k-2} \)};
	\node [buchst] (A) [right=\TikzDistBlocks pt of A] {\( b \)};}
\node [pkk] (Fix2) [right=\TikzDistBlocks pt of A] {\( \PBlock{k-2} \)};
\node [buchst] (A) [right=\TikzDistBlocks pt of Fix2] {\( b \)};
\foreach \x in {1,...,2}{
	\node [pkk] (A) [right=\TikzDistBlocks pt of A] {\( \PBlock{k-2} \)};
	\node [buchst] (A) [right=\TikzDistBlocks pt of A] {\( b \)};}
\node [pkk] (A) [right=\TikzDistBlocks pt of A] {\( \PBlock{k-2} \)};
\node (C1) [below=0.4 of Fix2.south east] {};
\path (C1) +(-\TikzUVLength pt, -0.6) node(B1) {};
\node (B2) [below=1.4 of Fix1.south west] {};
\path (B2) +(\TikzUVLength pt, -0.6) node(C2) {};
\draw (B1) rectangle (C1) node[midway]{\( v_{2} \)};
\draw (B2) rectangle (C2) node[midway]{\( u_{2} \)};
\end{tikzpicture}
\normalsize
\caption{By shifting the suffix \( v_{2} \)  we obtain the prefix \( u_{2} \). For brevity we use \( b \DefAs a_{k-1} \).\label{figV2ToU2}}
\end{figure}
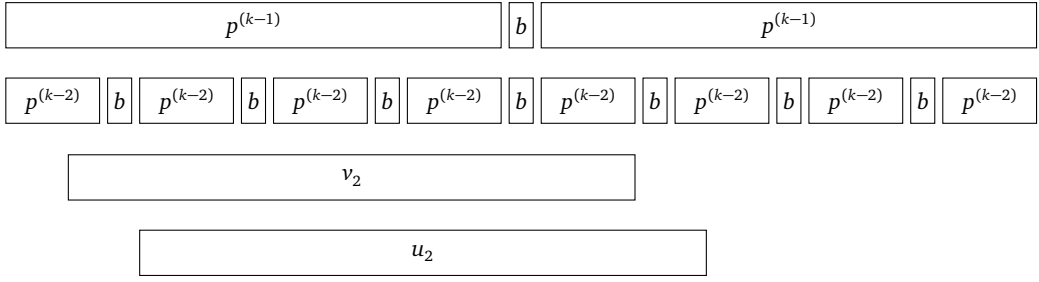

If \(x_{j} \) denotes the \( j \)-th vertex on this path, counted from \( v_{2} \), then we have
\begin{align*}
x_{j} &= \Word{ \Restr{ \PBlock{k-1} }{ \Length{ \PBlock{k-2} } + 1 - \widetilde{r} + j }{ \Length{ \PBlock{k-1} } } }{ a_{k-1} }{ \Restr{ \PBlock{k-1} }{ 1 }{ \Length{ \PBlock{ k-2 } } + j } } \\
& = \Big[ \Word{ \Restr{ \PBlock{k-2} }{ \Length{ \PBlock{k-2} } + 1 - \widetilde{r} + j }{ \Length{ \PBlock{k-2} } } }{ a_{k-1} }{ \PBlock{k-2} }{ \ldots }{ \PBlock{k-2} } \Big] \, \Word{ a_{k-1} }{ \PBlock{k-2} }{ a_{k-1} }{ \Restr{ \PBlock{k-2} }{ 1 }{ j-1 } } \, .
\end{align*}
Since the \( \PBlock{k-2} \)-blocks are palindromes, the reflection of \( x_{j} \) is given by
\begin{align*}
\Rev{ x_{j} } &= \Word{ \Rev{ \Restr{ \PBlock{k-2} }{ 1 }{ j - 1 } } }{ a_{k-1} }{ \PBlock{k-2} }{ \ldots }{ \PBlock{k-2} }{ a_{k-1} }{ \Rev{ \Restr{ \PBlock{k-2} }{ \Length{ \PBlock{k-2} } + 1 - \widetilde{r} + j }{ \Length{ \PBlock{k-2} } } } } \\
&= \Word{ \Restr{ \PBlock{k-2} }{ \Length{ \PBlock{k-2} } - j + 2 }{ \Length{ \PBlock{k-2} } } }{ a_{k-1} }{ \PBlock{k-2} }{ \ldots }{ \PBlock{k-2} }{ a_{k-1} }{ \Restr{ \PBlock{k-2} }{ 1 }{ \widetilde{r} - j } } \\
&= x_{\widetilde{r} +1 - j} \, .
\end{align*}
This is precisely the \( j \)-th vertex on the path, counted from \( u_{2} \). Thus the correspondence between reflected words and the reflected graph holds along all paths.
\end{proof}

\begin{cor}
The number of palindromes of length \( L \) is precisely the number of arcs in \( \Debruijn{L} \) that have an even number of edges.
\end{cor}

Since the number of edges of each arc was given in Figure~\ref{fig:deBruijn1ToL0},~\ref{fig:deBruijnAllg1} and~\ref{fig:deBruijnAllg2}, we can deduce a formula for the palindrome complexity. As before we write \( r \DefAs L \bmod ( \Length{ \PBlock{ k-1 } } + 1 ) \) and \( \widetilde{r} \DefAs L \bmod ( \Length{ \PBlock{ k-2 } } + 1 ) \).

\begin{cor}
\label{cor:PaliComp}
For \( 0 \leq L \leq \Length{ \PBlock{0} } \) the palindrome complexity is given by
\[ \Pali( L ) = ( \Card{ \Alphab_{0} } - 1) \cdot (L \bmod 2) + 1 \, .\]
For \( k \geq 1 \) and \( \Length{ \PBlock{ k-1 } } + 1 \leq L \leq \Length{ \PBlock{ k } } \) the palindrome complexity is given by
\begin{align*}
\Pali( L ) = {} & ( \Card{ \Alphab_{k} } - 1 ) \cdot ( L \bmod 2 ) + (( \Length{ \PBlock{ k-1 } } + 1 - r) \bmod 2) \\
& {} + ( r \bmod 2 ) \cdot \begin{cases}
1 & \text{if } \Length{ \PBlock{ k-1 } } + 1 \leq L \leq \Length{ \PBlock{ k } } - \Length{ \PBlock{ k-1 } } - 1 \\
\CharFkt{ \Alphab_{k+1} }( a_{k} ) & \text{if } \Length{ \PBlock{ k } } - \Length{ \PBlock{ k-1 } } \leq L \leq \Length{ \PBlock{ k } }
\end{cases} \\
& {} + \begin{cases}
( \widetilde{r} \bmod 2 ) + ( (\Length{ \PBlock{k-2} } + 1 - \widetilde{r}) \bmod 2 ) - ( L \bmod 2 ) \\
0
\end{cases} \\
& \quad \begin{cases}
\text{if } a_{k-1} \in \Alphab_{k} \text{ and }  \Length{ \PBlock{ k-1 } } + 1 \leq L \leq 2 \Length{ \PBlock{k-1} } - \Length{ \PBlock{k-2} } \\
\text{if } a_{k-1} \notin \Alphab_{k} \text{ or } 2 \Length{ \PBlock{k-1} } - \Length{ \PBlock{k-2} } < L \leq \Length{ \PBlock{ k } }
\end{cases} \hspace{-0.7em} .
\end{align*}
\end{cor}

\begin{exmpl}[period doubling]
We have \( \Alphab = \AlphabEv = \Set{ a, b } \) and \( \Length{ \PBlock{k} } + 1 = 2^{k+1} \). The previous corollary reproduces precisely the results from \cite{Dam_LocSymPD}: the first values of the palindrome complexity are
\[ \Pali( 0 ) = 1 \, , \; \Pali( 1 ) = 2 \, , \; \Pali( 2 ) = 1 \;\; \text{and} \quad \Pali( 3 ) = 3 \, . \]
For \( k \geq 2 \) the numbers \( \Length{\PBlock{k-2}} + 1 = 2^{k-1} \) and \( \Length{\PBlock{k-1}} + 1 = 2^{k} \) are even and do therefore not affect our computations modulo two. Hence the palindrome complexity for \( k \geq 2 \) and \( 2^{ k } \leq L \leq 2^{ k+1 } - 1 \) is given by
\[ \Pali( L ) = \begin{cases}
4\cdot ( L \bmod 2 )  & \text{if } 2^{ k } \leq L \leq 2^{k+1} - 2^{k-1} - 1 \\
3 \cdot ( L \bmod 2 )  & \text{if } 2^{k+1}  - 2^{k-1} -1 < L \leq 2^{ k+1 } - 1
\end{cases} \, . \qedhere\]
\end{exmpl}

\begin{exmpl}[generalised Grigorchuk subshift]
For \( 1 \leq L \leq \Length{ \PBlock{0} } \) the palindrome complexity is given by 
\[ \Pali( L ) = ( \Card{ \Alphab_{0} } - 1 ) \cdot ( L \bmod 2 ) + 1 \, . \]
For \( \Length{ \PBlock{ 0 } } + 1 \leq L \leq \Length{ \PBlock{ 1 } } \) the general form simplifies to
\begin{align*}
\Pali( L ) = {} & ( L \bmod 2 ) \cdot \Bigg( \Card{ \Alphab_{1} } + \begin{cases}
1 & \text{if } \Length{ \PBlock{ 0 } } + 1 \leq L \leq \Length{ \PBlock{ 1 } } - \Length{ \PBlock{ 0 } } - 1 \\
\CharFkt{ \Alphab_{2} }( a_{1} ) & \text{if } \Length{ \PBlock{ 1 } } - \Length{ \PBlock{ 0 } } \leq L \leq \Length{ \PBlock{ 1 } }
\end{cases} \Bigg) \, .
\end{align*}
Since \( \Length{ \PBlock{k} } + 1 \) is even for all \( k \geq 0 \), we obtain \( ( r \bmod 2 ) = ( L \bmod 2 ) = ( \widetilde{r} \bmod 2 ) \) for all \( k \geq 2 \). Thus the palindrome complexity for \( \Length{ \PBlock{ k-1 } } + 1 \leq L \leq \Length{ \PBlock{ k } } \) with \( k \geq 2 \) is given by
\begin{align*}
\Pali( L ) =  ( L \bmod 2 ) \cdot \Bigg( & \Card{ \Alphab_{k} } + \begin{cases}
1 & \text{if } \Length{ \PBlock{ k-1 } } + 1 \leq L \leq \Length{ \PBlock{ k } } - \Length{ \PBlock{ k-1 } } - 1 \\
\CharFkt{ \Alphab_{k+1} }( a_{k} ) & \text{if } \Length{ \PBlock{ k } } - \Length{ \PBlock{ k-1 } } \leq L \leq \Length{ \PBlock{ k } }
\end{cases} \\
& {} + \begin{cases}
1 & \text{if } a_{k-1} \in \Alphab_{k} \text{ and }  L \leq 2 \Length{ \PBlock{k-1} } - \Length{ \PBlock{k-2} } \\
0 & \text{if } a_{k-1} \notin \Alphab_{k} \text{ or } L > 2 \Length{ \PBlock{k-1} } - \Length{ \PBlock{k-2} }
\end{cases} \Bigg) \, . \qedhere
\end{align*}
\end{exmpl}

\begin{exmpl}[Grigorchuk subshift]
We have \( \Card{ \Alphab_{0} } = 4 \) and \( a_{0} \notin \Alphab_{1} \), as well as \( \Card{ \Alphab_{k} } = 3 \) and \( a_{k} \in \Alphab_{k+1} \) for all \( k \geq 1 \). Together with \( \Length{ \PBlock{k} } + 1 = 2^{k+1} \) this yields
\[ \Pali( L ) = \begin{cases}
4 \cdot ( L \bmod{ 2 } ) & \text{if } 1 \leq L \leq 3 \\
5 \cdot ( L \bmod{ 2 } ) & \text{if } 2^{k} \leq L \leq  2^{k+1} - 2^{k-1} - 1 \text{ for } k \geq 2\\
4 \cdot ( L \bmod{ 2 } ) & \text{if } 2^{k+1} - 2^{k-1} \leq L \leq 2^{k+1}  -1 \text{ for } k \geq 2 
\end{cases} \, . \qedhere \]
\end{exmpl}

%%%%%%%%%%%%%%%%%%%%%%%%
\section{Repetitivity}
\label{sec:Repe}

Repetitivity describes how long a block of letters has to be to ensure that all words of a given length appear in this block. More precisely, we define the \emph{repetitivity function}\index{repetitivity function} as 
\begin{align*}
\Repe \colon \NN \to \NN \, ,\; L \mapsto \min \Set{ \tilde{L} : {} & {} \text{every } u \in \Langu{ \Subshift }_{L} \text{ is contained in every } v \in \Langu{ \Subshift }_{\tilde{L}} } \, .
\end{align*}
It is easy to see that the repetitivity function is strictly increasing. For all sufficiently large \( L \) we describe it by an explicit formula in Theorem~\ref{thm:Repe}. Afterwards we characterise \( \alpha\)-repetitivity of simple Toeplitz subshifts as well as the special case of linear repetitivity (\( \alpha = 1 \)). The necessary tools are introduced in the first subsection.

\subsection{Detailed description of the coding sequence}

In the following we introduce two notions which describe combinatorial properties of the coding sequence \( (a_{k}) \). The first one is the function
\[ \AllL \colon \NN_{0} \to \NN \, , \; k \mapsto \min \Set{ j > k : \Set{ a_{k+1}, \ldots , a_{j} } = \Alphab_{k+1} } \, . \]
It counts how many positions after \( a_{k} \) it happens for the first time that every letter that occurs after \( a_{k} \) has occurred at least once. In particular \( \AllL( k ) \geq k + \Card{ \Alphab_{k+1} } \geq k + \Card{ \AlphabEv } \) holds for all \( k \in \NN_{0} \). 

\begin{prop}
The function \( \AllL \) is monotonically increasing.
\end{prop}

\begin{proof}
For \( \Alphab_{k+2} = \Alphab_{k+1} \) the definition of \( \AllL \) yields
\[ \AllL( k+1 ) = \min \Set{ j : \Set{ a_{k+2}, \ldots a_{j} } = \Alphab_{k+1} } \geq \min \Set{ j : \Set{ a_{k+1}, \ldots a_{j} } = \Alphab_{k+1} } = \AllL( k ) \, . \]
For \( \Alphab_{k+2} = \Alphab_{k+1} \setminus \Set{ a_{k+1} } \) we obtain
\begin{align*}
\AllL( k+1 ) &= \min \Set{ j : \Set{ a_{k+2}, \ldots , a_{j} } = \Alphab_{k+1} \setminus \Set{ a_{k+1} } } \\
&= \min \Set{ j : \Set{ a_{k+1}, \ldots , a_{j} } = \Alphab_{k+1} } \\
&= \AllL( k ) \, . \qedhere
\end{align*}
\end{proof}

Note that \( \AllL \) is in general not strictly monotonically increasing. Therefore, the second notion we introduce is the sequence \( ( \AllLInc{ i } )_{i \in \NN_{0}} \), defined by \( \AllLInc{ 0 } = 0 \) and \( \AllL( \AllLInc{ i } ) = \AllL( \AllLInc{ i } + 1 ) = \ldots = \AllL( \AllLInc{ i+1 } - 1 ) < \AllL( \AllLInc{ i+1 } ) \). It denotes those positions where \( \AllL \) actually increases. Below, we prove two useful characterisations of \( ( \AllLInc{ i } ) \). Both are based on the idea that \( \AllL(k) = \AllL(k+1) \) holds if \( a_{k+1} \) appears in \( \Set{ a_{k+2} , \ldots , a_{\AllL( k )} } \).

\begin{prop}
\label{prop:AllLIncRecur}
The sequence \( (\AllLInc{ i }) \) is given by \( \AllLInc{ 0 } = 0 \) and the recurrence relation \( \AllLInc{ i+1 } \DefAs  \max \Set{ j \leq \AllL( \AllLInc{ i } ) : \Set{ a_{j} , a_{j+1} , \ldots , a_{\AllL( \AllLInc{ i } )} } = \Alphab_{\AllLInc{ i }+1} } \).
\end{prop}

\begin{proof}
Let \( (\AllLInc{ i }) \) denote the previously defined sequence and let \( (\widetilde{m}_{i}) \) denote the sequence defined in this proposition. By definition \( \AllLInc{ 0 } = 0 = \widetilde{m}_{0} \) holds. We proceed by induction, so assume that \( \AllLInc{ i } = \widetilde{m}_{i} \) holds. Because of \( \Set{ a_{\widetilde{m}_{i}+1} , \ldots , a_{\AllL( \widetilde{m}_{i} )} } = \Alphab_{ \widetilde{m}_{i}+1 } \), we have \( \widetilde{m}_{i+1} \geq \widetilde{m}_{i} +1 \). Combined with \( \Set{ a_{\widetilde{m}_{i+1}}, \ldots , a_{\AllL( \widetilde{m}_{i} )} } = \Alphab_{\widetilde{m}_{i} + 1} \supseteq \Alphab_{\widetilde{m}_{i+1} } \), we obtain the inequality
\[ \AllL( \widetilde{m}_{i+1}-1 ) =  \min \Set{ j > \widetilde{m}_{i+1} - 1 : \Set{ a_{ \widetilde{m}_{i+1} }, \ldots , a_{j} } = \Alphab_{ \widetilde{m}_{i+1} } } \leq \AllL( \widetilde{m}_{i} ) = \AllL( \AllLInc{ i } ) \, . \]
This yields \( \widetilde{m}_{i+1} - 1 \leq \AllLInc{ i } \leq \AllLInc{ i+1 } - 1 \). Conversely, \( \AllL( \AllLInc{ i+1 } - 1 ) = \AllL( \AllLInc{ i } ) \) holds by definition, which implies \( \Set{ a_{ \AllLInc{ i+1 } } , \ldots , a_{ \AllL( \AllLInc{ i } ) } } = \Alphab_{ \AllLInc{ i+1 } } \subseteq \Alphab_{ \AllLInc{ i } + 1 } \). This yields \( \widetilde{m}_{i+1} \geq \AllLInc{ i+1 } \).
\end{proof}

\begin{prop}
\label{prop:AkAAllLK}
For \( k \geq 1 \) the equality \( a_{k} = a_{ \AllL( k ) } \) holds if and only if \( k = \AllLInc{ i } \) holds for some \( i \geq 1 \).
\end{prop}

\begin{proof}
First we assume \( k = \AllLInc{ i } \). The definition of \( \AllL( \AllLInc{ i } ) \) yields
\[ \Set{ a_{ \AllLInc{ i } + 1 }, \ldots, a_{ \AllL( \AllLInc{ i } ) - 1 } } = \Alphab_{ \AllLInc{ i } + 1 } \setminus \Set{ a_{ \AllL( \AllLInc{ i } ) } } \subseteq \Alphab_{ \AllLInc{ i } } \setminus  \Set{ a_{ \AllL( \AllLInc{ i } ) } } \, . \]
Since \( \AllL \) is increasing at \( \AllLInc{ i } \), we also have
\[ \Alphab_{\AllLInc{ i }} = \Set{ a_{ \AllLInc{ i } } , \ldots , a_{ \AllL( \AllLInc{ i } - 1 ) } } \subseteq \Set{ a_{\AllLInc{ i }} , \ldots , a_{ \AllL( \AllLInc{ i } ) - 1 } } \subseteq \Alphab_{\AllLInc{ i }} \, . \]
Thus \( a_{\AllLInc{ i }} = a_{\AllL( \AllLInc{ i } )} \) follows. Conversely we now assume \( a_{k} = a_{\AllL( k )} \) for some \( k \geq 1 \). Consequently we have
\[ \Set{ a_{k} , \ldots , a_{ \AllL( k ) - 1 } } = \Set{ a_{k+1} , \ldots , a_{ \AllL( k ) } } = \Alphab_{k+1} = \Alphab_{k} \, ,\]
and thus \( \AllL( k-1 ) \leq \AllL( k ) - 1 \).
\end{proof}

To illustrate the notions introduced above, we discuss some examples. The results will also play a role when we study \( \alpha \)-repetitivity and the Boshernitzan condition in Subsection~\ref{subsec:AlphaRepe} and Section~\ref{sec:BoshCond} respectively. 

\begin{exmpl}
\label{exmpl:AllLAlphEv2}
For every simple Toeplitz subshift with \( \Card{ \AlphabEv } = 2 \), the sequence \( (a_{k}) \) eventually alternates between the two letters in \( \AlphabEv \). This implies \( \AllL( k ) = k+2 \) and hence \( \AllLInc{ i+1 } = \AllLInc{ i } + 1 \) for all sufficiently large \( k \).
\end{exmpl}

\begin{exmpl}[generalised Grigorchuk subshift]
\label{exmpl:GrigAllLInc}
The case \( \Card{ \AlphabEv } = 2 \) was discussed in the previous example. For \( \Card{ \AlphabEv } = 3 \) we obtain \( \Alphab_{k} = \AlphabEv = \Set{ b , c, d } \) for all \( k \geq 1 \). The definition of  \( \AllL( \AllLInc{ i } ) \) implies \( a_{ \AllL( \AllLInc{ i } )} \notin \Set{ a_{ \AllLInc{ i } + 1 } , \ldots , a_{ \AllL( \AllLInc{ i } ) - 1} } \). Thus the letters \( a_{ \AllL( \AllLInc{ i }) - 2 } \), \( a_{ \AllL( \AllLInc{ i } ) - 1 } \) and \( a_{ \AllL( \AllLInc{ i } ) } \) are pairwise different and \( \AlphabEv = \Set{ a_{ \AllL( \AllLInc{ i } ) - 2 } , a_{ \AllL( \AllLInc{ i } ) - 1 } , a_{ \AllL( \AllLInc{ i } ) } } \) holds. Now Proposition~\ref{prop:AllLIncRecur} yields \( \AllLInc{ i+1 } = \AllL( \AllLInc{ i } ) - 2 \) for all \( i \geq 0 \). For the special case of the standard Grigorchuk subshift, we have \( \AllL( k ) = k+3 \) for all \( k \geq 0 \), and thus \( \AllLInc{ i+1 } = \AllLInc{ i } + 1 \).
\end{exmpl}

\begin{exmpl}[non-(B) example]
\label{exmpl:BoshExAllL}
Recall from Example~\ref{exmpl:BoshSubshiftDef} that \( (a_{k}) \) is given by
\[ (a_{k})_{k \in \NN_{0}} = (\hspace{-0.3em} \underbrace{a, b}_{2 \text{ letters}} \hspace{-0.5em}, c, \underbrace{a, b, a, b}_{4 \text{ letters}}, d, \ldots , c, \underbrace{a, b, \ldots, a, b}_{4l \text{ letters}}, d, \underbrace{a, b, \ldots, a, b}_{4l+2 \text{ letters}}, c, \ldots ) \, . \]
In this sequence, the \( l \)-th neither-\( a \)-nor-\( b \) letter has the index \( ( l+1 )^{2} - 2  \). When we have reached the next \( c \) and the next \( d \), we have reached all letters, since there are \( a \)'s and \( b \)'s in between them. This yields
\[ \AllL( k ) = (l+2)^{2} - 2 \quad \text{for } k \in \NN_{0} \text{ with } l^{2} - 2 \leq k < ( l+1 )^{2} - 2 \, . \]
In particular, \( ( \AllLInc{ i } ) \) is given by \( \AllLInc{ 0 } = 0 \) and \( \AllLInc{ i } = (i+1)^{2} - 2 \) for \( i \geq 1 \). Moreover \( \AllL( \AllLInc{ i } ) = \AllLInc{ i+2 } \) holds. It will turn out that this behaviour of \( \AllLInc{ i } \) and \(  \AllL( \AllLInc{ i } ) \) is the reason that the subshift is a counterexample to both, \( \alpha \)-repetitivity (see Example~\ref{exmpl:AlphaRepeBosh}) and the Boshernitzan condition (Example~\ref{exmpl:BoshExBoshCond}). In fact, this example was specifically designed in \cite{LiuQu_Simple} to violate the Boshernitzan condition.
\end{exmpl}

\subsection{The repetitivity of simple Toeplitz subshifts}

In the following we prove lower and upper bounds for the repetitivity function. It turns out that they can only be satisfied simultaneously if \( \Repe( L+1 ) - \Repe( L ) = 1 \) holds for all sufficiently large \( L \). This yields an explicit formula for the repetitivity function.

Before we begin, we repeat briefly which words occur in a simple Toeplitz subshift: according to Proposition~\ref{prop:WordsContained} every word of length \( \Length{ \PBlock{k} } + 1 \) or less is a subword of \( \Word{ \PBlock{k} }{ a }{ \PBlock{k} } \) for some \( a \in \Alphab_{k+1} \). By Proposition~\ref{prop:UpBdCompBlock} and Corollary~\ref{cor:CompGrowth}, all subwords of length \( \Length{ \PBlock{k} } + 1 \) that start in \( \Word{ \PBlock{k} }{ a } \) for \( a \in \Alphab_{k+1} \setminus \Set{ a_{k} } \) and in \( \Word{ \PBlock{k-1} }{ a_{k} } \) (if \( a_{k} \in \Alphab_{k+1} \) holds) are pairwise different. To bound the repetitivity from below, we show that certain words do not occur in a large blocks.

\begin{prop}
\label{prop:RepeLB1}
For every \( i \geq 1 \) we have the lower bound
\[ \Repe( \Length{ \PBlock{ \AllLInc{ i } } }  - \Length{ \PBlock{ \AllLInc{ i } - 1 } } + 1 ) \geq 2 \cdot ( \Length{ \PBlock{ \AllL( \AllLInc{ i } ) -1 } } + 1 ) \, . \]
\end{prop}

\begin{proof}
Let \( v \) denote the suffix of length \( \Length{ \PBlock{ \AllLInc{ i }} } + 1 - \Length{ \PBlock{ \AllLInc{ i } - 1 } } \) of \( \Word{ \PBlock{\AllLInc{ i }} }{ a_{ \AllL( \AllLInc{ i } ) } } \). Because of \( a_{ \AllL( \AllLInc{ i } ) } = a_{\AllLInc{ i }} \) (see Proposition~\ref{prop:AkAAllLK}) we obtain the decomposition
\[ \Word{ \PBlock{\AllLInc{ i }} }{ a_{ \AllL( \AllLInc{ i } ) } } = \PBlock{ \AllLInc{ i }-1 } \, \underbrace{ \Word{ a_{ \AllL( \AllLInc{ i } ) } }{ \PBlock{ \AllLInc{ i }-1 } }{ \ldots }{ \PBlock{ \AllLInc{ i }-1 } }{ a_{ \AllL( \AllLInc{ i } ) } }{ \PBlock{ \AllLInc{ i }-1 } }{ a_{ \AllL( \AllLInc{ i } ) } }}_{\text{the suffix } v } \; ,
\]
with \( n_{\AllLInc{ i }} \) consecutive single letters \( a_{ \AllL( \AllLInc{ i } ) } \). Let \( c \in \Alphab_{ \AllL( \AllLInc{ i } ) + 1 } \setminus \Set{ a_{ \AllL( \AllLInc{ i } ) } } \) and consider the word \( u \DefAs \Word{ \PBlock{ \AllL( \AllLInc{ i } ) - 1 } }{ c }{ \PBlock{ \AllL( \AllLInc{ i } ) - 1 } } \). We decompose \( \PBlock{ \AllL( \AllLInc{ i } ) - 1 } \) into \( \PBlock{ \AllLInc{ i } } \)-blocks and single letters from \( \Set{ a_{ \AllLInc{ i }+1 }, \ldots , a_{ \AllL( \AllLInc{ i } ) - 1 } } \). Since the definition of \( \AllL \) implies \( a_{ \AllL( \AllLInc{ i } ) } \notin \Set{ a_{ \AllLInc{ i }+1 }, \ldots , a_{ \AllL( \AllLInc{ i } ) - 1 } } \), there are at most \( n_{\AllLInc{ i }} - 1 \) consecutive occurrences of \( a_{ \AllL( \AllLInc{ i } ) } \) in the \( \PBlock{ \AllLInc{ i } - 1 } \)-block decomposition of \( u \). Thus, \( v \) does not occur in \( u \) in such a way that the \( \PBlock{ \AllLInc{ i } - 1 } \)-blocks align. It does not occur at a different position either, since the subwords of length \( \Length{ \PBlock{ \AllLInc{ i } - 1 } } + 1 \) of \( \Word{ \PBlock{ \AllLInc{ i } - 1 } }{ a }{ \PBlock{ \AllLInc{ i } - 1 } } \) are pairwise different. This yields
\[ \Repe( \Length{ \PBlock{ \AllLInc{ i } } } + 1 - \Length{ \PBlock{ \AllLInc{ i } - 1 } } ) = \Repe( \Length{ v } ) > \Length{ u } = 2 \cdot \Length{ \PBlock{ \AllL( \AllLInc{ i } ) - 1 } } + 1 \, . \qedhere \]
\end{proof}

\begin{prop}
\label{prop:RepeLB2}
For every \( i \geq 1 \) we have the lower bound
\[ \Repe( \Length{ \PBlock{ \AllLInc{ i } } } + 2 ) \geq 2 \cdot ( \Length{ \PBlock{ \AllL( \AllLInc{ i } ) - 1 } } + 1 ) + \Length{ \PBlock{ \AllLInc{ i } } } + 1 \, . \]
\end{prop}

\begin{proof}
The proof is similar to the previous one: we show that \( v \DefAs \Word{ a_{ \AllL( \AllLInc{ i } ) } }{ \PBlock{\AllLInc{ i }} }{ a_{ \AllLInc{ i }+1 } } \) does not occur in \( u \DefAs \Word{ \PBlock{ \AllL( \AllLInc{ i } ) -1 } }{ c }{ \PBlock{ \AllL( \AllLInc{ i } ) -1 } }{ a_{ \AllL( \AllLInc{ i } ) } }{ \PBlock{ \AllLInc{ i } } } \), with \( c \in \Alphab_{ \AllL( \AllLInc{ i } ) + 1 } \setminus \Set{ a_{ \AllL( \AllLInc{ i } ) } } \). Recall from the previous proof that \( a_{ \AllL( \AllLInc{ i } ) } \) does not appear as a single letter in the \( \PBlock{ \AllLInc{ i } } \)-block decomposition of \( \PBlock{ \AllL( \AllLInc{ i } ) - 1 } \). Thus any two \( \PBlock{ \AllLInc{ i } } \)-blocks in \( u \!=\! \PBlock{ \AllL( \AllLInc{ i } ) - 1 } c \, \PBlock{ \AllL( \AllLInc{ i } ) - 1 } a_{ \AllL( \AllLInc{ i } ) } \, \PBlock{ \AllLInc{ i } } \) are separated by a non-\( a_{ \AllL( \AllLInc{ i } ) } \) letter, except for the last two blocks. Consequently \( v = \Word{ a_{ \AllL( \AllLInc{ i } ) } }{ \PBlock{\AllLInc{ i }} }{ a_{ \AllLInc{ i }+1 } } \) does not occur in \( u \).
\end{proof}

\begin{rem}
\label{rem:CompLowBd}
For \( \AllLInc{ i+1 } = \AllLInc{ i } + 1 \) and \( n_{\AllLInc{ i }+1} = 2 \) Proposition~\ref{prop:RepeLB1} and~\ref{prop:RepeLB2} bound the repetitivity function at the same length, since
\[ \Length{ \PBlock{ \AllLInc{ i+1 } } } - \Length{ \PBlock{ \AllLInc{ i+1 } - 1 } } + 1 = \Length{ \PBlock{ \AllLInc{ i+1 } - 1 } } + 2  = \Length{ \PBlock{ \AllLInc{ i } } } + 2 \, .\] 
holds. The lower bound in Proposition~\ref{prop:RepeLB1} is stronger, as a short computation shows:
\begin{align*}
2 \cdot ( \Length{ \PBlock{ \AllL( \AllLInc{ i+1 } ) - 1 } } + 1 ) 
& \geq 2 \cdot n_{ \AllL( \AllLInc{ i+1 } ) - 1 } \cdot ( \Length{ \PBlock{ \AllL( \AllLInc{ i } ) - 1 } } + 1 ) \\
& = 2 \cdot ( \Length{ \PBlock{ \AllL( \AllLInc{ i } ) - 1 } } + 1 ) + 2 \cdot ( n_{\AllL( \AllLInc{ i+1 } ) - 1} - 1) ( \Length{ \PBlock{ \AllL( \AllLInc{ i } ) - 1 } } + 1 ) \\
& > 2 \cdot ( \Length{ \PBlock{ \AllL( \AllLInc{ i } ) - 1 } } + 1 ) + \Length{ \PBlock{ \AllLInc{ i } } } + 1 \, .
\end{align*}
Essentially, the case \( n_{\AllLInc{ i }+1} = 2 \) is special since \( \Word{ a_{\AllLInc{ i }+1} }{ \PBlock{\AllLInc{ i }} }{ a_{\AllLInc{ i }+1} } \) is not contained in \( \PBlock{ \AllLInc{ i }+1 } \). Moreover, \( \AllLInc{ i+1 } = \AllLInc{ i } + 1 \) implies that the letter \( a_{\AllLInc{ i }+1} = a_{\AllLInc{ i+1 }} = a_{\AllL( \AllLInc{ i+1 } ) } \) is not in \( \Set{ a_{ \AllLInc{ i } + 2 }\), \( \ldots, a_{ \AllL( \AllLInc{ i+1 } ) - 1} } \). Thus we need to look at a much longer word to see \( \Word{ a_{ \AllLInc{ i } + 1 } }{ \PBlock{ \AllLInc{ i } } }{ a_{\AllLInc{ i }+1} } \) than for \( n_{\AllLInc{ i }+1} > 2 \) or for \( \AllLInc{ i+1 } > \AllLInc{ i } + 1 \). 
\end{rem}

Now we prove upper bounds for the repetitivity. The idea is to show that all sufficiently long words contain all \( \Word{ \PBlock{k} a \PBlock{k} }\) for \( a \in \Alphab_{k+1} \), which contain all words of length \( \Length{ \PBlock{k} } + 1 \). As for the lower bound, we prove the inequality at two values.

\begin{prop}
\label{prop:RepeUB1}
For every  \( i \geq 1\) we have the upper bound
\[ \Repe( \Length{ \PBlock{ \AllLInc{ i } } } - \Length{ \PBlock{ \AllLInc{ i } - 1 } } ) \leq 2 \cdot ( \Length{ \PBlock{ \AllL( \AllLInc{ i } - 1 ) - 1 } } + 1 ) + \Length{ \PBlock{ \AllLInc{ i } } } - \Length{ \PBlock{ \AllLInc{ i } - 1 } } - 1 \, . \]
\end{prop}

\begin{proof}
To shorten notation we write \( \widetilde{m} \DefAs \AllL( \AllLInc{ i-1 } ) - 1 = \AllL( \AllLInc{ i } - 1 ) - 1 \). It suffices to show that every word of length \( 2 \cdot ( \Length{ \PBlock{ \widetilde{m} } } + 1 ) + \Length{ \PBlock{ \AllLInc{ i } } } - \Length{ \PBlock{ \AllLInc{ i } - 1 } } - 1 \) contains all subwords of length \( \Length{ \PBlock{ \AllLInc{ i } } } - \Length{ \PBlock{ \AllLInc{ i } - 1 } } \) of \( \Word{ \PBlock{ \AllLInc{ i } } }{ a }{ \PBlock{ \AllLInc{ i } } } \) for \( a \in \Alphab_{ \AllLInc{ i }+1 } \).

Every word of length \( 2 \cdot ( \Length{ \PBlock{ \widetilde{m} } } + 1 ) + \Length{ \PBlock{ \AllLInc{ i } } } - \Length{ \PBlock{ \AllLInc{ i } - 1 } } - 1 \) contains \( \PBlock{ \widetilde{m} } \). The decomposition of \( \PBlock{ \widetilde{m} } \) into \( \PBlock{\AllLInc{ i }} \)-blocks shows that \( \PBlock{ \widetilde{m} } \) contains \( \Word{ \PBlock{ \AllLInc{ i } } }{ a }{ \PBlock{ \AllLInc{ i } } } \) for all \( a \in \Set{  a_{ \AllLInc{ i }+1 } , \ldots, a_{\widetilde{m}} } \). Because of \( \Alphab_{ \AllLInc{ i }+1 } \subseteq \Alphab_{\AllLInc{ i }} = \Set{ a_{\AllLInc{ i }}, \ldots, a_{ \AllL( \AllLInc{ i }-1 ) } } \), this leaves only \( a \in \Set{ a_{\AllLInc{ i }} , a_{ \AllL( \AllLInc{ i }-1 ) } } \) to deal with. For \( a = a_{\AllLInc{ i }} \) it suffices to study words that start in \( \Word{ \PBlock{\AllLInc{ i }-1} }{ a_{\AllLInc{ i }} } \). Since their length is \( \Length{ \PBlock{ \AllLInc{ i } } } - \Length{ \PBlock{ \AllLInc{ i } - 1 } } \), they are contained in \( \PBlock{\AllLInc{ i }} \), which is contained in \( \PBlock{\widetilde{m}} \) and thus in \( u \). For \( a = a_{ \AllL( \AllLInc{ i }-1 ) } \) we need a refined version of the arguments above:

Every word \( u \) of length \( 2 \cdot ( \Length{ \PBlock{ \widetilde{m} }} + 1 ) + \Length{ \PBlock{ \AllLInc{ i } } } - \Length{ \PBlock{ \AllLInc{ i } - 1 } } - 1 \) contains necessarily a \( \PBlock{ \widetilde{m} } \)-block together with both neighbouring letters. At least one of the letters is \( a_{ \AllL( \AllLInc{ i }-1 ) } \) and without loss of generality we assume that it is the right one. To the right of it, the next \( \PBlock{ \AllLInc{ i } } \)-block begins. We distinguish two cases: if there are at least \( \Length{ \PBlock{\AllLInc{ i }} } - \Length{ \PBlock{ \AllLInc{ i }-1 } } - 1 \) letters of \( u \) to the right of \( a_{ \AllL( \AllLInc{ i }-1 ) } \), then \( u \) contains all subwords of \( \Word{ \PBlock{\AllLInc{ i }} }{ a_{ \AllL( \AllLInc{ i }-1 ) } }{ \PBlock{\AllLInc{ i }} } \) which start in \( \Word{ \PBlock{\AllLInc{ i }} }{ a_{ \AllL( \AllLInc{ i }-1 ) } } \) and we are done. If there are \( 0 \leq j < \Length{ \PBlock{ \AllLInc{ i } } } - \Length{ \PBlock{ \AllLInc{ i }-1 } } - 1 \) letters of \( u \) to the right of \( a_{ \AllL( \AllLInc{ i }-1 ) } \), then there are \( 2 (\Length{ \PBlock{ \widetilde{m} } } + 1 ) + \Length{ \PBlock{ \AllLInc{ i } } } - \Length{ \PBlock{ \AllLInc{ i }-1 } } - j - 2 \) letters of \( u \) to the left of \( a_{ \AllL( \AllLInc{ i }-1 ) } \). They form a \( \PBlock{ \widetilde{m} } \)-block, a single letter, another \( \PBlock{ \widetilde{m} } \)-block, another single letter and the rightmost \( \Length{ \PBlock{\AllLInc{ i }} } - \Length{ \PBlock{ \AllLInc{ i }-1 } } - 2 - j \) letters of \( \PBlock{\AllLInc{ i }} \), see Figure~\ref{fig:LRRemainders}.

\begin{figure}
\centering
\footnotesize
% set parameters for the image
\pgfmathsetmacro{\TikzDistBlocks}{3} % distance between two blocks in image in pt
\pgfmathsetmacro{\TikzBuchstLength}{9} % length of a single-letter-block in pt
% compute block length from parameters
\pgfmathsetmacro{\TikzLengthPkkk}{(0.95*\linewidth - 3*( \TikzBuchstLength + 2 * \TikzDistBlocks) ) / 3} % length of pkkk block in pt
\begin{tikzpicture}
[every node/.style ={rectangle, outer sep=0pt, inner sep=0pt},
buchst/.style={minimum width=\TikzBuchstLength pt, minimum height=0.6cm, draw},
pkkk/.style={minimum width=\TikzLengthPkkk pt, minimum height=0.6cm, draw}]
\node [right] at (0,0) [buchst] (Fix1) {\( \star \)};
\node [pkkk] (A) [right=\TikzDistBlocks pt of Fix1] {\( \PBlock{ \widetilde{m} } \)};
\node [buchst] (A) [right=\TikzDistBlocks pt of A] {\( \star \)};
\node [pkkk] (A) [right=\TikzDistBlocks pt of A] {\( \PBlock{ \widetilde{m} } \)};
\node [buchst] (Fix2) [right=\TikzDistBlocks pt of A] {\( a \)};
\draw (Fix1.north west)++(-\TikzDistBlocks pt, 0) -- ++(-\TikzLengthPkkk / 2 pt, 0);
\draw (Fix1.north west)++(-\TikzDistBlocks pt, 0) -- node [midway, left=4*\TikzDistBlocks pt]{\( \PBlock{\AllLInc{ i }} \)} ++(0, -0.6) -- ++(-\TikzLengthPkkk / 2 pt, 0);
\draw (Fix2.north east)++(\TikzDistBlocks pt, 0) -- ++(\TikzLengthPkkk / 2 pt, 0);
\draw (Fix2.north east)++(\TikzDistBlocks pt, 0) -- node [midway, right=4*\TikzDistBlocks pt]{\( \Restr{ \PBlock{\AllLInc{ i }} }{ 1 }{ j } \)} ++(0, -0.6) -- ++(\TikzLengthPkkk / 2 pt, 0);
\end{tikzpicture}
\normalsize
\caption{Decomposition of a word of length \( 2 \cdot ( \Length{ \PBlock{ \widetilde{m} }} + 1 ) + \Length{ \PBlock{ \AllLInc{ i } } } - \Length{ \PBlock{ \AllLInc{ i } - 1 } } - 1 \) into \( \PBlock{ \AllLInc{i} } \)-blocks and single letters. To shorten notation we use \( a \DefAs a_{ \AllL(  \AllLInc{ i }-1 ) }\).\label{fig:LRRemainders}}
\end{figure}

At least one of the single letters is \( a_{ \AllL( \AllLInc{ i }-1 ) } \). If it is the right letter, then \( u \) contains \( \Word{ \PBlock{\AllLInc{ i }} }{ a_{ \AllL( \AllLInc{ i }-1 ) } }{ \PBlock{\AllLInc{ i }} } \) and we are done. If it is the left letter, then note that the right end of \( u \) contains all subwords of length \( \Length{ \PBlock{\AllLInc{ i }} } - \Length{ \PBlock{ \AllLInc{ i }-1 } } \) of \( \Word{ \PBlock{\AllLInc{ i }} }{ a_{ \AllL( \AllLInc{ i }-1 ) } }{ \PBlock{\AllLInc{ i }} } \) that start in \( \Restr{ \PBlock{\AllLInc{ i }} }{ 1 }{ \Length{ \PBlock{ \AllLInc{ i }-1 } } + j + 2 } \). Since the left end of \( u \) is \( \Word{ \Restr{ \PBlock{\AllLInc{ i }} }{ \Length{ \PBlock{ \AllLInc{ i }-1 }} + j + 3 }{ \Length{ \PBlock{\AllLInc{ i }} } } }{ a_{ \AllL( \AllLInc{ i }-1 ) } }{ \PBlock{\AllLInc{ i }} } \), it contains the remaining subwords, which finishes the proof.
\end{proof}

\begin{prop}
\label{prop:RepeUB2}
For every  \( i \geq 1\) we have the upper bound
\[ \Repe( \Length{ \PBlock{\AllLInc{ i }} } + 1 ) \leq 2 \cdot ( \Length{ \PBlock{ \AllL( \AllLInc{ i } ) - 1 } } + 1 ) + \Length{ \PBlock{ \AllLInc{ i }-1 } } \, . \]
\end{prop}

\begin{proof}
The proof is similar to the previous one. We write \( \widetilde{m} = \AllL( \AllLInc{ i } ) - 1 \) and show that every word \( u \) of length \( 2 \cdot ( \Length{ \PBlock{ \widetilde{m} } } + 1 ) + \Length{ \PBlock{ \AllLInc{ i }-1 } } \) contains all subwords of length \( \Length{ \PBlock{\AllLInc{ i }} } + 1 \) of \( \Word{ \PBlock{\AllLInc{ i }} }{ a }{ \PBlock{\AllLInc{ i }} } \) with \( a \in \Alphab_{ \AllLInc{ i }+1 } \).

It follows from \( \Length{u} = 2 \cdot ( \Length{ \PBlock{ \widetilde{m} } } + 1 ) + \Length{ \PBlock{ \AllLInc{ i }-1 } } \) that \( u \) contains \( \PBlock{ \widetilde{m} } \). The decomposition of \( \PBlock{ \widetilde{m} } \) into \( \PBlock{\AllLInc{ i }} \)-blocks shows that \( u \) contains \( \Word{ \PBlock{\AllLInc{ i }} }{ a }{ \PBlock{\AllLInc{ i }} } \) for all  \( a \in \Set{ a_{ \AllLInc{ i }+1 }, \ldots , a_{ \widetilde{m} } } = \Alphab_{ \AllLInc{ i }+1 } \setminus \Set{ a_{ \AllL( \AllLInc{ i } ) } } \). To treat \( a = a_{ \AllL( \AllLInc{ i } ) } = a_{\AllLInc{ i }} \), we note that \( u \) necessarily contains \( \PBlock{ \widetilde{m} } \) together with both neighbouring letters. At least one of them is \( a_{ \widetilde{m}+1 } = a_{\AllLInc{ i }} \) and without loss of generality we assume that it is the right one. If there are at least \( \Length{ \PBlock{\AllLInc{ i }} } \) letters of \( u \) to the right of \( a_{\AllLInc{ i }} \), then we are done. If there are \( 0 \leq j < \Length{ \PBlock{\AllLInc{ i }} } \) letters of \( u \) to the right of \( a_{\AllLInc{ i }} \), then there are \( \Length{ \PBlock{ \AllLInc{ i }-1 } } - 1 - j +2( 1+ \Length{ \PBlock{ \widetilde{m} } } ) \) letters to the left of \( a_{\AllLInc{ i }} \). They form a \( \PBlock{ \widetilde{m} } \)-block, a single letter, another \( \PBlock{ \widetilde{m} } \)-block, another single letter and the rightmost \( \Length{ \PBlock{ \AllLInc{ i }-1 } } - 1 - j \) letters of \( \PBlock{ \AllLInc{ i }-1 } \). If the single letter to the right is \( a_{\AllLInc{ i }} \), then we are done. If the single letter to the left is \( a_{\AllLInc{ i }} \), then the right end of \( u \) contains all subwords of length \( \Length{ \PBlock{\AllLInc{ i } } } + 1 \) of \( \Word{ \PBlock{\AllLInc{ i }} }{ a_{\AllLInc{ i }} }{ \PBlock{\AllLInc{ i }} } \) that start in \( \Restr{ \PBlock{ \AllLInc{ i }-1 } }{ 1 }{ 1+j } \). The left end of \( u\) is \( \Word{ \Restr{ \PBlock{ \AllLInc{ i }-1 } }{ 2+j }{ \Length{ \PBlock{\AllLInc{ i }} } } }{ a_{\AllLInc{ i }} }{ \PBlock{\AllLInc{ i }} } \), which contains the remaining subwords.
\end{proof}

\begin{rem}
Similar to Remark~\ref{rem:CompLowBd} we note that
\[ \Length{ \PBlock{\AllLInc{ i+1 }} } - \Length{ \PBlock{ \AllLInc{ i+1 }-1 } } = \Length{ \PBlock{ \AllLInc{ i+1 }-1 } } + 1  = \Length{ \PBlock{ \AllLInc{ i } } } + 1 \] 
holds for \( \AllLInc{ i+1 } = \AllLInc{ i }+1 \) and \( n_{\AllLInc{ i+1 }} = 2 \). Thus Proposition~\ref{prop:RepeUB1} and~\ref{prop:RepeUB2} bound the repetitivity at the same length. A short computation shows that the bound in Proposition~\ref{prop:RepeUB2} is stronger:
\begin{align*}
2 \cdot ( \Length{ \PBlock{ \AllL( \AllLInc{ i } ) - 1 } } + 1 ) + \Length{ \PBlock{ \AllLInc{ i }-1 } } 
&< 2 \cdot ( \Length{ \PBlock{ \AllL( \AllLInc{ i } ) - 1 } } + 1 ) + \Length{ \PBlock{\AllLInc{ i }} } \\
& = 2 \cdot ( \Length{ \PBlock{ \AllL( \AllLInc{ i } ) - 1 } } + 1 ) + \Length{ \PBlock{\AllLInc{ i+1 }} } - \Length{ \PBlock{ \AllLInc{ i+1 } - 1 } } - 1 \, . \qedhere
\end{align*}
\end{rem}

We now deduce an explicit formula for the repetitivity. We show that the inequalities above can only be satisfied simultaneously if they are actually equalities and \( \Repe( L+1 ) - \Repe( L ) = 1 \) holds for all \( L \) except \( \Length{ \PBlock{\AllLInc{ i } } } + 1 \) and \( \Length{ \PBlock{\AllLInc{ i }} } - \Length{ \PBlock{ \AllLInc{ i }-1 } } \).

\begin{thm}
\label{thm:Repe}
For \( i \geq 1 \) and \( \Length{ \PBlock{\AllLInc{ i }} } - \Length{ \PBlock{ \AllLInc{ i }-1 } } + 1 \leq L \leq \Length{ \PBlock{\AllLInc{ i+1 }} } - \Length{ \PBlock{ \AllLInc{i+1}-1 } } \) the repetitivity function is given by
\[ \Repe( L ) \! = \! \begin{cases}
\! 2 \! \cdot \! \Length{ \PBlock{ \AllL( \AllLInc{ i } ) - 1 } } \! + \! 1\! - \! \Length{ \PBlock{\AllLInc{ i }} } \! + \! \Length{ \PBlock{ \AllLInc{ i }-1 } } \! + \! L \! & \text{for } \Length{ \PBlock{\AllLInc{ i }} } \! - \! \Length{ \PBlock{ \AllLInc{ i }-1 } } \! + \! 1 \! \leq \! L \! \leq \! \Length{ \PBlock{\AllLInc{ i } } } \! + \! 1 \\
 \! 2 \! \cdot \! \Length{ \PBlock{ \AllL( \AllLInc{ i } ) - 1 } } \! + \! 1\!  + \! L \! & \text{for } \Length{ \PBlock{\AllLInc{ i }} } \! + \! 2 \! \leq \! L \! \leq \! \Length{ \PBlock{\AllLInc{ i+1 }} } \! - \! \Length{ \PBlock{ \AllLInc{ i+1 }-1 } } \end{cases} \, .\]
\end{thm}

\begin{proof}
Recall that \( \Repe( L+1 ) - \Repe( L ) \geq 1 \) holds for all \( L \). First we consider the repetitivity between \( \Length{ \PBlock{\AllLInc{ i }} } - \Length{ \PBlock{ \AllLInc{ i }-1 } } + 1 \) and \( \Length{ \PBlock{\AllLInc{ i } } } + 1 \):
\begin{align*}
& 2 \cdot ( \Length{ \PBlock{ \AllL( \AllLInc{ i } ) - 1 } } + 1 ) + \Length{ \PBlock{ \AllLInc{ i }-1 } } \\
\geq {} & \Repe( \Length{ \PBlock{\AllLInc{ i }} } + 1 ) \quad \text{by Proposition~\ref{prop:RepeUB2}} \\
= {} & \Repe( \Length{ \PBlock{\AllLInc{ i }} } - \Length{ \PBlock{ \AllLInc{ i }-1 } } + 1) + \sum \nolimits_{L = \Length{ \PBlock{\AllLInc{ i }} } - \Length{ \PBlock{ \AllLInc{ i }-1 } } + 1 }^{ \Length{ \PBlock{\AllLInc{ i }} } } \big[ \Repe( L+1 ) - \Repe( L ) \big] \\
\geq {} & 2 \cdot ( \Length{ \PBlock{ \AllL( \AllLInc{ i } ) - 1 } } + 1 ) + \Length{ \PBlock{ \AllLInc{ i }-1 } } \quad \text{by Proposition~\ref{prop:RepeLB1}} \, .
\end{align*}
This proves the claim for \( \Length{ \PBlock{\AllLInc{ i }} } - \Length{ \PBlock{ \AllLInc{ i }-1 } } + 1 \leq L \leq \Length{ \PBlock{\AllLInc{ i }} } + 1 \). If \( \AllLInc{ i+1 } = \AllLInc{ i } + 1 \) and \( n_{\AllLInc{ i+1 }} = 2 \) hold, then \( \Length{ \PBlock{\AllLInc{ i }} } + 1 = \Length{ \PBlock{\AllLInc{i+1}} } - \Length{ \PBlock{ \AllLInc{ i+1} -1 } } \) follows and we are done. Otherwise we have yet to consider the repetitivity between \( \Length{ \PBlock{\AllLInc{ i }} } + 2 \) and \( \Length{ \PBlock{\AllLInc{ i+1 }} } - \Length{ \PBlock{ \AllLInc{ i+1 }-1 } } \):
\begin{align*}
& 2 \cdot ( \Length{ \PBlock{ \AllL( \AllLInc{ i } ) - 1 } } + 1 ) + \Length{ \PBlock{\AllLInc{ i+1}} } - \Length{ \PBlock{ \AllLInc{ i+1 } - 1 } } - 1 \\
\geq {} & \Repe( \Length{ \PBlock{\AllLInc{ i+1 }} } - \Length{ \PBlock{ \AllLInc{ i+1 }-1 } } ) \quad \text{by Proposition~\ref{prop:RepeUB1}} \\
= {} & \Repe( \Length{ \PBlock{\AllLInc{ i }} } + 2 ) + \sum \nolimits_{L = \Length{ \PBlock{\AllLInc{ i }} } + 2 }^{ \Length{ \PBlock{\AllLInc{ i+1 }} } - \Length{ \PBlock{ \AllLInc{ i+1 }-1 } } - 1 } \big[ \Repe( L+1 ) - \Repe( L ) \big] \\
\geq {} & 2 \cdot ( \Length{ \PBlock{ \AllL( \AllLInc{ i } ) - 1 } } + 1 ) + \Length{ \PBlock{\AllLInc{ i+1 }} } - \Length{ \PBlock{\AllLInc{ i+1 } - 1 } } - 1 \quad \text{by Proposition~\ref{prop:RepeLB2}} \, . \qedhere
\end{align*}
\end{proof}

\begin{cor}
For \( \AllLInc{ i+1 } = \AllLInc{ i }+1 \) and \( n_{\AllLInc{ i+1 }} = 2 \) the repetitivity simplifies to
\[ \Repe( L ) = 2 \cdot \Length{ \PBlock{ \AllL( \AllLInc{ i } ) - 1 } } + 1 - \Length{ \PBlock{\AllLInc{ i }} } + \Length{ \PBlock{ \AllLInc{ i }-1 } } +  L \]
for \( \Length{ \PBlock{\AllLInc{ i }} } - \Length{ \PBlock{ \AllLInc{ i }-1 } } + 1 \leq L \leq \Length{ \PBlock{ \AllLInc{ i }+1 } } - \Length{ \PBlock{\AllLInc{ i }} } \) .
\end{cor}

\begin{rem}
Roughly speaking, the jump between \( \Repe( \Length{ \PBlock{\AllLInc{ i }} } + 1 ) \) and \( \Repe( \Length{ \PBlock{\AllLInc{ i }} } + 2 ) \) is caused by the fact that all words of length \( \Length{ \PBlock{\AllLInc{ i }} } + 1 \) are subwords of \( \Word{ \PBlock{\AllLInc{ i }} }{ a }{ \PBlock{\AllLInc{ i }} } \) with \( a \in \Alphab_{ \AllLInc{ i }+1 } \). In contrast, this is not the case for words of length \( \Length{ \PBlock{\AllLInc{ i }} } + 2 \). The jump between \( \Repe( \Length{ \PBlock{\AllLInc{ i }} } - \Length{ \PBlock{ \AllLInc{ i }-1 } } ) \) and \( \Repe( \Length{ \PBlock{\AllLInc{ i }} } - \Length{ \PBlock{ \AllLInc{ i }-1 } } + 1 ) \) is caused by a similar reason:  while all subwords of length \( \Length{ \PBlock{\AllLInc{ i }} } - \Length{ \PBlock{ \AllLInc{ i }-1 } } \) of \( \Word{ \PBlock{\AllLInc{ i }} }{ a_{\AllLInc{ i }} }{ \PBlock{\AllLInc{ i }} } \) are contained in \( \PBlock{\AllLInc{ i }} \), this is not true for subwords of length \( \Length{ \PBlock{\AllLInc{ i }} } - \Length{ \PBlock{ \AllLInc{ i }-1 } } + 1 \). For \( n_{\AllLInc{ i+1 }} = 2 \) and \( \AllLInc{ i+1 }= \AllLInc{ i } + 1 \) the positions of the two jumps coincide. Therefore Proposition~\ref{prop:RepeLB1} and~\ref{prop:RepeUB2} provide stronger bounds than Proposition~\ref{prop:RepeLB2} and~\ref{prop:RepeUB1}, since the latter take only one of the two jumps into account.
\end{rem}

\subsection{Alpha-repetitivity}
\label{subsec:AlphaRepe}

In this subsection we study linear repetitivity and, more general, \( \alpha \)-repetitivity of simple Toeplitz subshifts. A subshift is said to be \emph{linearly repetitive}\index{linearly repetitive}\index{repetitivity!linear repetitivity} if there exists a constant \( C \) with \( \frac{ \Repe( L ) }{ L } \leq C \) for all \( L \geq 1 \). For a lower bound, note that \( 1 < \frac{ \Repe( L ) }{ L } \) holds for all \( L \), since the repetitivity function is strictly increasing. This notion was generalised in \cite[Definition 2.9]{GKMSS_AperioJarnik}, where a subshift is called \emph{\( \alpha \)-repetitive}\index{alpha-repetitive@\(\alpha\)-repetitive}\index{repetitivity!alpha-repetitivity@\(\alpha\)-repetitivity} for \( \alpha \geq 1 \) if \( 0 < \limsup_{ L \to \infty } \frac{ \Repe( L ) }{ L^{\alpha} } < \infty \) holds. In \cite[Theorem~4.10]{DKMSS_Regul-Article}, \( \alpha \)-repetitivity was characterised for \( l \)-Grigorchuk subshifts. Below we give a characterisation for general simple Toeplitz subshifts. For \( l \)-Grigorchuk subshifts we recover precisely the result from \cite{DKMSS_Regul-Article}.

\begin{prop}
\label{prop:AlphRepe}
Let \( \alpha \geq 1 \). A simple Toeplitz subshift is \( \alpha \)-repetitive if and only if the following inequalities hold:
\[ 0 < \limsup_{ i \to \infty } \frac{ \Length{ \PBlock{ \AllL( \AllLInc{ i } ) - 1 } } + 1 }{ ( \Length{ \PBlock{ \AllLInc{ i } } } + 1 )^{ \alpha } } = \limsup_{ i \to \infty } \frac{ n_{0} \cdot \ldots \cdot n_{ \AllL( \AllLInc{ i } ) - 1} }{ n_{0}^{\alpha} \cdot \ldots \cdot n_{ \AllLInc{ i } }^{\alpha} }  < \infty \, . \]
\end{prop}

\begin{proof}
By Theorem~\ref{thm:Repe} we have \( \frac{ \Repe( L ) }{ L^{\alpha} } = \frac{ \text{const} }{ L^{\alpha} } + L^{ 1-\alpha } \), where the constant depends on \( i \) and on whether \( \Length{ \PBlock{\AllLInc{ i }} } - \Length{ \PBlock{ \AllLInc{ i }-1 } } + 1 \leq L \leq \Length{ \PBlock{\AllLInc{ i }} } + 1 \) or \( \Length{ \PBlock{ \AllLInc{ i }} } + 2 \leq L \leq \Length{ \PBlock{ \AllLInc{ i+1 }} } - \Length{ \PBlock{ \AllLInc{ i+1 }-1 } } \) holds. For every \( i \), this quotient is maximal either at \( L_{1}( i ) \DefAs \Length{ \PBlock{\AllLInc{ i }} } - \Length{ \PBlock{ \AllLInc{ i }-1 } } + 1 \) or at \( L_{2}( i ) \DefAs \Length{ \PBlock{\AllLInc{ i }} } + 2 \). For \( L_{1} \) we obtain the bounds
\[ 2 \cdot \frac{ \Length{ \PBlock{ \AllL( \AllLInc{ i } ) -1 } } + 1 }{ ( \Length{ \PBlock{ \AllLInc{ i } } } + 1 )^{\alpha} } < \frac{ \Repe( L_{1}(i) )}{ L_{1}(i)^{\alpha} } <  \frac{2}{(1 - \frac{ 1 }{  n_{\AllLInc{ i }}  }) ^{\alpha}} \cdot \frac{ \Length{ \PBlock{ \AllL( \AllLInc{ i } ) - 1 } } + 1 }{ ( \Length{ \PBlock{ \AllLInc{ i } }} + 1)^{\alpha} } \leq 2^{1+\alpha} \cdot \frac{ \Length{ \PBlock{ \AllL( \AllLInc{ i } ) - 1 } } + 1 }{ ( \Length{ \PBlock{ \AllLInc{ i } } } + 1)^{\alpha}  } \, .\]
Similarly we obtain for \( L_{2} \) 
\[ 2^{1-\alpha} \cdot \frac{ \Length{ \PBlock{ \AllL( \AllLInc{ i } ) - 1 } } + 1 }{ ( \Length{ \PBlock{ \AllLInc{ i } } } + 1 )^{\alpha} } = \frac{ 2 \cdot ( \Length{ \PBlock{ \AllL( \AllLInc{ i } ) - 1 } } + 1 ) }{ ( 2 \cdot ( \Length{ \PBlock{ \AllLInc{ i } } } + 1) )^{\alpha} } < \frac{ \Repe( L_{2}(i) ) }{L_{2}(i)^{\alpha}} < 3 \cdot \frac{ \Length{ \PBlock{ \AllL( \AllLInc{ i } ) - 1 } } + 1 }{ ( \Length{ \PBlock{ \AllLInc{ i } } } + 1 )^{\alpha} } \, .\]
First we assume \( 0 < \limsup_{ i \to \infty } \frac{ \Length{ \PBlock{ \AllL( \AllLInc{ i } ) - 1 } } + 1 }{( \Length{ \PBlock{ \AllLInc{ i } } } + 1 )^{\alpha}} < \infty \). Then the subshift is \( \alpha \)-repetitive by the bounds given above:
\[ 0 < \limsup_{ i \to \infty } 2 \cdot \frac{ \Length{ \PBlock{ \AllL( \AllLInc{ i } ) - 1 } } + 1 }{ ( \Length{ \PBlock{ \AllLInc{ i } } } + 1 )^{\alpha}} \leq \limsup_{ L \to \infty } \frac{ \Repe( L ) }{ L^{\alpha} } \leq \limsup_{ i \to \infty } 2^{1+\alpha} \cdot  \frac{ \Length{ \PBlock{ \AllL( \AllLInc{ i } ) - 1 } } + 1 }{ ( \Length{ \PBlock{ \AllLInc{ i } } } + 1)^{\alpha}  } < \infty \, .\]
Conversely we now assume that the subshift is \( \alpha \)-repetitive. Then the above bounds yield
\[  0 < 2^{-1-\alpha} \cdot \limsup_{ L \to \infty } \frac{ \Repe( L ) }{ L^{\alpha} } < \limsup_{ i \to \infty } \frac{ \Length{ \PBlock{ \AllL( \AllLInc{ i } ) - 1 } } + 1 }{ ( \Length{ \PBlock{ \AllLInc{ i } } } + 1)^{\alpha} } < 2^{-1+\alpha} \cdot \limsup_{ L \to \infty } \frac{ \Repe( L ) }{ L^{\alpha} } < \infty \, . \qedhere \]
\end{proof}

\begin{cor}
\label{cor:LinRepe}
A simple Toeplitz subshift is linearly repetitive if and only if the sequence \( \left( \prod_{ j = \AllLInc{ i }+1 }^{ \AllL( \AllLInc{ i } ) - 1} n_{j} \right)_{i \geq 1} \) is bounded from above.
\end{cor}

\begin{proof}
By Proposition~\ref{prop:AlphRepe} with \( \alpha = 1 \), linear repetitivity is equivalent to
\[ 0 < \limsup_{ i \to \infty } \prod_{ j = \AllLInc{ i }+1 }^{ \AllL( \AllLInc{ i } ) - 1 } n_{j} < \infty \, .\]
Now the claim follows from the fact that the product is always bounded away from zero because of \( n_{j} \geq 2 \) and \( \AllL( k ) \geq k + 2 \).
\end{proof}

\begin{cor}
On the one hand, if the sequence \( (n_{k})_{k \geq 0} \) is bounded, then the subshift is linearly repetitive if and only if the sequence \( ( \AllL( \AllLInc{ i } ) - \AllLInc{ i } )_{i \geq 0} \) is bounded. On the other hand, if the sequence \( ( \AllL( \AllLInc{ i } ) - \AllLInc{ i } )_{i \geq 0} \) is bounded, then the subshift is linearly repetitive if and only if the sequence \( (n_{k})_{k \geq 0} \) is bounded. 
\end{cor}

\begin{exmpl}[Grigorchuk subshift]
Clearly \( (n_{k}) = (2, 2, 2, \ldots ) \) is bounded and we saw in Example~\ref{exmpl:GrigAllLInc} that \( ( \AllL( \AllLInc{ i } ) - \AllLInc{ i } )_{i \geq 0} \) is bounded as well. Thus the Grigorchuk subshift is linearly repetitive. Alternatively, linear repetitivity can be derived from the fact that the subshift is generated by a primitive substitution (see Remark~\ref{rem:SubstGrig}) and every such subshift is linearly repetitive (\cite[Lemma~2.3]{Solom_ComposTiling}, \cite[Proposition~25]{DurHostSkau_SubstDS}).
\end{exmpl}

\begin{cor}
If \( (n_{k}) \) is a constant sequence, then the subshift is \( \alpha \)-repetitive if and only if \( -\infty < \limsup_{i \to \infty} \big[ \AllL( \AllLInc{ i } ) - \alpha \cdot \AllLInc{ i } \big] < \infty \) holds.
\end{cor}

\begin{proof}
Let \( n \) denote the constant value of \( (n_{k}) \). By Proposition~\ref{prop:AlphRepe} the subshift is \( \alpha \)-repetitive if and only if the following holds:
\begin{alignat*}{2} % aligns at two points, alternates between right aligned and left aligned
&& 0 & < \limsup_{ i \to \infty } \frac{ n_{0} \cdot \ldots \cdot n_{ \AllL( \AllLInc{ i } ) - 1 } }{ n_{0}^{\alpha} \cdot \ldots \cdot n_{ \AllLInc{ i } }^{\alpha}  } < \infty \\
& \Longleftrightarrow & 0 & < \limsup_{ i \to \infty } \, n^{ \AllL( \AllLInc{ i } ) - \alpha \cdot ( \AllLInc{ i } + 1 ) } < \infty \\
& \Longleftrightarrow \quad & - \infty & < \limsup_{ i \to \infty } \big[ \AllL( \AllLInc{ i } ) - \alpha \cdot ( \AllLInc{ i } + 1) \big] < \infty \, .\tag*{\qedhere} % \qedhere does not work with alignat
\end{alignat*}
\end{proof}

\begin{exmpl}[non-(B) example]
\label{exmpl:AlphaRepeBosh}
By Example~\ref{exmpl:BoshExAllL} we have \( \AllLInc{ 0 } = 0 \), \( \AllLInc{ i } = (i+1)^{2} - 2 \) for \( i \geq 1 \), and \( \AllL( \AllLInc{ i } ) = \AllLInc{ i+2} \). Hence we obtain
\[ \AllL( \AllLInc{ i } ) - \alpha \cdot \AllLInc{ i } = ( 1-\alpha ) \cdot (( i+1 )^{2} - 2) + 4i + 8 \, . \]
Since the left hand side tends to \( + \infty \) for \( \alpha = 1 \) and to \( - \infty \) for \( \alpha > 1 \), there is no \( \alpha \geq 1 \) such that the subshift is \( \alpha \)-repetitive.
\end{exmpl}

The special case where \( \AllL( \AllLInc{i} ) - \AllLInc{i} \) is constant is discussed next. Note that the corollary covers in particular the case where \( n_{j+1} = n_{j}^{\sqrt[c-1]{\alpha}} \) holds for all \( j \geq 0 \).

\begin{cor}
Let \( ( \AllL( \AllLInc{ i } ) - \AllLInc{ i } )_{i \geq 0} \) be a constant sequence with value \( c \). If \( n_{j + c - 1}  = n_{j}^{\alpha} \) holds for all \( j \geq 0 \), then the subshift is \( \alpha \)-repetitive.
\end{cor}

\begin{proof}
This follows from Proposition~\ref{prop:AlphRepe}, since the product
\[ \frac{ n_{0} \cdot \ldots \cdot n_{ \AllL( \AllLInc{ i } ) - 1 } }{ n_{0}^{\alpha} \cdot \ldots \cdot n_{\AllLInc{ i }}^{\alpha} } = \frac{ \prod_{j = 0}^{ c - 2} n_{j} \cdot \prod_{j = c-1}^{ \AllLInc{ i }+c-1 } n_{j} }{ \prod_{j = 0}^{\AllLInc{ i }} n_{j}^{\alpha} } = \prod_{j = 0}^{ c - 2} n_{j} \]
is positive, finite and independent of \( i \).
\end{proof}

%%%%%%%%%%%%%%%%%%%%%%%%
\section{The Boshernitzan condition}
\label{sec:BoshCond}

The Boshernitzan condition is a weaker analogue of linear repetitivity. Based on a result from \cite{LiuQu_Simple} we characterise it for simple Toeplitz subshifts in terms of \( \AllL \) and \( (\AllLInc{i}) \). Our result takes a particular simple form if either \( n_{k}=2^{j_{k}} \) or \( \Card{ \AlphabEv } = 3 \) holds. This includes generalised Grigorchuk subshifts, where our characterisation of the Boshernitzan condition for the subshift has exactly the same form as a characterisation of the Boshernitzan condition for actions of the underlying self-similar group. Moreover our results will play in role in Chapter~\ref{chap:UnifCocycCantor}, since by \cite{BeckPogo_SpectrJacobi} the Boshernitzan condition implies Cantor spectrum for Jacobi operators (see Corollary~\ref{cor:BoshSTCantor}).

Since simple Toeplitz subshifts are uniquely ergodic (Corollary~\ref{cor:STStrictErgod}), there exists a unique \( \Shift \)-invariant ergodic probability measure \( \ErgodMeas \). We define
\[ \MaxCylin( L ) \DefAs \min_{u \in \Langu{ \Subshift }_{L}} \ErgodMeas( \Cylin{u}{1} ) = \min_{u \in \Langu{ \Subshift }_{L}} \ErgodMeas( \Set{ \InfWord \in \Subshift : \Restr{ \InfWord }{ 1 }{ L } = u } ) \, , \]
where \( \Cylin{u}{1} \) denotes the cylinder set, see Section~\ref{sec:PrelimSubsh}. A subshift is said to satisfy the \emph{Boshernitzan condition (\ref{eqn:BoshCond})}\index{Boshernitzan condition}\index{B@(B)} if
\begin{equation*}
\label{eqn:BoshCond}
\limsup_{L \to \infty} L \cdot \MaxCylin( L )  > 0 \tag{B}
\end{equation*}
holds. As we will discuss later (see Section~\ref{sec:AsNoPP}), the measure \( \ErgodMeas( \Cylin{u}{1} ) \) is precisely the frequency of \( u \) in \( \InfWord \) and describes how often \( u \) occurs ``on average''. Thus the Boshernitzan condition requires the same kind of bound as linear repetitivity, but only for the average gap length between occurrences of a word, instead of every gap length. Some results from \cite{LiuQu_Simple} about the Boshernitzan condition for simple Toeplitz subshifts are stated below:

\begin{prop}[{\cite[Proposition~4.1]{LiuQu_Simple}}]
\label{prop:BoshCondA2}
Every simple Toeplitz subshift with \( \Card{ \AlphabEv } = 2 \) satisfies (\ref{eqn:BoshCond}).
\end{prop}

\begin{exmpl}[period doubling]
\label{exmpl:PDExBosh}
By the previous proposition, the period doubling subshift satisfies the Boshernitzan condition. In fact it is even linearly repetitive. This follows either from Corollary~\ref{cor:LinRepe} or, more abstractly, from the fact that the period doubling substitution (see Remark~\ref{rem:SubstPD}) is primitive. It was proved in \cite[Proposition~25]{DurHostSkau_SubstDS} that primitive substitutions always generate linearly repetitive subshifts.
\end{exmpl}

The following, very useful description of \( \MaxCylin( L ) \) for \( \Card{ \AlphabEv } \geq 3 \) is from \cite{LiuQu_Simple}. For convenience we state it in our notation. Let \( \EventNr \) be such that \( a_{k} \in \AlphabEv \) holds for all \( k \geq \EventNr \). We define \( s_{j} \DefAs n_{0} \cdot \ldots \cdot n_{j-1} \). Recall that we write \( \AllL( k ) \DefAs \min \Set{ j > k : \Set{ a_{k+1}, \ldots , a_{j} } = \Alphab_{k+1} } \).

\begin{prop}[{\cite[Proposition~4.2]{LiuQu_Simple}}]
\label{prop:LiuQuEta}
For a simple Toeplitz subshift with \( \Card{ \AlphabEv } \geq 3 \) there exist constants \( 0 < c_{1} \leq c_{2} \) such that for every \( L >  s_{\EventNr} \), and \( j \) defined by the property \( s_{j-1} < L \leq s_{j} \), the following holds:
\begin{tightenumerate}
\item{If \( s_{j-1} < L < 2s_{j-1} \), then \( c_{1} \cdot \MaxCylin( L ) \leq \min \Big\{ \, \frac{ \Big\lceil \frac{ 2s_{ j-1 } - L }{ s_{ j-2 } } \Big\rceil }{ s_{ \min \Set{ i > j-1 : \, a_{i} = a_{j-2} } } } \; , \; \frac{1}{ s_{ \AllL( j-3 ) } } \, \Big\} \leq c_{2} \cdot \MaxCylin( L ) \) .}
\item{If \( 2s_{j-1} \leq L \leq s_{j} \), then \( c_{1} \cdot \MaxCylin( L ) \leq \frac{1}{ s_{\AllL( j-2 )} }  \leq c_{2} \cdot \MaxCylin( L ) \) .}
\end{tightenumerate}
\end{prop}

In \cite[Corollary~4.1]{LiuQu_Simple} a characterisation of the Boshernitzan condition for simple Toeplitz subshifts is derived. Below we give a different characterisation in terms of \( \AllL \) and \( (n_{k}) \).

\begin{prop}
\label{prop:BoshCond}
A simple Toeplitz subshift satisfies the Boshernitzan condition if and only if there exists a sequence \( (k_{r}) \) of natural numbers with \( \lim_{r \to \infty} k_{r} = \infty \) such that \( \big( \prod_{ j = k_{r}+1 }^{ \AllL( k_{r}-1 ) - 1 } n_{j} \big)_{r} \) is bounded.
\end{prop}

\begin{proof}
For \( \Card{ \AlphabEv }= 2 \) the claim is trivially true: by Example~\ref{exmpl:AllLAlphEv2} we have \( \AllL( k_{r}-1 ) = k_{r}+1 \) for all sufficiently large \( k_{r} \). Clearly the empty product \( \prod_{ j = k_{r}+1 }^{ \AllL( k_{r}-1 )-1 } n_{j} \) is always bounded and by Proposition~\ref{prop:BoshCondA2} the Boshernitzan condition is always satisfied for \( \Card{ \AlphabEv }= 2 \).

Let now \( \Card{ \AlphabEv } \geq 3 \). First we assume that \( \big( \prod_{ j = k_{r}+1 }^{ \AllL( k_{r}-1 )-1 } n_{j} \big)_{r} \) is bounded. Let the sequence \( (L_{r})\) be given by \( L_{r} \DefAs s_{ k_{r}+1 } \). By Proposition~\ref{prop:LiuQuEta} we obtain \( \MaxCylin( L_{r} ) \geq ( c_{2} \cdot s_{ \AllL( k_{r}-1 ) } )^{-1} \), which yields
\[ \limsup_{ L \to \infty } L \cdot \MaxCylin( L ) \geq \limsup_{ r \to \infty } L_{r} \cdot \MaxCylin( L_{r} ) \geq \frac{ 1 }{ c_{2} } \limsup_{ r \to \infty } \frac{ 1 }{ \prod_{ j=k_{r}+1 }^{ \AllL( k_{r}-1 ) - 1 } n_{j} } > 0 \, . \]
To prove the converse we assume that there is no sequence \( (k_{r}) \) for which the product is bounded. For every \( L \in \NN \) we define \( j \) by \( s_{j-1} < L \leq s_{j} \). If \( s_{j-1} < L < 2s_{j-1} \) holds, then Proposition~\ref{prop:LiuQuEta} yields
\[ L \cdot \MaxCylin( L ) < \frac{ 2 s_{j-1} }{ c_{1} } \cdot \frac{ 1 }{ s_{ \AllL( j-3 ) } }  = \frac{ 2 }{ c_{1} } \cdot \frac{ 1 }{ \prod_{ i=j-1 }^{ \AllL( j-3 )-1 } n_{i} } \xrightarrow{j \to \infty} 0 \, . \]
If \( 2s_{j-1} \leq L \leq s_{j} \) holds, then Proposition~\ref{prop:LiuQuEta} yields
\[ L \cdot \MaxCylin( L ) \leq \frac{ s_{j} }{ c_{1} } \cdot \frac{ 1 }{ s_{\AllL( j-2 )} } = \frac{ 1 }{ c_{1} } \cdot \frac{ 1 }{ \prod_{i=j}^{ \AllL( j-2 )-1 } n_{i} } \xrightarrow{ j \to \infty } 0 \, . \qedhere \]
\end{proof}

The following proposition shows that the sequence \( ( k_{r} ) \) can always be chosen as a subsequences of \( ( \AllLInc{ i } ) \). The reason is that among all \( k \in \Set{ \AllLInc{ i }+1, \ldots , \AllLInc{ i+1 } } \) the value of \( \prod_{ j = k+1 }^{ \AllL( k-1 )-1 } n_{j} \) is minimal for \( k  = \AllLInc{ i+1 } \).

\begin{prop}
\label{prop:BdKBdAllLInc}
There exists a sequence \( (k_{r}) \) with \( \lim_{ r \to \infty } k_{r} = \infty \) such that the products \( \big( \prod_{ j = k_{r}+1 }^{ \AllL( k_{r}-1 )-1 } n_{j} \big)_{r} \) are bounded if and only if there exists a subsequence \( (\AllLInc{ i_{r} })_{r} \) of \( (\AllLInc{ i }) \) with \( \lim_{r \to \infty} i_{r} = \infty \) such that the products \( \big( \prod_{ j = \AllLInc{ i_{r} }+1 }^{ \AllL( \AllLInc{ i_{r} }-1 )-1 } n_{j} \big)_{r} \) are bounded.
\end{prop}

\begin{proof}
On the one hand, if \( (\AllLInc{ i_{r} }) \) is such a subsequence, then clearly \( k_{r} \DefAs \AllLInc{ i_{r} } \) satisfies the claim. On the other hand, assume that \( ( k_{r} ) \) is a sequence as in the statement. For every \( r \), let \( i_{r} \) be such that \( \AllLInc{ i_{r} }-1 \leq k_{r}-1 < \AllLInc{ i_{r} } \) holds. Because of \( \AllL( \AllLInc{ i_{r} } - 1 ) = \AllL( k_{r}-1 ) \) we obtain
\( \prod_{ j = \AllLInc{ i_{r} }+1 }^{ \AllL( \AllLInc{ i_{r} }-1 )-1 } n_{j} \leq \prod_{ j = k_{r}+1 }^{ \AllL( k_{r}-1 )-1 } n_{j} \).
\end{proof}

\begin{exmpl}[non-(B) example]
\label{exmpl:BoshExBoshCond}
As we saw in Example~\ref{exmpl:BoshExAllL}, we have \( \AllL( \AllLInc{ i } ) = \AllLInc{ i+2} \) and \( \AllLInc{ i } = (i+1)^{2} - 2 \) for all \( i \geq 1 \). For every \( i \geq 1 \) this yields
\[ \prod_{ j = \AllLInc{ i }+1 }^{ \AllL( \AllLInc{ i }-1 )-1 } n_{j} = \prod_{ j = (i+1)^{2} - 1 }^{ (i+2)^{2} - 3 } n_{j} \geq 2^{2i + 2} \, .\]
By Proposition~\ref{prop:BoshCond} and~\ref{prop:BdKBdAllLInc}, the non-(B) example does therefore not satisfy (\ref{eqn:BoshCond}), which is of course the reason for its name. In fact, no simple Toeplitz subshift with the same coding sequence \( (a_{k}) \) satisfies (\ref{eqn:BoshCond}), independent of the chosen sequence \( (n_{k}) \). To demonstrate this behaviour is precisely why the non-(B) example was introduced in \cite{LiuQu_Simple}.
\end{exmpl}

\begin{rem}
By Corollary~\ref{cor:LinRepe}, a simple Toeplitz subshift is linearly repetitive if and only if \( \big( \prod_{ j = \AllLInc{ i }+1 }^{ \AllL( \AllLInc{ i } )-1 } n_{j} \big)_{i \geq 1} \) is bounded. By Proposition~\ref{prop:BoshCond} and~\ref{prop:BdKBdAllLInc} a simple Toeplitz subshift satisfies the Boshernitzan condition if and only if there exists a subsequence \( (\AllLInc{ i_{r} })_{r} \) such that \( \big( \prod_{ j = \AllLInc{ i_{r} }+1 }^{ \AllL( \AllLInc{ i_{r} }-1 )-1 } n_{j} \big)_{r} \) is bounded. This demonstrates that the Boshernitzan condition is indeed a weaker analogue of linear repetitivity. 
\end{rem}

\begin{cor}
\label{cor:BoshCondForA3}
For a simple Toeplitz subshift with \( \Card{ \AlphabEv } = 3 \) the Boshernitzan condition is satisfied if and only if \( \liminf_{ i \to \infty } n_{ \AllLInc{ i }+1 } < \infty \) holds.
\end{cor}

\begin{proof}
For all sufficiently large \( i \) we have \( \Set{ a_{ \AllLInc{ i } }, a_{ \AllLInc{ i }+1 }, a_{ \AllLInc{ i }+2 } } = \AlphabEv \) by the definitions of \( \AllL \) and \( (\AllLInc{ i }) \) and by a similar argument as in Example~\ref{exmpl:GrigAllLInc}. This yields \( \AllL( \AllLInc{ i }-1 ) = \AllLInc{ i }+2 \). Now the claim follows from Proposition~\ref{prop:BoshCond} and~\ref{prop:BdKBdAllLInc}.
\end{proof}

\begin{rem}
\label{rem:BdDiffForA3}
It was shown in \cite[Corollary~4.3]{LiuQu_Simple} that simple Toeplitz subshifts with \( \Card{ \AlphabEv } \geq 3 \), a bounded sequence \( ( n_{k} ) \) and \( \lim_{ k \to \infty } ( \AllL( k-2 ) - k ) = \infty \) do not satisfy the Boshernitzan condition. However, as the previous proof showed, the latter assumption cannot be satisfied for \( \Card{ \AlphabEv } = 3 \). In particular, every simple Toeplitz subshift with \( \Card{ \AlphabEv } = 3 \) and a bounded sequence \( ( n_{k} ) \) satisfies the Boshernitzan condition. Note that this is not true for \( \Card{ \AlphabEv } \geq 4 \), see Example~\ref{exmpl:BoshExBoshCond}.
\end{rem}

\begin{exmpl}[generalised Grigorchuk subshift]
We summarise the results from Proposition~\ref{prop:BoshCondA2} and Corollary~\ref{cor:BoshCondForA3}: a generalised Grigorchuk subshift satisfies (\ref{eqn:BoshCond}) if and only if either \( \Card{ \AlphabEv } = 2 \) holds, or \( \Card{ \AlphabEv } = 3 \) and \( \liminf_{ i \to \infty } n_{ \AllLInc{ i }+1 } < \infty \) hold.
\end{exmpl}

To conclude this section, we reformulate our results for generalised Grigorchuk subshifts in a way that relates them to self-similar groups. Recall from Example~\ref{exmpl:defGenGrigSubsh} that these subshifts are obtained from \( (n_{k}) = (2, 2, 2, \ldots ) \) and a sequence \( (b_{k}) \in \Alphab^{\NN_{0}} \), where \( b_{k} = b_{k+1} \) is allowed. When a letter \( b_{k} \) is repeated \( j \) times, we combine these occurrences to \( b_{k}^{2^{j}-1} \). Conversely, when we have \( n_{k} = 2^{j_{k}} \), we can interpret \( (a_{k}) \) as derived from a sequence \( (b_{k}) \) with repetitions and \( (n_{k})_{k \in \NN_{0}} = (2, 2, 2, \ldots ) \). Proposition~\ref{prop:BoshCond} can then be formulated in terms of \( (b_{k}) \):

\begin{cor}
\label{cor:BoshCondGenGrig}
Assume that all \( n_{k} \) have the form \( n_{k} = 2^{j_{k}} \) with \( j_{k} \in \NN \). Then the Boshernitzan condition is satisfied if and only if there exists a constant \( C \) and a sequence \( (t_{r}) \)  with \( \lim_{ r \to \infty } t_{r} = \infty \) such that \( \Set{ b_{ t_{r} } , \ldots , b_{ t_{r}+C } } = \Set{ b_{s} : s \geq t_{r} } \) holds.
\end{cor}

\begin{proof}
By Proposition~\ref{prop:BoshCond} with \( n_{k} = 2^{j_{k}} \), (\ref{eqn:BoshCond}) is satisfied if and only if there exists a sequence \( (k_{r})\) with \( \lim_{ r \to \infty } k_{r} = \infty \) such that \( \big( \sum \nolimits_{ i = k_{r}+1 }^{ \AllL( k_{r}-1 )-1 } j_{i} \big)_{r} \) is bounded. Assume that the letter \( a_{k} \) corresponds to the letters \( b_{t} = \ldots = b_{ t+j_{k}-1} \), the letter \( a_{k+1} \) corresponds to the letters \( b_{ t+j_{k} } = \ldots = b_{ t+j_{k}+j_{k+1}-1 } \) and so on, such that \( a_{\AllL( k-1 )} \) corresponds to \( b_{ t+j_{k}+\ldots+j_{ \AllL( k-1 )-1} } = \ldots = b_{ j_{k}+\ldots+j_{ \AllL( k-1 ) }-1 } \). Note that
\[ \Alphab_{k} = \Set{ a_{k}, \ldots , a_{\AllL( k-1 )} } = \Set{ b_{ t+j_{k}-1 }, b_{ t+j_{k} } , \ldots , b_{ t+j_{k}+\ldots+j_{ \AllL( k-1 )-1 } } } \]
holds. The set on the right contains \( 2 + \sum_{ i = k+1 }^{ \AllL( k-1 )-1 } j_{i} \) elements. If there exists a sequence \( ( k_{r} ) \) such that this sum is bounded, then
\[ \Set{ b_{ t_{r}+j_{ k_{r} }-1} , \ldots , b_{ t_{r}+j_{ k_{r} } + \ldots + j_{ \AllL( k_{r}-1 )-1} } } \]
has the claimed property. Conversely, assume that there exists a sequence \( ( t_{r} ) \) with \( \Set{ b_{ t_{r} } , \ldots , b_{ t_{r}+C } } = \Set{ b_{s} : s \geq t_{r} } \). Let \( a_{k_{r}} \) denote the letter that corresponds to \( b_{t_{r}} \). Then \( a_{ \AllL( k_{r}-1 )-1 } \) corresponds to a letter \( b_{t} \) with \( t \leq t_{r}+C \). Since \( j_{k} \) denotes the multiplicity of the letter \( b_{t} \) that corresponds to \( a_{k} \), we obtain
\[ \sum \nolimits_{ i = k_{r}+1 }^{ \AllL( k_{r}-1 )-1 } j_{i} \leq ( t_{r}+C ) - ( t_{r}+1 ) + 1 = C \, . \qedhere \]
\end{proof}

\begin{rem}
As the name suggests, generalised Grigorchuk subshifts are associated to elements in the family of Grigorchuk's groups. These groups are defined by sequences \( (\GenSeq_{t}) \) of automorphism-valued maps. Since the groups act on a set of rooted graphs, there exists a notion of a Boshernitzan condition, see Definition~\ref{defi:LRuBGrAct} in the appendix. Nagnibeda and Perez show in \cite{NagnPerez_SchreierGr} that the action of such a group satisfies the Boshernitzan condition if and only if there exists a constant \( C \) and a sequence \( (t_{r}) \) with \( \lim_{ r \to \infty } t_{r} = \infty \) such that for every \( r \) the equality \( \Set{ \GenSeq_{t_{r}} , \ldots , \GenSeq_{t_{r} + C} } = \Set{ \GenSeq_{s} : s \geq t_{r} } \) holds, see Proposition~\ref{prop:BoshCondGrAction}. This is completely analogous to our previous corollary for subshifts.
\end{rem}

%%%%%%%%%%%%%%%%%%%%%%%%
%%%%%%%%%%%%%%%%%%%%%%%%
%%%%%%%%%%%%%%%%%%%%%%%%
\chapter[Jacobi operators: almost sure absence of eigenvalues][Jacobi operators]{Jacobi operators: almost sure absence of eigenvalues}
\label{chap:JOSimpToep}

In this chapter we begin our study of Jacobi and Schrödinger operators. We recall their definitions and basic facts about their spectrum on aperiodic subshifts. A central concept is that of transfer matrices, which are matrix products that describe solutions of the eigenvalue equation. We remind the reader how repetitions of words exclude eigenvalues via so-called Gordon-type arguments. This allows us to prove that the pure point spectrum is empty for almost all elements of a simple Toeplitz subshift. In addition this chapter serves as an introduction for Chapter~\ref{chap:UnifCocycCantor}, where many of the aforementioned concepts will reappear in a more general form. Accordingly we sometimes only sketch the main idea and refer to the detailed proof of a more general result in the next chapter.

%%%%%%%%%%%%%%%%%%%%%%%%
\section{Jacobi operators and transfer matrices}
\label{sec:JOTransMat}

The prototype of the operators that we are interested in is the Schrödinger operator \( \Jac \DefAs \Delta + V \) on \( \ell^{2}(\ZZ) \). Here \( \Delta \) is the discrete Laplacian and \( V \) is a multiplication operator. It is called \emph{potential}\index{potential} and reflects the underlying aperiodic structure. In other words, we model a quasicrystal by a two-sided infinite, aperiodic arrangement of equally spaced atoms or ions. The Schrödinger operator describes their interaction with an electron. The spectrum of the operator encodes the possible energy values of the electron.

In the following we mostly consider Jacobi operators, which can have an off-diagonal term different from one (the Laplacian \( \Delta \) is, in other words, replaced by a weighted Laplacian). In \cite[Proposition~4.1]{GLN_SpectraSchreierAndSO} a connection between Jacobi operators and self-similar groups was proved: in certain cases, a Jacobi operator on a simple Toeplitz subshift is unitarily equivalent to the Laplacian on a Schreier graph, which is associated to the group. Therefore the study of Jacobi operators also sheds new light on self-similar groups. More details about this connection can be found in Section~\ref{sec:AppSchrGraphLapl} of the appendix.

\begin{defi}
Let \( \NDig \colon \Subshift \to \RR \setminus \Set{ 0 } \) and \( \Dig \colon \Subshift \to \RR \) be continuous functions on a subshift \( \Subshift \). For an element \( \InfWord \in \Subshift \), the operator \( \Jac_{\InfWord} \colon \ell^{2}( \ZZ ) \to \ell^{2}( \ZZ ) \) that is defined by
\[ ( \Jac_{\InfWord} \psi )( k ) \DefAs \NDig( \Shift^{k} \InfWord ) \psi( k-1 ) + \Dig( \Shift^{k} \InfWord ) \psi( k ) + \NDig( \Shift^{k+1} \InfWord ) \psi( k+1 ) \]
is called the \emph{Jacobi operator}\index{Jacobi operator} associated to \( \InfWord \). In the special case \( \NDig \equiv 1 \), the operator is called a \emph{Schrödinger operator}\index{Schrödinger operator}.
\end{defi}

For the operators that we consider in this thesis we make two additional assumptions. The first one ensures that \( \Jac_{\InfWord} \) actually ``sees'' the aperiodic structure of \( \InfWord \). More precisely, we define the dynamical system \( (\widetilde{\Subshift} , \widetilde{ \Shift }) \) by
\begin{equation}
\label{eqn:TildsShiftFG}
\widetilde{\Subshift} \DefAs \bigg\{ \, \bigg(\begin{matrix} \NDig( \InfWord ) \\ \Dig( \InfWord ) \end{matrix}\bigg) \, : \, \InfWord  \in \Subshift \, \bigg\} \qquad \text{and} \qquad \widetilde{ \Shift } \colon \widetilde{\Subshift} \to \widetilde{\Subshift} \, ,\; \bigg(\begin{matrix} \NDig( \InfWord ) \\ \Dig( \InfWord ) \end{matrix}\bigg) \mapsto \bigg(\begin{matrix} \NDig( \Shift \InfWord ) \\ \Dig( \Shift \InfWord ) \end{matrix}\bigg) \, ,
\end{equation}
with the product topology on \( \widetilde{\Subshift} \). We say that \( ( \NDig , \Dig ) \) is \emph{aperiodic}\index{aperiodic!pair of functions} if \( \widetilde{\Subshift} \) contains no \( \widetilde{ \Shift } \)-periodic element. For example, \( \NDig \) and \( \Dig \) must not both be constant. Our second assumption limits the range that \( \NDig \) and \( \Dig \) ``see'' of \( \InfWord \) by requiring both functions to be locally constant. With the usual topology on \( \Subshift \) (see Section~\ref{sec:PrelimSubsh}) this notion can be defined as follows:

\begin{defi}
\label{defi:LocConst}
A map \( f \), defined on a subshift \( \Subshift \), is called \emph{locally constant}\index{locally constant!map} if there exists a number \( J \in \NN_{0} \) such that \( f( \InfWord ) = f( \varrho ) \) holds for all \( \InfWord , \varrho \in \Subshift \) with \( \Restr{\InfWord}{-J}{J} = \Restr{\varrho}{-J}{J} \). In symbolic dynamics, locally constant maps are also known as \emph{(sliding) block codes}.
\end{defi}

\begin{rem}
\label{rem:LocConstFinVal}
Let \( f \) be locally constant. Then \( f \) takes only finitely many values, since the compactness of \( \Subshift \) implies that there exists a finite covering \( \Subshift = \cup_{ i = 1 , \ldots , k } U_{\InfWord_{i}} \) such that \( f \) is constant on each \( U_{\InfWord_{i}} \).
\end{rem}

\begin{rem}
While the Jacobi operators that we consider involve real-valued locally constant functions, Definition~\ref{defi:LocConst} is not restricted to this setting. In fact we will study matrix-valued locally constant functions in Chapter~\ref{chap:UnifCocycCantor}.
\end{rem}

The starting point for our investigation of the spectrum is the eigenvalue equation \( \Jac_{\InfWord} \varphi = E \varphi \) with \( E \in \RR \) and \( \varphi \in \RR^{\ZZ} \). Note that we do not assume \( \varphi \in \ell^{2}( \ZZ ) \), so \( \varphi \) is not necessarily an eigenfunction even if it solves the eigenvalue equation. Since such a solution is determined by its value at two consecutive positions, we collect them in a vector \( \Phi( j ) \DefAs \big( \begin{smallmatrix} \varphi( j+1 ) \\ \varphi( j ) \end{smallmatrix} \big) \). The definition of \( \Jac_{\InfWord} \) yields for \( j > 0 \) the equality
\[ \bigg(\begin{matrix} \varphi( j+1 ) \\ \varphi( j ) \end{matrix}\bigg) = \bigg(\begin{matrix} \frac{ E - \Dig( \Shift^{j} \InfWord ) }{ \NDig( \Shift^{j+1} \InfWord )} & - \frac{ \NDig(\Shift^{j} \InfWord) }{ \NDig( \Shift^{j+1} \InfWord )} \\ 1 & 0 \end{matrix}\bigg) \bigg(\begin{matrix} \varphi( j ) \\ \varphi( j-1 ) \end{matrix}\bigg) \]
and for \( j < 0 \) the equality
\[ \bigg(\begin{matrix} \varphi( j ) \\ \varphi( j-1 ) \end{matrix}\bigg) = \bigg(\begin{matrix} 0 & 1 \\ - \frac{ \NDig( \Shift^{j+1} \InfWord ) }{ \NDig( \Shift^{j} \InfWord ) } & \frac{ E - \Dig( \Shift^{j} \InfWord ) }{ \NDig( \Shift^{j} \InfWord ) } \end{matrix}\bigg) \bigg(\begin{matrix} \varphi( j+1 ) \\ \varphi( j ) \end{matrix}\bigg) \, . \]
The matrices above are called \emph{elementary transfer matrices}\index{elementary tansfer matrix}\index{transfer matrix!elementary}. They encode how the value of \( \varphi \) at two positions determines the value to the right (for \( j > 0 \)) or to the left (for \( j < 0 \)). This process is now iterated. To shorten notation, we introduce the \emph{transfer matrix map}\index{transfer matrix!map}, defined by
\begin{equation}
\label{eqn:TrMatMap}
\TrMat \colon \Subshift \to \GL \;\, , \;\; \InfWord \mapsto \bigg(\begin{matrix} \frac{ E - \Dig( \Shift \InfWord ) }{ \NDig( \Shift^{2} \InfWord ) } & - \frac{ \NDig( \Shift \InfWord ) }{ \NDig( \Shift^{2} \InfWord ) } \\ 1 & 0 \end{matrix}\bigg) \, .
\end{equation}
We obtain for \( j > 0 \) the equality
\[ \Phi( j ) = \TrMat( \Shift^{j-1} \InfWord ) \cdot \ldots \cdot \TrMat( \InfWord ) \, \Phi( 0 ) \]
and for \( j < 0 \) the equality
\[ \Phi( j ) = \TrMat( \Shift^{j} \InfWord )^{-1} \cdot \ldots \cdot \TrMat( \Shift^{-1} \InfWord )^{-1} \, \Phi( 0 ) \, .\]
The matrix products above are called \emph{transfer matrices}\index{transfer matrix}. Since they describe the behaviour of solutions of the eigenvalue equation, they are an important tool in the spectral theory of Jacobi operators. Both transfer matrix equations can be compressed into a single map. By a slight abuse of notation it is also called \( \TrMat \) and defined by
\begin{equation}
\label{eqn:TrMatCocyc}
\TrMat \colon \ZZ \times \Subshift \to \GL \, , \; ( j , \InfWord ) \mapsto \begin{cases}
\TrMat( \Shift^{j-1} \InfWord ) \cdot \ldots \cdot \TrMat( \Shift^{0} \InfWord ) & \text{for } j > 0 \\
\Id & \text{for } j = 0 \\
\TrMat( \Shift^{j} \InfWord )^{-1} \cdot \ldots \cdot \TrMat( \Shift^{-1} \InfWord )^{-1} & \text{for } j < 0
\end{cases} \, .
\end{equation}
This yields \( \Phi( j ) = \TrMat( j , \InfWord) \, \Phi( 0 )\) for all \( j \in \ZZ \). The same type of matrix products will also appear in Chapter~\ref{chap:UnifCocycCantor} when we study cocycles. In fact, transfer matrices are cocycles and they will be the main application of our more general approach in the next chapter.

%%%%%%%%%%%%%%%%%%%%%%%%
\section{The spectrum of Jacobi operators}
\label{sec:SpecJO}

For \( \InfWord \in \Subshift \) the spectrum of the associated Jacobi operator \( \Jac_{\InfWord} \) on \( \ell^{2}( \ZZ ) \) is given by
\[ \sigma( \Jac_{\InfWord} ) \DefAs \Set{ E \in \CC : E \cdot \Id - \Jac_{\InfWord} \text{ is not a bijection with bounded inverse} } \, . \]
In the following we review basic facts about the spectrum for the case that the subshift is aperiodic. For a more thorough discussion of spectral theory, see standard works such as \cite{ReedSimon1} or the surveys \cite{Suto_SOSurvey}, \cite{Dam_GordonInBook}, \cite{Dam_SurveyStricErgodSO}, \cite{Dam_SOSurvey} and the references therein. Many results on Jacobi operators and their spectrum can also be found in \cite{Teschl_JO}.

First note that \( \Jac_{\InfWord} \) is self-adjoint, which implies \( \sigma( \Jac_{\InfWord} ) \subseteq \RR \). Moreover since \( \Jac_{\InfWord} \) is bounded there is a spectral measure \( \mu_{\varphi} \) associated to every \( \varphi \in \ell^{2}( \ZZ ) \). We write:
\begin{tightitemize}
\tightlist
\item{\( \ell^{2}( \ZZ )_{\AC} \DefAs \{ \, \varphi \in \ell^{2}( \ZZ ) : \mu_{\varphi} \) is absolutely continuous with respect to the Lebesgue measure\( \, \}\).}
\item{\( \ell^{2}( \ZZ )_{\SC} \DefAs \{ \, \varphi \in \ell^{2}( \ZZ ) : \mu_{\varphi} \) is singular with respect to the Lebesgue measure and continuous\( \, \} \).}
\item{\( \ell^{2}( \ZZ )_{\PP} \DefAs \{ \, \varphi \in \ell^{2}( \ZZ ) : \mu_{\varphi} \) is a pure point measure\( \, \} \).}
\end{tightitemize}

\begin{defi}
The spectrum of the restriction of \( \Jac_{\InfWord} \) to each of the sets above is called: 
\begin{tightitemize}
\tightlist
\item{\( \sigma_{\AC}( \Jac_{\InfWord} ) \DefAs \sigma( \Jac_{\InfWord}|_{\ell^{2}( \ZZ )_{\AC}} ) \) the absolutely continuous spectrum of \( \Jac_{\InfWord} \).}
\item{\( \sigma_{\SC}( \Jac_{\InfWord} ) \DefAs \sigma( \Jac_{\InfWord}|_{\ell^{2}( \ZZ )_{\SC}} ) \) the singular continuous spectrum of \( \Jac_{\InfWord} \).}
\item{\( \sigma_{\PP}( \Jac_{\InfWord} ) \DefAs \Set{  E \in \RR : E \text{ is an eigenvalue of } \Jac_{ \InfWord } } \) the pure point spectrum of \( \Jac_{\InfWord} \).}
\end{tightitemize}
\end{defi}

Note that \( \Close{  \sigma_{\PP}( \Jac_{\InfWord} ) } = \sigma( \Jac_{\InfWord}|_{\ell^{2}( \ZZ )_{\PP}} ) \) holds. We obtain the following decomposition of the spectrum, see for example \cite[Section~VII.2]{ReedSimon1}:
\[ \sigma( \Jac_{\InfWord} ) = \sigma_{\AC}( \Jac_{\InfWord} ) \cup \sigma_{\SC}( \Jac_{\InfWord} ) \cup \Close{\sigma_{\PP}( \Jac_{\InfWord} )} \, . \]
Since every \( \InfWord \in \Subshift \) defines a different operator \( \Jac_{\InfWord} \), the spectra will in general also be different. However, in the situation that we typically consider in this thesis, the spectrum is independent of \( \InfWord \). In an almost sure sense, this is even true for the spectrum's components:

\begin{prop}
\label{prop:SpecIndepOmega}
Let \( \Subshift \) be a minimal and uniquely ergodic subshift. Then there exist sets \( \Spec , \Spec_{\AC} , \Spec_{\SC} , \Spec_{\PP} \subseteq \RR \) such that 
\[ \sigma( \Jac_{\InfWord} ) = \Spec \qquad \text{and} \qquad \sigma_{\AC}( \Jac_{\InfWord} ) = \Spec_{\AC} \;\; , \;\;\; \sigma_{\SC}( \Jac_{\InfWord} ) = \Spec_{\SC} \;\; , \;\;\; \Close{\sigma_{\PP}( \Jac_{\InfWord} )} = \Spec_{\PP} \]
hold for almost all \( \InfWord \in \Subshift \) with respect to the ergodic measure. The equality \( \sigma( \Jac_{\InfWord} ) = \Spec \) even holds for all \( \InfWord \in \Subshift \). This set is therefore referred to as the \emph{spectrum of the Jacobi operator}\index{spectrum of a Jacobi operator}\index{Jacobi operator!spectrum of} on the subshift. 
\end{prop}

The almost sure constancy can be found in \cite[Corollary~1]{Pastur_LyapuAndSc} and \cite[Theorem~9.4]{CFKS_SOApplicQC} for Schrödinger operators and in \cite[Theorem~5.3 and~5.4]{Teschl_JO} for Jacobi operators. For the complete \( \InfWord \)-independence of the spectrum, which is based on the strong convergence of an approximation by minimality, see \cite[Section~6.2]{Suto_SOSurvey} or \cite[Theorem~3.2]{Dam_SurveyStricErgodSO}. At least for Schrödinger operators, the equality \( \sigma_{\AC}( \Jac_{\InfWord} ) = \Spec_{\AC} \) also holds everywhere (\cite[Theorem~1.5]{LastSim_Spec1dSO}), but \( \sigma_{\SC}( \Jac_{\InfWord} ) \) and \( \sigma_{\PP}( \Jac_{\InfWord} ) \) will in general depend on \( \InfWord \), see \cite[Example~1]{JitoSim_SCSpecContCase}.

\begin{rem}
\label{rem:PerioAndRand}
As a side remark we comment on two extremal cases that are beyond the scope of this thesis: crystals and random (also: ``glassy'' or ``amorphous'') materials. In both cases, the spectrum of the associated Schrödinger operator has been thoroughly studied.

On the one hand, assume that \( \Jac_{ \InfWord } \) is \( P \)-periodic (``a crystal''). This yields a periodic difference equation which can be treated in the framework of Floquet/Bloch theory (\cite{Floqu_DiffEqCoefPerio}, \cite{Bloch_QMKristall}). An important role is played by the transfer matrix over a whole period. Since it has determinant one (see Equation~(\ref{eqn:TrMatMap})), its eigenvalues \( m_{\pm} \) satisfy \( m_{-} = \frac{1}{m_{+}} \). A short computation shows that \( \lvert m_{\pm} \rvert = 1 \) holds if and only if \(  \lvert \Tr( \TrMat( P , \InfWord ) ) \rvert \leq 2 \) holds. Moreover, the eigenvalues of \( \TrMat( P , \InfWord ) \) are connected to the asymptotic behaviour of solutions of the eigenvalue equation \( \Jac_{\InfWord} \varphi = E \varphi \). Combined, this yields the important relation:
\[ \sigma( \Jac_{ \InfWord } ) = \Set{ E \in \RR : \lvert \Tr( \TrMat( P , \InfWord ) ) \rvert \leq 2 } \, . \]
It can be shown that the spectrum of a periodic Schrödinger operator is purely absolutely continuous. A detailed treatment of this topic can be found in \cite[Subsection~VII.2.1]{CarmoLacr_SpecRandSO}.

On the other hand, assume that the value of each \( \InfWord( j ) \) is chosen randomly, independent from each other and with identical distributions (``a glassy material''). This is known as the Anderson model. The associated Schrödinger operator almost surely has pure point spectrum and exhibits so-called Anderson localisation, that is, exponentially decaying eigenfunctions, see \cite{And_AndersLoc}, \cite{CarmKleiMart_AndersBernou}, but also \cite[Chapter~9]{CFKS_SOApplicQC}, \cite[Chapter~VIII]{CarmoLacr_SpecRandSO} and \cite[Section~4]{Dam_SOSurvey}).
\end{rem}

In aperiodic subshifts there is a certain degree of both, disorder (which roughly corresponds to singular spectrum) and order (which roughly corresponds to continuous spectrum). Here and in the next subsection we discuss how aperiodicity indeed excludes absolutely continuous spectrum and how symmetries can be used to exclude pure point spectrum. For aperiodic subshifts where both approaches can be applied, this yields purely singular continuous spectrum.

The line of arguments which deduces absence of \( \AC \)-spectrum from aperiodicity is known as Kotani theory (\cite{Kotani_KotTheo}, \cite{Sim_KotaniLattice}, \cite{Kotani_AperioNoAC}) and is summarised in the following theorem. For Schrödinger operators it goes back to Kotani and was later generalised to Jacobi operators. The \( \InfWord \)-independence of \( \sigma_{\AC}( \Jac_{\InfWord} ) \), which was discussed in the paragraph after Proposition~\ref{prop:SpecIndepOmega}, is also used.

\begin{thm}
\label{thm:KotaniNoAC}
Let \( \Subshift \) be minimal and uniquely ergodic and assume that \( ( \NDig , \Dig ) \) is aperiodic. Then \( \Spec_{\AC} = \emptyset \) holds, that is, the spectrum of the Jacobi operator on \( \Subshift \) has no absolutely continuous part.
\end{thm}

We do not give a proof here, but outline its main idea. The proof relies on the relation between the spectrum and the asymptotic exponential growth rate of the norm of transfer matrices. More precisely, \( \Spec_{\AC} \) is the essential closure of the set of those energies, for which the Lyapunov exponent is zero (see \cite{Ishii_LyapuAndSC}, \cite{Pastur_LyapuAndSc}, \cite{Kotani_KotTheo} for Schrödinger operators and \cite[Theorem~5.17]{Teschl_JO} for Jacobi operators). If that set had positive Lebesgue measure, then almost all elements would be determined by their restriction to \( \NN \) respectively \( -\NN \), see \cite[Lemma~2.2]{Kotani_AperioNoAC}, \cite[Theorem~5.20]{Teschl_JO}. If \( \NDig \) and \( \Dig \) are locally constant, then this would imply that \( ( f , g ) \) are eventually constant (see also \cite[Theorem~1.1]{Reml_ACSpecJO} for Jacobi operators on the half-line and \cite[proof of Theorem~3]{BeckPogo_SpectrJacobi} for a discussion of the extension to \( \ZZ ) \)). Thus the aperiodicity of \( ( f , g ) \) implies Lebesgue measure zero for the spectrum.

%%%%%%%%%%%%%%%%%%%%%%%%
\section{Gordon-type arguments}
\label{sec:JOGordon}

As we have just seen, aperiodic subshifts are ``random enough'' to have no \( \AC \)-spectrum. In this section we discuss the converse, namely how symmetries (that is, ``order'') can be used to exclude \( \PP \)-spectrum. At the end of the section we comment briefly on reflection symmetry. However, our main focus lies on methods that use repetitions, also known as \emph{powers of a word}\index{power of a word}\index{word!power of} and denoted for example by \( u^{3} \DefAs u u u \). This approach goes back to Gordon (\cite{Gordon_3Block}) and is known as the \emph{three-block Gordon argument}\index{three-block Gordon argument}\index{Gordon argument!three-block}. Its basic form is sketched below. Additional background information can for instance be found in \cite{Dam_GordonInBook}. 

\begin{prop}[three-block Gordon argument, \cite{Gordon_3Block}]
\label{prop:ThreeGordon}
If there exists a sequence \( ( l_{i} ) \) of natural numbers with \( \lim_{i \to \infty} l_{i} = \infty \) such that \( \Restr{\InfWord}{-2l_{i}+1}{-l_{i}} = \Restr{\InfWord}{-l_{i}+1}{0} = \Restr{\InfWord}{1}{ l_{i}} \) holds for all \( i \in \NN \), then \( \Jac_{\InfWord} \) has no eigenvalues.
\end{prop}

A more general version of this statement will be proved in detail in Proposition~\ref{prop:GordonCocyc}. Here we present the underlying mechanism only in the simplest case, where \( f \equiv 1 \) holds and \( g \) only depends on \( \InfWord( 0 ) \):

\begin{proof}[Sketch of proof]
Let \( \varphi \) be a solution of the eigenvalue equation \( \Jac_{\InfWord} \varphi = E \varphi \). Then \( \Phi( j ) \DefAs \big( \begin{smallmatrix} \varphi( j+1 ) \\ \varphi( j ) \end{smallmatrix} \big) \) is given by \( \Phi( j ) = \TrMat( j , \InfWord) \, \Phi( 0 )\) for all \( j \in \ZZ \). Clearly \( \Restr{\InfWord}{-2l_{i}+1}{-l_{i}} = \Restr{\InfWord}{-l_{i}+1}{0} = \Restr{\InfWord}{1}{l_{i}} \) implies \( \TrMat( l_{i} , \Shift^{-2l_{i}} \InfWord ) = \TrMat( l_{i} , \Shift^{-l_{i}} \InfWord ) = \TrMat( l_{i} , \InfWord ) \). Applying the transfer matrix over \( \Restr{\InfWord}{-l_{i}+1}{0} \) twice therefore yields the transfer matrix over \(\Restr{ \InfWord }{ -2l_{i}+1 }{ 0 } \). By the Cayley/Hamilton theorem for \( \SL \)-matrices we obtain
\[ \TrMat( 2 l_{i} , \Shift^{-2l_{i}} \InfWord ) = \TrMat( l_{i} , \Shift^{-l_{i}} \InfWord )^{2} = \Tr( \TrMat( l_{i} , \Shift^{-l_{i}} \InfWord ) ) \TrMat( l_{i} , \Shift^{-l_{i}} \InfWord ) - \Id \, .\]
Now multiplication with  \( \Phi( -2l_{i} ) \) and \( \Phi( -l_{i} ) \) respectively yields the equations:
\begin{alignat}{3}
\Phi( 0 ) &=  \TrMat( 2 l_{i} , \Shift^{-2l_{i}} \InfWord ) \, \Phi( -2l_{i} )  &&= \Tr( \TrMat( l_{i} , \Shift^{-l_{i}} \InfWord ) ) \Phi( -l_{i} ) - \Phi( -2l_{i} ) \, , \label{eqn:ThreeGordon}\\
\Phi( l_{i} ) &=  \TrMat( 2 l_{i} , \Shift^{-l_{i} }\InfWord ) \, \Phi( - l_{i} )  &&= \Tr( \TrMat( l_{i} , \Shift^{-l_{i}} \InfWord ) ) \Phi( 0 ) - \Phi( -l_{i} ) \, . \notag
\end{alignat}
If we assume \( \Phi ( 0 ) \neq \Big( \begin{smallmatrix} 0 \\ 0 \end{smallmatrix} \Big) \), then clearly \( \Phi( -2l_{i} ) \), \( \Phi( -l_{i} ) \) and \( \Phi( l_{i} ) \) cannot all tend to zero for \( i \to \infty \). Hence no non-zero solution of the eigenvalue equation is square summable.
\end{proof}

\begin{rem}
By symmetry the same result holds if \( \Restr{\InfWord}{-l_{i}+1}{0} = \Restr{\InfWord}{1}{l_{i}} = \Restr{\InfWord}{l_{i}+1}{ 2l_{i}} \) is satisfied. Moreover, the repetitions could be aligned at any fixed position instead of the origin.
\end{rem}

\begin{exmpl}[no eigenvalues for leading words]
\label{exmpl:AbsenSpWords}
Recall from Example~\ref{exmpl:SpWBlocks} that the leading words of a simple Toeplitz subshift are the elements \( \SpW{ \AEv } \DefAs \lim_{k \to \infty} \Word{ \PBlock{k} }{ \Origin{ \AEv } }{ \PBlock{k} } \) with \( \AEv \in \AlphabEv \). We apply the three-block Gordon argument: for fixed \( \AEv \in \AlphabEv \) we choose a subsequence \( ( a_{k_{i}} )_{i} \) of the coding sequence such that \( a_{k_{i}} = \AEv \) holds for all \( i \). Since every \( \PBlock{k} \) is a prefix and a suffix of \( \PBlock{k+1} \), the leading word \( \SpW{ \AEv } \) looks around the origin like
\begin{alignat*}{3}
\SpW{ \AEv } = {} && \Word{ \ldots }{ \PBlock{k_{i}} } & \Word{ \Origin{\AEv} }{ \PBlock{k_{i}} }{ \ldots } \\
= {} && \Word{ \ldots }{ \PBlock{k_{i}-1} }{ a_{k_{i}} }{ \PBlock{k_{i}-1} } & \Word{ \Origin{\AEv} }{ \PBlock{k_{i}-1} }{ a_{k_{i}} }{ \PBlock{k_{i}-1} }{ \ldots } \, .
\end{alignat*}
Now Proposition~\ref{prop:ThreeGordon} with \( l_{i} \DefAs \Length{ \PBlock{k_{i}-1} } + 1 \) yields \( \sigma_{\PP}( \Jac_{\SpW{ \AEv }} ) = \emptyset \).
\end{exmpl}

\begin{rem}
\label{rem:GenAbsPP}
Absence of eigenvalues for a single \( \InfWord \in \Subshift \) already implies \emph{generic absence}\index{generic}, that is, absence for a dense \( G_{\delta} \) set of elements, see \cite[Proposition~6.1]{Dam_GordonInBook} for the Schrödinger case. This follows from the fact that \( \Set{ \InfWord \in \Subshift : \sigma_{\PP}( \Jac_{\InfWord} ) = \emptyset } \) is a \( G_{\delta} \) set (\cite[Theorem~1.1]{Simon_SCSpec1}) and if it is non-empty, then it is dense by minimality. When the three-block Gordon argument applies, it therefore proves generic absence of pure point spectrum. However, it can be shown that every minimal, aperiodic subshift contains at least one element which has no square centred at the origin (\cite[Theorem~2]{Dam_SCSpec2}, based on a result from \cite{DuvMigRest_LocGlobPeriod}). In particular the three-block Gordon argument alone is not sufficient to prove uniform absence of eigenvalues.
\end{rem}

By Equation~(\ref{eqn:ThreeGordon}) there is another way to exclude eigenvalues: if the trace is bounded, then a twofold repetition on one side of the origin suffices. This criterion is known as the \emph{two-block Gordon argument}\index{two-block Gordon argument}\index{Gordon argument!two-block}. It was proved in \cite[Lemma~1 and Proposition~2]{Suto_QuasiPerSO}.

\begin{prop}[two-block Gordon argument]
\label{prop:TwoGordon}
Let \( E \in \RR \). If there exists a sequence \( ( l_{i} ) \) of natural numbers with \( \lim_{i \to \infty} l_{i} = \infty \) and a constant \( C \in \RR \) such that \( \Restr{\InfWord}{-2l_{i}+1}{-l_{i}} = \Restr{\InfWord}{-l_{i}+1}{0} \) and \( \lvert \Tr( \TrMat( l_{i} , \Shift^{-l_{i}} \InfWord ) ) \rvert \leq C \) hold for all \( i \in \NN \), then \( E \) is not an eigenvalue of \( \Jac_{\InfWord} \).
\end{prop}

Provided that there are arbitrarily long squares, we can therefore exclude point spectrum if, for all \( E \in \sigma( \Jac_{\InfWord} ) \), the trace of the corresponding transfer matrix is bounded. Fortunately, periodic approximation shows that being in the spectrum is closely related to bounded traces. While we do not use this relation in the following, it is crucial for the earlier results on Schrödinger operators on simple Toeplitz subshifts that we cite next. Since it also highlights the importance of transfer matrices and recurrence relations for words, we briefly sketch the main idea (compare \cite[Section~4]{Dam_GordonInBook}): assume that there is a sequence \( ( \InfWord_{k} ) \) of \( P_{k} \)-periodic words, such that the associated Schrödinger operators \( \Jac_{ \InfWord_{k} } \) converge strongly to \( \Jac_{\InfWord} \). As mentioned in Remark~\ref{rem:PerioAndRand}, the periodic operators have the spectrum \( \sigma( \Jac_{ \InfWord_{k} } ) = \Set{ E \in \RR : \lvert \Tr( \TrMat( P_{k} , \InfWord_{k} ) ) \rvert \leq 2 } \). It can be shown that the strong convergence implies
\[ \sigma( \Jac_{ \InfWord } ) \subseteq \bigcap_{j \in \NN} \Close{ \bigcup_{k \geq j} \Set{ E \in \RR : \lvert \Tr( \TrMat( P_{k} , \InfWord_{k} ) ) \rvert \leq 2 } } \; . \]
The idea is to choose the approximating sequence \( ( \InfWord_{k} ) \) such that it yields a ``nice'' recurrence relation for \( \Tr( \TrMat( P_{k} , \InfWord_{k} ) ) \), known as the \emph{trace map}\index{trace map}. Then this recurrence has to be used to show that \( \bigcup_{k \geq j} \Set{ E \in \RR : \lvert \Tr( \TrMat( P_{k} , \InfWord_{k} ) ) \rvert \leq 2 } \) is already a closed set. This yields
\begin{align*}
\sigma( \Jac_{ \InfWord } ) &\subseteq \bigcap_{j \in \NN} \bigcup_{k \geq j} \Set{ E \in \RR : \lvert \Tr( \TrMat( P_{k} , \InfWord_{k} ) ) \rvert \leq 2 } \\
&= \Set{ E \in \RR : \lvert \Tr( \TrMat( P_{k} , \InfWord_{k} ) ) \rvert \leq 2 \text{ for infinitely many } k } \, , 
\end{align*}
which is precisely what is needed for the two-block Gordon argument. The existence of such trace maps for arbitrary substitution systems was for instance proved in \cite{KolNori_TrMapSubst} and \cite{AvishBeren_TrMapSubust}.

\begin{exmpl}[period doubling]
\label{exmpl:TrMapPD}
Recall from Example~\ref{exmpl:PD} that the period doubling subshift is defined by the limit of the words \( \PBlock{k+1} \DefAs \Word{ \PBlock{k} }{ a_{k+1} }{ \PBlock{k} } \) with \( (a_{k})_{k \in \NN_{0}} = ( a , b , a , b , \ldots ) \). For the approximation we use the periodic words \( \InfWord_{k} \DefAs ( \Word{ \PBlock{k} }{ a_{k+1} } )^{\infty} \) with period length \( P_{k} = 2^{k+1} \). Note that the words \( \Word{ \PBlock{k} }{ a_{k+1} } \) are generated by the powers of the period doubling substitution, see Remark~\ref{rem:SubstPD}. We consider the trace \( \tau_{E , k} \DefAs \Tr( \TrMat( P_{k} , \InfWord_{k} ) ) \) for \( \NDig \equiv 1 \) and \( \Dig( \InfWord ) \) only depending on \( \InfWord(0) \). In \cite[Section~1]{BellBovGhez_PD} it is shown that the recurrence relation for \( \PBlock{k} \) implies the trace map \( \tau_{E , k+1} = \tau_{E , k} ( \tau_{E , k-1}^{2} - 2 ) - 2 \). Consequently, if \( \lvert \tau_{E , j} \rvert > 2 \) and \( \lvert \tau_{E , j+1} \rvert > 2 \) hold, then \( \lvert \tau_{E , k} \rvert > 2 \) holds for all \( k \geq j \). This yields
\begin{equation}
\label{eqn:NonEscTr}
\bigcup_{k \geq j} \Set{ E \in \RR : \lvert \tau_{E , k} \rvert \leq 2 } = \Set{ E \in \RR : \lvert \tau_{E , j} \rvert \leq 2 } \cup \Set{ E \in \RR : \lvert \tau_{E , j+1} \rvert \leq 2 } \, .
\end{equation}
Since \( \TrMat( P_{k} , \InfWord_{k} ) \) is a product of \( P_{k} \) elementary transfer matrices, \( \tau_{E , k} \) is a polynomial in \( E \) and hence continuous. Therefore the sets on the right hand side in Equation~(\ref{eqn:NonEscTr}) are indeed closed.
\end{exmpl}

Similar arguments were also used for Schrödinger operators on the Fibonacci subshift (\cite{Suto_QuasiPerSO}) and on Sturmian subshifts in general (\cite{BellIochScopTest_SpecProp}), see Appendix~\ref{app:Sturm} for definitions. In addition, Liu and Qu generalised the previous example to arbitrary simple Toeplitz subshifts and obtained a result analogous to Equation~(\ref{eqn:NonEscTr}), see \cite[Lemma~3.2]{LiuQu_Simple}. This implies that, for all energies in the spectrum of the Schrödinger operator, the trace over \( \Word{ \PBlock{k} }{ a_{k+1} } \) is bounded for infinitely many \( k \), which yields the following result:

\begin{prop}[{\cite[Theorem~1.3]{LiuQu_Simple}}]
\label{prop:LiuQuN4NoEV}
Let \( \Subshift \) be a simple Toeplitz subshift with coding sequences \( (a_{k}) \) and \( (n_{k}) \). Let \( f \equiv 1 \) and let \( \Dig \) only depend on \( \InfWord(0) \). If \( n_{k} \geq 4 \) holds for all \( k \geq 0 \), then \( \Jac_{\InfWord} \) has purely singular continuous spectrum for every \( \InfWord \in \Subshift \).
\end{prop}

For a detailed proof the reader is referred to \cite{LiuQu_Simple}. The main idea is to use Kotani theory (see Theorem~\ref{thm:KotaniNoAC}) to exclude \( \AC \)-spectrum. To prove absence of \( \PP \)-spectrum, consider a subsequence \( ( k_{i} ) \) such that the trace over \( \Word{ \PBlock{k_{i}} }{ a_{k_{i}+1} } \) is bounded. Because of \( n_{k_{i}} \geq 4 \), the decomposition of \( \InfWord \) into \( \PBlock{k_{i}} \)-blocks has a square at least on one side of the origin. Thus the two-block Gordon argument applies. As the next example shows, \( n_{k} \geq 4 \) is not a necessary condition:

\begin{exmpl}[no eigenvalues for period doubling, \cite{Damanik_ScForPD}]
For the period doubling subshift it was shown in \cite[Theorem~1]{Damanik_ScForPD} that the Schrödinger operator \( \Jac_{\InfWord} \) has empty point spectrum for every \( \InfWord \in \Subshift \). The proof analyses the possible arrangements of the words \( \Word{ \PBlock{k} }{ a } \) and \( \Word{ \PBlock{k} }{ b } \) around the origin. It turns out that there is always either a threefold repetition or a twofold repetition whose trace can be controlled by the recurrence relation from \cite{BellBovGhez_PD} (see Example~\ref{exmpl:TrMapPD} above). 
\end{exmpl}
  
As the period doubling example shows, there are simple Toeplitz subshifts where every element contains growing cubes or growing squares with bounded trace. However, there are also simple Toeplitz words without squares at their origin (let alone cubes). In this case one cannot prove absence of eigenvalues by Gordon-type arguments. Such words are for example contained in the Grigorchuk subshift:
 
\begin{exmpl}[no squares in alternating words]
\label{exmpl:NoSquaresUnSpW}
In the alternating word \( \widetilde{ \InfWord } \) of the Grigorchuk subshift (see Example~\ref{exmpl:DefAlterWord}), there are no squares on either side of the origin: for arbitrary \( L \in \NN \), let \( k \) be such that \( \Length{\PBlock{ k }}+1 \leq L < \Length{\PBlock{ k+1 }} + 1 \) holds. We consider the decomposition of \( \widetilde{ \InfWord } \) in \( \PBlock{k} \)-blocks, here for even \( k \):
\[  \widetilde{ \InfWord } = \Word{ \ldots }{ \PBlock{k} }{ a_{k+1} }{ \PBlock{k} }{a_{k+4}}{ \PBlock{k} }{ a_{k+1} }{ \PBlock{k} }{ a_{k+2} }{ \underline{\PBlock{k}} }{ a_{k+1} }{ \PBlock{k} }{ a_{k+3} }{ \PBlock{k} }{ a_{k+1} }{ \PBlock{k} }{ \ldots } \, , \]
where the underlined \( \PBlock{k} \) contains the origin. For odd \( k \), the arrangement of the letters \( a_{j} \) is simply reflected at the origin. We compare the prefixes of length \( \Length{\PBlock{ k }} + 1 \) of \( \Restr{ \widetilde{ \InfWord } }{ 1 }{ L } \) and  \( \Restr{ \widetilde{ \InfWord } }{ L+1 }{ 2L } \): while the prefix of the first word is a subword of \( \Word{ \PBlock{k} }{ a_{k+1} }{ \PBlock{ k } } \), the prefix of the second word is either in \( \Word{ \PBlock{k} }{ a_{k+3} }{ \PBlock{ k } } \) or in \( \Word{ \PBlock{k} }{ a_{k+1} }{ \PBlock{ k } } \). In the first case, it differs from \( \Restr{  \widetilde{ \InfWord }  }{ 1 }{ \Length{\PBlock{ k }} +1 } \) because of \( a_{k+1} \neq a_{k+3} \). In the second case, it starts at a different position inside \( \PBlock{k} \) than \( \Restr{  \widetilde{ \InfWord }  }{ 1 }{ \Length{\PBlock{ k }} +1 } \) does, since \(  L < \Length{\PBlock{ k+1 }} + 1 = 2 \Length{\PBlock{ k }} + 2 \) holds. Since the subwords of length \( \Length{\PBlock{ k }} + 1 \) of \( \Word{ \PBlock{k} }{ a }{ \PBlock{k} } \) are pairwise different (Subsection~\ref{subsec:IneqCompGrow}), we obtain \( \Restr{  \widetilde{ \InfWord }  }{ 1 }{ L } \neq \Restr{ \widetilde{ \InfWord }  }{ L+1 }{ 2L } \). Similar considerations for \( a_{k+1} \), \( a_{k+2} \) and \( a_{k+4} \) show that there are no squares to the left either.
\end{exmpl}

It is also possible to exclude eigenvalues with the help of palindromes (reflection symmetry) instead of powers of words (translation symmetry). Examples can be found in \cite[Theorem~5.1]{HofKnillSim_Pali} (based on a similar result in \cite{JitoSim_SCSpecContCase} for the continuous case) and in \cite[Theorem~1]{DamGhezRaym_Pali}. Similar to the difference between the three-block Gordon argument and the two-block Gordon argument, the result in \cite{HofKnillSim_Pali} requires stronger combinatorial properties, while the result in \cite{DamGhezRaym_Pali} requires additional control over the trace. Here we only recall a corollary from \cite{HofKnillSim_Pali} that has an immediate consequence for Schrödinger operators on simple Toeplitz subshifts:

\begin{prop}[{\cite[Corollary~7.3]{HofKnillSim_Pali}}]
Let \(f \equiv 1 \) and let \( g( \InfWord ) \) only depend on \( \InfWord( 0 ) \). If \( \Subshift \) is minimal, uniquely ergodic and contains arbitrarily long palindromes, then \( \Jac_{\InfWord} \) has purely singular continuous spectrum for a generic subset of \( \Subshift \).
\end{prop}

Recall that every simple Toeplitz subshift is minimal, uniquely ergodic (see Corollary~\ref{cor:STStrictErgod}) and contains arbitrarily long palindromes (namely the \( \PBlock{k} \)-blocks, see Example~\ref{exmpl:PBlockPalindr}). This yields the following result, which we will improve further in the next section:

\begin{cor}
\label{cor:GenNoPPSimToep}
In every simple Toeplitz subshift there is a generic subset of elements for which the associated Schrödinger operator has purely singular continuous spectrum.
\end{cor}

\begin{rem}
For a subshift that is generated by a primitive substitution (such as the Grigorchuk subshift, see Remark~\ref{rem:SubstGrig}), the elements to which the palindrome criterion from \cite{HofKnillSim_Pali} applies form a set of measure zero. This was proved in \cite[Theorem~1.3]{DamZare_PaliCompl} by studying the palindrome complexity. The criterion can therefore not be used to show almost sure or uniform absence of pure point spectrum for the Grigorchuk subshift.
\end{rem}

%%%%%%%%%%%%%%%%%%%%%%%%
\section{Almost sure absence of eigenvalues for simple Toeplitz subshifts}
\label{sec:AsNoPP}

By Corollary~\ref{cor:GenNoPPSimToep} the pure point spectrum of the Schrödinger operator \( \Jac_{\InfWord} \) is empty for a generic subset of every simple Toeplitz subshift. In the following we improve this result and show that, even for Jacobi operators, \( \sigma_{\PP}( \Jac_{\InfWord} ) = \emptyset \) holds for almost every \( \InfWord \in \Subshift \) with respect to the unique ergodic probability measure \( \ErgodMeas \). A similar result was shown for the Grigorchuk subshift in \cite[Theorem~4.6]{GLN_SpectraSchreierAndSO}. Our proof is also inspired by \cite{GLN_SpectraSchreierAndSO}: since \( \Jac_{\InfWord} \) and \( \Jac_{ \Shift \InfWord } \) are unitarily equivalent, the set \( \Set{ \InfWord \in \Subshift : \sigma_{\PP}( \Jac_{\InfWord} ) = \emptyset } \) is \( \Shift \)-invariant. By ergodicity, absence of pure point spectrum on a set of positive \( \ErgodMeas \)-measure implies therefore \( \ErgodMeas \)-almost sure absence. Since a threefold repetition excludes pure point spectrum by the Gordon argument, the following sufficient condition can be derived (see \cite[Proposition~6.4]{Dam_GordonInBook} for the Schrödinger case):

\begin{prop}
\label{prop:PosMesAsNoPP}
Let \( \Subshift \) be a uniquely ergodic subshift and \( (l_{i}) \) a sequence of natural numbers with \( \lim_{i \to \infty} l_{i} = \infty \). If
\[ \limsup_{i \to \infty} \ErgodMeas \big( \Set{ \InfWord \in \Subshift : \Restr{\InfWord}{-2l_{i}+1}{-l_{i}} = \Restr{\InfWord}{-l_{i}+1}{0} = \Restr{\InfWord}{1}{l_{i}} } \big) > 0 \]
holds, then \( \Jac_{\InfWord} \) has no eigenvalues for almost all \( \InfWord \in \Subshift \) with respect to the ergodic measure. 
\end{prop}

To prove that there is a set of positive measure with threefold repetitions in its elements, we first recall standard notions on frequencies; for more details see for example \cite[Section~IV.2]{Queff_SubstDynSys}, \cite[Example~1.4]{Dam_GordonInBook} or \cite[Subsection~2.2]{Dam_SurveyStricErgodSO}.

\begin{defi}
\label{defi:FreqAndCo}
For \( u, v \in \Langu{ \Subshift } \), let \( \Copies{u}{v} \) denote the \emph{number of copies}\index{copies!number of}\index{overlapping copies} of \( u \) that appear in \( v \), where occurrences of \( u \) may overlap. Let \( \DisCopies{u}{v} \) denote the \emph{maximal number of non-overlapping copies}\index{non-overlapping copies}\index{copies!number of non-overlapping} of \( u \) in \( v \). For \( \RWord \in \Alphab^{\NN} \) we say that \( u \in \Langu{ \Subshift } \) \emph{has frequency \( c_{u} \) in \( \RWord \) }\index{frequency}\index{word!frequency of} if, for every \( j \in \NN \), we have \( c_{u} = \lim_{L \to \infty} \frac{1}{L} \Copies{u}{ \Restr{ \RWord }{ j }{ j+L-1 } } \) and the convergence is uniform in \( j \) (thus  \( c_{u} \) is sometimes called \emph{uniform frequency} as well). Similarly frequencies in \( \RWord \in \Alphab^{- \NN_{0}} \) are defined.
\end{defi}

The reason that we introduce frequencies here is that they are connected to the measure of cylinder sets: let \( \Subshift \DefAs \Subshift( \RWord ) \) be a subshift, generated by a one-sided infinite word \( \RWord \in \Alphab^{\NN} \), see Definition~\ref{defi:SubshiftLangu}. Then \( \Subshift \) is uniquely ergodic if and only if, for every \( u \in \Langu{ \Subshift } \), the frequency \( c_{u} \) in \( \RWord \) exists. In that case \( \ErgodMeas( \Cylin{u}{j} ) = c_{u} \) holds for every \( u \in \Langu{ \Subshift } \) and every \( j \in \ZZ \), compare \cite[Example~1.4]{Dam_GordonInBook}. Here \( \Cylin{u}{j} \) denotes the cylinder set of elements with word \( u \) at position \( j \), see Section~\ref{sec:PrelimSubsh}. We use the high frequency of \( \PBlock{k} \)-blocks in every \( \InfWord \in \Subshift \) to prove the main result of this chapter:

\begin{thm}
\label{thm:AsNoPP}
For every simple Toeplitz subshift \( \Subshift \) there exists a sequence \( ( l_{i} ) \) of natural numbers with \( \lim_{i \to \infty} l_{i} = \infty \) and
\[ \ErgodMeas \big( \Set{ \InfWord \in \Subshift : \Restr{\InfWord}{-2l_{i}+1}{-l_{i}} = \Restr{\InfWord}{-l_{i}+1}{0} = \Restr{\InfWord}{1}{ l_{i}} } \big) > \frac{ 1 }{ 3^{ 1 + \Card{ \AlphabEv } } } \, .\]
In particular \( \sigma_{\PP}( \Jac_{\InfWord} ) = \emptyset \) holds for \( \ErgodMeas \)-almost every \( \InfWord \in \Subshift \) by Proposition~\ref{prop:PosMesAsNoPP}.
\end{thm}

\begin{proof}
To shorten notation, we denote the set of elements with a threefold repetition around the origin by
\( R_{3} \DefAs \Set{ \InfWord \in \Subshift : \Restr{\InfWord}{-2l_{i}+1}{-l_{i}} = \Restr{\InfWord}{-l_{i}+1}{0} = \Restr{\InfWord}{1}{ l_{i}} } \). Now recall from Equation~(\ref{eqn:PBlockRecur}) the recurrence relation
\( \PBlock{-1} \DefAs \epsilon \) and \( \PBlock{k+1} \DefAs \Word{ \PBlock{k} }{ a_{k+1} }{ \PBlock{k} }{ \ldots }{ \PBlock{k} } \) for simple Toeplitz subshifts. For the one-sided infinite limit \( \PBlock{ \infty } \DefAs \lim_{k \to \infty} \PBlock{k} \) this yields \( \Copies{ \PBlock{k} }{ \Restr{ \PBlock{ \infty } }{ 1 }{ L } } \geq \Big\lfloor \frac{ L }{ \Length{\PBlock{ k }}+1 } \Big\rfloor \) for every \( L \in \NN \). By unique ergodicity (Corollary~\ref{cor:STStrictErgod}) we obtain for every \( j \in \ZZ \) the relation
\[ \ErgodMeas( \Cylin{ \PBlock{k} }{j} ) = \lim_{ L \to \infty} \Copies{ \PBlock{k} }{ \Restr{ \PBlock{ \infty } }{ 1 }{ L } } \cdot \frac{ 1 }{ L } \geq \frac{ 1 }{ \Length{\PBlock{ k }}+1 } \, . \]

It remains to relate cylinder sets to repetitions in \( \InfWord \). First we treat the case where the coding sequence \(( n_{k}) \) is eventually bounded by three (any greater number would do as well). The idea is that the reoccurrence of a letter in \( (a_{k}) \) causes a repetition. Since all \( n_{k} \) are small, the \( \PBlock{k} \)-block that contains the repetition is short and hence frequent. Here are the details:

For all sufficiently large \( k \) both \( \Alphab_{k} = \AlphabEv \) and \( n_{k} \leq 3 \) hold. Since at least one letter appears twice in \( a_{k} , a_{k+1} \ldots , a_{k+\Card{ \AlphabEv }} \) we can choose a sequence \( k_{i} \) with \( \lim_{i \to \infty} k_{i} = \infty \), such that for all \( i \in \NN \) we have \( n_{k_{i}} \leq 3 \) and \( a_{k_{i}} = a_{h_{i}} \) for some \( h_{i} \in \Set{ k_{i}+1 , k_{i}+2, \ldots , k_{i}+\Card{ \AlphabEv } } \). We define \( l_{i} \DefAs \Length{ \PBlock{k_{i}-1} } + 1 \) and \( r_{i} \DefAs \Length{ \PBlock{h_{i}-1} } + 1 \). By the recurrence relation for \( \PBlock{k} \)-blocks, the start of \( \PBlock{h_{i}} \) looks like:
\begingroup
\setlength{\arraycolsep}{1pt}
\[ \begin{array}{r *{13}{c} l}
\PBlock{h_{i}} = & \multicolumn{6}{c}{ \PBlock{h_{i}-1} } & a_{h_{i}} & \multicolumn{6}{c}{ \PBlock{h_{i}-1} } & \ldots \\
= & & \ldots & \PBlock{k_{i}-1} & a_{k_{i}} & \PBlock{k_{i}-1} & & a_{k_{i}} & & \PBlock{k_{i}-1} & a_{k_{i}} & \PBlock{k_{i}-1} & \ldots & & \ldots \, .
\end{array} \]
\endgroup 
For each \( \InfWord \in \Cylin{ \PBlock{h_{i}} }{ -r_{i} + 1 } \) we conclude that every \( \varrho \in \Set{ \InfWord , \ldots , \Shift^{ l_{i}-1 } \InfWord } \) lies in \( R_{3} \), see Figure~\ref{fig:CylinThreeRepe1}. Note that no element occurs twice in \( \cup_{j=0}^{l_{i}-1} \, \Shift^{j} ( \Cylin{ \PBlock{h_{i}} }{ -r_{i} + 1 } ) \): for every \( \InfWord \in \Cylin{ \PBlock{h_{i}} }{ -r_{i}+1 } \), the elements \( \InfWord , \ldots , \Shift^{ l_{i}-1 } \InfWord \) are pairwise different by aperiodicity. For two distinct elements \( \InfWord_{1} , \InfWord_{2} \in \Cylin{ \PBlock{h_{i}} }{ -r_{i}+1 } \) we obviously have \( \Shift^{j_{1}} \InfWord_{1} \neq \Shift^{j_{2}} \InfWord_{2} \) for \( j_{1} = j_{2} \in \Set{ 0 , \ldots , l_{i}-1 }\). For \( j_{1} \neq j_{2} \), the shifted elements are different since \( \PBlock{h_{i}} \) occurs in them at different positions. This yields the inclusion
\[ R_{3} \supseteq \DisjCupDisp_{j =0}^{ l_{i}-1 } \Shift^{j} ( \Cylin{ \PBlock{h_{i}} }{ -r_{i}+1 } ) \, . \]
Since the sets are disjoint and the measure \( \ErgodMeas \) is \( \Shift \)-invariant, we obtain the estimate
\[ \ErgodMeas( R_{3} ) \geq \ErgodMeas \Bigg( \DisjCupDisp_{j =0}^{ l_{i}-1 } \Shift^{j}( \Cylin{ \PBlock{h_{i}} }{ -r_{i}+1 } ) \Bigg) =  l_{i} \cdot \ErgodMeas( \Cylin{ \PBlock{h_{i}} }{ -r_{i}+1 } ) \geq \frac{ 1 }{ n_{h_{i}} \cdot \ldots \cdot n_{k_{i}} } \geq \frac{ 1 }{ 3^{ 1+\Card{ \AlphabEv} } } \; . \]

This finishes the first case. The remaining case concerns subshifts with a subsequence \( (k_{i}) \) such that \( n_{k_{i}} \geq 4 \) holds for every \( i \in \NN \). We use that every \( \PBlock{k_{i}} \) consists of ``many'' \( \PBlock{k_{i}-1} \)-blocks, which ensures the existence of ``many'' repetitions. To make this precise, define again \( l_{i} \DefAs \Length{\PBlock{ k_{i}-1 }} +1 \). For each \( \InfWord \in \Cylin{ \Word{a_{ k_{i}+1 } }{ \PBlock{k_{i}} } }{ -2l_{i} } \), every \( \varrho \) in \( \Set{ \InfWord , \ldots , \Shift^{ ( n_{k_{i}} - 3 )l_{i} - 1 } \InfWord } \) lies in \( R_{3} \), see Figure~\ref{fig:CylinThreeRepe2}. By the same argument as in the first case we obtain the inclusion
\[ R_{3} \supseteq \DisjCupDisp_{ j = 0 }^{ ( n_{k_{i}} - 3 )l_{i} - 1 } \Shift^{j} ( \Cylin{ \Word{ a_{ k_{i}+1 } }{ \PBlock{k_{i}} } }{ -2l_{i} } ) \, . \]
The disjointness of the sets and the \( \Shift \)-invariance of \( \ErgodMeas \) yield the desired estimate:
\[ \ErgodMeas( R_{3} ) \geq ( n_{k_{i}} - 3 ) l_{i} \cdot \frac{ n_{k_{i}+1}-1 }{ \Length{ \PBlock{ k_{i}+1 } } + 1 } = \frac{ n_{k_{i}} - 3  }{ n_{ k_{i} } } \cdot \frac{  n_{ k_{i}+1 } - 1 }{  n_{ k_{i}+1 } } \geq \frac{ 1 }{ 4 } \cdot \frac{ 1 }{ 2 } = \frac{ 1 }{ 8 } \, , \]
where we used that \( \Word{ a_{ k_{i}+1 } }{ \PBlock{k_{i}} } \) occurs \( ( n_{ k_{i}+1 } - 1) \)-times in \( \PBlock{ k_{i}+1 } \).
\end{proof}

\begin{figure}
\centering
\footnotesize
% set parameters for the image
\pgfmathsetmacro{\TikzCodSeqNkkk}{3} % number of pkk in pkkk
\pgfmathsetmacro{\TikzCodSeqNkk}{3} % number of pk in pkk
\pgfmathsetmacro{\TikzDistBlocks}{3} % distance between two blocks in image in pt
\pgfmathsetmacro{\TikzBuchstLength}{9} % length of a single-letter-block in pt
% compute block length from parameters
\pgfmathsetmacro{\TikzLengthPkkk}{0.8*\linewidth} % length of pkkk block in pt
\pgfmathsetmacro{\TikzLengthPkk}{(\TikzLengthPkkk - (\TikzCodSeqNkkk - 1) * ( \TikzBuchstLength + 2 * \TikzDistBlocks) ) / \TikzCodSeqNkkk} % length of pkk block in pt
\pgfmathsetmacro{\TikzLengthPk}{(\TikzLengthPkk - (\TikzCodSeqNkk - 1) * ( \TikzBuchstLength + 2 * \TikzDistBlocks) ) / \TikzCodSeqNkk} % length of pk block in pt
\begin{tikzpicture}
[every node/.style ={rectangle},
buchst/.style={minimum width=\TikzBuchstLength pt, draw, outer sep=0pt, inner sep=0pt, minimum height=0.6cm},
pkkk/.style={minimum width=\TikzLengthPkkk pt, draw, outer sep=0pt, inner sep=0pt, minimum height=0.6cm},
pkk/.style={minimum width=\TikzLengthPkk pt, draw, outer sep=0pt, inner sep=0pt, minimum height=0.6cm},
pk/.style={minimum width=\TikzLengthPk pt, draw, outer sep=0pt, inner sep=0pt, minimum height=0.6cm}]
\node [right] at (0,0) [pkkk] (Fix2) {\( \PBlock{h_{i}} \)};
\node [right] at (0,-1) [pkk] (A) {\( \PBlock{h_{i}-1} \)};
\node [buchst] (A) [right=\TikzDistBlocks pt of A] {\( b \)};
\node [pkk] (A) [right=\TikzDistBlocks pt of A] {\( \PBlock{h_{i}-1} \)};
\node [pkk, draw=none] (A) [right=\TikzDistBlocks pt of A] {\( \ldots \)};
\node [right] at (0,-2) [pk, draw=none] (A) {\( \ldots \)};
\node [buchst, draw=none] (A) [right=\TikzDistBlocks pt of A] {};
\node [pk] (A) [right=\TikzDistBlocks pt of A] {\( u \)};
\node [buchst] (A) [right=\TikzDistBlocks pt of A] {\( c \)};
\node [pk] (A) [right=\TikzDistBlocks pt of A] {\( u \)};
\node [buchst] (Fix1) [right=\TikzDistBlocks pt of A] {\( b \)};
\node [pk] (A) [right=\TikzDistBlocks pt of Fix1] {\( u \)};
\node [buchst] (A) [right=\TikzDistBlocks pt of A] {\( c \)};
\node [pk] (A) [right=\TikzDistBlocks pt of A] {\( u \)};
\node [buchst, draw=none] (A) [right=\TikzDistBlocks pt of A] {};
\node [pk, draw=none] (A) [right=\TikzDistBlocks pt of A] {\( \ldots \)};
\node [pkk, draw=none] (A) [right=\TikzDistBlocks pt of A] {\( \ldots \)};
\draw[dashed] (Fix1.north east) ++ ( \TikzDistBlocks/2 pt , 0 ) -- + (0 , 2.4 ) node[at end, left, yshift=-3]{0\vphantom{1}} node[at end, right, yshift=-3]{1\vphantom{0}};
\draw (Fix2.north east) -- +( 0.7 , 0 );
\draw (Fix2.south east) -- +( 0.7 , 0 );
\node [right=1 of Fix2.east]{\( \InfWord \)};
\draw (Fix2.north west) -- +( -0.7 , 0 );
\draw (Fix2.south west) -- +( -0.7 , 0 );
\end{tikzpicture}
\normalsize
\caption{An element \( \InfWord \in \Cylin{ \PBlock{h_{i}} }{ -r_{i}+1 } \). To shorten notation we use \( u \DefAs \PBlock{k_{i}-1} \), \( b \DefAs a_{h_{i}} \) and \( c \DefAs a_{k_{i}} \). Because of \( b = c \), each of the words \( \InfWord , \ldots , \Shift^{ l_{i}-1 } \InfWord \) has three repetitions around the origin.\label{fig:CylinThreeRepe1}}
\end{figure}

\begin{figure}
\centering
\footnotesize
% set parameters for the image
\pgfmathsetmacro{\TikzCodSeqNkkk}{5} % number of pkk in pkkk
\pgfmathsetmacro{\TikzDistBlocks}{3} % distance between two blocks in image in pt
\pgfmathsetmacro{\TikzBuchstLength}{9} % length of a single-letter-block in pt
% compute block length from parameters
\pgfmathsetmacro{\TikzLengthPkkk}{0.8*\linewidth} % length of pkkk block in pt
\pgfmathsetmacro{\TikzLengthPkk}{(\TikzLengthPkkk - (\TikzCodSeqNkkk - 1) * ( \TikzBuchstLength + 2 * \TikzDistBlocks) ) / \TikzCodSeqNkkk} % length of pkk block in pt
\begin{tikzpicture}
[every node/.style ={rectangle},
buchst/.style={minimum width=\TikzBuchstLength pt, draw, outer sep=0pt, inner sep=0pt, minimum height=0.6cm},
pkkk/.style={minimum width=\TikzLengthPkkk pt, draw, outer sep=0pt, inner sep=0pt, minimum height=0.6cm},
pkk/.style={minimum width=\TikzLengthPkk pt, draw, outer sep=0pt, inner sep=0pt, minimum height=0.6cm},
pk/.style={minimum width=\TikzLengthPk pt, draw, outer sep=0pt, inner sep=0pt, minimum height=0.6cm}]
\node [right] at (0,0) [pkk] (A) {\( \PBlock{k_{i}-1} \)};
\node [buchst] (A) [right=\TikzDistBlocks pt of A] {\( b \)};
\node [pkk] (A) [right=\TikzDistBlocks pt of A] {\( \PBlock{k_{i}-1} \)};
\node [buchst] (Fix1) [right=\TikzDistBlocks pt of A] {\( b \)};
\node [pkk] (A) [right=\TikzDistBlocks pt of Fix1] {\( \PBlock{k_{i}-1} \)};
\foreach \x in {4,...,\TikzCodSeqNkkk}{
	\node [buchst] (A) [right=\TikzDistBlocks pt of A] {\( b \)};
	\node [pkk] (A) [right=\TikzDistBlocks pt of A] {\( \PBlock{k_{i}-1} \)};}
\node [right] at (0,1) [pkkk] (Fix2) {\( \PBlock{k_{i}} \)};
\draw[dashed] (Fix1.south east) ++ ( \TikzDistBlocks/2 pt , 0 ) -- + (0 , 2 ) node[at end, left, yshift=-3]{0\vphantom{1}} node[at end, right, yshift=-3]{1\vphantom{0}};
\draw (Fix2.north east) -- +( 0.7 , 0 );
\draw (Fix2.south east) -- +( 0.7 , 0 );
\node [right=1 of Fix2.east]{\( \InfWord \)};
\node [left=0 of Fix2.west]{\( a_{ k_{i}+1 }\)};
\draw (Fix2.north west) -- +( -1.2 , 0 );
\draw (Fix2.south west) -- +( -1.2 , 0 );
\end{tikzpicture}
\normalsize
\caption{An element \( \InfWord \in \Cylin{ a_{ k_{i}+1 } \, \PBlock{k_{i}} }{ -2l_{i} } \) with \( b \DefAs a_{k_{i}} \). There are three repetitions around the origin of each of the words \( \InfWord , \ldots , \Shift^{ ( n_{k_{i}} - 3 )l_{i} - 1 } \InfWord \).\label{fig:CylinThreeRepe2}}
\end{figure}

Combining Theorem~\ref{thm:AsNoPP} and Theorem~\ref{thm:KotaniNoAC} we obtain the following corollary:

\begin{cor}
Let  \( \Subshift \) be a simple Toeplitz subshift with \( \Card{ \AlphabEv } \geq 2 \). Let \( \NDig \colon \Subshift \to \RR \setminus \Set{ 0 } \) and \( \Dig \colon \Subshift \to \RR \) be locally constant functions such that \( ( \NDig , \Dig ) \) is aperiodic. Then the Jacobi operator \( \Jac_{\InfWord} \) has purely singular continuous spectrum for almost every \( \InfWord \in \Subshift \) with respect to the ergodic measure.
\end{cor}

%%%%%%%%%%%%%%%%%%%%%%%%
%%%%%%%%%%%%%%%%%%%%%%%%
%%%%%%%%%%%%%%%%%%%%%%%%
\chapter[Uniformity of cocycles and Cantor spectrum][Uniformity of cocycles]{Uniformity of cocycles and Cantor spectrum}
\label{chap:UnifCocycCantor}

In the previous chapter we defined a Jacobi operator \( \Jac_{\InfWord} \) for every \( \InfWord \in \Subshift \). We discussed that the singular continuous part and the pure point part of its spectrum may depend on \( \InfWord  \in \Subshift \). This makes it difficult to obtain uniform statements about the spectral type of \( \Jac_{\InfWord} \). However, the spectrum as a set is independent of \( \InfWord\), see Proposition~\ref{prop:SpecIndepOmega}. In this chapter we study conditions under which \( \sigma( \Jac_{\InfWord} ) \) becomes a Cantor set. For this purpose, the asymptotic exponential behaviour of the transfer matrices plays a crucial role. It is known that uniformity of the transfer matrices implies Cantor spectrum of Lebesgue measure zero, see\cite{DamLenz_Boshern} and \cite{BeckPogo_SpectrJacobi}. As main result of this chapter, we prove this uniformity (and thus Cantor spectrum for Jacobi operators) for every simple Toeplitz subshift. In fact, we prove that every locally constant cocycle is uniform if the underlying subshift satisfies some conditions that we call the ``leading sequence conditions \LCond{}'', see Theorem~\ref{thm:LCondUniform}.

First we recall the notion of cocycles, with transfer matrices as main examples, and how their uniformity is related to Cantor spectrum. In Section~\ref{sec:LSCSubsh} we define \LCond{}-subshifts, discuss basic properties and give sufficient conditions for a subshift to satisfy \LCond{}. We explicitly check these conditions for simple Toeplitz subshifts (Subsection~\ref{subsec:LCondST}) and Sturmian subshifts (Subsection~\ref{subsec:LCondSturm}). In Section~\ref{sec:UnifCoycLCond} we prove our main result, namely that every locally constant cocycle over an \LCond{}-subshift is uniform. For Sturmian subshifts this is a well-known property (see \cite[Theorem~4]{Lenz_ErgodTheo1dimSO}, which even treats the quasi-Sturmian case). However, the \LCond{}-property of Sturmian subshifts demonstrates how \LCond{}-subshifts can serve as a unifying framework for different classes of subshifts. Most of the results from Section~\ref{sec:LSCSubsh} to \ref{sec:UnifCoycLCond} are based on joint work with Rostislav Grigorchuk, Daniel Lenz and Tatiana Nagnibeda and can also be found in \cite{GLNS_LeadingSeq_Arxiv}.

%%%%%%%%%%%%%%%%%%%%%%%%
\section{Cocycles and their relation to Cantor spectrum}
\label{sec:CocycCantor}

In Section~\ref{sec:JOTransMat} we associated an elementary transfer matrix \( \TrMat( \InfWord ) \) to every \( \InfWord \in \Subshift \). This was done via locally constant functions \( \NDig \) and \( \Dig \) that relate the product \( \TrMat( j, \InfWord ) = \TrMat( \Shift^{j-1} \InfWord ) \cdot \ldots \cdot \TrMat( \Shift^{0} \InfWord ) \) to solutions of the eigenvalue equation of \( \Jac_{ \InfWord } \). Now we consider the more general notion of cocycles: to every \( \InfWord \in \Subshift \), a matrix is associated in a locally constant way. In analogy to the transfer matrix case (see Equation~(\ref{eqn:TrMatCocyc})), we consider the product of these matrices, the so-called cocycle. While cocycles can be defined for matrices in \( \GL \), we consider only \( \SL \)-matrices to simplify computations.

\begin{defi}
\label{defi:Cocyc}
Let \( \Cocyc \colon \Subshift \to \SL \) be a continuous map. The \emph{cocycle associated to \( \Cocyc \)}\index{cocycle}, which we also denote by \( \Cocyc \), is defined as
\[ \Cocyc \colon \ZZ \times \Subshift \to \SL \; , \;\; ( j , \InfWord ) \mapsto \begin{cases} 
\Cocyc( \Shift^{j-1} \InfWord ) \cdot \ldots \cdot \Cocyc( \Shift^{0} \InfWord ) & \text{for } j > 0 \\
\Id & \text{for } j = 0 \\
\Cocyc( \Shift^{j} \InfWord )^{-1} \cdot \ldots \cdot \Cocyc( \Shift^{-1} \InfWord )^{-1} & \text{for } j < 0
\end{cases} \, . \]
\end{defi}

\begin{defi}
A cocycle \( \Cocyc \colon \ZZ \times \Subshift \to \SL \) is called \emph{locally constant}\index{locally constant!cocycle} if the underlying map \( \Cocyc \colon \Subshift \to \SL \) is locally constant, see Definition~\ref{defi:LocConst}.
\end{defi}

\begin{exmpl}
\label{exmpl:TrMod}
Our main example are transfer matrices, but with the definition from Section~\ref{sec:JOTransMat}, they are in general not in \( \SL \). Thus, we follow \cite{BeckPogo_SpectrJacobi} and define for \( E \in \RR \) the \emph{modified transfer matrix map}\index{modified transfer matrix map}\index{transfer matrix!modified map}
\[ \TrMod \colon \Subshift \to \SL \; , \;\; \InfWord \mapsto \bigg(\begin{matrix} \frac{ E - \Dig( \Shift \InfWord ) }{ \NDig( \Shift^{2} \InfWord ) } & - \frac{ 1 }{ \NDig( \Shift^{2} \InfWord ) } \\ \NDig( \Shift^{2} \InfWord ) & 0 \end{matrix}\bigg) \, . \]
Similar to Equation~(\ref{eqn:TrMatCocyc}), we define the \emph{modified transfer matrices}\index{modified transfer matrix}\index{transfer matrix!modified} as the matrix product that is given by the map
\[ \TrMod \colon \ZZ \times \Subshift \to \SL \; , \;\; ( j , \InfWord ) \mapsto \begin{cases}
\TrMod( \Shift^{j-1} \InfWord ) \cdot \ldots \cdot \TrMod( \Shift^{0} \InfWord ) & \text{for } j > 0 \\
\Id & \text{for } j = 0 \\
\TrMod( \Shift^{j} \InfWord )^{-1} \cdot \ldots \cdot \TrMod( \Shift^{-1} \InfWord )^{-1} & \text{for } j < 0
\end{cases} \, . \]
Clearly \( \TrMod \) is a cocycle. If \( \NDig \) and \( \Dig \) are locally constant, then \( \TrMod \) is locally constant as well. Just as the unmodified transfer matrix, \( \TrMod \) describes solutions of the eigenvalue equation \( \Jac_{\InfWord} \varphi = E \varphi \). If \( \varphi \) is such a solution, then we have
\[ \bigg(\begin{matrix} \varphi(j+1) \\ \NDig( \Shift^{j+1} \InfWord ) \varphi(j) \end{matrix}\bigg) = \TrMod( j, \InfWord ) \, \bigg(\begin{matrix} \varphi(1) \\ \NDig( \Shift \InfWord ) \varphi(0) \end{matrix}\bigg) \, . \qedhere\]
\end{exmpl}

Because of its relation to solution of \( \Jac_{\InfWord} \varphi = E \varphi \), we want to study the asymptotic exponential behaviour of \( \Vert \TrMod( j , \InfWord ) \rVert \). It turns out that this behaviour is the same for almost all \( \InfWord \in\Subshift \). In fact, the same is true for every locally constant cocycle. This follows from the multiplicative ergodic theorem by Furstenberg and Kesten (\cite{FurstenKest_ProdRandMatr}), which is based on Kingman's subadditive ergodic theorem (\cite{Kingm_ErgodTheoSubadd}, \cite[Theorem~1]{King_SubaddErgod}). Below, the theorem is stated in the form that was given in \cite[page 798]{Furman_MultiErgodThm}.

\begin{thm}[multiplicative ergodic theorem]
\label{thm:MET}
Let \( (X, \ErgodMeas, T) \) be an ergodic system. Let \( \Cocyc \colon X \to \GL \) be a measurable function, with both \( \ln ( \lVert \Cocyc(x) \rVert ) \) and  \( \ln( \lVert \Cocyc(x)^{-1} \rVert ) \) in \( L^{1}( \ErgodMeas ) \). Then there exists a constant \( \Lyapu{ \Cocyc } \) such that the convergence
\[ \lim_{j \to \infty} \frac{1}{j} \ln( \lVert \Cocyc( \Shift^{j-1}x ) \cdot \hdots \cdot \Cocyc( x ) \rVert ) = \Lyapu{ \Cocyc } \]
holds for \( \ErgodMeas \)-almost every \( x \in X \) and in \( L^{1}( \ErgodMeas ) \).
\end{thm}

If \( \Cocyc \colon \Subshift \to \SL \) is locally constant, then it is also measurable and takes only finitely many values. Thus \( \ln ( \lVert \Cocyc( \InfWord ) \rVert ) \) and  \( \ln( \lVert \Cocyc( \InfWord )^{-1} \rVert ) \) are bounded and so the previous theorem applies. Note that, even for a uniquely ergodic and minimal system, the limit holds only almost surely and need not exist in every point. Counterexamples can for instance be found in \cite[Proposition~6.4]{Herm_DiffeoPosEntro} or \cite[Theorem~2.2]{Walt_UniErgRandProd}. However, we are especially interested in cases where the limit even holds uniformly.

\begin{defi}[{\cite[Section~4]{Furman_MultiErgodThm}}]
A map \( \Cocyc \colon \Subshift \to \SL \) is called \emph{uniform}\index{uniform!map} if the limit \( \lim_{j \to \infty} \frac{1}{j} \ln( \lVert \Cocyc(j , \InfWord) \rVert ) \) exists for all \( \InfWord \in \Subshift \) and the convergence is uniform in \( \InfWord \).
\end{defi}

\begin{defi}
A cocycle \( \Cocyc \colon \ZZ \times \Subshift \to \SL \) is called \emph{uniform}\index{uniform!cocycle} if the underlying map \( \Cocyc \colon \Subshift \to \SL \) is uniform.
\end{defi}

\begin{rem}
\label{rem:ThmByWeiss}
For minimal systems, the existence of the limit \( \lim_{j \to \infty} \frac{1}{j} \ln( \lVert \Cocyc(j , \InfWord) \rVert ) \) for all \( \InfWord \in \Subshift \) already implies uniformity of \( \Cocyc \colon \Subshift \to \SL \). To see this, we use that \( \ln( \lVert \Cocyc( j , \cdot ) \rVert ) \) as a function from \( \Subshift \) to \( \RR \) satisfies \( \ln( \lVert \Cocyc( k+j , \InfWord ) \rVert ) \leq \ln( \lVert \Cocyc( k , \InfWord ) \rVert ) + \ln( \lVert \Cocyc( j , \Shift^{k} \InfWord ) \rVert ) \) for all \( j , k \in \NN \) and all \( \InfWord \in \Subshift \). The claim then follows from a result by Benjamin Weiss, which was published in \cite[Theorem~1.2]{GLNS_LeadingSeq_Arxiv}. However, in the following we will not make use of this fact. Instead we prove a uniform lower bound for \( \frac{1}{j} \ln( \lVert \Cocyc( j , \InfWord ) \rVert ) \) and use a result by Lenz to show uniformity (see Proposition~\ref{prop:LowBdUnif}).
\end{rem}

For Schrödinger operators it was shown in \cite[Corollary~2.1]{Lenz_SingSpec1dQC} that uniformity of the transfer matrix implies Cantor spectrum. By a result of Beckus and Pogorzelski, this holds for Jacobi operators as well:

\begin{prop}[{\cite[Theorem~3]{BeckPogo_SpectrJacobi}}]
\label{prop:UnifCantorSpec}
Let \( \Subshift \) be a minimal, uniquely ergodic and aperiodic subshift. Consider the continuous maps \( \NDig \colon \Subshift \to \RR \setminus \Set{0} \) and \( \Dig \colon \Subshift \to \RR \) which take finitely many values, and the corresponding family of Jacobi operators \( (\Jac_{\InfWord})_{\InfWord \in \Subshift} \). Suppose that \( ( \NDig , \Dig ) \) is aperiodic and that the transfer matrix \( \TrMat \) is uniform for every \( E \in \RR \). Then the spectrum \( \sigma( \Jac_{\InfWord} ) \) is a Cantor set of Lebesgue measure zero.
\end{prop}

A major tool to establish uniformity of the transfer matrix is the Boshernitzan condition (see Section~\ref{sec:BoshCond} for the definition). The following proposition is obtained in \cite{BeckPogo_SpectrJacobi}, using a result from \cite{DamLenz_Boshern}:

\begin{prop}[{\cite[Corollary~4]{BeckPogo_SpectrJacobi}}]
\label{prop:BoshCantor}
Let \( \Subshift\) be a minimal, aperiodic subshift such that the Boshernitzan condition holds. Consider the family of corresponding Jacobi operators \( \Set{ \Jac_{\InfWord} }_{\InfWord \in \Subshift \,} \), where the continuous maps \( \NDig \) and \( \Dig \) take only finitely many values and \( ( \NDig , \Dig ) \) is aperiodic. Then the transfer matrix map \( \TrMat \colon \Subshift \to \GL \) is uniform for every \( E \in \RR \). In particular, the spectrum \( \sigma( \Jac_{\InfWord} ) \) is a Cantor set of Lebesgue measure zero.
\end{prop}

By combining this result with Proposition~\ref{prop:BoshCondA2}, Proposition~\ref{prop:BoshCond}, Corollary~\ref{cor:BoshCondForA3} and Remark~\ref{rem:LocConstFinVal} we obtain:

\begin{cor}
\label{cor:BoshSTCantor}
Let \( \Subshift \) be a simple Toeplitz subshift. If there exists a sequence \( (k_{r}) \) of natural numbers with \( \lim_{r \to \infty} k_{r} = \infty \) such that \( \prod_{ j = k_{r}+1 }^{\AllL( k_{r}-1 )-1} n_{j} \) is bounded, then every aperiodic Jacobi operator on the subshift has Cantor spectrum of Lebesgue measure zero. In particular this is the case if either \( \Card{\AlphabEv} = 2 \) holds or \( \Card{\AlphabEv} = 3 \) and \( \liminf_{i \to \infty} n_{\AllLInc{i}+1} < \infty \) hold.
\end{cor}

Consequently, there are some simple Toeplitz subshifts on which Jacobi operators exhibit Cantor spectrum, for instance the period doubling subshift (see Example~\ref{exmpl:PDExBosh}). However, there are also simple Toeplitz subshifts for which the Boshernitzan condition does not hold, see Example~\ref{exmpl:BoshExBoshCond}. Nevertheless it could be proved for arbitrary simple Toeplitz subshifts in \cite[Theorem~1.1]{LiuQu_Simple} that the transfer matrices of Schrödinger operators are uniform for all \( E \in \RR \). In Section~\ref{sec:UnifCoycLCond} we generalise this result. We prove that every locally constant cocycle over a so-called \LCond{}-subshift is uniform. In particular this yields uniformity of \( \TrMod \), but to deduce Cantor spectrum we need uniformity of the unmodified transfer matrix \( \TrMat \). To prove this we use the following relation between \( \TrMat \) and \( \TrMod \) that was established in \cite{BeckPogo_SpectrJacobi}: 

\begin{prop}[{\cite[Lemma~5]{BeckPogo_SpectrJacobi}}]
\label{prop:UnifTrMatTrMod}
Let \( \Subshift \) be a minimal and uniquely ergodic subshift and let \( \InfWord \in \Subshift \). Then the limit \( \lim_{j \to \infty} \frac{1}{j} \ln( \lVert \TrMat( j, \InfWord ) \rVert ) \) exists if and only if the limit \( \lim_{j \to \infty} \frac{1}{j} \ln( \lVert \TrMod( j, \InfWord ) \rVert ) \) exists and in this case, they are equal. Moreover, \( \TrMat \) is uniform if and only if \( \TrMod \) is uniform.
\end{prop}

The following two observations (see \cite[Proposition~1.1]{GLNS_LeadingSeq_Arxiv}) will be helpful in our study of the asymptotic behaviour of cocycles. They demonstrate that taking inverses does not change the norm and that omitting a finite piece at the beginning does not change the asymptotic behaviour.
 
\begin{prop}
\label{prop:CocycInvCut}
Let \( \Subshift \) be a subshift and let \( \Cocyc \colon \Subshift \to \SL \) be a locally constant map. Then the following holds:
\begin{tightenumerate}
\item{For every \( j\in \NN_{0} \) and every \( \InfWord \in \Subshift \) we have \( \lVert \Cocyc( -j , \InfWord ) \rVert = \lVert \Cocyc( j , \Shift^{-j} \InfWord ) \rVert \).}
\item{For every \( J \in \ZZ \) there exists a constant \( c(J) > 0 \) with
\[ \ln( \lVert \Cocyc( j-J , \Shift^{J} \InfWord ) \rVert ) - c(J) \leq \ln( \lVert \Cocyc( j , \InfWord ) \rVert ) \leq \ln( \lVert \Cocyc( j-J , \Shift^{J} \InfWord ) \rVert ) + c(J) \]
for all \( \InfWord \in \Subshift \) and all \( j \in \ZZ \).}
\end{tightenumerate}
\end{prop}

\begin{proof}
For every \( B\in \SL \) we have \( \lVert B \rVert = \lVert B^{-1} \rVert \). Part (a) then follows from a short computation:
\[ \lVert \Cocyc( -j , \InfWord ) \rVert = \lVert \Cocyc( \Shift^{-j} \InfWord )^{-1} \cdot \ldots \cdot \Cocyc( \Shift^{-1} \InfWord )^{-1} \rVert = \lVert \Cocyc( \Shift^{-1} \InfWord ) \cdot \ldots \cdot \Cocyc( \Shift^{-j} \InfWord ) \rVert = \lVert \Cocyc( j , \Shift^{-j} \InfWord ) \rVert \, . \]
For part (b) we define \( c(J) \DefAs \sup \Set{ \ln( \lVert \Cocyc( J , \InfWord ) \rVert ) : \InfWord \in \Subshift } \), which is finite since \( \Cocyc \) is locally constant. Moreover, the definition of cocycles yields \( \Cocyc( j , \InfWord ) = \Cocyc( j-J , \Shift^{J} \InfWord ) \cdot \Cocyc( J , \InfWord ) \) for all \( j , J \in \ZZ \), which can be checked by an enumeration of cases. Conceptually this means that the matrix product associated to \( j \) shifts of \( \InfWord \) is split into the product associated to \( J \) shifts of \( \InfWord \) and the product associated to \( j-J \) shifts of \( \Shift^{J} \InfWord \). With \( C \DefAs \Cocyc( j-J , \Shift^{J} \InfWord ) \) and \( D \DefAs \Cocyc( J , \InfWord ) \), the claim follows from the inequality
\[\lVert C \rVert \frac{ 1 }{ \lVert D \rVert } = \lVert C \rVert \frac{ 1 }{ \lVert D^{-1} \rVert } \leq \lVert C D \rVert \leq \lVert C \rVert \lVert D \rVert \, , \]
which holds for all \( C,D \in \SL \). 
\end{proof}

%%%%%%%%%%%%%%%%%%%%%%%%
\section{(LSC)-subshifts}
\label{sec:LSCSubsh}

In the following we introduce \LCond{}-subshifts. Their defining properties are chosen in such a way that cocycles on these subshifts are particularly well-behaved. Since the definition is hard to check directly, we give sufficient combinatorial conditions for \LCond{}. In Section~\ref{sec:LCondExmpl} we will prove that both simple Toeplitz subshifts and Sturmian subshifts satisfy these sufficient conditions. The study of \LCond{}-subshifts stems from a collaboration with Rostislav Grigorchuk, Daniel Lenz and Tatiana Nagnibeda, see \cite{GLNS_LeadingSeq_Arxiv}.

\subsection{Definition and basic properties} 
\label{subsec:DefLCond}

We now introduce the leading sequence condition for subshifts. Essentially it says that there are finitely many elements which contain all \( u \in \Langu{ \Subshift } \) around their origins. In addition cocycles are required to have the same asymptotic exponential behaviour along all these elements. The third condition provides additional control over this behaviour.

\begin{defi}
\label{defi:LCond}
A subshift \( \Subshift \) satisfies the \emph{leading sequence condition}~\LCond{}\index{LSC@(LSC)}\index{leading sequence condition} if there exists a number \( r \in \NN \) and elements \( \LWord{1} , \ldots , \LWord{\LNr} \in \Subshift \) with the following properties:
\begin{tightdescription}
\item[\hypertarget{LSCa}{\LCond{\(\boldsymbol{\ \alpha}\)}}]{There exist integers \( L_{0} , I_{0} \in \NN_{0} \) such that
\[ \Langu{\Subshift}_{L} = \bigcup_{i=1}^{\LNr} \bigcup_{j= -I_{0}-L}^{I_{0}-1} \Restr{ \LWord{i} }{ j+1 }{ j+L } \]
holds for all  \( L \in \NN \) with \( L \geq L_{0} \).}
\item[\hypertarget{LSCb}{\LCond{\(\boldsymbol{\ \beta}\)}}]{For every locally constant map \( A \colon \Subshift \to \SL \), there exists a number \( c \in \RR \) such that for every \( i \in \Set{ 1 , \ldots , \LNr } \) the limits
\[ \lim_{j \to \pm \infty} \frac{ 1 }{ \lvert j \rvert } \ln( \lVert \Cocyc( j , \InfWord^{(i)}) \rVert ) \]
exist and all limits have the value \( c \).}
\item[\hypertarget{LSCc}{\LCond{\(\boldsymbol{\ \gamma}\)}}]{For every locally constant map \( \Cocyc \colon \Subshift \to \SL \), every \( i \in \Set{ 1, \ldots , \LNr } \) and every \( \Phi \in \RR^{2} \setminus \Set{ 0 } \),
at most one of the limits \( \lim_{j \to \pm \infty} \frac{1}{ \lvert j \rvert } \ln( \lVert \Cocyc( j , \LWord{i} ) \Phi \rVert ) \) is negative.}
\end{tightdescription}
Sometimes we call \LConda{} the \emph{combinatorial leading sequence condition}, while \LCondb{} and \LCondc{} are referred to as \emph{cocycle leading sequence condition}. The words \( \LWord{1}, \ldots, \LWord{\LNr} \) are called the \emph{leading sequences}\index{leading sequence} or \emph{leading words}\index{leading word}\index{word!leading} of the subshift.
\end{defi}

\begin{rem}
\label{rem:LNr-I0}
In \LConda{} we require that every \( u \in \Langu{ \Subshift } \) intersects \( \Restr{ \LWord{i} }{ -I_{0} }{ I_{0} } \) for some \( i \in \Set{ 1 , \ldots , \LNr } \). Without loss of generality we may assume \( I_{0} = 0 \):
if \( \LWord{1} , \ldots , \LWord{\LNr} \) are leading sequences with \( I_{0} > 0 \), then it follows easily from Proposition~\ref{prop:CocycInvCut} that \( \Shift^{ -I_{0} } \LWord{1} , \ldots , \Shift^{ I_{0} } \LWord{1} , \ldots , \Shift^{-I_{0}} \LWord{\LNr} , \ldots , \Shift^{I_{0}} \LWord{ \LNr } \) are leading sequences with \( I_{0} = 0 \). However, our general philosophy is that leading sequences are few elements which describe the behaviour of all elements. Thus we want to avoid leading sequences that are ``essentially the same'' and ``contain no new information'' (in the sense that they belong to the same orbit). We will therefore allow positive \( I_{0} \) in specific examples, while we will occasionally set \( I_{0} = 0 \) in proofs.
\end{rem}

It turns out that the seemingly arbitrary leading sequence conditions have rather strong combinatorial implications. Roughly speaking, the frequency of a word can be expressed as the norm of a cocycle, and by \LCondb{} it is therefore the same in all leading sequences. By combining \LConda{} and \LCondc{} we find that every \( u \in \Langu{ \Subshift } \) occurs in some \( \LWord{i} \) infinitely often. The application of \LConda{} to the gaps between two consecutive occurrences of \( u \) shows that \( u \) even has positive frequency. Since the frequencies in all leading sequences agree, every word has positive frequency in every leading sequence, which implies unique ergodicity and minimality. Below we discuss these arguments in detail. We begin with a result on the complexity.

\begin{prop}
\label{prop:LCondCompl}
Let \( \Subshift \) be an \LCond{}-subshift with \( \LNr \) leading sequences. Then its complexity is bounded by \( \Comp( L ) \leq \LNr L + 2 r I_{0} \) for all sufficiently large \( L \in \NN \).
\end{prop}

\begin{proof}
For each of the \( \LNr \) leading sequences there are at most (\( 2 I_{0} + L \))-many words of length \( L \) that intersect \( \Restr{ \LWord{i} }{ -I_{0} }{ I_{0} } \). Now the claim follows immediately from \LConda{}.
\end{proof}

Now we consider recurrence, unique ergodicity and minimality of \LCond{}-subshifts. By \LConda{}, every \( u \in \Langu{ \Subshift } \) occurs in some \( \LWord{i} \). It turns out that \LCondc{} implies that \( u \) has to occur infinitely often in this \( \LWord{i} \). This is known as recurrence.

\begin{defi}
A two-sided infinite word \( \InfWord \in \Alphab^{\ZZ} \) is called \emph{recurrent}\index{recurrent}\index{word!recurrent} if every finite subword of \( \InfWord \) occurs infinitely often in \( \InfWord \).
\end{defi}

\begin{prop}
\label{prop:LCondRecur}
Let \( \Subshift \) be an \LCond{}-subshift. Then every leading word \( \LWord{i} \), \( i \in \Set{ 1 ,\ldots , \LNr } \), is recurrent.
\end{prop}

\begin{proof}
Assume that there exist \( u \in \Langu{ \Subshift } \) and \( i \in \Set{ 1 , \ldots , \LNr } \) such that \( \LWord{i} \) contains \( u \), but only finitely many times. Let \( J \in \NN \) be such that \( \Restr{ \LWord{i} }{ -J }{ J } \) contains all occurrences of \( u \) in \( \LWord{i} \). Clearly this implies that \( \Restr{ \Shift^{j} \LWord{i} }{ -J }{ J } = \Restr{ \LWord{i} }{ -J }{ J } \) holds only for \( j = 0 \). Using this, we show that there is a locally constant cocycle \( \Cocyc \) such that the norm of \( \Cocyc( j , \LWord{i} ) ( \begin{smallmatrix} 1 \\ 0 \end{smallmatrix} ) \) decays exponentially in both directions, which contradicts \LCondc{}. Let \( \Cocyc \) be defined by
\[ \Cocyc \colon \Subshift \to \SL \quad , \;\; \InfWord \mapsto \begin{cases}
\big( \begin{smallmatrix} 0 & -1 \\ 1 & 0 \end{smallmatrix} \big) & \text{if } \Restr{ \InfWord }{ -J }{ J } = \Restr{ \LWord{i} }{ -J }{ J }  \\
\big( \begin{smallmatrix} 2 & 0 \\ 0 & \frac{1}{2} \end{smallmatrix} \big) & \text{if } \Restr{ \InfWord }{ -J }{ J } \neq \Restr{ \LWord{i} }{ -J }{ J } \\
\end{cases} \, .\]
For \( j > 0 \) this yields
\[  \Cocyc( j , \LWord{i} ) \begin{pmatrix} 1 \\ 0 \end{pmatrix}  = \begin{pmatrix} 2 & 0 \\ 0 & \frac{1}{2} \end{pmatrix}^{j-1} \cdot \begin{pmatrix} 0 & -1 \\ 1 & 0 \end{pmatrix} \cdot \begin{pmatrix} 1 \\ 0 \end{pmatrix} =  \begin{pmatrix} 0 \\ \frac{1}{2^{j-1}} \end{pmatrix} \, . \]
Similarly we obtain for \( j < 0 \):
\[  \Cocyc( j , \LWord{i} ) \begin{pmatrix} 1 \\ 0 \end{pmatrix} = \Bigg( \begin{pmatrix} 2 & 0 \\ 0 & \frac{1}{2} \end{pmatrix}^{-1} \Bigg)^{j} \begin{pmatrix} 1 \\ 0 \end{pmatrix} = \begin{pmatrix} \frac{1}{2^{j}} & 0 \\ 0 & 2^{j} \end{pmatrix} \cdot \begin{pmatrix} 1 \\ 0 \end{pmatrix} = \begin{pmatrix} \frac{1}{2^{j}} \\ 0 \end{pmatrix} \, . \qedhere \]
\end{proof}

Now we apply \LConda{} to the gaps between two occurrences of \( u \) in \( \LWord{i} \). This allows us to strengthen the previous result considerably:

\begin{prop}
\label{prop:LCondBdGap}
Let \( \Subshift \) be an \LCond{}-subshift. If \( i \in \Set{ 1 , \ldots , \LNr } \) is such that \( u \in \Langu{ \Subshift } \) occurs in \( \LWord{i} \), then \( u \) occurs with bounded gaps in \( \LWord{i}|_{-\NN_{0}} \) or in \( \LWord{i}|_{\NN} \).
\end{prop}

\begin{proof}
Fix \( u \in \Langu{ \Subshift } \) and let \( i \in \Set{ 1 , \ldots , \LNr } \) be arbitrary. Let \( j^{(i)}_{1} < \ldots < j^{(i)}_{N_{i}} \) denote the first positions of all (possibly overlapping) occurrences of \( u \) that either intersect \( \Restr{ \LWord{i} }{ -I_{0} }{ I_{0} } \) or are the first occurrence of \( u \) to the left or to the right of \( [-I_{0} , \, I_{0} ] \), see Figure~\ref{fig:DefLTilde}. Clearly the set \( \Set{  j^{(i)}_{1} , \ldots , j^{(i)}_{N_{i}} } \) is finite for every \(  i \in \Set{ 1 , \ldots , \LNr } \), but some of these sets might be empty (we will see in Proposition~\ref{prop:LSPosFreq} that this is not the case). However, at least one of the sets contains at least two elements, since \( u \) intersects at least one \( \Restr{ \LWord{i} }{ -I_{0} }{ I_{0} } \) and occurs infinitely often in this \( \LWord{i} \) (Proposition~\ref{prop:LCondRecur}). Therefore
\[ \widetilde{L}_{u} \DefAs \max \big\{ \; j^{(i)}_{ l+1 } -  j^{(i)}_{ l } : i \in \Set{ 1 , \ldots , \LNr } \text{ with } N_{i} \geq 2 \, , \; l \in \Set{ 1 , \ldots , N_{i}-1 } \; \big\} \]
is well-defined. By definition, \( \widetilde{L}_{u} - \Length{u} \) is then the largest possible gap between two occurrences of \( u \), see again Figure~\ref{fig:DefLTilde}. Consequently, no word of the form \( \Word{ u }{ v }{ u } \) with \( \Length{v} > \widetilde{L}_{u} - \Length{u} \) can intersect any \( \Restr{ \LWord{i} }{ -I_{0} }{ I_{0} } \). Since all finite words intersect some \( \Restr{ \LWord{i} }{ -I_{0} }{ I_{0} } \), no word \( \Word{ u }{ v }{ u } \) with \( \Length{v} > \widetilde{L}_{u} - \Length{u} \) exists in \( \Langu{ \Subshift } \). Combined with Proposition~\ref{prop:LCondRecur} this yields the claim.
\end{proof}

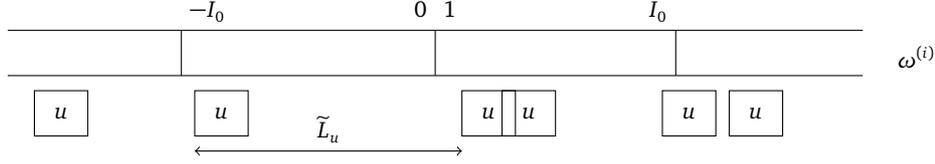
\begin{figure}
\centering
\footnotesize
% set parameters for the image
\pgfmathsetmacro{\TikzLLength}{5} % length of one letter, in pt
\pgfmathsetmacro{\TikzULength}{4*\TikzLLength} % length of u, in pt
\pgfmathsetmacro{\TikzILength}{18*\TikzLLength} % length of I_{0} (one side of the window around the origin), in pt
\begin{tikzpicture}
[uWord/.style={draw, rectangle, outer sep=0pt, inner sep=0pt, minimum height=0.6cm, minimum width=\TikzULength pt}]
\draw ( -160 pt, 0 ) --( 160 pt, 0 );
\draw ( -160 pt, -0.6 ) --( 160 pt, -0.6 );
\node at (180 pt, -0.35){\( \LWord{i}\)};
\draw ( 0 , -0.6 ) -- ( 0 , 0 ) node[at end, above left]{0\vphantom{\( 1 I_{0} \)}} node[at end, above right]{1\vphantom{\( 0 I_{0} \)}};
\draw ( -\TikzILength-\TikzLLength pt , -0.6 ) -- ( -\TikzILength-\TikzLLength pt , 0 ) node[at end, above right]{\( -I_{0} \)\vphantom{\( 0 1 \)}};
\draw ( \TikzILength pt , -0.6 ) -- ( \TikzILength pt , 0 ) node[at end, above left]{\( I_{0} \)\vphantom{\( 0 1 \)}};
\node [right] at (-30*\TikzLLength pt , -1.1) [uWord] {\( u \)};
\node [right] at (-18*\TikzLLength pt , -1.1) [uWord] (A) {\( u \)};
\node [right] at (   2*\TikzLLength pt , -1.1) [uWord] (B) {\( u \)};
\node [right] at (  5*\TikzLLength pt , -1.1) [uWord] {\( u \)};
\node [right] at ( 17*\TikzLLength pt , -1.1) [uWord] {\( u \)};
\node [right] at ( 22*\TikzLLength pt , -1.1) [uWord] {\( u \)};
\draw[<->] ( -18*\TikzLLength pt, -1.6  ) -- ( 2*\TikzLLength pt, -1.6  ) node[midway, above]{\( \widetilde{L}_{u} \)};
\end{tikzpicture}
\normalsize
\caption{Definition of \( \widetilde{L}_{u} \) as the maximal distance between two occurrences of \( u \) that intersect \( \Restr{ \LWord{i} }{ -I_{0} }{ I_{0} } \) or lie next to it.\label{fig:DefLTilde}}
\end{figure}

Now recall from Definition~\ref{defi:FreqAndCo} that \( u \in \Langu{ \Subshift } \) is said to have frequency \( c_{u} \) in \( \RWord \in \Alphab^{\NN} \) if, for every \( j \in \NN \), we have \( c_{u} = \lim_{L \to \infty} \frac{1}{L} \Copies{u}{ \Restr{ \RWord }{ j }{ j+L-1 } } \) and the convergence is uniform in \( j \) (and similar for \( \RWord \in \Alphab^{- \NN_{0}} \) ). The previous result, combined with \LCondb{}, implies that every \( u \in \Langu{ \Subshift } \) has positive frequency in each half of every leading sequence:

\begin{prop}
\label{prop:LSPosFreq}
Let \( \Subshift \) be an \LCond{}-subshift with \( \LNr \) leading sequences. Then for every \( u \in \Langu{ \Subshift } \) there exists a constant \( c_{u} > 0 \) such that for all \( \RWord \in \Set{ \LWord{ 1 }|_{-\NN_{0}} , \ldots , \LWord{ \LNr }|_{-\NN_{0}} , \LWord{ 1 }|_{\NN} , \ldots \), \( \LWord{ \LNr }|_{\NN} } \) the frequency of \( u \) in \( \RWord \) exists and is equal to \( c_{u} \).
\end{prop}

\begin{proof}
Let \( u \in \Langu{ \Subshift } \) be fixed. By Proposition~\ref{prop:LCondBdGap} we assume without loss of generality that \( u \) occurs with bounded gaps in some \( \LWord{m}|_{\NN} \), with \( m \in \Set{ 1 , \ldots , \LNr } \) (the case of the negative half-line is similar). Let \( L_{u} \) be such that for every \( j > 0 \), the word \( u \) occurs at least once in \( \Restr{ \LWord{m} }{ j }{ j +L_{u} - 1} \). This yields
\[ \liminf_{L \to \infty} \frac{1}{L} \cdot \Copies{u}{ \Restr{ \LWord{m} }{ 1 }{ L } } \geq \liminf_{L \to \infty} \frac{1}{L} \cdot \Bigg\lfloor \frac{ L }{ L_{u} } \Bigg\rfloor = \frac{1}{ L_{u} } > 0 \, .\]
Now we consider the locally constant map \( \Cocyc_{u} \colon \Subshift \to \SL \) that is defined by \( \Cocyc_{u}( \InfWord ) \DefAs \left(\begin{smallmatrix} 2  & 0 \\ 0 & \frac{ 1 }{ 2 } \end{smallmatrix} \right) \) if \( \Restr{ \InfWord }{ 1 }{ \Length{u} } = u \) holds, and \( \Cocyc_{u}( \InfWord ) \DefAs \Id \) otherwise. For \( L \geq 1 \) we obtain
\[ \frac{ 1 }{ \ln(2) } \cdot \frac{ \ln ( \lVert \Cocyc_{u}( L , \LWord{m} ) \rVert ) }{ L } = \frac{ \Copies{ u }{ \Restr{ \LWord{m} }{ 1 } {\Length{u}+L-1 } } }{ L } \, .\] 
Let \( c_{u} \in \RR \) denote the limit of this term for \( L \to \infty \), which exists by \LCondb{}. By the estimate above we obtain \( c_{u} \geq \frac{ 1 }{ L_{u} } > 0 \). Again by \LCondb{}, all other limits for \( i \in \Set{ 1 , \ldots , r } \) and \( L \to \pm \infty \) exist as well and are all equal to \( c_{u} \). Note in particular that for \( L \leq -1 \) the term
\[ \frac{ 1 }{ \ln(2) } \cdot \frac{ \ln ( \lVert \Cocyc_{u}( L , \LWord{i} ) \rVert ) }{ -L } = \frac{ \Copies{ u }{ \Restr{ \LWord{i} }{ L+1 }{ \Length{u}-1 } } }{ -L } \]
tends to \( c_{u} \) as \( L \) tends to minus infinity. Next we prove the uniform convergence of \( \frac{1}{L} \Copies{u}{ \Restr{ \LWord{1} }{ j }{ j+L-1 } } \) with respect to the initial position \( j \in \NN \) (all other limits are similar). The idea is that every sufficiently long word \( \Restr{ \LWord{1} }{ j }{ j+L-1 } \) also occurs close to the origin. Either the part left of the origin, right of the origin, or both are long enough that we can apply the considerations from above. If one part is short, then dividing by \( L \) will let this term tend to zero. Here are the details:

Let \( \varepsilon > 0 \). Let \( L_{0} \) denote the minimum length that is required in \LConda{}. Let \( L_{1} \) be such that for all \( L \geq L_{1} \), all \( i \in \Set{ 1 , \ldots , \LNr } \) and all \( k \in \Set{ -I_{0} -\Length{u}+1 , \ldots , I_{0} } \) we have
\[ \Big\lvert \frac{ \Copies{ u }{ \Restr{ ( \Shift^{k} \LWord{i} ) }{ -L+1 }{ \Length{u}-1 } } }{ L } - c_{u} \Big\rvert < \varepsilon \qquad \text{and} \qquad \Big\lvert \frac{ \Copies{ u }{ \Restr{ ( \Shift^{k} \LWord{i} ) }{ 1 } {\Length{u}+L-1 } } }{ L } - c_{u} \Big\rvert < \varepsilon \, .\]
Let \( L_{2}, L_{3} \in \NN \) be large enough that
\[ \max \Big( \bigcup_{i=1}^{\LNr} \Big\{ \; \frac{ \Copies{ u }{ \Restr{ \LWord{i} }{ 2-L_{1} }{ \Length{u}-1 } } }{ L_{2} } \, , \, \frac{ \Copies{ u }{ \Restr{ \LWord{i} }{ 1 }{ L_{1}+\Length{u}-2 } } }{ L_{2} }  \; \Big\} \Big) < \varepsilon \]
and \( \frac{ \Length{u} + L_{1} -1 }{ L_{3} } < \varepsilon \) hold. Now consider \( L \geq \max \Set{ L_{0} , 2 L_{1}+\Length{u}-1, L_{2} , L_{3} } \).

Let \( j \in \NN \) be arbitrary. By \LConda{} there exist \( i \in \Set{ 1 , \ldots , \LNr } \) and \( k \in \Set{ -I_{0} -L+1 , \ldots , I_{0} } \) such that \( \Restr{ \LWord{1} }{ j }{ j+L-1 } = \Restr{ \LWord{i} }{ k }{ k+L-1 } \) holds. For \( k \geq 1 \) we have
\[ \frac{ \Copies{ u }{ \Restr{ \LWord{i} }{ k }{ k+L-1 } } }{ L } = \frac{ L - \Length{u} + 1 }{ L } \cdot \frac{ \Copies{ u }{ \Restr{ ( \Shift^{k-1} \LWord{i} ) }{ 1 }{ \Length{u} + ( L - \Length{u} + 1) - 1 } } }{ L - \Length{u} + 1 } \, . \]
By assumption \( I_{0}-1 \geq k-1 \geq 0 \) and \( L-\Length{u}+1 \geq L_{1} \) hold. Consequently, the second factor is \( \varepsilon \)-close to \( c_{u} \). Moreover, the first factor is \( \varepsilon \)-close to \( 1 \) because of \( L \geq L_{3} \). Similarly we get for \( k \leq \Length{u} - L \):
\[ \frac{ \Copies{ u }{ \Restr{ \LWord{i} }{ k }{ k+L-1 } } }{ L } = \frac{ L-\Length{u}+1 }{ L } \cdot \frac{ \Copies{ u }{ \Restr{ ( \Shift^{ k+L-\Length{u} } \LWord{i} ) }{ - ( L - \Length{u} + 1 )+1 }{ \Length{u}-1 } } }{ L-\Length{u}+1 }\, , \]
where the second factor is again \( \varepsilon \)-close to \( c_{u} \) and the first factor is \( \varepsilon \)-close to \( 1 \). Now only \( \Length{u}-L+1 \leq k \leq 0 \) remains. We split the word in a part left and a part right of the origin and reduce our considerations to the two previous cases:
\[ \frac{ \Copies{ u }{ \Restr{ \LWord{i} }{ k }{ k+L-1 } } }{ L } = \frac{ -k+1 }{ L } \cdot \frac{ \Copies{ u }{ \Restr{ \LWord{i} }{ k }{ \Length{u}-1 } } }{ -k+1 } + \frac{ k+L-\Length{u} }{ L } \cdot \frac{ \Copies{ u }{ \Restr{ \LWord{i} }{ 1 }{ k+L-1 } } }{ k+L-\Length{u} } \, . \]
If either \(  -k+1 \) or \( k+L-\Length{u} \) is less than \( L_{1} \), then the respective summand is less than \( \varepsilon \) because of \( L \geq L_{2} \). Since \( L \geq 2 L_{1}+\Length{u}-1 \) implies that the other term has to be at least \( L_{1} \), the respective frequency is \( \varepsilon \)-close to \( c_{u} \), while its coefficient is  \( \varepsilon \)-close to \( 1 \). If \(  -k+1 \) and \( k+L-\Length{u} \) are both greater or equal to \( L_{1} \), then
\[ \frac{ -k+1 }{ L } \cdot \frac{ \Copies{ u }{ \Restr{ \LWord{i} }{ k }{ \Length{u}-1 } } }{ -k+1 } + \frac{ k+L-\Length{u} }{ L } \cdot \frac{ \Copies{ u }{ \Restr{ \LWord{i} }{ 1 }{ k+L-1 } } }{ k+L-\Length{u} } + \frac{ \Length{u}-1 }{ L } \cdot c_{u} \]
is a weighted average over three terms that are \( \varepsilon \)-close to \( c_{u} \). Since \( L \geq L_{3} \) implies that the third summand is less than \( \varepsilon \), the quotient \( \frac{ \Copies{ u }{ \Restr{ \LWord{i} }{ k }{ k+L-1 } } }{ L } \) differs at most by \( 2 \varepsilon \) from \( c_{u} \).
\end{proof}

\begin{cor}
\label{cor:LSCLangu}
The language of an \LCond{}-subshift \( \Subshift \) is given by
\[ \Langu{ \Subshift } = \Langu{ \LWord{ 1 }|_{-\NN_{0}} } = \ldots = \Langu{ \LWord{ \LNr }|_{-\NN_{0}} } = \Langu{ \LWord{ 1 }|_{\NN} } = \ldots = \Langu{ \LWord{ \LNr }|_{\NN} } \, .\]
\end{cor}
 
\begin{proof}
This follows immediately from Proposition~\ref{prop:LSPosFreq} and the definition of \( \Langu{ \Subshift } \).
\end{proof}

As mentioned after Definition~\ref{defi:FreqAndCo}, existence of frequencies is closely related to unique ergodicity. In addition, positivity of frequencies is related to minimality: a subshift \( \Subshift( \RWord ) \) defined by the language of \( \RWord \in \Alphab^{\NN}\) is minimal and uniquely ergodic if and only if, for every \( u \in \Langu{ \RWord } \), the frequency of \( u \) in \( \RWord \) exists and is positive (\cite[Example~1.4]{Dam_GordonInBook}, but see also \cite[Corollary~IV.14.]{Queff_SubstDynSys} and \cite[Proposition~4.4]{BaakeGrimm_Aperio}). By Corollary~\ref{cor:LSCLangu} and Proposition~\ref{prop:LSPosFreq} this applies in particular to \LCond{}-subshifts:

\begin{cor}
\label{cor:LSCMinUniErgo}
Every \LCond{}-subshift is minimal and uniquely ergodic.
\end{cor}

\subsection{A sufficient combinatorial condition for (LSC)}
\label{subsec:LCondComb}

In this subsection we relate cocycles to combinatorial properties of words. For this we use subadditive functions. Therefore we first recall this notion and discuss two examples. They will later play a role in our proofs, see for example in Proposition~\ref{prop:LCondSuffi}.

\begin{defi}
A function \( F \colon \Langu{ \Subshift } \to \RR \) is called a \emph{subadditive}\index{subadditive function} if \( F( \Word{ u }{ v } ) \leq F( u ) + F( v ) \) holds for all \( u, v \in \Langu{ \Subshift } \) with \( \Word{ u }{ v } \in \Langu{ \Subshift} \).
\end{defi}

\begin{rem}
For the limit of a subadditive function \( F \colon \Langu{ \Subshift } \to \RR \) we write \( \AsAv{F} \DefAs \lim_{L \to \infty} \max \Set{ \frac{ F(u) }{ L } : u \in \Langu{ \Subshift }_{L} } \). It exists and is equal to \( \inf_{L \in \NN} \max \Set{ \frac{ F(u) }{ L } : u \in \Langu{ \Subshift }_{L} } \) by Fekete's Lemma, since \( \widetilde{F} \colon \NN \to \RR \, , \; L \mapsto \max \Set{ F(u) : u \in \Langu{ \Subshift }_{L} } \) is subadditive as well.
\end{rem}

\begin{exmpl}
\label{exmpl:CopiesSubAdd}
Fix a word \( u \in \Langu{ \Subshift } \) and recall that \( \DisCopies{u}{v} \) denotes the maximal number of non-overlapping copies of \( u \) in \( v \) (Definition~\ref{defi:FreqAndCo}). The function
\[ F \colon \Langu{ \Subshift } \to \RR \; , \quad v \mapsto - \DisCopies{u}{v} \]
is subadditive: for \( v_{1}, v_{2} \in \Langu{ \Subshift } \) with \( \Word{ v_{1} }{ v_{2} } \in \Langu{ \Subshift } \), the non-overlapping copies of \( u \) in \( \Word{ v_{1} }{ v_{2} } \) are at least all the copies of \( u \) in \( v_{1} \), plus all the copies of \( u \) in \( v_{2} \).
\end{exmpl}

\begin{exmpl}
\label{exmpl:MaxCocycSubadd}
Let \( \Cocyc \) be a locally constant cocycle. The function
\[ F \colon \Langu{ \Subshift } \to \RR \; , \quad v \mapsto \max \Set{ \ln( \lVert \Cocyc( \Length{v} , \InfWord ) \rVert ) : \InfWord \in \Subshift \text{ with } \Restr{ \InfWord }{ 1 }{ \Length{v} } = v } \]
is subadditive: for \( v_{1}, v_{2} \in \Langu{ \Subshift } \) with \( \Word{ v_{1} }{ v_{2} } \in \Langu{ \Subshift } \), we have
\begin{align*}
F( v_{1} v_{2} ) &= \max_{\Restr{ \InfWord }{ 1 }{ \Length{v_{1}v_{2} } } = v_{1} v_{2} } \hspace{-1.4em} \Set{ \ln( \lVert \Cocyc( \Length{\Word{ v_{1} }{ v_{2} }} , \InfWord ) \rVert ) } \\
&\leq \max_{\Restr{ \InfWord }{ 1 }{ \Length{v_{1}v_{2} } } = v_{1} v_{2} } \hspace{-1.2em} \Set{ \ln( \lVert \Cocyc( \Shift^{\Length{v_{1}} + \Length{v_{2}} -  1} \InfWord ) \cdot \ldots \cdot \Cocyc( \Shift^{\Length{v_{1}}} \InfWord ) \rVert \! \cdot\! \lVert \Cocyc( \Shift^{\Length{v_{1}} -  1} \InfWord ) \cdot \ldots \cdot \Cocyc( \Shift^{0} \InfWord ) \rVert ) } \\
&\leq \max_{\Restr{ \InfWord }{ \Length{v_{1}}+1 }{ \Length{v_{1}v_{2} } } = v_{2} } \hspace{-1.6em} \Set{ \ln( \lVert \Cocyc( \Shift^{\Length{v_{1}} + \Length{v_{2}} -  1} \InfWord ) \cdot \ldots \cdot \Cocyc( \Shift^{\Length{v_{1}}} \InfWord ) \rVert ) } \\
& \quad \, + \max_{\Restr{ \InfWord }{ 1 }{ \Length{v_{1} } } = v_{1} } \hspace{-0.5em} \Set{ \ln( \lVert \Cocyc( \Shift^{\Length{v_{1}} -  1} \InfWord ) \cdot \ldots \cdot \Cocyc( \Shift^{0} \InfWord ) \rVert ) } \\
&= F( v_{2} ) + F( v_{1} ) \, . \hspace{27em} \blacklozenge
\end{align*}%
\renewcommand{\qedsymbol}{}% omit end symbol
\end{exmpl}

\vspace{-2\baselineskip}

In the following we prove sufficient combinatorial conditions for \LCondb{} and \LCondc{}. We make the simplifying assumption that there exists a one-sided infinite word \( \RWord \in \Alphab^{\NN} \) and finite words \( \LOrigin{1} , \ldots , \LOrigin{\LNr} \) such that the leading sequences are given by \( \LWord{i} = \Word{ \Rev{ \RWord } }{ \LOrigin{i} }{ \RWord } \). Here \( \Rev{ \RWord } \in \Alphab^{-\NN} \) denotes the \emph{reflection of \( \RWord \)}\index{reflection!of an infinite word}\index{word!reflection of an infinite} and is defined as \( \Rev{ \RWord }( -j ) = \RWord( j ) \). Accordingly, we now focus on limits of one-sided infinite words. First we recall a condition from Lenz, known as \emph{uniform positivity of quasiweights (PQ)}\index{PQ@(PQ)}\index{uniform positivity of quasiweights}:

\begin{defi}[{\cite[Section~1]{Lenz_UETfinAlphab}}]
\label{def:PQ} 
We say that \( \RWord \in \Alphab^{\NN} \) satisfies (PQ) if there exists a constant \( c > 0 \) such that the inequality 
\begin{equation}
\label{eqn:DefiPQ}
\tag{PQ}
\liminf_{L \to \infty} \; \frac{ j }{ L } \cdot \DisCopies{ \Restr{ \RWord }{ 1 }{ j } }{ \Restr{ \RWord }{ 1 }{ L } } \geq c
\end{equation}
holds for every prefix \( \Restr{ \RWord }{ 1 }{ j } \) of \( \RWord \). 
\end{defi}

The following two propositions show that (\ref{eqn:DefiPQ}) is closely related to limits of arbitrary subadditive functions. Both statements are variations of results in \cite[Section~3]{Lenz_UETfinAlphab} and the proofs are inspired by methods from \cite{Lenz_UETfinAlphab} as well.

\begin{prop}
\label{prop:PQPrefixLimit}
Let \( \Subshift \) be a subshift, \( \InfWord \in \Subshift \) an element and \( \RWord \DefAs \Restr{ \InfWord }{ 1 }{ \infty } \). Let \( F \colon \Langu{ \Subshift } \to \RR \) be a subadditive function and assume that \( \RWord \) satisfies (\ref{eqn:DefiPQ}). If there exists a sequence \( ( \PBlock{k} )_{k} \) of prefixes of \( \RWord \) with \( \lim_{ k \to \infty } \Length{ \PBlock{ k } } = \infty \) and 
\[ \lim_{k \to \infty} \frac{ F( \PBlock{k} ) }{ \Length{ \PBlock{k} } }  = \lim_{L \to \infty} \max \Big\{ \; \frac{ F(u) }{ L } : u \in \Langu{ \Subshift }_{L} \; \Big\} \AsDef \AsAv{F} \, , \]
then the limit \( \lim_{L \to \infty} \frac{ F( \Restr{ \RWord }{ 1 }{ L } ) }{ L } \) exists (and is obviously equal to \( \AsAv{F} \)).
\end{prop}

\begin{proof}
Since \( \frac{ F( \Restr{ \RWord }{ 1 }{ L } ) }{ L } \leq \max \Set{ \frac{ F(u) }{ L } : u \in \Langu{ \Subshift }_{L} } \) holds for every \( L \in \NN \), we obtain
\[ \limsup_{ L \to \infty } \frac{ F( \Restr{ \RWord }{ 1 }{ L } ) }{ L } \leq \limsup_{L \to \infty} \max_{ u \in \Langu{ \Subshift }_{L} } \frac{ F(u) }{ L } = \AsAv{F} \, . \]
To prove \( \liminf_{ L \to \infty } \frac{ F( \Restr{ \RWord }{ 1 }{ L } ) }{ L } \geq \AsAv{F} \), we assume the contrary, that is, we assume that there exist \( \delta > 0 \) and a sequence \( ( u_{k} ) \) of prefixes of \( \RWord \) with \( \Length{ u_{k} } \to \infty \), such that \( \frac{ F( u_{k} ) }{ \Length{ u_{k} } } \leq \AsAv{F} - \delta \) holds for all \( k \in \NN \). Since \( \RWord \) satisfies (\ref{eqn:DefiPQ}), there is a constant \( c > 0 \) such that
\[ \liminf_{L \to \infty} \; \frac{ \DisCopies{ u_{k} }{ \Restr{ \RWord }{ 1 }{ L } } }{ L } \geq c \frac{ 1 }{ \Length{u_{k}} } \]
holds for every \( k \in \NN \). In particular, for every \( k \) there is a constant \( L_{k} \) such that for every \( L \geq L_{k} \), the word \( \Restr{ \RWord }{ 1 }{ L } \) contains at least \( \frac{ 3c }{ 4 } \frac{ L }{ \Length{ u_{k} } } \geq 6 \) disjoint copies of \( u_{k} \). Therefore we can write \( \Restr{ \RWord }{ 1 }{ L } = \Word{ x_{1} }{ u_{k} }{ x_{2} }{ u_{k} }{ \ldots}{ u_{k} }{ x_{s+1} } \)
with (possibly empty) finite words \( x_{1} , \ldots , x_{s+1} \) and \( s \geq \frac{ 3c }{ 4 } \frac{ L }{ \Length{ u_{k} } } \). By combining every other copy of \( u_{k} \) with its neighbouring \( x_{i} \)'s, we obtain \(\Restr{ \RWord }{ 1 }{ L } = \Word{ y_{1} }{ u_{k} }{ y_{2} }{ u_{k} }{ \ldots}{ u_{k} }{ y_{t+1} } \) with \( t \geq \frac{ c }{ 4 } \frac{ L }{ \Length{ u_{k} } } \) and \( \Length{y_{i}} \geq \Length{u_{k}} \) for all \( i \in \Set{ 1 , \ldots , t+1 } \). Now we choose \( \varepsilon \) with \( 0 < \varepsilon < \frac{ c }{ 8 } \delta \). Moreover, we fix a sufficiently large \( k \) such that we have
\[ \frac{ F( u ) }{ L_{0} } \leq \lim_{L \to \infty} \max_{  u \in \Langu{ \Subshift }_{L} } \frac{ F(u) }{ L } + \varepsilon = \AsAv{F} + \varepsilon \]
for all \( L_{0} \geq \Length{ u_{k} } \) and all \( u \in \Langu{ \Subshift }_{L_{0}} \). For all \( L \geq L_{k} \), subadditivity of \( F \) now yields
\begin{align*}
\frac{ F( \Restr{ \RWord }{ 1 }{ L } ) }{ L } & \leq \frac{ t \cdot \Length{u_{k}} }{ L } \frac{ F( u_{k} ) }{ \Length{u_{k}} } + \sum_{i=1}^{t+1} \frac{ \Length{y_{i}} }{ L } \frac{ F( y_{i} )}{ \Length{y_{i}} } \\
& \leq \frac{ t \cdot \Length{u_{k}} }{ L } ( \AsAv{F} - \delta ) + \sum_{i=1}^{t+1} \frac{ \Length{y_{i}} }{ L } ( \AsAv{F} + \varepsilon ) \\
& = \AsAv{F} - t \cdot \frac{ \Length{u_{k}} }{ L } \cdot \delta + \big( \sum_{i=1}^{t+1} \frac{ \Length{y_{i}} }{ L } \big) \cdot \varepsilon \; , \;\; \text{since } t \cdot \Length{u_{k}} + \sum_{i=1}^{t+1} \Length{y_{i}} = L \\
& < \AsAv{F} - \frac{ c }{ 8 } \delta \; , \;\; \text{since } t \geq \frac{ c }{ 4 } \frac{ L }{ \Length{ u_{k} } } \text{ and } \varepsilon < \frac{ c }{ 8 } \delta \, .
\end{align*}
Clearly this contradicts \( \lim_{k \to \infty} \frac{ F( \PBlock{k} ) }{ \Length{ \PBlock{k} } } = \AsAv{F} \).
\end{proof}

For our study of \LCond{}-subshifts the proposition above suffices. Since the converse statement highlights the close connection between (\ref{eqn:DefiPQ}) and limits of subadditive functions, we include it here without proof. A proof can be found in \cite[Lemma~4.1~(b)]{GLNS_LeadingSeq_Arxiv}. 

\begin{prop}
\label{prop:PrefixLimitPQ}
Let \( \Subshift \) be a subshift, \( \InfWord \in \Subshift \) an element and \( \RWord \DefAs \Restr{ \InfWord }{ 1 }{ \infty } \). If the limit \( \lim_{L \to \infty} \frac{ F( \Restr{ \RWord }{ 1 }{ L } ) }{ L } \) exists for every subadditive function \( F \colon \Langu{ \Subshift } \to \RR \), then \( \RWord \) satisfies (\ref{eqn:DefiPQ}).
\end{prop}

Finally we also prove a sufficient combinatorial condition for \LCondc{}. Similar to what we have seen in Proposition~\ref{prop:ThreeGordon}, we use a three-block Gordon argument to show that repetitions prevent the norm from decaying in both directions.

\begin{prop}
\label{prop:GordonCocyc}
Let \( \Subshift \) be a subshift and \( \InfWord \in \Subshift \). 
If there exists a sequence \( ( l_{i} ) \) of natural numbers with \( \lim_{i \to \infty} l_{i} = \infty \) such that for all \( i \in \NN \) we have
\[ \Restr{\InfWord}{-2l_{i}+1}{-l_{i}} = \Restr{\InfWord}{-l_{i}+1}{0} = \Restr{\InfWord}{1}{l_{i}} \qquad \text{or} \qquad \Restr{\InfWord}{-l_{i}+1}{0} = \Restr{\InfWord}{1}{l_{i}} = \Restr{\InfWord}{l_{i}+1}{2 l_{i}} \; , \]
then for every locally constant \( \Cocyc \colon \Subshift \to \SL \) and every \( \Phi \in \RR^{2} \setminus \Set{ 0 } \), the norm \( \lVert \Cocyc( j , \InfWord) \Phi \rVert \) tends to zero for at most one of the limits \( j \to \pm \infty \). In particular, there cannot be exponential decay in both directions, that is, \LCondc{} holds for \( \InfWord \).
\end{prop}

\begin{proof}
We only treat the case \( \Restr{\InfWord}{-2l_{i}+1}{-l_{i}} = \Restr{\InfWord}{-l_{i}+1}{0} = \Restr{\InfWord}{1}{l_{i}} \) since the second case is similar. In its simplest form we have already seen the Gordon argument in Proposition~\ref{prop:ThreeGordon}, and the reader is encouraged to reread our earlier sketch since it highlights the main ideas without going into technical details. Here we prove the Gordon argument in full generality for locally constant cocycles: for the underlying map \( \Cocyc \) there exists a number \( J \) such that \( \Cocyc( \InfWord ) \) only depends on \( \Restr{\InfWord}{-J}{J} \). Without loss of generality we assume that all \( l_{i} \) are larger than \( 2 J \). Since \( \Cocyc \) depends on \( \Restr{\InfWord}{-J}{J} \), the values of \( \Cocyc( l_{i} , \Shift^{-2l_{i}+1} \InfWord) \),  \( \Cocyc( l_{i} , \Shift^{-l_{i}+1} \InfWord) \) and \( \Cocyc( l_{i} , \Shift^{1} \InfWord) \) are in general not equal. However, we have
\[ \Cocyc( l_{i}-J , \Shift^{-2l_{i}+J+1} \InfWord ) = \Cocyc( l_{i}-J , \Shift^{-l_{i}+J+1} \InfWord ) \quad \text{and} \quad \Cocyc( l_{i}-J , \Shift^{-l_{i}+1} \InfWord ) = \Cocyc( l_{i}-J , \Shift^{1} \InfWord ) \, ,\]
see Figure~\ref{fig:3EpsGord}. With \( \varrho \DefAs \Shift^{J} \InfWord \) we obtain
\[ \Cocyc( l_{i} , \Shift^{-2l_{i}+1} \varrho) = \Cocyc( l_{i} , \Shift^{-l_{i}+1} \varrho) \qquad \text{and} \qquad \Cocyc( l_{i}-2J , \Shift^{-l_{i}+1} \varrho) = \Cocyc( l_{i}-2J , \Shift^{1} \varrho) \, .\]
For \( l_{i} \to \infty \) this can be seen as a \emph{three-minus-epsilon repetition}\index{three@(three minus epsilon)-block Gordon argument}\index{Gordon argument!three minus@(three minus epsilon)-block} of \( \varrho \). In the following we show by a Gordon-type argument that this prevents decay for \( \varrho \), which then implies the claim for \( \InfWord \).

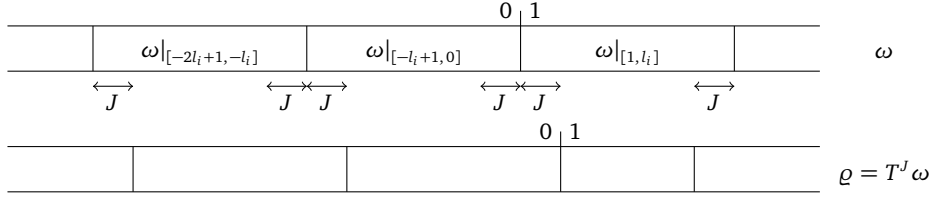
\begin{figure}
\centering
\footnotesize
% set parameters for the image
\pgfmathsetmacro{\TikzWLength}{80} % length l_{i} of one of the repeated words, in pt
\pgfmathsetmacro{\TikzRLength}{0.4*\TikzWLength} % length of rest of element that is displayed, in pt
\pgfmathsetmacro{\TikzJLength}{15} % length of the J letters (on one side) that the cocycle depends on
\begin{tikzpicture}
% \InfWord
\draw ( - \TikzWLength - \TikzRLength pt, 0 ) --( 2*\TikzWLength + \TikzRLength pt, 0 );
\draw ( - \TikzWLength - \TikzRLength pt, -0.6 ) --( 2*\TikzWLength + \TikzRLength pt, -0.6 );
\draw ( \TikzWLength pt , -0.6 ) -- ( \TikzWLength pt , 0.2 ) node[at end, left]{0\vphantom{1}} node[at end, right]{1\vphantom{0}};
\foreach \x in {-1,0,2}{
	\draw ( \x*\TikzWLength pt, -0.6 ) -- ( \x*\TikzWLength pt, 0 );}
\foreach \x in {0,1,2}{
	\draw[<->] ( -\TikzWLength + \x*\TikzWLength pt, -0.8  ) -- ( -\TikzWLength +\TikzJLength + \x*\TikzWLength pt, -0.8  ) node[midway, below]{\( J \)};
	\draw[<->] ( \x*\TikzWLength pt, -0.8  ) -- ( -\TikzJLength + \x*\TikzWLength pt, -0.8  ) node[midway, below]{\( J \)};	
	}
\node at ( -0.5*\TikzWLength pt, -0.35) {\( \Restr{\InfWord}{-2l_{i}+1}{-l_{i}} \)};	
\node at ( 0.5*\TikzWLength pt, -0.35) {\( \Restr{\InfWord}{-l_{i}+1}{0} \)};
\node at ( 1.5*\TikzWLength pt, -0.35) {\( \Restr{\InfWord}{1}{l_{i}} \)};
\node at (2.7*\TikzWLength pt, -0.35){\( \InfWord \)};
% \varrho
\draw ( - \TikzWLength - \TikzRLength pt, -1.6 ) --( 2*\TikzWLength + \TikzRLength pt, -1.6 );
\draw ( - \TikzWLength - \TikzRLength pt, -2.2 ) --( 2*\TikzWLength + \TikzRLength pt, -2.2 );
\draw ( \TikzJLength pt, -2.2 ) -- ( \TikzJLength pt, -1.6 );
\draw ( \TikzJLength + \TikzWLength pt, -2.2 ) -- ( \TikzJLength + \TikzWLength pt, -1.4 ) node[at end, left]{0\vphantom{1}} node[at end, right]{1\vphantom{0}};
\draw ( -\TikzWLength + \TikzJLength pt, -2.2 ) -- ( -\TikzWLength +\TikzJLength pt, -1.6 );
\draw ( 2*\TikzWLength - \TikzJLength pt, -2.2 ) -- ( 2*\TikzWLength - \TikzJLength pt, -1.6 );
\node at (2.7*\TikzWLength pt, -1.95){\( \varrho = \Shift^{J} \InfWord \)};
\end{tikzpicture}
\normalsize
\caption{Threefold repetition in \( \InfWord \) and three-minus-epsilon repetition in \( \varrho \).\label{fig:3EpsGord}}
\end{figure}

It is easy to see that \( \Cocyc( l_{i} , \Shift^{-l_{i}+1} \varrho ) \) and \( \Cocyc( l_{i} , \Shift^{1} \varrho ) \) differ by \( 4J \) factors:
\begin{align*}
\Cocyc( l_{i} , \Shift^{-l_{i}+1} \varrho ) &= \Cocyc( \Shift^{0} \varrho ) \cdot \ldots \cdot \Cocyc( \Shift^{-2J+1} \varrho ) \cdot \Cocyc( l_{i}-2J, \Shift^{-l_{i}+1} \varrho ) \\
&= \Cocyc( \Shift^{0} \varrho ) \cdot \ldots \cdot \Cocyc( \Shift^{-2J+1} \varrho ) \cdot \Cocyc( l_{i}-2J, \Shift^{1} \varrho ) \\
&= \underbrace{\Cocyc( \Shift^{0} \varrho ) \cdot \ldots \cdot \Cocyc( \Shift^{-2J+1} \varrho )\cdot \Cocyc( \Shift^{l_{i}-2J+1} \varrho )^{-1} \cdot \ldots \cdot \Cocyc( \Shift^{l_{i}} \varrho )^{-1} }_{ \AsDef \, B} \\*
&\quad \cdot \Cocyc( \Shift^{l_{i}} \varrho ) \cdot \ldots \cdot \Cocyc( \Shift^{l_{i}-2J+1} \varrho ) \cdot \Cocyc( l_{i}-2J, \Shift^{1} \varrho ) \\
&= B \cdot \Cocyc( l_{i} , \Shift^{1} \varrho ) \, .
\end{align*}
Since \( \Cocyc \colon \Subshift \to \SL \) is locally constant, we have \( 1 \leq \sup_{ \InfWord \in \Subshift } \lVert \Cocyc( \InfWord ) \rVert \AsDef S < \infty \). Now \( \lVert A( \cdot ) \rVert = \lVert A( \cdot )^{-1} \rVert \) implies \( 1 \leq \lVert B \rVert \leq S^{4J} \). The Cayley/Hamilton theorem yields
\begin{align*}
0 &= \Cocyc( l_{i} , \Shift^{-l_{i}+1} \varrho )^{2} - \Tr( \Cocyc( l_{i} , \Shift^{-l_{i}+1} \varrho ) ) \, \Cocyc( l_{i} , \Shift^{-l_{i}+1} \varrho ) + \Id \\
&= \Cocyc( 2 l_{i} , \Shift^{-2l_{i}+1} \varrho) - \Tr( \Cocyc( l_{i} , \Shift^{-l_{i}+1} \varrho ) ) \, \Cocyc( l_{i} , \Shift^{-l_{i}+1} \varrho ) + \Id \, ,
\end{align*}
as well as
\begin{align*}
0 &= \Cocyc( l_{i} , \Shift^{-l_{i}+1} \varrho )^{2} - \Tr( \Cocyc( l_{i} , \Shift^{-l_{i}+1} \varrho ) ) \, \Cocyc( l_{i} , \Shift^{-l_{i}+1} \varrho ) + \Id \\
&= B \cdot \Cocyc( 2 l_{i} , \Shift^{-l_{i}+1} \varrho) - \Tr( \Cocyc( l_{i} , \Shift^{-l_{i}+1} \varrho ) ) \, \Cocyc( l_{i} , \Shift^{-l_{i}+1} \varrho ) + \Id \, .
\end{align*}
For \( \Phi_{0} \in \RR^{2} \setminus \Set{ 0 } \) and \( j \in \ZZ \) we now define \( \Phi_{j} \DefAs \Cocyc( j , \Shift^{1} \varrho) \Phi_{0} \). After multiplying the equalities above by \( \Phi_{-2l_{i}} \) and \( \Phi_{-l_{i}} \) respectively, we obtain
\begin{align}
0 &= \Phi_{0} - \Tr( \Cocyc( l_{i} , \Shift^{-l_{i}+1} \varrho ) ) \, \Phi_{-l_{i}} + \Phi_{-2l_{i}} \label{eqn:3EpsGord1} \\
\text{and} \qquad 0 &= B \, \Phi_{l_{i}} - \Tr( \Cocyc( l_{i} , \Shift^{-l_{i}+1} \varrho ) ) \, \Phi_{0}  + \Phi_{-l_{i}} \, . \label{eqn:3EpsGord2}
\end{align}
Now there are two cases: on the one hand, Equation~(\ref{eqn:3EpsGord1}) yields for \( \lvert \Tr( \Cocyc( l_{i} , \Shift^{-l_{i}+1} \varrho ) ) \rvert \leq 1 \) the relation
\[ \lVert \Phi_{-2l_{i}} \rVert + \lVert \Phi_{-l_{i}} \rVert \geq \lVert \Phi_{-2l_{i}} \rVert + \lVert \Tr( \Cocyc( l_{i} , \Shift^{-l_{i}+1} \varrho ) ) \Phi_{-l_{i}} \rVert \geq \lVert \Phi_{0} \rVert \geq \frac{\lVert \Phi_{0} \rVert}{ \lVert B \rVert } \, .\]
On the other hand, Equation~(\ref{eqn:3EpsGord2}) yields for \( \lvert \Tr( \Cocyc( l_{i} , \Shift^{1} \varrho ) ) \rvert \geq 1 \) the relation
\[ \lVert \Phi_{l_{i}} \rVert + \lVert \Phi_{-l_{i}} \rVert \geq \frac{ \lVert B \rVert \cdot \lVert \Phi_{l_{i}} \rVert + \lVert \Phi_{-l_{i}} \rVert }{ \lVert B \rVert } \geq \frac{ \lVert \Tr( \Cocyc( l_{i} , \Shift^{1} \varrho ) ) \Phi_{0} \rVert }{ \lVert B \rVert }  \geq \frac{ \lVert \Phi_{0} \rVert }{ \lVert B \rVert } \, . \]
By combining both cases, we obtain the inequality
\[ \max \Set{ \lVert \Phi_{-2l_{i}} \rVert , \lVert \Phi_{-l_{i}} \rVert , \lVert \Phi_{l_{i}} \rVert } \geq \frac{\lVert \Phi_{0} \rVert}{ 2 \cdot \lVert B \rVert } \, .\]
Since \( \lVert B \rVert \) is bounded, \( \lVert \Phi_{j} \rVert \) tends to zero for at most one of the limits \( j \to \pm \infty \).

Now it only remains to deduce the corresponding statement for  \( \lVert \Cocyc( j , \InfWord ) \Phi \rVert \): for an arbitrary \( \Phi \in \RR^{2} \setminus \Set{ 0 } \), define \( \Phi_{0} \DefAs \Cocyc( J+1 , \InfWord ) \Phi \), where \( J \) is such that \( \Cocyc( \InfWord ) \) only depends on \( \Restr{ \InfWord }{ -J }{ J } \). By the preceding considerations, \( \lVert \Phi_{j} \rVert \) cannot tend to zero in both directions. Now the claim follows from the definitions of \( \Phi_{j} \) and \( \varrho \), which yield
\[ \Phi_{j} = \Cocyc( j , \Shift^{ J+1 } \InfWord ) \Phi_{0} = \Cocyc( j , \Shift^{J+1} \InfWord ) \Cocyc( J+1 , \InfWord ) \Phi =\Cocyc( j+J+1 , \InfWord ) \Phi \, . \qedhere \]
\end{proof}

Combined the Propositions~\ref{prop:PQPrefixLimit} and~\ref{prop:GordonCocyc} yield combinatorial conditions that imply \LCond{}. In contrast to the conditions in Definition~\ref{defi:LCond}, they are relatively easy to check for specific examples. We will use them in Section~\ref{sec:LCondExmpl} to prove that simple Toeplitz subshifts and Sturmian subshifts satisfy \LCond{}. Note that the conditions are sufficient, but not necessary: every periodic subshift satisfies for example \LCond{} (see Subsection~\ref{subsec:LCondPeriod}), but their leading sequences will in general not have the palindromic form \( \LWord{i} = \Word{ \Rev{\RWord} }{\Origin{ \LOrigin{i} }}{ \RWord } \) that is required in the sufficient conditions below.

\begin{prop}
\label{prop:LCondSuffi}
A subshift \( \Subshift \) satisfies \LCond{} if there exists a one-sided infinite word \( \RWord \in \Alphab^{\NN} \) and finite words \( \LOrigin{1} , \ldots , \LOrigin{\LNr}\) such that the following conditions hold:
\begin{tightdescription}
\tightlist
\item[\hypertarget{LSCaa}{\LCond{\(\boldsymbol{\ \alpha'}\)}}]{For all \( i \in \Set{ 1 , \ldots , \LNr } \) we have \( \LWord{i} \DefAs \Word{ \Rev{\RWord} }{\Origin{ \LOrigin{i} }}{ \RWord } \in \Subshift \). Moreover, there exist \( L_{0}, I_{0} \in \NN_{0} \) such that
\[ \Langu{\Subshift}_{L} = \bigcup_{i=1}^{\LNr} \bigcup_{j= -I_{0}-L}^{I_{0}-1} \Restr{ \LWord{i} }{ j+1 }{ j+L } \]
holds for all  \( L \in \NN \) with \( L \geq L_{0} \).}
\item[\hypertarget{LSCbb}{\LCond{\(\boldsymbol{\ \beta'}\)}}]{The one-sided infinite word \( \RWord \) satisfies (\ref{eqn:DefiPQ}). Moreover, there is a sequence \( ( \Restr{\RWord}{1}{L_{k}} )_{k} \) of prefixes of \( \RWord \) with \( \lim_{k \to \infty} L_{k} = \infty \), such that for every subadditive function \( F \colon \Langu{ \Subshift } \to \RR \), the limit \( \lim_{k \to \infty} \frac{ F( \Restr{\RWord}{1}{L_{k}} ) }{ L_{k} } \) exists and is equal to \( \lim_{L \to \infty} \max \Set{ \frac{ F(u) }{ L } : u \in \Langu{ \Subshift }_{L} } \).}
\item[\hypertarget{LSCcc}{\LCond{\(\boldsymbol{\ \gamma'}\)}}]{For every \( i \in \Set{ 1 , \ldots , \LNr } \) there exists a sequence \( ( l_{k} ) \) of natural numbers with \( \lim_{k \to \infty} l_{k} = \infty \) such that for all \( k \in \NN \) we have either
\begin{alignat*}{5}
& \Restr{\LWord{i}}{-2 l_{k}+1}{-l_{k}} &&= \Restr{\LWord{i}}{-l_{k}+1}{0} &&= \Restr{\LWord{i}}{1}{l_{k}} \, , \\
\text{or} \qquad & \Restr{\LWord{i}}{-l_{k}+1}{0} &&= \Restr{\LWord{i}}{1}{l_{k}} &&= \Restr{\LWord{i}}{l_{k}+1}{2 l_{k}} \; .
\end{alignat*}}
\end{tightdescription}
\end{prop}

\begin{proof}
Clearly \LCondaa{} implies \LConda{}. Moreover, \LCondcc{} implies \LCondc{} by Proposition~\ref{prop:GordonCocyc}. Thus, only the validity of \LCondb{} remains to show. For this, let \( \Cocyc \colon \Subshift \to \SL \) be a locally constant map and consider the function
\[ F_{\Cocyc}( u ) \DefAs \max \Set{ \ln( \lVert \Cocyc( \Length{u} , \InfWord ) \rVert ) : \InfWord \in \Subshift \text{ with } \Restr{ \InfWord }{ 1 }{ \Length{u} } = u } \, . \]
Since it is subadditive (see Example~\ref{exmpl:MaxCocycSubadd}), the limit \( \lim_{L \to \infty} \frac{ F_{\Cocyc}( \Restr{\RWord}{1}{L} ) }{ L } \) exists by \LCondbb{} and Proposition~\ref{prop:PQPrefixLimit}. For every \( L \in \NN \), let \( \InfWord_{L} \in \Subshift\) be such that \( \Restr{ \InfWord_{L} }{ 1 }{ L } = \Restr{\RWord}{1}{L} \) and \( F_{\Cocyc}( \Restr{\RWord}{1}{L} ) = \ln( \lVert \Cocyc( L , \InfWord_{L} ) \rVert ) \) hold. Since \( \Cocyc \) depends on a finite patch around the origin, the first few and last few matrices in \( \Cocyc( L , \LWord{i} ) \) and \( \Cocyc( L , \InfWord_{L} ) \) may differ. However, according to Proposition~\ref{prop:CocycInvCut} we can ignore these finite pieces in limit considerations. Therefore we obtain for every \( i \in \Set{ 1 , \ldots , \LNr } \):
\[ \lim_{L \to \infty} \frac{ \ln( \lVert \Cocyc( L , \LWord{i} ) \rVert ) }{ L } = \lim_{L \to \infty} \frac{ \ln( \lVert \Cocyc( L , \InfWord_{L} ) \rVert ) }{ L }  = \lim_{L \to \infty} \frac{ F_{\Cocyc}( \Restr{\RWord}{1}{L} ) }{ L } \, , \]
and in particular the limit on the left hand side exists.

Now note that for every \( u \in \Langu{ \Subshift } \) we have \( \Rev{u} \in \Langu{ \Subshift } \) as well: since both, \( \RWord \) and \( \Rev{ \RWord } \) occur in the subshift, the reflections of all their subwords are in \( \Langu{ \Subshift } \). By \LCondaa{}, every word \( u \) occurs around the origin of some \( \LWord{i} \), and by the repetitions in \LCondcc{}, there is an occurrence of \( u \) that lies completely in \( \RWord \) or \( \Rev{ \RWord } \). We can therefore define \( \Rev{F_{\Cocyc}} \colon \Langu{ \Subshift } \to \RR \), \( u \mapsto F_{\Cocyc}( \Rev{u} ) \), which is a subadditive function as well. Hence the limit \( \lim_{L \to\infty} \frac{ F_{\Cocyc}( \Rev{ \Restr{\RWord}{1}{L} } ) }{ L } \) exists. Using \( \lVert  \Cocyc( -L , \InfWord ) \rVert = \lVert \Cocyc( \Shift^{-1} \InfWord ) \cdot \ldots \cdot \Cocyc( \Shift^{-L} \InfWord ) \rVert \), the same argument as above shows the relation
\[ \lim_{L \to \infty} \frac{ \ln( \lVert \Cocyc( -L , \LWord{i} ) \rVert ) }{ L } = \lim_{L \to \infty} \frac{ \Rev{F_{\Cocyc}}( \Restr{ \RWord }{1}{L} ) }{ L } = \lim_{L \to \infty} \frac{ F_{\Cocyc}( \Rev{ \Restr{\RWord}{1}{L} } ) }{ L } \]
and in particular the existence of the limit on the left. The repetitions in \LCondcc{} yield \( \Rev{ \Restr{ \RWord }{ 1 }{ l_{k} - \Length{\LOrigin{i}} } } = \Restr{ \RWord }{ 1 }{ l_{k} - \Length{\LOrigin{i}} } \). Since \( l_{k} \) gets arbitrarily large, the limits for \( \Cocyc( -L , \LWord{i} ) \) and \( \Cocyc( L , \LWord{i} ) \) all agree, which gives \LCondb{}.
\end{proof}

%%%%%%%%%%%%%%%%%%%%%%%%
\section{Examples of (LSC)-subshifts}
\label{sec:LCondExmpl}

In this section we show that simple Toeplitz subshifts and Sturmian subshifts satisfy the leading sequence condition. Consequently, the results that we will prove in Section~\ref{sec:UnifCoycLCond} for \LCond{}-subshifts apply to them. Most importantly, it will follow that aperiodic Jacobi operators associated to a simple Toeplitz subshift always have Cantor spectrum (Theorem~\ref{thm:LCondCantorSpec}). For Sturmian subshifts this is well-known, but we discuss them nevertheless to show that \LCond{}-subshifts cover several interesting classes. As an additional example we also prove the leading sequence conditions for periodic subshifts.

\subsection{Periodic subshifts}
\label{subsec:LCondPeriod}

Our focus in this thesis is on aperiodic subshifts, since the spectral properties of Jacobi operators are rather well-understood in the periodic case (see Remark~\ref{rem:PerioAndRand}). In particular, the main application of the leading sequence property deals with aperiodic subshifts (Theorem~\ref{thm:LCondCantorSpec}). Nevertheless we briefly discuss periodic subshifts since they are the simplest examples of \LCond{}-subshifts. Below we give a direct proof of the \LCond{}-property. However, note that it can also be concluded from the corresponding result for simple Toeplitz subshifts in Theorem~\ref{thm:SimpToepLCond}, since the aperiodicity of simple Toeplitz subshifts is not used in Subsection~\ref{subsec:LCondST}.

Let \( \widetilde{ \InfWord } \in \Alphab^{\ZZ} \) be periodic with period \( P \in \NN \), that is, for all \( j , l \in \ZZ \) we have \( \widetilde{ \InfWord }( j ) = \widetilde{ \InfWord }( j + lP ) \). The associated subshift \( \Subshift \) is given by
\[ \Subshift \DefAs \Subshift( \widetilde{ \InfWord } ) = \Close{ \Set{ \Shift^{k} \widetilde{ \InfWord } : k \in \ZZ } } = \big\{ \; \widetilde{ \InfWord } , \Shift^{1} \widetilde{ \InfWord } , \ldots , \Shift^{P-1} \widetilde{ \InfWord } \; \big\}  \, . \]  
The subshift has a single leading sequence, which we choose as \( \LWord{1} \DefAs \widetilde{ \InfWord } \) (but we could have chosen any other \( \InfWord \in \Subshift \) as well). We set \( I_{0} \DefAs \big\lceil \frac{ P }{ 2 } \big\rceil \). First we show that every finite word intersects \( \Restr{ \widetilde{ \InfWord } }{ -I_{0} }{ I_{0} } \): every \( u \in \Langu{ \Subshift } \) occurs in \( \widetilde{ \InfWord } \). Since \( \widetilde{ \InfWord } \) is \( P \)-periodic and \( \Length{ \Restr{ \widetilde{ \InfWord } }{ -I_{0} }{ I_{0} } } > P \) holds, a shifted occurrence of \( u \) has its first letter in \( \Restr{ \widetilde{ \InfWord } }{ -I_{0} }{ I_{0} } \). Secondly, by Proposition~\ref{prop:GordonCocyc} the condition \LCondc{} follows immediately from periodicity due to the repetitions around the origin. Now let \( \Cocyc \) be a locally constant cocycle. It only remains to show that the limits \( \lim_{j \to \pm \infty} \frac{ 1 }{ \lvert j \rvert } \ln( \lVert \Cocyc( j , \widetilde{ \InfWord } ) \rVert ) \) exist and are equal. 

For \( j \to \infty \) we first consider the subsequence given by \( j_{l} \DefAs 2^{l} P \). Recall that \( \Cocyc( s , \InfWord ) = \Cocyc( s-t , \Shift^{t} \InfWord ) \cdot \Cocyc( t , \InfWord ) \) holds for all \( s , t \in \ZZ \) and all \( \InfWord \in \Subshift \), see the proof of Proposition~\ref{prop:CocycInvCut}. Since \( \widetilde{ \InfWord } \) is \( P \)-periodic, we obtain
\[ \Cocyc( 2^{l+1} P , \widetilde{ \InfWord } ) = \Cocyc( 2^{l} P , \Shift^{ 2^{l} P } \widetilde{ \InfWord } ) \cdot \Cocyc( 2^{l}P , \InfWord ) = \Cocyc( 2^{l} P , \widetilde{ \InfWord } ) \cdot \Cocyc( 2^{l}P , \InfWord ) \, .\]
Thus subadditivity yields
\[ \frac{ \ln( \lVert \Cocyc( j_{l+1} , \widetilde{ \InfWord } ) \rVert ) }{ j_{l+1} } \leq \frac{ 2 \cdot \ln( \lVert \Cocyc( 2^{l} P , \widetilde{ \InfWord } ) \rVert ) }{ 2^{l+1} P } = \frac{ \ln( \lVert \Cocyc( j_{l} , \widetilde{ \InfWord } ) \rVert ) }{ j_{l} } \, , \]
that is, the sequence \( ( \frac{ \ln( \lVert \Cocyc( j_{l} , \widetilde{ \InfWord } ) \rVert ) }{ j_{l} } )_{l} \) is monotonically decreasing. Since it is also bounded from below by zero, it converges to some value \( \Lyapu{ \Cocyc } \in [0 , \infty ) \). To deal with arbitrary \( j \in \NN \), we express them as weighted average of \( j_{l} \)'s and use the following relation:

\emph{Let \( ( x_{l} )_{l} \) be a convergent sequence of real numbers with limit \( x \). Let \( \lambda_{ l , k} \) be non-negative numbers with \( \sum_{l=1}^{k} \lambda_{ l , k } = 1 \) for every \( k \in\NN \) and \( \lim_{k \to \infty} \lambda_{ l , k } = 0 \) for every \( l \in \NN \). Then the sequence \( ( \sum_{l=1}^{k} \lambda_{ l , k } x_{l}  )_{k} \) of weighted averages of \( x_{l} \)'s converges to \( x \) as well.} 

This can easily be seen by splitting the sum: there exists \( L \in \NN \) such that \( \sum_{l=L+1}^{k} \lambda_{ l , k } \lvert x_{l} - x \rvert \) is close to zero for all sufficiently large \( k \) due to \( x_{l} \to x\) and \( \sum_{l=1}^{L} \lambda_{ l , k } \lvert x_{l} - x \rvert \) is close to zero due to \( \lambda_{ l , k } \to 0 \). Now we proceed with the main proof. We define \( D \DefAs \max \Set{ \ln( \lVert \Cocyc( \InfWord ) \rVert ) : \InfWord \in \Subshift } \) and write \( j \in \NN \) in the form \( j = \sum_{ l = 0 }^{  L-1 } m_{l} 2^{l} P + r \) with \( m_{l} \in \Set{ 0 ,1 } \), minimal \( L \in \NN \) and \( r \in \Set{ 0 , \ldots , P-1 } \). As before, \( P \)-periodicity and subadditivity yield an upper bound:
\begin{align*}
\frac{ \ln( \lVert \Cocyc( j , \widetilde{ \InfWord } ) \rVert ) }{ j }
&\leq \frac{ \ln( \lVert \Cocyc( r , \widetilde{ \InfWord } ) \rVert ) + \sum_{ l=0 }^{L-1} \ln( \lVert \Cocyc( m_{l}2^{l} P , \widetilde{ \InfWord } ) \rVert )  }{ j } \\
&\leq \frac{ r }{ j } \cdot D + \sum_{ l=0 }^{L-1} \frac{ m_{l} 2^{l} P }{ j } \cdot \frac{ \ln( \lVert \Cocyc( 2^{l} P , \widetilde{ \InfWord } ) \rVert ) }{ 2^{l} P } \, .
\end{align*}
Since the right hand side is a weighted average of \( ( \frac{ \ln( \lVert \Cocyc( j_{l} , \widetilde{ \InfWord } ) \rVert ) }{ j_{l} } )_{l} \), it converges to \( \Lyapu{ \Cocyc } \). To deduce a lower bound, we note that \( 2^{L} P = P - r + j + \sum_{ l=0 }^{ L-1 } ( 1 - m_{l} ) 2^{l} P  \) holds. Thus \( P \)-periodicity and subadditivity yield
\[ \ln( \lVert \Cocyc( 2^{L} P , \widetilde{ \InfWord } ) \rVert ) \leq (P-r) D + \ln( \lVert \Cocyc( j , \widetilde{ \InfWord } ) \rVert ) + \sum_{ l=0 }^{ L-1 } ( 1-m_{l} ) \ln( \lVert \Cocyc( 2^{l} P , \widetilde{ \InfWord } ) \rVert ) \, .\]
Rearranging the terms we obtain
\begin{align*}
\frac{ \ln( \lVert \Cocyc( j , \widetilde{ \InfWord } ) \rVert ) }{ j }
&\geq \frac{ \ln( \lVert \Cocyc( 2^{L} P , \widetilde{ \InfWord } ) \rVert ) - \sum_{l=0}^{L-1} ( 1-m_{l} ) \ln( \lVert \Cocyc( 2^{l} P , \widetilde{ \InfWord } ) \rVert ) - ( P-r ) D}{ j } \\
&= \frac{ \ln( \lVert \Cocyc( 2^{L} P , \widetilde{ \InfWord } ) \rVert ) }{ 2^{L} P } + \frac{ 2^{L} P - j }{ j } \Bigg[ \frac{ \ln( \lVert \Cocyc( 2^{L} P , \widetilde{ \InfWord } ) \rVert ) }{ 2^{L} P}\\
& \qquad - \Bigg( \sum_{ l=0 }^{ L-1 } \frac{ ( 1-m_{l} ) 2^{l} P }{ 2^{L} P - j } \cdot \frac{ \ln( \lVert \Cocyc( 2^{l} P , \widetilde{ \InfWord } ) \rVert ) }{ 2^{l} P } + \frac{ P-r }{ 2^{L} P - j } \cdot D \Bigg) \Bigg] \, .
\end{align*}
Since the term in parentheses is a weighted average, it converges to \( \Lyapu{ \Cocyc } \). Moreover \( \frac{ 2^{L} P - j }{ j } \) is bounded, so the second summand converges to zero. Hence upper and lower bound agree, so the limit \( \lim_{ j \to \infty } \frac{ \ln( \lVert \Cocyc( j , \widetilde{ \InfWord } ) \rVert ) }{ j } \) exists and is equal to \( \Lyapu{ \Cocyc } \).

Finally we show that the limit for \( j \to - \infty \) is equal to \( \Lyapu{ \Cocyc } \) as well. For \( j = l P + r \) with \( l \in \NN_{0} \) and \( 0 \leq r < P \), the \( P \)-periodicity of \( \widetilde{ \InfWord } \) yields
\[ \lVert \Cocyc( -j , \widetilde{ \InfWord } ) \rVert \!=\! \lVert \Cocyc( \Shift^{-1} \widetilde{ \InfWord } ) \cdot \ldots \cdot \Cocyc( \Shift^{-j} \widetilde{ \InfWord } ) \rVert \!=\! \lVert \Cocyc( \Shift^{(l+1)P - 1} \widetilde{ \InfWord } ) \cdot \ldots \cdot \Cocyc( \Shift^{P-r} \widetilde{ \InfWord } ) \rVert \!=\! \lVert \Cocyc( j , \Shift^{P-r} \widetilde{ \InfWord } ) \rVert \, . \]
Now the claim follows from Proposition~\ref{prop:CocycInvCut}, as a finite shift does not change the limit.

\subsection{Simple Toeplitz subshifts}
\label{subsec:LCondST}

We check that the conditions from Proposition~\ref{prop:LCondSuffi}, which are sufficient for \LCond{}, apply to simple Toeplitz subshifts. First we discuss the words \( \LWord{i} \) and their combinatorial properties: recall from Subsection~\ref{subsec:STDefiPBlock} that a simple Toeplitz subshift can be defined from palindromic blocks \( \PBlock{k} \). As before, we let \( \AlphabEv \) denote the set of letters that occur infinitely often in the coding sequence \( ( a_{k} ) \). We set \( \LNr \DefAs \Card{ \AlphabEv } \) and \( \AlphabEv \DefAs \Set{ a^{(1)}, \hdots , a^{(\LNr)} } \). Now we define \( \LNr \) words of length one by \( \LOrigin{i} \DefAs a^{(i)} \). As discussed in Example~\ref{exmpl:SpWBlocks}, the so-called ``leading words''
\[ \LWord{i} \DefAs \LWord{a^{(i)}} = \lim_{k \to \infty} \Word{ \PBlock{k} }{ \Origin{ a^{(i)} } }{ \PBlock{k} } =\Word{ \Rev{ \PBlock{\infty} } }{ \Origin{\LOrigin{i}} }{ \PBlock{\infty} }  \]  
belong to \( \Subshift \). Below we will see that they are indeed leading words in the sense of \LCond{}-subshifts. This is also the underlying reason that their combinatorial structure plays such an important role for the properties of the subshift (see for instance Example~\ref{exmpl:AbsenSpWords}). First we show that all sufficiently long finite words occur in some \( \LWord{i} \) around the origin.

\begin{prop}
Let \( \EventNr \) be such that \( a_{k} \in \AlphabEv \) holds for all \( k \geq \EventNr \). Then for every \( L \geq \Length{\PBlock{ \EventNr }} \) and every \( u \in \Langu{ \Subshift }_{L} \), there are numbers \( i \in \Set{ 1 , \ldots , \LNr } \) and \( j \in \Set{ 1, \ldots, L } \) with \( u = \Restr{ \LWord{i} }{ -j+1 }{ -j+L } \).
\end{prop}

\begin{proof}
Let \( k \) be such that \( \Length{\PBlock{ k }} < L \leq \Length{\PBlock{ k+1 }} \) holds. We choose an element \( \InfWord \in \Subshift \) which contains \( u \) and decompose it as \( \InfWord = \Word{ \ldots }{ \PBlock{k+1} }{ \star }{ \PBlock{k+1} }{ \ldots } \) with single letters \( \star \in \Set{ a_{j} : j \geq k+2 } \), see Equation~(\ref{eqn:DecompPBlock}). Now we distinguish two cases: firstly, if \( u \) is contained in \( \PBlock{k+1} = \Word{ \PBlock{k} }{ a_{k+1} }{ \PBlock{k} }{ \ldots }{ \PBlock{k} } \), then \( u \) contains at least once the single letter \( a_{k+1} \). We choose \( i \) such that \( \LOrigin{i} = a_{k+1} \) holds. Then \( u \) can be found in \( \LWord{i} = \Word{ \ldots }{ \PBlock{k+1} }{ \Origin{\LOrigin{i}} }{  \PBlock{k+1} }{ \ldots } \) such that \( \LWord{i}(0) \) coincides with the position of \( a_{k+1} \) in \( u \). Secondly, if \( u \) is not contained in a single \( \PBlock{k+1} \)-block, then there is a letter \( a \in \Set{ a_{l} : l \geq k+2 } \subseteq \AlphabEv \) such that \( u \) is contained in \( \Word{ \PBlock{k+1} }{ a }{ \PBlock{k+1} } \) and \( a \) is contained in \( u \). We choose \( i \) such that \( \LOrigin{i} = a \) holds. Around the origin, \( \LWord{i} \) has the form \( \Word{ \ldots }{ \PBlock{k+1} }{ \Origin{\LOrigin{i}} }{  \PBlock{k+1} }{ \ldots } \) and therefore contains \( u \) as claimed.
\end{proof}

Now we check that the leading words contain the necessary repetitions for \LCondcc{}:

\begin{prop}
For every \( i \in \Set{ 1 , \ldots , \LNr } \) there exists a sequence \( ( l_{k} ) \) of natural numbers with \( \lim_{k \to \infty} l_{k} = \infty \) such that for all \( k \in \NN \) we have
\[\Restr{\LWord{i}}{-2 l_{k}+1}{-l_{k}} = \Restr{\LWord{i}}{-l_{k}+1}{0} = \Restr{\LWord{i}}{1}{l_{k}} \, . \]
\end{prop}

\begin{proof}
For every \( \LOrigin{i} = a^{(i)} \in \AlphabEv \), there exists a sequence \( ( k_{j}) \) with \( \lim_{j \to \infty} k_{j} = \infty \) and \( a_{k_{j}} = a^{(i)} \). Now \( l_{k} \DefAs \Length{ \PBlock{k_{j}-1} } + 1 \) has the claimed property since \( \LWord{i} \) looks around the origin like
\[ \Word{ \ldots }{ \PBlock{k_{j}} }{ \Origin{a^{(i)}} }{ \PBlock{k_{j}} }{ \ldots } \, = \, \Word{ \ldots }{ \PBlock{k_{j}-1} }{ a_{k_{j}} }{ \PBlock{k_{j}-1} }{ \Origin{ a_{k_{j}} } }{ \PBlock{k_{j}-1} }{ a_{k_{j}} }{ \ldots } \; . \qedhere \]
\end{proof}

It only remains to check \LCondbb{}. First we show that along the prefixes \( \PBlock{k} \) of \( \PBlock{\infty} \), the asymptotic average of every subadditive function exists. Afterwards we prove that \( \PBlock{ \infty } \) satisfies (\ref{eqn:DefiPQ}).

\begin{prop}
Let \( F \colon \Langu{ \Subshift } \to \RR \) be a subadditive function. Then \( \lim_{k \to \infty} \frac{ F( \PBlock{k} ) }{ \Length{ \PBlock{k} } } = \lim_{L \to \infty} \max \Set{ \frac{ F( u ) }{ L } : u \in \Langu{ \Subshift }_{L} } \AsDef \AsAv{F} \) holds.
\end{prop}

\begin{proof}
Since \( \limsup_{k \to \infty} \frac{ F( \PBlock{k} ) }{ \Length{ \PBlock{k} } } \leq \AsAv{F} \) holds by definition, only \( \liminf_{k \to \infty} \frac{ F( \PBlock{k} ) }{ \Length{ \PBlock{k} } } \geq \AsAv{F} \) remains to show: define \( D \DefAs \max \Set{ \lvert F( a ) \rvert : a \in \Alphab } \), fix an arbitrary \( k \in \NN \), let \( L \geq \Length{ \PBlock{k} } \) and let \( u \in \Langu{ \Subshift }_{L}\). Since every \( \InfWord \in \Subshift \) can be decomposed into \( \PBlock{k} \)-blocks and single letters, we obtain \( u = \Word{ v }{ \star }{ \PBlock{k} }{ \star }{ \PBlock{k} }{ \ldots }{ \PBlock{k} }{ \star }{ w } \), where \( v \) is a suffix and \( w \) is a prefix of \( \PBlock{k} \). In this decomposition of \( u \), there are at most \( \frac{ L }{ \Length{\PBlock{ k }} } \)-many \( \PBlock{k} \)-blocks and \( \frac{ L }{ \Length{\PBlock{ k }} } + 1 < \frac{ 2L }{ \Length{\PBlock{ k }} } \) single letters. This yields
\[ \frac{ F( u ) }{ L } \leq \frac{ \Length{v} \cdot D }{ L } + \frac{ \frac{ 2L }{ \Length{\PBlock{ k }} } \cdot D }{ L } + \frac{ \frac{ L }{ \Length{\PBlock{ k }} } \cdot F( \PBlock{k} )}{ L } + \frac{ \Length{w} \cdot D }{ L } \, . \]
Since \( u \) was arbitrary, we obtain for every \( k \) and every \( L \geq \Length{\PBlock{ k }} \):
\[ \max_{u \in \Langu{ \Subshift }_{L}} \frac{ F( u  )}{ L } \leq \frac{ \Length{\PBlock{ k }} \cdot D }{ L } + \frac{ 2 D }{ \Length{\PBlock{ k }} } + \frac{ F( \PBlock{k} ) }{ \Length{\PBlock{ k }} } + \frac{ \Length{\PBlock{ k }} \cdot D }{ L } \, . \]
Taking the limit \( L \to \infty \) yields \( \AsAv{F} \leq \frac{ 2 D }{ \Length{\PBlock{ k }} } + \frac{ F( \PBlock{k} ) }{ \Length{\PBlock{ k }} } \) for every fixed \( k \). The claim now follows by taking \( \liminf_{k \to \infty} \) on both sides.
\end{proof}

\begin{prop}
For every \( j \in \NN \) the inequality
\[ \liminf_{L \to \infty} \, \DisCopies{ \Restr{ \PBlock{\infty} }{ 1 }{ j } }{ \Restr{ \PBlock{\infty} }{ 1 }{ L } } \cdot  \frac{ j }{ L } \geq \frac{1}{8} \]
holds, that is, \( \PBlock{\infty} \in \Alphab^{\NN} \) satisfies (\ref{eqn:DefiPQ}).
\end{prop}

\begin{proof}
Fix \( j \in \NN \), let \( K \) be such that \( \Length{\PBlock{ K }} < j \leq \Length{\PBlock{ K+1 }} \) holds and let \( m \) be such that \( m \cdot ( \Length{ \PBlock{K} } + 1 ) \leq j < (m+1) \cdot ( \Length{ \PBlock{K} } + 1 ) \) holds. In particular this implies \( 1 \leq m < n_{K+1} \). It is easy to see that
\[ \DisCopies{ \Restr{ \PBlock{\infty} }{ 1 }{ j } }{ \PBlock{K+1} } \cdot \frac{ j }{ \Length{\PBlock{ K+1 }} + 1 } \geq \Bigg\lfloor \frac{ n_{K+1} }{ m+1 } \Bigg\rfloor \cdot \frac{ m( \Length{\PBlock{ K }} +1 ) }{ n_{K+1}( \Length{ \PBlock{K} } +1 ) } = \Bigg\lfloor \frac{ n_{K+1} }{ m+1 } \Bigg\rfloor \cdot \frac{ m }{ n_{K+1} } \]
holds. We distinguish two cases: on the one hand, if \( \frac{ n_{K+1} }{ m+1 } < 2 \) holds, then we obtain
\[ \Bigg\lfloor \frac{ n_{K+1} }{ m+1 } \Bigg\rfloor \cdot \frac{ m }{ n_{K+1} } = 1 \cdot \frac{ m }{ m+1 } \frac{ m+1 }{ n_{K+1} } > \frac{ 1 }{ 4 } \, . \]
On the other hand, if \( \frac{ n_{K+1} }{ m+1 } \geq 2 \) holds, then we obtain
\[ \Bigg\lfloor \frac{ n_{K+1} }{ m+1 } \Bigg\rfloor \cdot \frac{ m }{ n_{K+1} } > \Bigg( \frac{ n_{K+1} }{ m+1 } - 1 \Bigg) \frac{ m }{ n_{K+1} } = \Bigg( 1 - \frac{ m+1 }{ n_{K+1} } \Bigg) \frac{ m }{ m+1 } \geq \frac{ 1 }{ 4 } \, . \]
Now a short computation yields the desired lower bound:
\begin{align*}
\DisCopies{ \Restr{ \PBlock{\infty} }{1}{j} }{\Restr{ \PBlock{\infty} }{ 1 }{ L} } \! \cdot \! \frac{ j }{L} &\geq \DisCopies{ \Restr{ \PBlock{\infty} }{1}{j} }{ \PBlock{K+1} } \! \cdot \! \frac{ j }{ \Length{\PBlock{ K+1 }} + 1 } \cdot \DisCopies{ \PBlock{K+1} }{ \Restr{ \PBlock{\infty} }{1}{L} } \! \cdot \! \frac{ \Length{\PBlock{ K+1 }} + 1 }{ L } \\
& > \frac{1}{4} \cdot \Big( \frac{ L }{ \Length{\PBlock{ K+1 }} + 1 } - 1 \Big) \cdot \frac{ \Length{\PBlock{ K+1}} +1 }{ L } \\
& = \frac{1}{4} \cdot \Big( 1 - \frac{ \Length{\PBlock{ K+1 }} + 1 }{ L } \Big) \\
& \geq \frac{1}{4} \cdot \Big( 1 - \frac{1}{2} \Big) \quad \text{for all sufficiently large } L \, . \hspace{8.1em} \qed
\end{align*}
\renewcommand{\qedsymbol}{}% omit end symbol
\end{proof}

\vspace{-2\baselineskip}

\begin{thm}
\label{thm:SimpToepLCond}
Every simple Toeplitz subshift satisfies \LCond{}.
\end{thm}

\begin{proof}
By the preceding propositions, every simple Toeplitz subshift satisfies the conditions of Proposition \ref{prop:LCondSuffi}.
\end{proof}

\subsection{Sturmian subshifts}
\label{subsec:LCondSturm}

We check that the conditions from Proposition~\ref{prop:LCondSuffi}, which are sufficient for \LCond{}, apply to Sturmian subshifts. Background information on Sturmian subshifts can be found in Appendix~\ref{app:Sturm}. In particular, we discuss in Proposition~\ref{prop:SturmDefBlocks} how a sequence \( ( n_{k} )_{k \in \NN} \) and the recurrence relation
\[ s_{0} \DefAs b \;\; , \quad s_{1} \DefAs \Word{ b^{n_{1}-1} }{ a } \;\; , \quad s_{k+1} \DefAs \Word{ s_{k}^{n_{k+1}} }{ s_{k-1} } \quad (k\geq 1) \]
define a Sturmian subshift: by \( \PBlock{k} \) we denote the word \( s_{k} \) without its last two letters. We write \( \lim_{k \to \infty} \PBlock{k} \AsDef \PBlock{\infty} \in \Set{ a , b }^{\NN} \) and let \( \Subshift = \Subshift( \PBlock{\infty} ) \DefAs \Set{ \InfWord \in \Alphab^{\ZZ} : \Langu{ \InfWord } \subseteq \Langu{ \PBlock{\infty} } } \) denote the associated Sturmian subshift. We define \( \LOrigin{1} \DefAs \Word{ a }{ b } \) and \( \LOrigin{2} \DefAs \Word{ b }{ a } \). In the following we show that \( \Subshift \) satisfies \LCond{} with leading sequences \( \LWord{i} \DefAs \Word{ \Rev{ \PBlock{\infty} } }{ \Origin{\LOrigin{i}} }{ \PBlock{\infty} } \) for \( i \in \Set{ 1, 2 } \). 

\begin{prop}
The infinite words \( \LWord{1} , \LWord{2} \) are elements of \( \Subshift \). For every sufficiently long \( u \in \Langu{ \Subshift } \) there are numbers \( i \in \Set{ 1 , 2 } \) and \( k \in \Set{ 1, \ldots, \Length{u}+1 } \) with \( u = \Restr{ \LWord{i} }{ -k+1 }{ -k+\Length{u} } \).
\end{prop}

\begin{proof}
In Corollary~\ref{cor:LWordInSturm} in the appendix we show \( \LWord{1} , \LWord{2} \in \Subshift \). The second part of the claim can be concluded from \cite[Lemma~3.2]{DamLenz_UniSpectProp2}. For completeness, we give an explicit proof as well, which is based on Proposition~\ref{prop:SturmPali}: let \( K \in \NN \) be minimal such that \( u \) is contained in \( \PBlock{K} \). Assume that \( K \geq 6 \) holds. In the decomposition \( \PBlock{K} = \Word{ \PBlock{K-1} }{ \LOrigin{i} }{ \ldots }{ \PBlock{K-1} }{ \LOrigin{i} }{ \PBlock{K-2} } \), the word \( u \) intersects \( \LOrigin{i} \) at least once. To the left of the first such intersection, \( u \) is equal to a suffix of \( \PBlock{K-1} \), which clearly is a suffix of \( \PBlock{ K } \) as well. To the right of this intersection, \( u \) is equal to a prefix of \( \PBlock{K} \). Aligning the intersection with the occurrence of \( \LOrigin{i} \) at the origin of \( \LWord{i} = \Word{ \ldots }{ \PBlock{K} }{ \Origin{ \LOrigin{i} } }{ \PBlock{K} }{ \ldots } \) proves the claim.
\end{proof}

\begin{prop}
The word \( \PBlock{\infty} \in \Alphab^{\NN} \) satisfies (\ref{eqn:DefiPQ}). 
\end{prop}

\begin{proof}
Fix \( j \in \NN \), let \( k \) be the unique index given by \( \Length{s_{k-1}} \leq j < \Length{s_{k}} \) and note that \( \DisCopies{ \Restr{ \PBlock{\infty} }{ 1 }{ j } }{ \Restr{ \PBlock{\infty} }{ 1 }{ L } } \geq \DisCopies{ s_{k} }{ \Restr{ \PBlock{\infty} }{ 1 }{ L } } \cdot \DisCopies{ \Restr{ \PBlock{\infty} }{ 1 }{ j } }{ s_{k} } \) holds. We study the factors separately: in the decomposition of \( \Restr{ \PBlock{\infty} }{ 1 }{ L } \) into \( s_{k} \)-blocks and \( s_{k-1} \)-blocks, all  \( s_{k-1} \)-blocks are isolated, see Proposition~\ref{prop:SturmPali}~(e). Using \( \Length{ s_{k} } > \Length{ s_{k-1} } \), we obtain \( \DisCopies{ s_{k} }{ \Restr{ \PBlock{\infty} }{ 1 }{ L } } \geq \frac{ L }{ 2 \Length{s_{k}} } - 1 \geq \frac{ L }{ 3 \Length{s_{k}} } \) for all sufficiently large \( L \). Now we consider \( \Restr{ \PBlock{\infty} }{ 1 }{ j } \). Let \( t , r \in \NN \) be given by \( j = t \Length{s_{k-1}} + r \) and \( 0 \leq r < \Length{s_{k-1}} \). Note that we have \( 1 \leq t \leq n_{k} \). If \( 2t+2 \leq n_{k} \) holds, we use \( \DisCopies{ \Restr{ \PBlock{\infty} }{ 1 }{ j } }{ s_{k} } \geq \lfloor \frac{ n_{k} }{ t+1 } \rfloor \geq \frac{ n_{k} }{ t+1 } -1 \), which yields 
\[ \frac{ j }{ L } \DisCopies{ \Restr{ \PBlock{\infty} }{ 1 }{ j } }{ \Restr{ \PBlock{\infty} }{ 1 }{ L } } \geq \frac{ t \Length{s_{k-1}} }{ L } \cdot \Big( \frac{ n_{k} }{ t+1 } -1 \Big) \cdot \frac{ L }{ 3 \Length{s_{k}} } = \frac{ t }{ 3 ( t+1 ) } \Big( 1 - \frac{ t+1 }{ n_{k} } \Big) \geq \frac{ 1 }{ 12 } \, . \]
If \( 2t+1 \geq n_{k} \) holds, we use \( \DisCopies{ \Restr{ \PBlock{\infty} }{ 1 }{ j } }{ s_{k} } \geq 1 \) and obtain
\[ \frac{ j }{ L } \DisCopies{ \Restr{ \PBlock{\infty} }{ 1 }{ j } }{ \Restr{ \PBlock{\infty} }{ 1 }{ L } } \geq \frac{ t \Length{s_{k-1}} }{ L } \cdot 1 \cdot \frac{ L }{ 3 \Length{s_{k}} } \geq \frac{ t }{ 3 ( n_{k} + 1 ) } \geq \frac{1}{12} \, .\qedhere \]
\end{proof}

\begin{prop}
For every subadditive function \( F \colon \Langu{ \Subshift } \to \RR \), the limit \( \lim_{k \to \infty} \frac{ F( s_{k} ) }{ \Length{s_{k}} } \) exists (but might be equal to minus infinity).
\end{prop}

\begin{proof}
Subadditivity of \( F \) combined with the decomposition \( s_{k} = \Word{ s_{k-1}^{n_{k}} }{ s_{k-2} } \) yields for all \( k \geq 2 \) the inequality
\begin{equation}
\label{eqn:SubAdRecSturm}
\frac{ F( s_{k} ) }{ \Length{s_{k}} } \leq \frac{ n_{k} \Length{s_{k-1}} }{  n_{k} \Length{s_{k-1}} + \Length{s_{k-2}} } \cdot \frac{ F( s_{k-1} ) }{ \Length{s_{k-1}} } + \frac{ \Length{s_{k-2}} }{  n_{k} \Length{s_{k-1}} + \Length{s_{k-2}} } \cdot \frac{ F( s_{k-2} ) }{ \Length{s_{k-2}} } \, .
\end{equation}
Since this is a weighted average of \( \frac{ F( s_{k-1} ) }{ \Length{s_{k-1}} } \) and \( \frac{ F( s_{k-2} ) }{ \Length{s_{k-2}} } \), the sequence \( \big( \frac{ F( s_{k} ) }{ \Length{s_{k}} } \big)_{k} \) is bounded from above by \( \max \big\{ \, \frac{ F( s_{0} ) }{ \Length{s_{0}} } , \frac{ F( s_{1} ) }{ \Length{s_{1}} } \, \big\} \). Moreover, Equation~(\ref{eqn:SubAdRecSturm}) implies
\[ \frac{ F( s_{k} ) }{ \Length{s_{k}} } - \frac{ F( s_{k-1} ) }{ \Length{s_{k-1}} } \leq - \Big( \frac{ F( s_{k-1} ) }{ \Length{s_{k-1}} } - \frac{ F( s_{k-2} ) }{ \Length{s_{k-2}} } \Big) \frac{ \Length{s_{k-2}} }{ \Length{s_{k}} } \leq (-1)^{k-1} \Big( \frac{ F( s_{1} ) }{ \Length{s_{1}} } - \frac{ F( s_{0} ) }{ \Length{s_{0}} } \Big) \prod_{i=2}^{k} \frac{ \Length{s_{i-2}} }{ \Length{s_{i}} } \, . \]
Now \( 0 < \frac{ \Length{s_{i-2}} }{ \Length{s_{i}} } = \frac{ \Length{s_{i-2}} }{ n_{i} \Length{s_{i-1}} + \Length{ s_{i-2} } } < \frac{1}{2} \) yields for all \( k\geq 2 \) the inequality
\[ \frac{ F( s_{k} ) }{ \Length{s_{k}} } - \frac{ F( s_{k-1} ) }{ \Length{s_{k-1}} } \leq \Big\lvert \frac{ F( s_{1} ) }{ \Length{s_{1}} } - \frac{ F( s_{0} ) }{ \Length{s_{0}} } \Big\rvert \cdot \frac{1 }{ 2^{k-1} } \, . \]
Finally, we use a telescoping sum and obtain for \( k > l \geq 1 \):
\[ \frac{ F( s_{k} ) }{ \Length{ s_{k} } } = \frac{ F( s_{l} ) }{ \Length{ s_{l} } } + \sum_{i=l+1}^{k} \Big( \frac{ F( s_{i} ) }{ \Length{ s_{i} } } - \frac{ F( s_{i-1} ) }{ \Length{ s_{i-1} } } \Big) \leq \frac{ F( s_{l} ) }{ \Length{ s_{l} } } + \Big\lvert \frac{ F( s_{1} ) }{ \Length{s_{1}} } - \frac{ F( s_{0} ) }{ \Length{s_{0}} } \Big\rvert \cdot \frac{1 }{ 2^{l-1} } \, .\]
Now consider a subsequence \( ( s_{l} )\) such that \( \big( \frac{ F( s_{l} ) }{ \Length{ s_{l} } } \big)_{l} \) converges to \( F_{\infty} \DefAs \liminf_{k \to \infty} \frac{ F( s_{k} ) }{ \Length{ s_{k} } } \). For sufficiently large \( l \), the second summand gets arbitrarily close to zero, while the first summand gets arbitrarily close to \( F_{\infty} \) (if \( F_{\infty} \) is finite) or becomes smaller than every constant (if \( F_{\infty} \) is equal to minus infinity). Consequently, also \( \frac{ F( s_{k} ) }{ \Length{ s_{k} } } \) gets arbitrarily close to \( F_{\infty} \) or becomes smaller than every constant for all sufficiently large \( k \).
\end{proof}

\begin{prop}
For every subadditive function \( F \colon \Langu{ \Subshift } \to \RR \) we have \( \lim_{k \to \infty} \frac{ F( s_{k} ) }{ \Length{s_{k}} } = \lim_{L \to \infty} \max \big\{ \, \frac{ F(u) }{ L } : u \in \Langu{ \Subshift }_{L} \, \big\} \AsDef \AsAv{F} \).
\end{prop}

\begin{proof}
We give a proof by direct calculation, but a similar result can also be found in \cite[Theorem~11]{Lenz_ErgodTheo1dimSO}. We write \( F_{\infty} \DefAs \lim_{k \to \infty} \frac{ F( s_{k} ) }{ \Length{ s_{k} } } \), which exists by the previous proposition. Since \( F_{\infty} \leq \AsAv{F} \) is clear, we only have to show \( F_{\infty} \geq \AsAv{F} \): let \( D \DefAs \max \Set{   F( a ) , F( b ) } \), \( k \in \NN \), \( L \geq \Length{ s_{k} } \) and \( u \in \Langu{ \Subshift }_{L} \). Then \( u \) can be decomposed into \( s_{k} \)-blocks, \( s_{k-1} \)-blocks, a suffix \( v \) of \( s_{k} \) or \( s_{k-1} \) and a prefix \( w \) of \( s_{k} \), see Proposition~\ref{prop:SturmPali}. Let \( r(u) \) and \( t(u) \) denote the number of \( s_{k} \)-blocks and \( s_{k-1} \)-blocks respectively. Subadditivity yields
\begin{align*}
\frac{ F( u ) }{ L } &\leq \frac{ r(u) \cdot \Length{ s_{k} } }{ L } \cdot \frac{F( s_{k} )}{ \Length{s_{k}} } + \frac{ t(u) \cdot \Length{ s_{k-1} } }{ L } \cdot \frac{F( s_{k-1} )}{ \Length{s_{k-1}} } + \frac{ \Length{v} \cdot D }{ L } + \frac{ \Length{w} \cdot D }{ L } \\
&\leq \frac{ r(u) \cdot \Length{ s_{k} } + \Length{v} + \Length{w} }{ L } \cdot \frac{F( s_{k} )}{ \Length{s_{k}} } + \frac{ t(u) \cdot \Length{ s_{k-1} } }{ L } \cdot \frac{F( s_{k-1} )}{ \Length{s_{k-1}} } + \frac{ 2 \Length{s_{k}} }{ L } \cdot \Big\lvert D - \frac{F( s_{k} )}{ \Length{s_{k}} } \Big\rvert \, .
\end{align*}
Now we distinguish two cases: first assume that \( F_{\infty} \) is finite. Let \( \varepsilon > 0 \) and choose \( k \) such that \( \big\lvert \frac{ F( s_{l} ) }{ \Length{s_{l}} } - F_{\infty} \big\rvert < \varepsilon \) holds for all \( l \geq k-1 \). For all \( L \geq \Length{s_{k}} \) and all \( u \in \Langu{ \Subshift }_{L} \) we obtain
\[ \frac{ F( u ) }{ L } \leq \underbrace{ \frac{ r(u) \cdot \Length{ s_{k} } + \Length{v} + \Length{w} + t(u) \cdot \Length{ s_{k-1} } }{ L } }_{= 1} ( F_{\infty} + \varepsilon ) + \frac{ 2 \Length{s_{k}} }{ L } \cdot \Big\lvert D - \frac{F( s_{k} )}{ \Length{s_{k}} } \Big\rvert \, . \]
Taking the maximum over all \( u \in \Langu{ \Subshift }_{L} \) and then the limit \( L \to \infty \) yields \( \AsAv{F} \leq F_{\infty} + \varepsilon \) for all \( \varepsilon > 0 \) and thus \( \AsAv{F} \leq F_{\infty} \).

For the second case, assume that \( F_{\infty} = - \infty \) holds, let \( C \in \RR \) be arbitrary and choose \( k \) such that \( \frac{ F( s_{l} ) }{ \Length{s_{l}} } < C \) holds for all \( l \geq k-1 \). For all \( L \geq \Length{s_{k}} \) and all \( u \in \Langu{ \Subshift }_{L} \) we obtain
\[ \frac{ F( u ) }{ L } \leq \frac{ r(u) \cdot \Length{ s_{k} } + \Length{v} + \Length{w} + t(u) \cdot \Length{ s_{k-1} } }{ L } C + \frac{ 2 \Length{s_{k}} }{ L } \cdot \Big\lvert D - \frac{F( s_{k} )}{ \Length{s_{k}} } \Big\rvert \, . \]
Again, taking the maximum over all \( u \in \Langu{ \Subshift }_{L} \) and then the limit \( L \to \infty \) yields \( \AsAv{F} \leq C \) for all \( C \in \RR \) and thus \( \AsAv{F} = - \infty = F_{\infty} \).
\end{proof}

\begin{prop}
For every \( k \geq 2 \) we have
\begin{alignat*}{5}
\Restr{ \LWord{1} }{ -\Length{ s_{2k} }+1 }{ 0 } &= \Restr{ \LWord{1} }{1 }{ \Length{s_{2k}} } &&= \Restr{ \LWord{1} }{\Length{s_{2k}}+1 }{ 2\Length{s_{2k}} } &&= s_{2k}\\
\text{and} \quad \Restr{ \LWord{2} }{ -\Length{ s_{2k+1} }+1 }{ 0 } &= \Restr{ \LWord{2} }{1 }{ \Length{s_{2k+1}} } &&= \Restr{ \LWord{2} }{\Length{s_{2k+1}}+1 }{ 2\Length{s_{2k+1}} } &&= s_{2k+1} \, .
\end{alignat*}
\end{prop}

\begin{proof}
By Proposition~\ref{prop:SturmPali} we have for all \( k \geq 4 \) the decomposition
\[ s_{k+2} = \Word{ s_{k+1} }{ s_{k+1}^{n_{k+2}-1} }{ s_{k} } = \Word{ s_{k}^{ n_{k+1} } }{ s_{k-1} }{ s_{k+1}^{n_{k+2}-1} }{ s_{k} } = \Word{ s_{k}^{ n_{k+1} } }{ s_{k-1} }{ \PBlock{k} }{ \ldots } = \Word{ s_{k} }{ s_{k} }{ \PBlock{k-1} }{ \ldots } \, , \]
which implies \( \Restr{ \LWord{i} }{ 1 }{ 2 \Length{s_{k} } } = \Word{ s_{k} }{ s_{k} } \) for \( i \in \Set{ 1 , 2 } \) and \( k \geq 4 \). In addition, Proposition~\ref{prop:SturmPali} yields for all \( k \geq 2 \) the equalities
\[ \Restr{ \LWord{1} }{ -\Length{ s_{2k} }+1 }{ 0 } = \Word{ \PBlock{2k} }{ a }{ b } =  s_{2k} \quad \text{and} \quad \Restr{ \LWord{2} }{ -\Length{ s_{2k+1} }+1 }{ 0 } = \Word{ \PBlock{2k+1} }{ b }{ a } =  s_{2k+1} \, .  \qedhere\]
\end{proof}

\begin{thm}
\label{thm:SturmLCond}
Every Sturmian subshift satisfies \LCond{}.
\end{thm}

\begin{proof}
By the preceding propositions, every Sturmian subshift satisfies the conditions of Proposition \ref{prop:LCondSuffi}.
\end{proof}

%%%%%%%%%%%%%%%%%%%%%%%%
\section{Uniformity of cocycles for (LSC)-subshifts and Cantor spectrum}
\label{sec:UnifCoycLCond}

As we have seen, \LCond{}-subshifts provide a framework that contains several interesting classes of subshifts. In this section we show that all locally constant cocycles on \LCond{}-subshifts are uniform. This proves in a unified way that Jacobi operators have Cantor spectrum, both on Sturmian subshifts (where it is well-known, see \cite[Theorem~4]{Lenz_ErgodTheo1dimSO} combined with \cite[Theorem~3]{BeckPogo_SpectrJacobi}) and on simple Toeplitz subshifts (where it was only known for Schrödinger operators, see \cite[Theorem~1.2]{LiuQu_Simple}). This is the main result of this chapter and it relies on two statements. They are due to Ruelle (or, in a probabilistic version, to Oseledec \cite{Oseledec_MET}) and to Lenz:

\begin{prop}[\cite{Ruel_DiffDynSys}, see {\cite[Theorem~9.12]{CFKS_SOApplicQC}} for this version]
\label{prop:Ruelle}
Let \( (A_{j}) \) be a sequence of matrices in \( \SL \) with \( \lim_{j \to \infty} \frac{ \ln( \lVert A_{j} \rVert ) }{ j } = 0 \) and \( \lim_{j \to \infty} \frac{ \ln ( \lVert \Cocyc_{j} \ldots \Cocyc_{1} \rVert ) }{ j  } \AsDef c > 0 \). Then there exists a unique one-dimensional subspace \( V \subset \RR^{2} \) with
\[ \lim_{ j \to \infty} \frac{ \ln( \lVert \Cocyc_{j} \ldots \Cocyc_{1} \Phi \rVert )}{ j } = - c \qquad \text{and} \qquad \lim_{ j \to \infty } \frac{ \ln ( \lVert \Cocyc_{j} \ldots \Cocyc_{1} \Psi \rVert ) }{ j } = + c\]
for all \( \Phi \in V \setminus \Set{0} \) and all \( \Psi \notin V \).
\end{prop}

\begin{prop}[{\cite[Theorem~3]{Lenz_NonUnifCocy}}]
\label{prop:LowBdUnif}
Let \( \Subshift \) be minimal and uniquely ergodic. Then a continuous map \( \Cocyc \colon \Subshift \to \SL \) is uniform with limit value \( \Lyapu{ \Cocyc } > 0 \), if and only if there exist \( m \in \NN \) and \( \delta > 0 \) such that \( \frac{1}{j} \ln( \lVert \Cocyc( j , \InfWord ) \rVert ) \geq \delta \) holds for all \( \InfWord \in \Subshift \) and all \( j \geq m \).
\end{prop}

We can now prove the main result of this chapter.

\begin{thm}
\label{thm:LCondUniform}
If \( \Subshift \) satisfies \LCond{}, then every locally constant map \( \Cocyc \colon \Subshift \to \SL \) is uniform.
\end{thm}

\begin{proof}
Without loss of generality we assume \( I_{0} = 0 \), see Remark~\ref{rem:LNr-I0}. Moreover, let \( L \in \NN \) be such that \( \Cocyc( \InfWord ) \) only depends on \( \Restr{ \InfWord }{ -L }{ L } \). Define \( C \DefAs \max \Set{ \ln( \lVert \Cocyc( \InfWord ) \rVert ) : \InfWord \in \Subshift } \), which is finite since \( \Cocyc \) takes only finitely many values. We use that every word occurs around the origin of some \( \LWord{i} \) by \LConda{} and that we have some control over the \( \LWord{i} \) by \LCondb{} and \LCondc{}. More precisely, let \( j_{0} \in \NN \) be such that for every \( \InfWord \) and every \( j \geq j_{0} \) there exist \( i \in \Set{ 1 , \ldots , \LNr }  \) and \( j_{-} \geq 1 , \, j_{+} \geq 0 \) with \( \Restr{ \InfWord }{ 0 }{ j-1 } = \Restr{ \LWord{i} }{ -j_{-}+1 }{ j_{+} } \). Obviously we have
\[ \Cocyc( j_{-}+j_{+} , \Shift^{-j_{-}} \LWord{i} ) = \Cocyc( j_{+} , \LWord{i} ) \cdot \Cocyc( -j_{-} , \LWord{i} )^{-1} \, . \]
By \LCondb{} there exist a constant \( c \geq 0 \) with \( c = \lim_{ j \to \pm \infty} \frac{ \ln( \lVert \Cocyc( j , \LWord{i} ) \rVert ) }{ \lvert j \rvert } \) for every \( i \in \Set{ 1, \ldots,  \LNr } \). First we consider the case \( c = 0 \):

Let \( \varepsilon > 0 \) and let \( j_{1} \in \NN \) be such that \( \frac{ 1 }{ \lvert j \rvert } \ln( \lVert \Cocyc( j , \LWord{i} ) \rVert ) < \frac{ \varepsilon }{3} \) holds for all \( i \in \Set{ 1 , \ldots , \LNr } \) and all \( j \in \ZZ \) with \( \lvert j \rvert \geq j_{1} \). Let \( j_{2} \in \NN \) be such that
\[ \frac{ 1 }{ j_{2} } \cdot \max \Set{ \ln( \lVert \Cocyc( l , \LWord{i} ) \rVert ) : l \in \ZZ \text{ with } \lvert l \rvert < j_{1} \, , \; i \in \Set{ 1 , \ldots , r } } \leq \frac{ \varepsilon }{ 3 } \]
holds. Let \( j_{3} \in \NN \) be such that \( \frac{ 4 C L }{ j_{3} } \leq \frac{ \varepsilon }{ 3 } \) holds. For arbitrary \( J \geq \max \Set{ j_{0} , j_{2} , j_{3} } \) and arbitrary \( \InfWord \in \Subshift \), let \( i \in \Set{ 1 , \ldots , \LNr } \) be such that \( \Restr{ \InfWord }{ 0 }{ J-1 } = \Restr{ \LWord{i} }{ -j_{-}+1 }{ j_{+} } \) holds for suitable \( j_{-} \geq 1 , \, j_{+} \geq 0 \). Note that we have
\[ \frac{ \ln( \lVert \Cocyc( j_{-} + j_{+} , \Shift^{ -j_{-} } \LWord{i} ) \rVert ) }{ J } \leq \frac{ \ln( \lVert \Cocyc( j_{+} , \LWord{i} ) \rVert ) }{ J }+\frac{  \ln( \lVert A( -j_{-} , \LWord{i} ) \rVert ) }{ J } \, , \]
where both summands are less or equal \(  \frac{ \varepsilon }{ 3 } \): if \( j_{\pm} \) is greater or equal \( j_{1} \), then this follows from the definition of \( j_{1} \), and if \( j_{\pm} \) is smaller than \( j_{1} \), then this follows from \( J \geq j_{2} \). Since \( \Cocyc( J , \InfWord ) \) and \( \Cocyc( j_{-} + j_{+} , \Shift^{ -j_{-} } \LWord{i} ) \) differ at most in the first \(  L-1 \) matrices and in the last \( L+1 \) matrices, we obtain
\[ \frac{ \ln( \lVert \Cocyc( J , \InfWord) \rVert ) }{ J } \leq  \frac{ \ln( \lVert \Cocyc( j_{-} + j_{+} , \Shift^{ -j_{-} } \LWord{i} ) \rVert ) }{ J } + \frac{ 4 C L }{J} \leq \varepsilon \, . \]

Now we treat the case \( c > 0 \): by Proposition~\ref{prop:LowBdUnif} it suffices to show that there exists a number \( \delta > 0 \) with \( \frac{ \ln( \lVert \Cocyc( j , \InfWord ) \rVert ) }{ j } \geq \delta \) for all \( \InfWord \in \Subshift \) and all sufficiently large \( j \). First we note that the assumptions of Proposition~\ref{prop:Ruelle} are satisfied because of \LCondb{}. This implies for every \( i \in \Set{ 1 , \ldots, \LNr } \) the existence of a one-dimensional subspace \( V^{(i)}_{+} \subset \RR^{2} \) with
\[ \lim_{j \to \infty} \frac{ \ln( \lVert \Cocyc( j , \LWord{i} ) v \rVert ) }{ j } = - c \quad \text{for all } v \in V^{(i)}_{+} \setminus \Set{ 0 } \, , \]
and another one-dimensional subspace \( V^{(i)}_{-} \subset \RR^{2} \) with
\[ \lim_{j \to \infty} \frac{ \ln ( \lVert \Cocyc( -j , \LWord{i} ) v \rVert ) }{ j } = - c \quad \text{for all } v \in V^{(i)}_{-} \setminus \Set{ 0 } \, . \]
Now fix \( i \in \Set{ 1 , \ldots , \LNr } \). By \LCondc{} we obtain \( V^{(i)}_{+} \neq V^{(i)}_{-} \) and in particular, Proposition~\ref{prop:Ruelle} yields the equality
\[ \lim_{j \to \infty} \frac{ \ln( \lVert \Cocyc( j , \LWord{i} ) v \rVert ) }{ j } = c \quad \text{for all } v \in V^{(i)}_{-} \setminus \Set{ 0 } \, . \]
To prove uniformity, we fix some \( \hat{v} \in  V^{(i)}_{-} \setminus \Set{ 0 } \) and proceed similar to the case \( c=0 \), but with \( \lVert \Cocyc \hat{v}\rVert \) instead of \( \lVert \Cocyc \rVert \). More precisely, let \( j_{1} \in \NN \) be such that we have
\[ \frac{ \ln( \lVert \Cocyc( -j , \LWord{i} ) \hat{v} \rVert ) }{ j } \leq - \frac{ c }{ 2 } \qquad \text{and} \qquad \frac{ \ln( \lVert \Cocyc( j , \LWord{i} ) \hat{v} \rVert ) }{ j } \geq \frac{ c }{ 2 } \]
for all \( j \geq j_{1} \). Let \( j_{2} \in \NN \) be such that
\[ \frac{ 1 }{ j_{2} } \cdot \Big\lvert \; \max \Set{ \ln( \lVert \Cocyc( l , \LWord{i} ) \hat{v} \rVert ) : l \in \ZZ \text{ with } \lvert l \rvert < j_{1} } \; \Big\rvert \leq \frac{ c }{ 8 } \]
holds. Let \( j_{-} \geq 1 , \, j_{+} \geq 0 \) be such that \( J \DefAs j_{-} + j_{+} \geq \max \Set{ 2 j_{1} , j_{2} } \) holds. We write \( \hat{u} \DefAs \Cocyc( -j_{-} , \LWord{i} ) \hat{v} \) and obtain
\begin{align*}
\ln ( \lVert \Cocyc( J , \Shift^{-j_{-}} \LWord{i} ) \rVert ) & \geq \ln \left( \frac{ \lVert \Cocyc( J , \Shift^{-j_{-}} \LWord{i} ) \hat{u} \rVert }{ \lVert \hat{u} \rVert } \right) \\
&= \ln( \lVert \Cocyc( j_{+} , \LWord{i} ) \hat{v} \rVert ) - \ln( \lVert \Cocyc( -j_{-} , \LWord{i} ) \hat{v} \rVert ) \, .
\end{align*}
Because of \( j_{-} + j_{+} = J \geq 2 j_{1} \), either \( j_{-} \) or \( j_{+} \) or both are greater or equal \( j_{1} \). Accordingly, we distinguish three cases: for \( j_{-} , j_{+} \geq j_{1} \), the definitions of \( j_{1} \) and \( \hat{v} \) imply
\[ \frac{ \ln( \lVert \Cocyc( J , \Shift^{-j_{-}} \LWord{i} ) \rVert ) }{ J } \geq \frac{ \ln( \lVert \Cocyc( j_{+} , \LWord{i} ) \hat{v} \rVert ) }{ j_{+} } \frac{ j_{+} }{ J } - \frac{ \ln( \lVert \Cocyc( -j_{-} , \LWord{i} ) \hat{v} \rVert ) }{ j_{-} } \frac{ j_{-} }{ J } \geq \frac{ c }{ 2 } \, .\]
For \( j_{-} < j_{1} \) and \( j_{+} \geq j_{1} \), we have \( \frac{ j_{+} }{ J } \geq \frac{ 1 }{ 2 } \) and \( j_{-} < j_{1} \). Because of \( J \geq j_{2} \) this implies
\[ \frac{ \ln( \lVert \Cocyc( J , \Shift^{-j_{-}} \LWord{i} ) \rVert ) }{ J } \geq \frac{ \ln( \lVert \Cocyc( j_{+} , \LWord{i} ) \hat{v} \rVert ) }{ j_{+} } \frac{ j_{+} }{ J } - \frac{ \ln( \lVert \Cocyc( -j_{-} , \LWord{i} ) \hat{v} \rVert ) }{ J } > \frac{ c }{ 8 } \, . \]
For \( j_{-} \geq j_{1} \) and \( j_{+} < j_{1} \), similar considerations yield
\[ \frac{ \ln( \lVert \Cocyc( J , \Shift^{-j_{-}} \LWord{i} ) \rVert ) }{ J } \geq \frac{ \ln( \lVert \Cocyc( j_{+} , \LWord{i} ) \hat{v} \rVert ) }{ J } - \frac{ \ln( \lVert \Cocyc( -j_{-} , \LWord{i} ) \hat{v} \rVert ) }{ j_{-} } \frac{ j_{-} }{ J } > \frac{ c }{ 8 } \, . \]
The above bounds hold for all \( i \in \Set{ 1 , \ldots , \LNr } \). Hence we deduce from Proposition~\ref{prop:CocycInvCut} and \( \Restr{ \InfWord }{ 0 }{ J-1 } = \Restr{ \LWord{i} }{ -j_{-}+1 }{ j_{+} } \) that there exists a constant \( C \) with
\[ \frac{ \ln( \lVert \Cocyc( J , \InfWord ) \rVert ) }{ J } \geq \frac{ \ln( \lVert \Cocyc( J , \Shift^{-j_{-}} \LWord{i} ) \rVert ) }{ J } - \frac{ C }{ J } \geq \frac{ c }{ 8 } - \frac{ C }{ J } \, . \]
Since \( C \) and \( c \) are independent of \( \InfWord \), we obtain \( \frac{ \ln( \lVert \Cocyc( J , \InfWord ) \rVert ) }{ J } \geq \delta > 0 \) for all \( \InfWord \in \Subshift \) and all sufficiently large \( J \), which proves the claim.
\end{proof}

The main application of the previous theorem is to transfer matrix cocycles, see Example~\ref{exmpl:TrMod}. Their uniformity on \LCond{}-subshifts has consequences for the spectrum of Jacobi operators:

\begin{thm}
\label{thm:LCondCantorSpec}
Let \( \Subshift \) be an \LCond{}-subshift. Let \( f \colon \Subshift \to \RR \setminus \Set{ 0 } \) and \( g \colon \Subshift \to \RR \) be locally constant functions such that \( ( \NDig , \Dig ) \) is aperiodic. Let \( (\Jac)_{\InfWord \in \Subshift} \) be the family of Jacobi operators that is defined by \( \NDig \) and \( \Dig \). Then the spectrum \( \Spec = \sigma( \Jac_{\InfWord} ) \) is a Cantor set of Lebesgue measure zero.
\end{thm}

\begin{proof}
Since \( \NDig \) and \( \Dig \) are locally constant, the cocycle associated to the modified transfer matrix map
\[ \TrMod \colon \Subshift \to \SL \, ,\; \InfWord \mapsto \begin{pmatrix} \frac{ E - \Dig( \Shift \InfWord ) }{ \NDig( \Shift^{2} \InfWord ) } & - \frac{ 1 }{ \NDig( \Shift^{2} \InfWord ) } \\ \NDig( \Shift^{2} \InfWord ) & 0 \end{pmatrix} \]
is locally constant as well. By Theorem~\ref{thm:LCondUniform} it is therefore uniform. This implies uniformity of the (unmodified) transfer matrices \( \TrMat \) (\cite{BeckPogo_SpectrJacobi}, see Proposition~\ref{prop:UnifTrMatTrMod} above). As \( \NDig \) and \( \Dig \) take only finitely many values (see Remark~\ref{rem:LocConstFinVal}), we can apply \cite[Theorem~3]{BeckPogo_SpectrJacobi} (see Proposition~\ref{prop:UnifCantorSpec}), which yields the desired result.
\end{proof}

\begin{exmpl}
\label{exmpl:SturmUnif}
Sturmian subshifts (see Appendix~\ref{app:Sturm} for a definition and basic properties) satisfy \LCond{} by Theorem~\ref{thm:SturmLCond}. Thus every locally constant cocycle is uniform, which is a well-known fact (see \cite[Theorem~4]{Lenz_ErgodTheo1dimSO}). In particular, aperiodic Jacobi operators over Sturmian subshifts have Cantor spectrum of Lebesgue measure zero. The corresponding result for Schrödinger operators was already shown in \cite[corollary of the main result]{BellIochScopTest_SpecProp}.
\end{exmpl}

\begin{exmpl}
Simple Toeplitz subshifts satisfy \LCond{} by Theorem~\ref{thm:SimpToepLCond}. Thus every aperiodic Jacobi operator over a simple Toeplitz subshift has Cantor spectrum of Lebesgue measure zero by Theorem~\ref{thm:LCondCantorSpec}. This strengthens the result from Corollary~\ref{cor:BoshSTCantor}, which only applied to those simple Toeplitz subshifts that satisfy the Boshernitzan condition.
\end{exmpl}

\begin{rem}
For Schrödinger operators (that is, \( \NDig \equiv 1 \)) with a function \( \Dig \) that only depends on \( \InfWord(0) \), uniformity of the associated cocycles over simple Toeplitz subshifts was studied by Liu and Qu in \cite{LiuQu_Simple}. In their analysis as well as in ours, subadditivity is used to establish the existence of certain limits. The result in \cite{LiuQu_Simple} then follows from a careful investigation of the trace map and the set of non-escaping energies. Here we are able to cover the more general class of locally constant Jacobi operators by proving uniformity from the characterisation by Lenz (see Proposition~\ref{prop:LowBdUnif}), a Gordon-type argument (Proposition~\ref{prop:GordonCocyc}) and the result of Ruelle (Proposition~\ref{prop:Ruelle}). 
\end{rem}

\appendix % from here on, all chapters are called appendices 

%%%%%%%%%%%%%%%%%%%%%%%%
%%%%%%%%%%%%%%%%%%%%%%%%
%%%%%%%%%%%%%%%%%%%%%%%%
\chapter{Sturmian subshifts}
\label{app:Sturm}

\pagestyle{NoSecruled} % use different page style (for chapter without sections)

In this appendix we review the notion of Sturmian subshifts. Similar to simple Toeplitz subshifts, they receive much interest since their elements are very structured. We recall several equivalent definitions for Sturmian subshifts and mention results on the spectrum of Schrödinger and Jacobi operators. Afterwards we discuss combinatorial properties, which we use in Subsection~\ref{subsec:LCondSturm} to show that Sturmian subshifts satisfy the leading sequence condition. While this appendix focuses on properties that are related to the topic of this thesis, much more is known about Sturmian subshifts. Additional information can for instance be found in \cite[Chapter~2]{Loth_AlgCombWords} and \cite[Chapter~6]{Fogg_Substitutions}.

Sturmian sequences were introduced in \cite[Section~2]{MorseHedl_Sturmian} via a balance property. The modern definition (see \cite{CovHed_MinBlockGr}, \cite{Cov_MinBlockGr2}) via the complexity function is slightly more restrictive since it excludes eventually periodic words:

\begin{defi}
A one-sided infinite word \( \RWord \in \Alphab^{\NN} \) is called \emph{Sturmian}\index{Sturmian!word}\index{word!Sturmian} if its complexity is given by \( \Comp( L ) = L+1 \) for all \( L \in \NN_{0} \).
\end{defi}

For a Sturmian word \( \RWord \in \Alphab^{\NN} \), we define the associated \emph{Sturmian subshift}\index{subshift!Sturmian}\index{Sturmian!subshift} by \( \Subshift( \RWord ) \DefAs \Set{ \InfWord \in \Alphab^{\ZZ} : \Langu{ \InfWord } \subseteq \Langu{ \RWord } } \). It is well-known that Sturmian subshifts are minimal (this follows from the minimality of irrational rotations, which are discussed below) and uniquely ergodic (this follows for example by \cite[Theorem~1.5]{Boshernitzan_UniErgodic} from \( \limsup_{ L \to \infty } \frac{ \Comp( L ) }{ L } = 1 < 3 \), see Remark~\ref{rem:STCompUniErg}). Note that minimality implies that every \( \InfWord \in \Subshift( \RWord ) \) has complexity \( \Comp( L ) = L+1 \).

\begin{rem}
Sturmian subshifts are always defined over an alphabet \( \Alphab = \Set{ a, b } \) of two letters, since \( \Card{ \Alphab } = \Comp( 1 ) = 2 \) holds by definition.
\end{rem}

\begin{rem}
The class of Sturmian subshifts is disjoint from the class of simple Toeplitz subshifts: for periodic simple Toeplitz subshifts, this is immediately clear. For aperiodic simple Toeplitz subshifts, we have \( \Card{ \AlphabEv } \geq 2 \) (Proposition~\ref{prop:STAperiod}) and
\[ \Comp( \Length{ \PBlock{k-1} } + 1 ) = ( \Card{ \AlphabEv } - 1) \cdot ( \Length{ \PBlock{k-1} } + 1 ) + \Length{ \PBlock{k-2} } + 1 > \Length{ \PBlock{k-1} } + 2 \]
for all sufficiently large \( k \) (Corollary~\ref{cor:CompLargeL}). This contradicts \( \Comp( L ) = L+1 \).
\end{rem}

If \( \InfWord \in \Alphab^{\ZZ} \) is periodic with period \( P \in \NN \), then it has complexity \( \Comp( L ) \leq P \) for all \( L \geq P \). It is therefore not a Sturmian word and also not in a Sturmian subshift. On the other hand, the famous Morse/Hedlund theorem states that non-periodic words cannot have lower complexity than Sturmian words:

\begin{prop}[{\cite[Theorem~7.4]{MorseHedl_SymbDyn}}]
\label{prop:MHThm}
A non-periodic \( \InfWord \in \Alphab^{\ZZ} \) has complexity \( \Comp( L ) \geq L - 1 + \Card{\Alphab} \).
\end{prop}

Since the complexity is a way to measure the order in a word, Sturmian words can be considered as the most ordered among the aperiodic words. While the above definition highlights the simplicity of Sturmian words, it does not provide a way to generate them. However, there are many other, equivalent ways to describe Sturmian words and subshifts. To mention just a few, Sturmian subshifts can be characterised by
\begin{tightitemize}
\item{mechanical words and cutting sequences (\cite[Subsection~2.1.2]{Loth_AlgCombWords}),}
\item{balanced words (\cite[Theorem~6.1.8]{Fogg_Substitutions}, \cite[Theorem~2.1.13]{Loth_AlgCombWords}),}
\item{return words (\cite[main theorem]{Vuill_RetSturm}).}
\end{tightitemize}
While all these notions are beyond the scope of this appendix, we focus now on yet another characterisation, namely by rotation maps. For \( \alpha \in [0 , 1) \), the \emph{rotation by \( \alpha \)}\index{rotation on the unit circle} on the unit circle is defined by
\[ \Rot{\alpha} \colon [0 , 1) \to [0 , 1) \, , \; x \mapsto ( x + \alpha ) \bmod 1 \, . \]
The rotation and an initial value \( x_{0} \in [0 , 1 ) \) generate an infinite word \( \RWord_{\alpha ,\, x_{0}} \in \Alphab^{\NN} \) by
\begin{equation}
\label{eqn:RotSeqSturm}
\RWord_{\alpha ,\, x_{0}}( j ) \DefAs \begin{cases}
b & \text{if } \Rot{\alpha}^{j}( x_{0} ) \in [0, 1-\alpha ) \\
a & \text{if } \Rot{\alpha}^{j}( x_{0} ) \in [1-\alpha , 1)
\end{cases} \, .
\end{equation}
For irrational \( \alpha \), the language of \( \RWord_{\alpha ,\, x_{0}} \) is independent of \( x_{0} \) (\cite[Proposition~2.1.18]{Loth_AlgCombWords}). Thus every \( \alpha \in [ 0 , 1 ) \setminus \QQ \) has an associated subshift \( \Subshift_{\alpha} \DefAs \Set{ \InfWord \in \Alphab^{\ZZ} : \Langu{ \InfWord } \subseteq \Langu{ \RWord_{\alpha ,\, x_{0}} } } \).

\begin{prop}[{\cite[Theorem~2.1.13]{Loth_AlgCombWords}}]
A one-sided infinite word \( \RWord \in \Alphab^{\NN} \) is Sturmian if and only if \( \RWord = \RWord_{\alpha ,\, x_{0}} \) holds for some irrational \( \alpha \in [0 , 1 ) \) and some \( x_{0} \in [0 , 1 ) \).
\end{prop}

\begin{exmpl}[Fibonacci subshift]
Let \( \alpha \DefAs \frac{ \sqrt{5} - 1 }{ 2 } \) denote the inverse of the golden ratio. The associated Sturmian subshift \( \Subshift_{\alpha} \) can also be generated from the substitution \( \sigma_{\text{Fib}} \colon a \mapsto \Word{ a }{ b } \, , \; b \mapsto a \) (see for example \cite[Proposition~5.4.9]{Fogg_Substitutions}), which is closely related to the Fibonacci sequence. Therefore, \( \Subshift_{\alpha} \) is known as the \emph{Fibonacci subshift}\index{Fibonacci subshift}\index{subshift!Fibonacci}. It is probably the most-studied subshift and there exists a large amount of literature both on its combinatorial properties (for example \cite{MigPir_RepeFib}, \cite{Drou_PaliFib}, \cite{HuaWen_RetWordsFib}) and on the spectrum of Schrödinger operators on it (see the remark below).
\end{exmpl}

\begin{rem}
\label{rem:SturmSpecSO}
As discussed in Chapter~\ref{chap:JOSimpToep}, the spectral properties of Schrödinger operators are closely related to the combinatorics of the underlying potential. It is therefore not surprising that Schrödinger operators on Sturmian subshifts (and, as a prototype, on the Fibonacci subshift) are particularly well-studied, see \cite{DGY_FibAlles} and the references therein. For the Fibonacci subshift, it was proved by Sütő (\cite{Suto_QuasiPerSO}, \cite{Suto_SingContSpect}) that the spectrum of Schrödinger operators is purely singular continuous and a Cantor set of Lebesgue measure zero. The Hausdorff dimension of the spectrum has been determined as well (\cite{DEGT_HDimFib}). Cantor spectrum (\cite{BellIochScopTest_SpecProp}) which is purely singular continuous (\cite{DamKillipLenz_UniSpectProp3}) could also be shown for all Sturmian subshifts. While the aforementioned results rely on analysing the trace map, Cantor spectrum can also be concluded from the Boshernitzan condition. The condition's validity for all Sturmian subshifts was proved in \cite[Theorem~4]{DamLenz_BoshLowCompl}. This implies uniformity of the transfer matrix cocycle, see Proposition~\ref{prop:BoshCantor}. More generally, it was shown in \cite[Theorem~4]{Lenz_ErgodTheo1dimSO} that every locally constant cocycle on a Sturmian subshift is uniform. It is this result that we reproduce in Example~\ref{exmpl:SturmUnif} by a different method, namely by our theory of \LCond{}-subshifts.
\end{rem}

To prove the leading sequence condition for Sturmian subshifts (Subsection~\ref{subsec:LCondSturm}), we use that they exhibit a block structure similar to simple Toeplitz subshifts (see Equation~(\ref{eqn:DecompPBlock})). It emerges from rotation maps: for every \( \alpha \in [  0 , 1 ) \setminus \QQ \), there exists a unique, non-terminating continued fraction expansion
\[ \alpha = \frac{ 1 }{ n_{1} + \frac{ 1 }{ n_{2} + \frac{ 1 }{ n_{3} + \ldots } } } \, , \]
which defines a sequence \( ( n_{k})_{k \in \NN} \) of natural numbers. Recursively we now define a sequence of words by:
\begin{equation}
\label{eqn:StrumRecur}
s_{0} \DefAs b \;\; , \quad s_{1} \DefAs \Word{ b^{n_{1}-1} }{ a } \;\; , \quad s_{k+1} \DefAs \Word{ s_{k}^{n_{k+1}} }{ s_{k-1} } \, .
\end{equation}
By \( \PBlock{k} \DefAs \Restr{ s_{k} }{ 1 }{ \Length{ s_{k} }-2 } \), with \( k \geq 2 \), we denote the word \( s_{k} \) with its last two letters omitted. Since \( \PBlock{ k } \) is a prefix of  \( \PBlock{k+1} \), the one-sided limit \( \PBlock{\infty} \DefAs \lim_{k \to \infty} \PBlock{k} \in \Alphab^{\NN} \) exists. In fact it is precisely the word that is generated by the rotation \( \Rot{\alpha} \):

\begin{prop}[{\cite[Proposition~1]{BellIochScopTest_SpecProp}}, {\cite[Proposition~2.2.24]{Loth_AlgCombWords}}]
\label{prop:SturmDefBlocks}
For every \( \alpha \in [  0 , 1 ) \setminus \QQ \), we have \( \RWord_{\alpha ,\, 0} = \PBlock{\infty} \). In particular this implies \( \Subshift( \PBlock{\infty} ) = \Subshift_{\alpha} \).
\end{prop}

 In the next proposition we collect some well-known combinatorial properties of the words \( s_{k} \) and \( \PBlock{k} \).

\begin{prop}
\label{prop:SturmPali}
For every \( k \geq 2 \) the following properties hold:
\begin{tightenumerate}
\item{The word \( s_{2k} \) ends with \( \Word{ a }{ b } \), that is, we have \( s_{2k} = \Word{ \PBlock{2k} }{ a }{ b } \).}
\item{The word \( s_{2k+1} \) ends with \( \Word{ b }{ a } \), that is, we have \( s_{2k+1} = \Word{ \PBlock{2k+1} }{ b }{ a } \).}
\item{We have \( \Word{ s_{k} }{ \PBlock{k+1} } = \Word{ s_{k+1} }{ \PBlock{k} } \).}
\item{The word \( \PBlock{k} \) is a palindrome.}
\item{For every \( j \) with \( 1 \leq j < k \), the word \( s_{k} \) can be decomposed into occurrences of \( s_{j} \) and  \( s_{j-1} \), such that the occurrences of \( s_{j-1} \) are isolated.}
\end{tightenumerate}
\end{prop}

\begin{proof}
We proceed by induction. Using Equation~(\ref{eqn:StrumRecur}) it is easy to check the claims for small \( k \in \NN \). Moreover \( s_{k} \) is a suffix of \( s_{k+2} \), which immediately implies (a) and (b). Now assume that (c) holds for some \( k \). Again by Equation~(\ref{eqn:StrumRecur}) we have \( \PBlock{ k+2 } = \Word{ s_{k+1}^{n_{k+2}} }{ \PBlock{k} } \), which yields
\[\Word{ s_{k+1} }{ \PBlock{k+2} } =\Word{  s_{k+1} }{ s_{k+1}^{n_{k+2}} }{ \PBlock{k} } = \Word{ s_{k+1}^{n_{k+2}} }{  s_{k+1} }{ \PBlock{k} } = \Word{ s_{k+1}^{n_{k+2}} }{  s_{k} }{ \PBlock{k+1} } = \Word{  s_{k+2} }{ \PBlock{k+1} } \, . \]
To prove (d), assume that \( \PBlock{k} \) and \( \PBlock{k-1} \) are palindromes. Let \( v_{k} \) denote the last two letters of \( s_{k} \). Because of (c) we can decompose \( \PBlock{ k+1 } = \Word{ s_{k}^{n_{k+1}} }{ \PBlock{k-1} } \) in two different ways:
\[
( \Word{ \PBlock{k} }{ v_{k} } )^{n_{k+1}} \, \PBlock{k-1} = \Word{ s_{k}^{n_{k+1}} }{ \PBlock{k-1} } = \PBlock{k-1} \, ( \Word{ v_{k-1} }{ \PBlock{k} } )^{n_{k+1}} \, .\]
Since \( \PBlock{k} \) and \( \PBlock{k-1} \) are palindromes and \( \Rev{ v_{k-1} } = v_{k} \) holds, the claim follows. Finally note that for every fixed \( k \), property (e) holds for \( j=k-1 \) by Equation~(\ref{eqn:StrumRecur}). Inductively, (e) can be proved for all \( j \) by decomposing \( s_{j} \) into \( s_{j-1} \) and \( s_{j-2} \), which finishes the proof.
\end{proof}

\begin{cor}
\label{cor:LWordInSturm}
Let \( \Rev{\PBlock{\infty}} \in \Alphab^{ -\NN } \) be defined by \( \Rev{\PBlock{\infty}}( -j ) = \PBlock{\infty}( j ) \). Then the two-sided infinite words \( \LWord{1} \DefAs \Word{ \Rev{\PBlock{\infty}} }{ a }{\Origin{ b }}{ \PBlock{\infty} } \) and \( \LWord{2} \DefAs \Word{ \Rev{\PBlock{\infty}} }{ b }{\Origin{ a }}{ \PBlock{\infty} } \) are elements of the Sturmian subshift \( \Subshift( \PBlock{\infty} ) \).
\end{cor}

\begin{proof}
Since all \( \PBlock{k} \) are palindromes and \( \PBlock{k-1} \) is a prefix of \( \PBlock{k} \), it is a suffix of \( \PBlock{k} \) as well. Combining Equation~(\ref{eqn:StrumRecur}) with Proposition~\ref{prop:SturmPali} we obtain the following decompositions:
\begin{alignat*}{5}
s_{2k+1} &= \Word{ \ldots }{ s_{2k} }{ s_{2k-1} } &&= \Word{ \ldots }{ \PBlock{2k} }{ a }{ b }{ \PBlock{2k-1} }{ b }{ a }&& = \Word{ \ldots }{ \PBlock{2k-1} }{ a }{ b }{ \PBlock{2k-1} }{ \ldots}\; , \\
s_{2k+2} &= \Word{ \ldots }{ s_{2k+1} }{ s_{2k} } &&= \Word{ \ldots }{ \PBlock{2k+1} }{ b }{ a }{ \PBlock{2k} }{ a }{ b }&& = \Word{ \ldots }{ \PBlock{2k} }{ b }{ a }{ \PBlock{2k} }{ \ldots} \; .
\end{alignat*}
Thus \( \Word{ \PBlock{k} }{ a }{ b }{ \PBlock{k} } \) and \( \Word{ \PBlock{k} }{ b }{ a }{ \PBlock{k} } \) are in the language of the subshift for all \( k \). In addition the palindromicity of \( \PBlock{k} \) implies that the ``left limit'' \( \lim_{k \to \infty} \PBlock{k} \in \Alphab^{-\NN} \) is equal to \( \Rev{\PBlock{\infty}} \).
\end{proof}

%%%%%%%%%%%%%%%%%%%%%%%%
%%%%%%%%%%%%%%%%%%%%%%%%
%%%%%%%%%%%%%%%%%%%%%%%%
\chapter[The family of Grigorchuk's groups and their Schreier graphs][Grigorchuk's groups and Schreier graphs]{The family of Grigorchuk's groups and their Schreier graphs}
\label{app:SelfSimGr}

\pagestyle{ruled} % restore standard style (was changed for Appendix A)

This appendix gives a brief overview over certain self-similar groups and their connection to simple Toeplitz subshifts. It is by no means an exhaustive treatment of the numerous topics that are related to self-similar groups. More information and additional details can for example be found in \cite{Mur_IntroSelfSimGr}, \cite{GNS_GrSpecSets} or \cite{BGS_BranchGr}. Here we merely give the basic definitions and highlight some results that concern the relation to simple Toeplitz subshifts. Essentially the connection between these two topics is caused by similarities between the subshifts on the one hand and certain graphs that are associated to self-similar groups on the other hand, the so-called Schreier graphs. Because of these similarities, the Jacobi operators on simple Toeplitz subshifts are unitarily equivalent to Laplacians on the Schreier graphs. Thus our result on Cantor spectrum, see Theorem~\ref{thm:LCondUniform}, yields new insight into self-similar groups as well. The aim of this appendix is to sketch the background for this application of our spectral result. It also motivates our definition of the (generalised) Grigorchuk subshift. Nevertheless the content of this appendix is entirely supplementary and not necessary for the understanding of any of the other chapters.

\section{The family of Grigorchuk's groups}

Self-similar groups are groups which act on a regular tree in a self-similar way. More precisely, let \( X \) be a finite set of cardinality \( r \geq 2 \). The \emph{\(r\)-regular rooted tree}\index{tree} is an undirected graph whose vertex set is \( \cup_{k = 0}^{\infty} X^{k} \). The root is the empty word \( \epsilon \). Between two vertices there is an edge if and only if they are of the form \( u \) and \( \Word{ u }{ a } \) with \( u \in \cup_{k = 0}^{\infty} X^{k} \) and \( a \in X \). As an example Figure~\ref{fig:GrigGrActionA} shows the binary tree (that is, \( r=2 \)) with \( X = \Set{ 0 , 1 } \). The background shading and the arrows can be ignored for now. A bijective map from the vertex set of a regular rooted tree to itself is called an \emph{automorphism of the tree} if neighbouring vertices are mapped to neighbouring vertices. In particular every automorphism preserves the root.

\begin{defi}
A group \( G \) of automorphisms on a regular rooted tree is called a \emph{self-similar group}\index{self-similar group}\index{group!self-similar} if, for every \( g \in G \) and every \( a \in X \), there exist elements \( h \in G \) and \( b \in X \) such that \( g( a u ) = b h( u ) \) holds for all \( u \in \cup_{k = 0}^{\infty} X^{k} \).
\end{defi}

In other words, for every group element \( g \) and every \( a \in X \), there is a group element \( h \) such that \( g \) acts on the subtree below \( a \) in the same way as \( h \) acts on the whole tree. Since self-similar groups often have interesting properties such as intermediate growth or non-elementary amenability, they have been heavily studied in the last decades, see for instance \cite{BGN_FractalGrSets}, \cite{BGS_BranchGr}, \cite{Nekrash_SelfSimGr} and the references therein. An important example is \emph{Grigorchuk's group}\index{Grigorchuk!group}\index{group!Grigorchuk's}, a self-similar group on the binary tree that was introduced in \cite{Grig_BurnsideRuss}. We recall its definition in the example below. More information can for instance be found in \cite{Grig_SolvUnsolv}.

\begin{exmpl}[Grigorchuk's group]
Consider the binary tree with \(  X = \Set{ 0, 1 } \). Grigorchuk's group is generated by four elements \( a \), \( b \), \( c \) and \( d \) which satisfy
\begin{align*}
a( \Word{ 0 }{ u } ) &= \Word{ 1 }{ u } \, , & b( \Word{ 0 }{ u } ) &= \Word{ 0 }{ a( u ) } \, , & c( \Word{ 0 }{ u } ) &= \Word{ 0 }{ a( u ) } \, , & d( \Word{ 0 }{ u } ) &= \Word{ 0 }{ u } \, , \\
a( \Word{ 1 }{ u } ) &= \Word{ 0 }{ u } \, , & b( \Word{ 1 }{ u } ) &= \Word{ 1 }{ c( u ) } \, , & c( \Word{ 1 }{ u } ) &= \Word{ 1 }{ d( u ) } \, , & d( \Word{ 1 }{ u } ) &= \Word{ 1 }{ b( u ) } \, .
\end{align*}
We see that the generator \( a \) maps each vertex to the vertex where the first letter is flipped, that is, it interchanges the subtree below \( 0 \) with the subtree below \( 1 \) (see Figure~\ref{fig:GrigGrActionA}). Clearly the action of \( a \) on \( \cup_{k = 0}^{\infty} X^{k} \) is an involution.

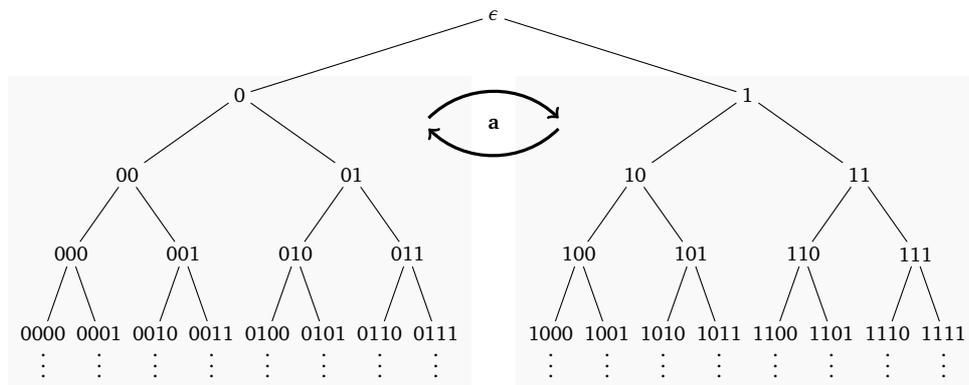
\begin{figure}
\centering
\scriptsize
\begin{tikzpicture}
[level distance=30, level 1/.style={sibling distance=190}, level 2/.style={sibling distance=84}, level 3/.style={sibling distance=42}, level 4/.style={sibling distance=21},
every node/.style ={outer sep=0pt, inner sep=2pt, rectangle},
every on background layer/.style={every node/.style ={fill=gray!5}}]
% with the foreach part, we create N children instead of one; for each of them, the code following in { } is executed
\node{\( \epsilon \)}
child foreach \a in {0,1}{
	node(\a){\a}
	child foreach \b in {0,1}{
		node {\a \b }
		child foreach \c in {0,1}{
			node {\a \b \c}
			child foreach \d in {0,1}{
				node(\a \b \c \d){\a \b \c \d}
				node (\a \b \c \d b) [below=-5pt of {\a \b \c \d}] {\( \vdots \)}
				}}}};
\begin{scope}[on background layer]
	\node(background0) [fit=(0) (0000) (0111) (0000b) (0111b)] {};
	\node(background1) [fit=(1) (1000) (1111) (1000b) (1111b)] {};
\end{scope}
\node (arrow0) [below left = 0.8 of background0.north east] {};
\node (arrow1) [below right = 0.8 of background1.north west] {};
\draw [very thick, ->] (arrow0) to [out=45,in=135] (arrow1);
\draw [very thick, ->] (arrow1) to [out=225,in=315] (arrow0);
\path (arrow0) -- (arrow1) node[midway]{\( \mathbf{a} \)};
\end{tikzpicture}
\normalsize
\caption{The  generator \( a \) of Grigorchuk's group, acting on the binary tree.\label{fig:GrigGrActionA}}
\end{figure}

\begin{figure}
\centering
\footnotesize
%
% action of b
\begin{minipage}[b]{0.6\textwidth}
\centering
\begin{tikzpicture}
[level distance=30, level/.style={sibling distance=120/(2^((#1)-1))}, 
every node/.style ={outer sep=0pt, inner sep=2pt},
vertex/.style ={circle, inner sep=1pt, fill},
every on background layer/.style={every node/.style ={rectangle, fill=gray!5}}]
\node[vertex]{}
	child foreach \a in {0,1}{
		node[vertex](\a){}
		child foreach \b in {0,1}{
			node[vertex](\a \b){}
			child foreach \c in {0,1}{
				node[vertex](\a \b \c){}
				child foreach \d in {0,1}{
					node[vertex](\a \b \c \d){}
					node (\a \b \c \d b) [below=-3pt of {\a \b \c \d}] {\( \vdots \)}
					}}}};
\begin{scope}[on background layer]
% define shaded rectangle hull of fit=... arguments
	\node(background00) [fit=(00) (0000b) (0011b)] {};
	\node(background01) [fit=(01) (0100b) (0111b)] {};
	\node(background100) [fit=(100) (1000b) (1001b)] {};
	\node(background101) [fit=(101) (1010b) (1011b)] {};
	\node(background110) [fit=(110) (1100b) (1101b), label=] {\\[4 ex] \(\mathbf{id}\)};
	\node(background111) [fit=(111) (1110b) (1111b)] {\\[4 ex] \(\mathbf{b}\)};
\end{scope}
% define end points for arrows
\node (arrow00) [below left = 0.4 of background00.north east] {};
\node (arrow01) [below right = 0.4 of background01.north west] {};
\node (arrow100) [below left = 0.2 of background100.north east] {};
\node (arrow101) [below right = 0.2 of background101.north west] {};
% draw arrows
\draw [thick, ->] (arrow00) to [out=45,in=135] (arrow01);
\draw [thick, ->] (arrow01) to [out=225,in=315] (arrow00);
\path (arrow00) -- (arrow01) node[midway]{\( \mathbf{a} \)};
\draw [thick, ->] (arrow100) to [out=45,in=135] (arrow101);
\draw [thick, ->] (arrow101) to [out=225,in=315] (arrow100);
\path (arrow100) -- (arrow101) node[midway]{\( \mathbf{a} \)};
\end{tikzpicture}
\subcaption{The action of \( b \) on the binary tree.\label{subfig:GrigGrActionB}}
\end{minipage}
\hfill
%
% portrait of b
\begin{minipage}[b]{0.3\textwidth}
\centering
\begin{tikzpicture}
[level distance = 15, sibling distance=20, 
every node/.style ={circle, outer sep=0pt, inner sep=1pt, fill, minimum height=0pt,},
every label/.style ={label distance=0.5ex, draw=none, fill=none, rectangle}]
\node{}
child{
	node[label=below:\( a \)]{} }
child{ node{}
	child{ node[label=below:\( a \)]{} }
	child{ node{}
		child{ node[label=below:\( id \)]{} }
		child{ node{}
			child{ node[label=below:\( a \)]{} }
			child{ node{}
				child{ node[label=below:\( a \)]{} }
				child{ node{}
					child{ node[label=below:\( id \)]{} }
					child{ node[label=below:\( \ddots \)]{} }}}}}};
\end{tikzpicture}
\subcaption{Portrait of \( b \).\label{subfig:GrigGrPortraitB}}
\end{minipage}
\normalsize
\caption{The generator \( b \) of Grigorchuk's group.\label{fig:GrigGrB}}
\end{figure}
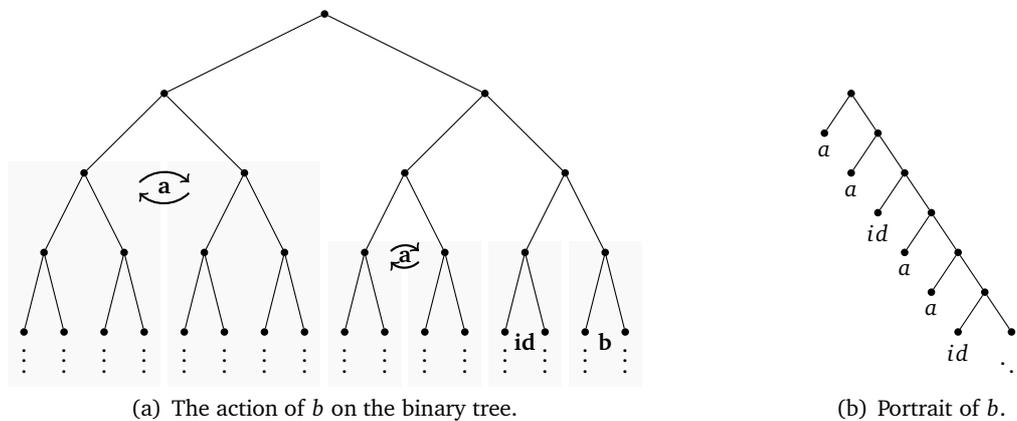

Now we discuss the action of \( b \), see Figure~\ref{subfig:GrigGrActionB}: the action of \( b \) on the subtree below \( 0 \) is the same as the action of \( a \) on the whole tree, that is, \( b \) interchanges the two branches. The action of \( b \) below \( 1 \) is the same as the action of \( c \) on the whole tree, that is, \( b \) acts below \( \Word{ 1 }{ 0 } \) like \( a \) and below \( \Word{ 1 }{ 1 } \) like \( d \). The definition of \( d \) implies that \( b \) acts below \( \Word{ 1 }{ 1 }{ 0 } \) like the identity. Moreover, \( b \) acts below \( \Word{ 1 }{ 1 }{ 1 } \) as it acts on the whole tree, that is, our the considerations have to be repeated for the vertices below \( \Word{ 1 }{ 1 }{ 1 } \). In Figure~\ref{subfig:GrigGrPortraitB} this action is depicted in the so-called \emph{portrait}\index{portrait of a generator} of \( b \). Since the action of \( a \) is an involution and since \( b \) acts on each subtree either like \( a \) or like the identity, the action of \( b \) is an involution as well. Similarly the actions of \( c \) and \( d \) are involutions, too.
\end{exmpl}

In \cite{Grig_DegrOfGrowthRuss}, the definition above was generalised to a whole family \( (G_{\GenSeq}) \) of groups. We review their definition in the next example.

\begin{defi}[\emph{family of Grigorchuk's groups}\index{Grigorchuk!family of groups}\index{group!family of Grigorchuk's groups}]
\label{defi:DefFamGrigGr}
Let \( a \),  \( b \),  \( c \) and \( d \) denote the generators of Grigorchuk's group (see previous example). Let \( \GenSeq = (\GenSeq_{k}) \in \Set{ \pi_{b}, \pi_{c}, \pi_{d} }^{\NN} \) denote a sequence whose values are the maps
\[ \pi_{b} \colon b \mapsto id ,\, c \mapsto a , \, d \mapsto a \;\; , \;\;\; \pi_{c} \colon b \mapsto a ,\, c \mapsto id , \, d \mapsto a \;\; , \;\;\; \pi_{d} \colon b \mapsto a ,\, c \mapsto a , \, d \mapsto id \, . \]
Later we will see that \( (\GenSeq_{k}) \) is related to the coding sequence \( (a_{k} ) \) of a simple Toeplitz subshift.

\begin{figure}
\centering
\footnotesize
%
% portrait of \hat{b}
\begin{minipage}[b]{0.3\textwidth}
\centering
\begin{tikzpicture}
[every node/.style ={inner sep=1pt, minimum height=0pt, fill, circle}, every label/.style ={label distance=0.5ex, draw=none, fill=none, rectangle}, level distance = 20, sibling distance=35]
\node{}
	child{ node[label=below:\( \GenSeq_{1}( b ) \)]{}
	}
	child{ node{}
		child{ node[label=below:\( \GenSeq_{2}( b ) \)]{}
		}
		child{ node{}
			child{ node[label=below:\( \GenSeq_{3}( b ) \)]{} }
			child{ node[label=below:\( \ddots \)]{} }
		}
	};
\end{tikzpicture}
\subcaption{The generator \( \hat{b} \).}
\end{minipage}
\hfill
%
% portrait of \hat{c}
\begin{minipage}[b]{0.3\textwidth}
\centering
\begin{tikzpicture}
[every node/.style ={inner sep=1pt, minimum height=0pt, fill, circle}, every label/.style ={label distance=0.5ex, draw=none, fill=none, rectangle}, level distance = 20, sibling distance=35]
\node{}
	child{ node[label=below:\( \GenSeq_{1}( c ) \)]{}
	}
	child{ node{}
		child{ node[label=below:\( \GenSeq_{2}( c ) \)]{}
		}
		child{ node{}
			child{ node[label=below:\( \GenSeq_{3}( c ) \)]{} }
			child{ node[label=below:\( \ddots \)]{} }
		}
	};
\end{tikzpicture}
\subcaption{The generator \( \hat{c} \).}
\end{minipage}
\hfill
%
% portrait of \hat{d}
\begin{minipage}[b]{0.3\textwidth}
\centering
\begin{tikzpicture}
[every node/.style ={inner sep=1pt, minimum height=0pt, fill, circle}, every label/.style ={label distance=0.5ex, draw=none, fill=none, rectangle}, level distance = 20, sibling distance=35]
\node{}
	child{ node[label=below:\( \GenSeq_{1}( d ) \)]{}
	}
	child{ node{}
		child{ node[label=below:\( \GenSeq_{2}( d ) \)]{}
		}
		child{ node{}
			child{ node[label=below:\( \GenSeq_{3}( d ) \)]{} }
			child{ node[label=below:\( \ddots \)]{} }
		}
	};
\end{tikzpicture}
\subcaption{The generator \( \hat{d} \).}
\end{minipage}
\normalsize
\caption{The generators \( \hat{b} \), \( \hat{c} \) and \( \hat{d} \) of the group \( G_{\InfWord} \).\label{fig:FamGrigGenerat}}
\end{figure}

For every sequence \( (\GenSeq_{k}) \) we now define a group \( G_{\GenSeq} \) of automorphisms of the binary tree. Each group is given by four generators \( a \), \( \hat{b} \), \( \hat{c} \) and \( \hat{d} \), where \( a \) is as above and the latter three are defined as follows: the action of \( \hat{b} \) below \( 0 \) is given by \( \GenSeq_{1}(b) \), below \( \Word{ 1 }{ 0 } \) by \( \GenSeq_{2}(b) \),  below \( \Word{ 1 }{ 1 }{ 0 } \) by \( \GenSeq_{3}(b) \) and so on. Similarly, the action of \( \hat{c} \) is defined by \( ( \GenSeq_{1}(c), \GenSeq_{2}(c), \GenSeq_{3}(c), \ldots ) \) and the action of \( \hat{d} \) by \( ( \GenSeq_{1}(d), \GenSeq_{2}(d), \GenSeq_{3}(d), \ldots ) \), see Figure~\ref{fig:FamGrigGenerat}. Note that the action of every generator is an involution. Note also that the family of groups contains in particular Grigorchuk's group, which corresponds to the periodic sequence \( \GenSeq = (\pi_{d}, \pi_{c}, \pi_{b} , \ldots ) \).
\end{defi}

\section{Schreier graphs and Laplacians}
\label{sec:AppSchrGraphLapl}

There is a connection between Toeplitz subshifts and self-similar groups, which has lately received increased attention, for instance in \cite{MBon_TopoFullGr} and \cite{GLN_SpectraSchreierAndSO}. The connection is based on the groups' Schreier graphs. In the following we recall their definition for the family of Grigorchuk's groups. We also mention a Boshernitzan condition for group actions, which is analogous to our characterisation of the Boshernitzan condition in Corollary~\ref{cor:BoshCondGenGrig}.

\begin{defi}
Let \( G_{\GenSeq} \) be a member of the family of Grigorchuk's groups. The group action on the binary tree \( \cup_{k = 0}^{\infty} \Set{ 0 ,1 }^{k} \) induces an action on the boundary \( \Set{ 0 , 1 }^{\NN} \) as well. We can therefore define the following directed, labelled graph with vertices given by \( ( \cup_{k=0}^{\infty} \Set{ 0 , 1 }^{k} ) \cup \Set{ 0 , 1 }^{\NN} \): there is an edge labelled by \( s \in \Set{ a , \hat{b} , \hat{c} , \hat{d} } \) from a vertex \( u \) to a vertex \( v \) if and only if \( s \) acts by \( s( u ) = v \). The connected component of a vertex \( u \in \Set{ 0 , 1 }^{k} \) is the \( k \)-th level of the tree and is called the \emph{level-\( k \) Schreier graph}\index{Schreier graph!level-\( k \)}\index{graph!Schreier graph!level-\( k \)} of \( G_{\GenSeq} \). The connected component of a vertex \( u \in \Set{ 0 , 1 }^{\NN} \) is its orbit under the group action on \( \Set{ 0 , 1 }^{\NN} \) and is called an \emph{orbital Schreier graph}\index{Schreier graph!orbital}\index{graph!Schreier graph!orbital} of \( G_{\GenSeq} \).
\end{defi}

\begin{rem}
Since \( a \), \( \hat{b} \), \( \hat{c} \) and \( \hat{d} \) act by involutions, we actually consider the edges as undirected. Moreover, for every Schreier graph we choose one of its vertices as the root. Hence we consider Schreier graphs as rooted, undirected, labelled graphs. The group \( G_{\GenSeq} \) acts on the set of these graphs by shifting the root according to the edge labels.
\end{rem}

Recall from Example~\ref{exmpl:defGenGrigSubsh} that the generalised Grigorchuk subshift is described by the limit of the words \( \PBlock{k+1} \DefAs \Word{ \PBlock{k} }{ a_{k+1} }{ \PBlock{k} } \). Note how the word ``at level \( k+1 \)'' is given by two copies of word ``at level \( k \)'', which are connected by \( a_{k+1} \). A similar structure can be found in the Schreier graphs that are associated to the family of Grigorchuk's groups:

\begin{exmpl}[family of Grigorchuk's groups]
Without proofs, we quickly review the shape of the Schreier graphs associated to \( G_{\GenSeq} \). More details can for example be found in \cite{Voro_SchrGrGrigGr} or in \cite[Subsection~1.2]{GLN_SpectraSchreierAndSO}. The level-\( (k+1) \) Schreier graph consists of two copies of the level-\( k \) Schreier graph, connected in a way that is specified by \( \GenSeq_{k} \). Consequently, all level-\(k\) Schreier graphs of \( G_{\GenSeq} \) are ``essentially linear'' and the connection between every second pair of neighbours corresponds to the action of the generator \( a \). Of the remaining connections, every second one corresponds to the action of \( \GenSeq_{1} \). Of those connections still remaining after that, every second one corresponds to the action of \( \GenSeq_{2} \), and so on.

To discuss orbital Schreier graphs, we write \( \overline{0} \DefAs \Word{ 0 }{ 0 }{ 0 }{ \ldots }\) respectively \( \overline{1} \DefAs \Word{ 1 }{ 1 }{ 1 }{ \ldots } \) for a constant one-sided infinite word. The orbital Schreier graph that corresponds to the rightmost ray \( \overline{1} \) in the binary tree is one-sided infinite. It is shown in Figure~\ref{subfig:OneSideSchreier} for the special case of Grigorchuk's group. While all other graphs in the orbit of \( \overline{1} \) under the group action are one-sided as well, all graphs in different orbits are two-sided infinite. They differ in their labels, but as unlabelled graphs they are all isomorphic and look like the graph sketched in Figure~\ref{subfig:TwoSideSchreier}.
\end{exmpl}

\begin{figure}
\centering
\footnotesize
%
% one-sided graph
\begin{minipage}[b]{0.9\textwidth}
\centering
\pgfmathsetmacro{\TikzNodeDist}{1} % horizontal distance between two vertices
\pgfmathsetmacro{\TikzNumNodes}{11} % number of vertices (should be odd)
\pgfmathsetmacro{\TikzLoopH}{0.8} % height of loop
\begin{tikzpicture}
[Vertex/.style ={inner sep=1pt, minimum height=0pt, fill, circle},
Label/.style ={inner sep=2pt,  draw=none, fill=none, rectangle}, 
Loop/.style={to path={ .. controls +(0.4, \TikzLoopH) and +(-0.4, \TikzLoopH) ..node[midway, above, Label]{ \(#1\) } (\tikztotarget)}},
CurveUp/.style={to path={ .. controls +(0.4, 0.15) and +(-0.4, 0.15) .. (\tikztotarget) node[midway, above, Label]{ \(#1\) } }},
CurveDown/.style={to path={ .. controls +(0.4, -0.15) and +(-0.4, -0.15) .. (\tikztotarget) node[midway, below, Label]{ \(#1\) } }}]
% list vertices (Gray code / reflected binary code). Array index starts at 0
\def\VertexName{{"1111","0111","0011","1011","1001","0001","0101","1101","1100","0100","0000","1000","1010","0010","0110","1110"}}
% draw first vertex + edges + labels
\node [Vertex, label={[yshift=-2ex, xshift=-2ex, font=\scriptsize]below left:\(1111 \overline{1}\)}](1) at (\TikzNodeDist , 0){};
\path[draw](1)
	.. controls +(-0.4 , \TikzLoopH) and +(0.4 , \TikzLoopH) .. node[midway, above, Label]{\( b \)}(1)
	.. controls +(-\TikzLoopH , 0.4) and +(-\TikzLoopH, -0.4) .. node[midway, left, Label]{\( c \)} (1)
	.. controls +(-0.4 , -\TikzLoopH) and +(0.4 , -\TikzLoopH) .. node[midway, below, Label]{\( d \)}(1);
% draw remaining vertices + labels
\foreach \x in {2,...,\TikzNumNodes}{
	\pgfmathparse{\VertexName[int(\x-1)]} % evaluates the array, stores result in \pgfmathresult
	\pgfmathsetmacro{\NameVar}{\pgfmathresult} % saves result in auxiliary variable
	\node [Vertex, label={[yshift=-2ex, font=\scriptsize]below:\NameVar \(\overline{1}\)}](\x) at (\x*\TikzNodeDist , 0){};}
% draw a connections
\foreach \end [evaluate=\end as \start using \end-1] in {2,4,...,\TikzNumNodes}{
	\path[draw] (\start) edge node[Label, above] {\( a \)} (\end);}
% draw d-as-loop connections
\foreach \end [evaluate=\end as \start using int(\end-1)] in {3,7,...,\TikzNumNodes}{
	\path[draw]	(\end)	edge[Loop=d] (\end)
				(\start)	edge[Loop=d] (\start)
						edge[CurveUp=b] (\end)
						edge[CurveDown=c] (\end);}
% draw c-as-loop connections
\path[draw] (5) edge[Loop=c] (5)
	(4) 	edge[Loop=c] (4)
		edge[CurveUp=b] (5)
		edge[CurveDown=d] (5);
% draw b-as-loop connections
\path[draw] (9) edge[Loop=b] (9)
	(8) 	edge[Loop=b] (8)
		edge[CurveUp=c] (9)
		edge[CurveDown=d] (9);
% draw end of graph
\node (End) [draw=none, fill=none] at (\TikzNumNodes*\TikzNodeDist + \TikzNodeDist/2 , 0){};
\draw[dotted, thick] (\TikzNumNodes)--(End);
\end{tikzpicture}
\vspace{-1ex}
\subcaption{The one-sided infinite graph associated to \( \overline{1} \).\label{subfig:OneSideSchreier}}
\vspace{2ex}
\end{minipage}
%
% two-sided graph
\begin{minipage}[b]{0.9\textwidth}
\centering
\pgfmathsetmacro{\TikzNodeDist}{1} % horizontal distance between two vertices
\pgfmathsetmacro{\TikzNumNodes}{11} % number of vertices (should be odd)
\pgfmathsetmacro{\TikzLoopH}{0.8} % height of loop
\begin{tikzpicture}
[Vertex/.style ={inner sep=1pt, minimum height=0pt, fill, circle},
Label/.style ={inner sep=2pt,  draw=none, fill=none, rectangle}, 
Loop/.style={to path={ .. controls +(0.4, \TikzLoopH) and +(-0.4, \TikzLoopH) ..node[midway, above, Label]{ \(#1\) } (\tikztotarget)}},
CurveUp/.style={to path={ .. controls +(0.4, 0.15) and +(-0.4, 0.15) .. (\tikztotarget) node[midway, above, Label]{ \(#1\) } }},
CurveDown/.style={to path={ .. controls +(0.4, -0.15) and +(-0.4, -0.15) .. (\tikztotarget) node[midway, below, Label]{ \(#1\) } }}]
% list vertices (Gray code / reflected binary code). Array index starts at 0
\def\VertexName{{"1001","0001","0101","1101","1100","0100","0000","1000","1010","0010","0110"}}
% draw vertices and labels
\foreach \x in {1,...,\TikzNumNodes}{
	\pgfmathparse{\VertexName[int(\x-1)]} 
	\pgfmathsetmacro{\NameVar}{\pgfmathresult}
	\node [Vertex, label={[yshift=-2ex, font=\scriptsize]below:\NameVar \(\overline{0}\)}](\x) at (\x*\TikzNodeDist , 0){};}
% draw start of graph
\node (Start) [draw=none, fill=none] at (\TikzNodeDist/2 , 0){};
\draw[dotted, thick] (Start)--(1);
% draw a connections
\foreach \end [evaluate=\end as \start using \end-1] in {2,4,...,\TikzNumNodes}{
	\path[draw] (\start) edge node[Label, above] {\( a \)} (\end);}
% draw d-as-loop connections
\foreach \end [evaluate=\end as \start using int(\end-1)] in {3,7,...,\TikzNumNodes}{
	\path[draw]	(\end)	edge[Loop=d] (\end)
				(\start)	edge[Loop=d] (\start)
						edge[CurveUp=b] (\end)
						edge[CurveDown=c] (\end);}
% draw c-as-loop connections
\path[draw] (9) edge[Loop=c] (9)
	(8) 	edge[Loop=c] (8)
		edge[CurveUp=b] (9)
		edge[CurveDown=d] (9);						
% draw b-as-loop connections
\path[draw] (5) edge[Loop=b] (5)
	(4) 	edge[Loop=b] (4)
		edge[CurveUp=c] (5)
		edge[CurveDown=d] (5);
% draw end of graph
\node (End) [draw=none, fill=none] at (\TikzNumNodes*\TikzNodeDist + \TikzNodeDist/2 , 0){};
\draw[dotted, thick] (\TikzNumNodes)--(End);
\end{tikzpicture}
%\vspace{1ex}
\subcaption{A two-sided infinite Schreier graph in the orbit of \( \overline{0} \).\label{subfig:TwoSideSchreier}}
\end{minipage}
\normalsize
\caption{An example of a one-sided infinite Schreier graph and of a two-sided infinite Schreier graph, see also \cite[Figure~2 and~3]{GLN_SpectraSchreierAndSO} and \cite[Figure~2]{Voro_SchrGrGrigGr}.}
\end{figure}

\begin{rem}
By Definition~\ref{defi:SToepHoleFill} we distinguish normal and extended Toeplitz words. The latter have a remaining undetermined position after the hole-filling process and correspond to the leading words in a simple Toeplitz subshift (Example~\ref{exmpl:SpWBlocks}). Analogous notions exist for the Schreier graphs of \( G_{\GenSeq} \) as well. To see this, we shift our focus towards isomorphism classes of graphs.

Two rooted, directed graphs with labelled edges are called isomorphic if there exists a bijection that preserves the graph structure, the labels and maps the root to the root. Moreover, a graph is called locally finite if only finitely many edges end at each vertex. Following \cite{Voro_SchrGrGrigGr} we denote by \( \mathcal{MG}_{0} \) the set of isomorphism classes of rooted, directed, connected, locally finite, labelled graphs. By \cite[Lemma~2.1]{Voro_SchrGrGrigGr}, a topology on \( \mathcal{MG}_{0} \) is given by the base \( \mathscr{U}_{ \Gamma , V } \), where \( \Gamma \) runs over all finite graphs in \( \mathcal{MG}_{0} \), \( V \) runs over all subsets of the vertices of \( \Gamma \) and \( \mathscr{U}_{ \Gamma , V } \) denotes the set of those isomorphism classes in \( \mathcal{MG}_{0} \) whose graphs have \( \Gamma \) as a subgraph and contain no edges that connect a vertex in \( V \) to a vertex outside of \( \Gamma \). Now we consider the set of isomorphism classes of the orbital Schreier graphs associated to \( G_{\GenSeq} \). When we take the set's closure in \( \mathcal{MG}_{0} \), it turns out that the isolated points in the closure are precisely the isomorphism classes of the one-sided infinite graphs, see \cite[Theorem~1.1~(iii)]{Voro_SchrGrGrigGr}. The closure without these isolated points is in the following denoted by \( \SchreiSp{\GenSeq} \). Clearly \( \SchreiSp{\GenSeq} \) contains all isomorphism classes of two-sided infinite Schreier graphs. They correspond to regular Toeplitz words. But since we took the closure, \( \SchreiSp{\GenSeq} \) contains also the isomorphism class of graphs that consist of two copies of the one-sided infinite graph, connected in a way that is specified by either \( \pi_{b} \), \( \pi_{c} \) or \( \pi_{d} \). They correspond to extended Toeplitz words (which similarly consist of two copies of \( \PBlock{\infty} \), connected by \( b \), \( c \) or \( d \)).
\end{rem}

In Subsection~\ref{subsec:AlphaRepe} we discuss the notion of linear repetitivity, which roughly speaking means that the gap between two occurrences of a word grows at most linearly with the word length. In the same chapter we also encounter the Boshernitzan condition (Section~\ref{sec:BoshCond}), a weaker analogue of linear repetitivity. Both notion exist not only for subshifts, but also for actions of \( G_{\GenSeq} \) on metric spaces. To make this precise, we denote by \( \lvert g \rvert \) the minimal number of generators that is necessary to obtain an element \( g \in G_{\GenSeq} \).

\begin{defi}[\cite{NagnPerez_SchreierGr}]
\label{defi:LRuBGrAct}
Let \( ( Y, d ) \) be a metric space and let \( G_{\GenSeq} \) be a member of the family of Grigorchuk's groups, equipped with an action on \( ( Y , d ) \). We say that this action
\begin{tightenumerate}
\item{\emph{is linearly repetitive}\index{linearly repetitive!for a group action}\index{group!linear repetitive action}, if there exists a constant \( C \geq 0 \) such that for all \( r > 0 \) and all \( y_{1} ,  y_{2} \in Y \) there exists an element \( g \in G_{\GenSeq} \) with \( \lvert g \rvert \leq C \cdot r \) such that \( d( y_{1} , g( y_{2} ) ) < \frac{1}{r} \) holds.}
\item{\emph{satisfies the Boshernitzan condition}\index{Boshernitzan condition!for a group action}\index{group!Boshernitzan condition for action}, if there exists a constant \( C \geq 0 \) and an increasing real-valued sequence \( (r_{k}) \) with \( \lim_{k \to \infty} r_{k} = \infty \) such that for all \( y_{1} ,  y_{2} \in Y \) there exists an element \( g \in G_{\GenSeq} \) with \( \lvert g \rvert \leq C \cdot r_{k} \) such that \( d( y_{1} , g( y_{2} ) ) < \frac{1}{r_{k}} \) holds.}
\end{tightenumerate}
\end{defi}

As in the previous remark, let \( \SchreiSp{\GenSeq} \) denote the closure of the set of isomorphism classes of orbital Schreier graphs of \( G_{\GenSeq} \) without its isolated points. For the action of \( G_{\GenSeq} \) on \( \SchreiSp{\GenSeq} \) by a shift of the root, Nagnibeda and Perez characterised the above notions. Note the strong resemblance of their result for the action of groups in Grigorchuk's family, and our result for the word combinatorics of generalised Grigorchuk subshifts in Corollary~\ref{cor:BoshCondGenGrig}. 

\begin{prop}[\cite{NagnPerez_SchreierGr}]
\label{prop:BoshCondGrAction}
The action of \( G_{\GenSeq} \) on \( \SchreiSp{\GenSeq} \)
\begin{tightenumerate}
\item{is linearly repetitive if and only if there exists a constant \( M \in \NN \) such that, for all \( t \in \NN_{0}  \), we have \( \Set{ \GenSeq_{t} , \ldots , \GenSeq_{t+M-1} } = \Set{ \GenSeq_{i} : i \geq t} \).}
\item{satisfies the Boshernitzan condition if and only if there exists a constant \( M \in \NN \) such that, for all \( r \in \NN_{0} \), there exists some \( t \geq r \) such that \( \Set{ \GenSeq_{t} , \ldots , \GenSeq_{t+M-1} } = \Set{ \GenSeq_{i} : i \geq t } \) holds.}
\end{tightenumerate}
\end{prop}

\begin{rem}
The Boshernitzan condition for group actions is actually not restricted to the family of Grigorchuk's groups. It likewise applies to the more general case of \emph{spinal groups}\index{spinal group}\index{group!spinal}, which were introduced in \cite{BarthSun_SpinalGr}. Their precise definition is beyond the scope of this appendix, but roughly speaking, a spinal group is generated by the following automorphisms of a \( p \)-regular tree (where \( p \) is a prime number): the generator \( a \) permutes the vertices directly below the root. In addition, there is a set \( B \) of generators. The action of a generator \( b \in B \) below vertices of the form \( \Word{ 1^{k} }{ 0 } \) is given by a sequence \( ( \GenSeq_{k} ) \) of maps. Below all other vertices, \( b \) acts as the identity, see Figure~\ref{fig:ActSpinGr}. For more details, the interested reader is referred to \cite{BarthSun_SpinalGr}, \cite[Section~2]{BGS_BranchGr} or \cite{NagnPerez_SchreierGr}. Some examples of spinal groups are the Fabrykowski/Gupta group (\cite{FabrGupta_FabGupGr}, \cite{FabrGupta_FabGupGr2}) and the members of the family of Grigorchuk's groups.

\begin{figure}
\centering
\scriptsize
\begin{tikzpicture}
[every node/.style ={inner sep=1pt, minimum height=0pt, fill, circle}, every label/.style ={label distance=0.5ex, draw=none, fill=none, rectangle}, level distance = 35, sibling distance=20]
\node{}
	child{ node[label=below:\( \GenSeq_{1}( b ) \)]{} }
	child{ node[label=below:\( id \)]{} }
	child{ node[label=below:\( id \)]{} }
	child{ node[label=below:\( id \)]{} }
	child{ node{}
		child{ node[label=below:\( \GenSeq_{2}( b ) \)]{} }
		child{ node[label=below:\( id \)]{} }
		child{ node[label=below:\( id \)]{} }
		child{ node[label=below:\( id \)]{} }
		child{ node{}
			child{ node[label=below:\( \GenSeq_{3}( b ) \)]{} }
			child{ node[label=below:\( id \)]{} }
			child{ node[label=below:\( id \)]{} }
			child{ node[label=below:\( id \)]{} }
			child{ node[label=below:\( \ddots \)]{} } } };
\end{tikzpicture}
\normalsize
\caption{Action of a generator \( b\in B \) of a spinal group that acts on the five-regular tree.\label{fig:ActSpinGr}}
\end{figure}
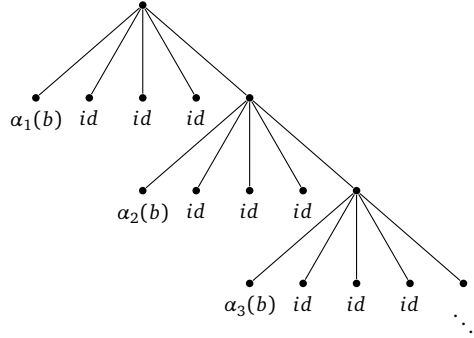

As for the family of Grigorchuk's groups, Schreier graphs also can be defined for spinal groups in general. As before, the level-(\( k+1 \)) Schreier graph consists of several copies of the level-\( k \) Schreier graph, connected in a way that is specified by \( \GenSeq_{k} \). But since we might have more than two copies, the resulting graph is not ``essentially linear''. It is therefore not clear how to associate a subshift to the graph.
\end{rem}

For the family of Grigorchuk's groups, the similarity between generalised Grigorchuk subshifts and Schreier graphs yields a relation between operators on these objects as well, see \cite{GLN_SpectraSchreierAndSO}, \cite{GLNS_LeadingSeq_Arxiv}. On the subshifts, these operators are precisely the Jacobi operators that are discussed in Chapter~\ref{chap:JOSimpToep}. On the Schreier graphs, so-called Laplacians can be defined: let \( G_{\GenSeq} \) be a member of the family of Grigorchuk's groups, let \( S \in \SchreiSp{\GenSeq} \) be an isomorphism class and let \( s \) be a Schreier graph in this class. Then every edge of \( s \) is labelled by one of the generators \( a \), \( \hat{b} \), \( \hat{c} \) or \( \hat{d} \). For the set of vertices of \( s \) and the set of edges of \( s \) we write \( \Vertices{ s } \) and \( \Edges{ s } \), respectively. The Laplacian on \( s \) is defined as follows:

\begin{defi}
Let \( \NDig \colon \Edges{ s } \to \RR \) be a function that only depends on the label of an edge. The Laplacian \( \Lap{s}{\NDig} \colon \ell^{2}( \Vertices{ s } ) \to \ell^{2}( \Vertices{ s } ) \) on the graph \( s \) is defined by
\[ ( \Lap{s}{\NDig} \varphi )( u ) \DefAs \sum_{e \, : \, u \stackrel{e}{\sim} v} \NDig( e ) \varphi( v ) \, , \]
where the sum is taken over all edges \( e \) such that \( u \) is one endpoint of \( e \). The other endpoint is denoted by \( v \) (and may be equal to \( u \)). 
\end{defi}

Since isomorphisms of labelled graphs preserve the labels, all representatives of an isomorphism class define the same Laplacian. Therefore we write \( \Lap{S}{\NDig} \) instead of \( \Lap{s}{\NDig} \). Now recall that \( \SchreiSp{(\pi_{d}, \pi_{c}, \pi_{b} , \ldots )} \) denotes the closure of the set of isomorphism classes of orbital Schreier graphs of Grigorchuk's group \( G_{\GenSeq} = G_{(\pi_{d}, \pi_{c}, \pi_{b} , \ldots )} \), without isolated points. For Laplacians on these classes the following relation was established in \cite{GLN_SpectraSchreierAndSO}:

\begin{prop}[{\cite[Proposition~4.1]{GLN_SpectraSchreierAndSO}}]
For every \( S \in  \SchreiSp{(\pi_{d}, \pi_{c}, \pi_{b} , \ldots )} \) and every real-valued function \( \NDig \) that is defined on the set of edge labels, there exists an element \( \InfWord \in \Subshift \) in the Grigorchuk subshift such that the Laplacian \( \Lap{S}{\NDig} \) is unitarily equivalent to a locally constant Jacobi operator \( \Jac_{\InfWord} \) on \( \ell^{2}( \ZZ ) \).
\end{prop}

In fact, this result extends to the whole family of Grigorchuk's groups and even to spinal groups on the binary tree (\cite[Theorem~7.1]{GLNS_LeadingSeq_Arxiv}). This is one of our main motivations for studying Jacobi operators on simple Toeplitz subshifts: in Chapter~\ref{chap:UnifCocycCantor} we establish Cantor spectrum of Lebesgue measure zero for aperiodic Jacobi operators (Theorem~\ref{thm:LCondCantorSpec}). We can apply this result to Laplacians on self-similar groups (see also \cite[Corollary~7.2]{GLNS_LeadingSeq_Arxiv}): if a spinal group on the binary tree satisfies a certain aperiodicity condition (namely that the edge weights associated to the eventual alphabet are not all equal), then the spectrum of the Laplacian on the associated orbital Schreier graphs is a Cantor set of Lebesgue measure zero.

%%%%%%%%%%%%%%%%%%%%%%%%
%%%%%%%%%%%%%%%%%%%%%%%%
%%%%%%%%%%%%%%%%%%%%%%%%
\backmatter % does not change the page numbering. prohibits sectional division numbering. float captions, equations, etc., are numbered continuously

\newcommand{\etalchar}[1]{$^{#1}$}

\cleardoublepage

\printindex

\end{document}